\documentclass[11pt]{article} 
\usepackage{babel}
\usepackage{amsmath}
\usepackage{lmodern}
\usepackage[T1]{fontenc}
\usepackage[babel=true]{microtype}
\usepackage{amsxtra}%
\usepackage{amsfonts}%
\usepackage{amssymb}%
\usepackage{amsthm}
\usepackage{graphicx}
\usepackage{epstopdf} 
\usepackage{color,hyperref}
\hypersetup{colorlinks,breaklinks,
             linkcolor=blue,urlcolor=blue,
           anchorcolor=blue,citecolor=blue}
       
\usepackage{subcaption}
\usepackage{appendix}
\usepackage[font=footnotesize]{caption}
\usepackage{calc}

\usepackage[sort]{natbib}
\usepackage[margin=0.75in]{geometry}

\newcommand{\eps}{\varepsilon}
\newcommand{\Z}{\mathbb{Z}}
\newcommand{\N}{\mathbb{N}}
\newcommand{\R}{\mathbb{R}}

\newcommand{\C}{\mathbb{C}}

\newcommand{\F}{\mathbb{F}}

\renewcommand{\phi}{\varphi}

\newcommand{\sign}{\text{sign}}
\newcommand{\mcl}{\mathcal{L}}

\newcommand{\El}{\mathcal{L}}

\usepackage{enumitem}
\setitemize{noitemsep,topsep=0pt,parsep=2pt,partopsep=0pt}
\renewcommand{\Re}{\mathrm{Re} \,}
\renewcommand{\Im}{\mathrm{Im} \,}

\def\XXint#1#2#3{{\setbox0=\hbox{$#1{#2#3}{\int}$ }
		\vcenter{\hbox{$#2#3$ }}\kern-.6\wd0}}

\newtheorem*{thm*}{Theorem}

\newtheorem{prop}{Proposition}
\newtheorem{lemma}[prop]{Lemma}
\newtheorem{corollary}[prop]{Corollary}

\newtheorem{thmlocal}[prop]{Theorem}

\newtheorem{hyp}{Hypothesis}

\newtheorem{remark}[prop]{Remark}

\numberwithin{equation}{section}
\numberwithin{prop}{section}

\newcommand{\Deff}{D_\mathrm{eff}}

\newcommand{\clin}{{c_\mathrm{lin}}}

\newcommand{\etalin}{\eta_\mathrm{lin}}
\newcommand{\uwt}{\mathbf{u}_\mathrm{wt}}
\newcommand{\ufr}{\mathbf{u}_\mathrm{fr}}
\newcommand{\lwt}{\mathcal{L}_\mathrm{wt}}
\newcommand{\afr}{\mathcal{A}_\mathrm{fr}}
\newcommand{\lfr}{\mathcal{L}_\mathrm{fr}}
\newcommand{\U}{\mathbf{U}}

\renewcommand{\u}{\mathbf{u}}
\newcommand{\g}{\mathbf{g}}
\newcommand{\G}{\mathbf{G}}
\newcommand{\nuwt}{\nu_\mathrm{wt}}
\newcommand{\lambdawt}{\lambda_\mathrm{wt}}
\newcommand{\Pwt}{P_\mathrm{wt}}
\newcommand{\Twt}{T^\mathrm{wt}}
\newcommand{\Pfr}{P_\mathrm{fr}}
\newcommand{\Tfr}{T^\mathrm{fr}}
\newcommand{\diag}{\mathrm{diag}}
\newcommand{\nufr}{\nu_\mathrm{fr}}
\newcommand{\e}{\mathbf{e}}
\newcommand{\w}{\mathbf{w}}
\newcommand{\q}{\mathbf{q}}
\renewcommand{\v}{\mathbf{v}}

\newcommand{\Gheat}{G^\mathrm{heat}}
\newcommand{\Godd}{G^\mathrm{odd}}

\newcommand{\alphaout}{\alpha_\mathrm{out}}
\newcommand{\uint}{\mathbf{u}_\mathrm{int}}
\newcommand{\alphaint}{\alpha_\mathrm{int}}

\newcommand{\re}{\mathrm{e}}
\newcommand{\ri}{\mathrm{i}}
\newcommand{\de}{\mathrm{d}}
\newcommand{\NT}{\widetilde{\mathcal{N}}}

\newcommand{\vt}{\widetilde{\mathbf{v}}}
\newcommand{\vf}{\mathring{\mathbf{v}}}
\newcommand{\uf}{\mathring{\mathbf{u}}}
\newcommand{\ut}{\widetilde{\mathbf{u}}}
\newcommand{\z}{\mathbf{z}}

\newcommand{\kwt}{k_\mathrm{wt}}

\newcommand\extrafootertext[1]{%
    \bgroup
    \renewcommand\thefootnote{\fnsymbol{footnote}}%
    \renewcommand\thempfootnote{\fnsymbol{mpfootnote}}%
    \footnotetext[0]{#1}%
    \egroup
}

\begin{document}
	\begin{center}
        {\fontsize{15}{15}\fontseries{b}\selectfont{Stability of coherent pattern formation through invasion in the FitzHugh-Nagumo system}}\\[0.2in] 
		Montie Avery$^1$, Paul Carter$^2$, Bj\"orn de Rijk$^3$, and Arnd Scheel$^4$ \\[0.1in]
		\textit{\footnotesize 
$^1$Department of Mathematics and Statistics, Boston University, 665 Commonwealth Ave, Boston, MA 02215, USA \\
$^2$Department of Mathematics, University of California, Irvine, 340 Rowland Hall, Irvine, CA, USA \\
$^3$Department of Mathematics, Karlsruhe Institute of Technology, Englerstra{\ss}e 2, 76131 Karlsruhe, Germany \\
$^4$School of Mathematics, University of Minnesota,  206 Church St. S.E., Minneapolis, MN 55455, USA
}
\end{center}

\begin{abstract}
We establish sharp nonlinear stability results for fronts that describe the creation of a periodic pattern through the invasion of an unstable state. The fronts we consider are {critical}, in the sense that they are expected to mediate pattern selection from compactly supported or steep initial data. 
We focus on pulled fronts, that is, on fronts whose propagation speed is determined by the linearization  about the unstable state in the leading edge, only. 
We present our analysis in the specific setting of the FitzHugh-Nagumo system, where pattern-forming uniformly translating fronts have recently been constructed rigorously~\cite{CarterScheel}, but our methods can be used to establish nonlinear stability of pulled pattern-forming fronts in general reaction-diffusion systems. 
This is the first stability result of critical pattern-selecting fronts and provides a rigorous foundation for heuristic, universal wave number selection laws in growth processes based on a marginal stability conjecture.
The main technical challenge is to describe the interaction between two separate modes of marginal stability, one associated with the spreading process in the leading edge, and one associated with the pattern in the wake. We develop tools based on far-field/core decompositions to characterize, and eventually control, the interaction between these two different types of diffusive modes. Linear decay rates are insufficient to close a nonlinear stability argument and we therefore need a sharper description of the relaxation in the wake of the front using a phase modulation ansatz. We control regularity in the resulting quasilinear equation for the modulated perturbation using nonlinear damping estimates.  
\end{abstract}

\extrafootertext{The authors gratefully acknowledge support from the US National Science Foundation, through NSF DMS-2205663 (A.S.), NSF-DMS-2202714 (M.A.), and NSF-DMS-2238127 (P.C.), and from the Deutsche Forschungsgemeinschaft (DFG, German Research Foundation) - Project-ID 491897824 (B.dR). The authors would also like to thank the Collaborative Research Center ``Wave Phenomena'' for hosting M.A. and P.C. while working on this project in August 2022, as well as the Institute for Computational and Experimental Research in Mathematics for hosting all authors in January 2023 through the Collaborate@ICERM program.}

\section{Introduction}

Invasion into unstable states plays an important role in the development of complex coherent structures in many physical systems. Unstable states can be observed as a transient for instance after a change in system parameters, or after the introduction of a novel external agent to which the system is unstable. In spatially extended systems, one expects localized fluctuations of the now unstable background state to grow and spread, leaving a new stable state in the wake of a propagating \emph{invasion front}. A fundamental problem then is to predict the speed of propagation of the invasion process as well as the new state selected in its wake. 

Predictions for invasion speeds in the mathematical literature have historically been restricted to systems with comparison principles. Since comparison principles are incompatible with complex pattern formation, invasion fronts in these systems typically select a new spatially homogeneous equilibrium in their wake. On the other hand, many interesting invasion phenomena have been observed in experiments, simulations, and formal analyses in {pattern-forming} models describing a large variety of physical systems; see~\cite{vanSaarloos} for a comprehensive review. Pattern-forming systems usually admit families of periodic patterns, parameterized by the wave number. An invasion process spreading into an unstable state typically selects one distinguished pattern and wave number out of this family; see~\cite{deelanger}. This wave number selection mechanism has been observed in many experiments across the sciences~\cite{vanSaarloos}, and has promising applications to nanoscale manufacturing technologies~\cite{NatureMaterials, Bradley}. However, beyond the heuristics~\cite{deelanger} and matched asymptotics~\cite{PhysRevE.61.R6063}, there do not appear to be mathematical results or techniques that describe these phenomena. 

In a first approach, one would hope to make such predictions rigorous by finding a unique speed and wave number for which there is a stable traveling front solution connecting the unstable state to a periodic pattern in the wake. Existence of pattern-forming fronts has been rigorously established near the onset of a Turing instability~\cite{ColletEckmann, EckmannWayne1, HaragusSchneider, Hexagons} and in phase separation problems~\cite{ScheelCoarsening1, ScheelCoarsening2}. However, such fronts typically exist and are stable for a continuum of speeds and associated wave numbers, so that this simple approach does not predict which of these speed-wave number pairs are selected by localized initial data. 

The \emph{marginal stability conjecture}~\cite{brevdo, bers1983handbook, deelanger, vanSaarloosMarginalstability, vanSaarloos, colleteckmannbook, AveryScheelSelection} asserts that speeds and associated wave numbers are determined by the distinguished front solutions which are \emph{marginally spectrally stable} in an appropriate sense, that is, their spectrum, in an appropriate space, touches the imaginary axis but is otherwise stable. There are two distinct scenarios for marginal spectral stability: the marginal stability may arise from marginal pointwise stability of the unstable state in the leading edge in a distinguished moving frame, or from marginally stable point spectrum of the invasion front itself. The former case is referred to as \emph{pulled}, or linearly determined, propagation, while the latter case is referred to as \emph{pushed}, or nonlinearly determined, propagation. Because the associated spectrum is marginally stable, or \emph{critical}, selected fronts are sometimes referred to as critical fronts. By contrast, faster-traveling \emph{supercritical} fronts arise for a continuum of speeds and associated wave numbers, and are often stable against restricted classes of perturbations. Stability of these supercritical fronts has been shown rigorously in some pattern-forming systems~\cite{Schneider1, Schneider2}, but they do not appear to be relevant to selection and propagation from compactly supported or steep initial data. 

The marginal stability conjecture was recently established for systems of parabolic equations~\cite{AveryScheelSelection, AverySelectionRD} based on a novel, conceptual approach that does not rely on comparison principles, thus providing a promising avenue towards the analysis of pattern-forming systems that inherently do not possess ordering properties. The analysis there does however assume that the marginally stable front selects a spatially constant, exponentially attracting state in its wake, rather than a periodic pattern.

The results in~\cite{AveryScheelSelection, AverySelectionRD} demonstrate that the key ingredient to establishing  the marginal stability conjecture is a sharp theory for nonlinear stability of invasion fronts against perturbations that do not alter the decay in the leading edge of the front. These sharp stability results predict the characteristic $\log$-shift of the front position when starting from compactly supported or steep initial data and allow to close the matched asymptotics arguments in~\cite{PhysRevE.61.R6063}. Our main result, informally stated below as Theorem~\ref{t: main}, provides precisely these sharp stability estimates in the case of a uniformly translating, pulled front that creates a pattern in its wake.  

To be concrete, we present our analysis in the setting of the FitzHugh-Nagumo system,
\begin{align}
\begin{split}
    u_t &= u_{xx} + u(u + a)(1-u-a)- w, \\
    w_t &= \eps (u - \gamma w), 
\end{split}    
\label{e: FHN not shifted}
\end{align}
which we write abstractly as a system for $\u = (u, w)^\top $,
\begin{align}
    \u_t = D \u_{xx} + F(\u; a, \gamma, \eps), \quad D = \begin{pmatrix}
    1 & 0 \\ 0 & 0 
    \end{pmatrix},\qquad 
    F(\u; a, \gamma, \eps) = \begin{pmatrix}
    u(u + a)(1-u-a) - w \\
    \eps(u - \gamma w) 
    \end{pmatrix} .
    \label{e: fhn}
\end{align}
The FitzHugh-Nagumo system~\eqref{e: fhn} models excitable and oscillatory media far from equilibrium and is ubiquitous across the sciences as both a phenomenological model and as a more rigorous simplification of more complex descriptions. It arose in this latter fashion first as a simplification of the Hodgkin-Huxley equations for signal propagation in nerve fibers and, together with variations and adaptions, has since been used to model, for instance, the onset of turbulence in fluids~\cite{Barkley}, carbon monoxide oxidation on platinum surfaces~\cite{Oxidation1, Oxidation2}, and cardiac arrhythmias~\cite{Cardiac}.  
Mathematically,~\eqref{e: fhn} is a scalar reaction-diffusion equation coupled to a linear ODE and, thus, one of the simplest models which could, and does exhibit spatio-temporal pattern formation.
\paragraph{Existence of pulled fronts.} 
We consider~\eqref{e: FHN not shifted} in the {oscillatory regime}, $ 0 < a < \frac{1}{2}$, and $0 < \gamma < 4$, with $0 < \eps \ll 1$. In this regime, $\u =(0,0)$ is the unique spatially constant equilibrium, but is unstable with perturbations growing in any translationally invariant norm. Linearizing at this unstable equilibrium, one finds a linear spreading speed $c_\mathrm{lin}(a,\gamma,\eps)=2\sqrt{a(1-a)}+\mathrm{O}(\eps)$ together with a characteristic exponential decay rate $-\eta_\mathrm{lin}(a,
\gamma,\eps)$, at which the linear equation is marginally stable; see~\cite[Lemma 2.1]{CarterScheel} or Lemma~\ref{l: linear spreading speed} below for details, as well as~\cite{HolzerScheelPG} for background on linear spreading speeds.
In addition to this spatially constant solution,~\eqref{e: FHN not shifted} admits stable time-periodic, $x$-independent solutions, commonly referred to as relaxation oscillations, and spatially periodic modulations of these oscillations, which are traveling waves $\u(x,t) = \u_\mathrm{wt}(x - ct)$ satisfying $\u_\mathrm{wt}(\xi)=\u_\mathrm{wt}(\xi+L)$ for some $L > 0$. These periodic traveling waves, or wave trains, exist for a range of periods $L_*<L<\infty$ with speeds $c=c(L)$.
Traveling front solutions $\u(x,t) = \u(x-ct)$ that connect $\u=(0,0)$ to a wave train $\u_\mathrm{wt}$ have been constructed in~\cite[Theorem 1.2]{CarterScheel} using dynamical systems techniques, in particular geometric singular perturbation theory.

\begin{figure}
	\centering
	\begin{subfigure}{0.24\textwidth}
		\includegraphics[width=1\textwidth]{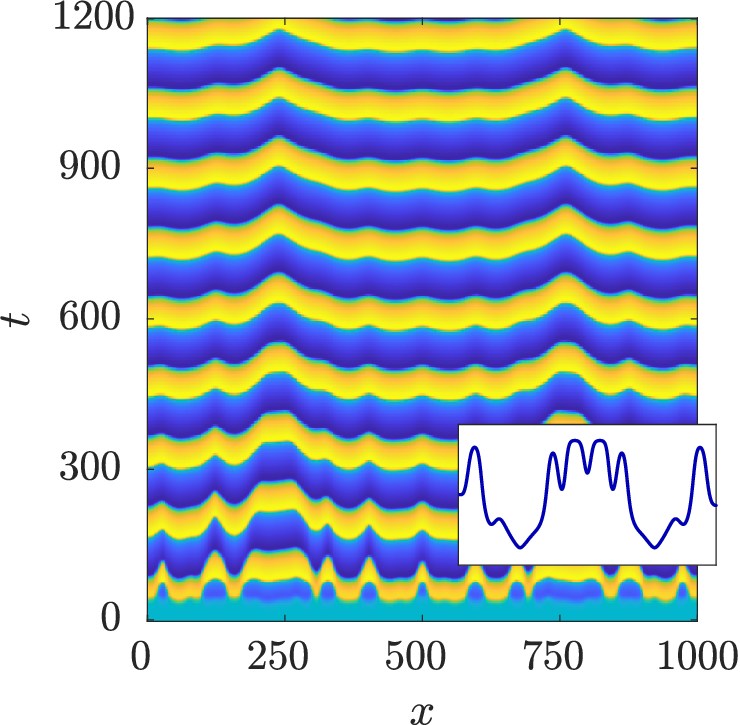}
	\end{subfigure}
	\hfill
	\begin{subfigure}{0.24\textwidth}
		\includegraphics[width=1\textwidth]{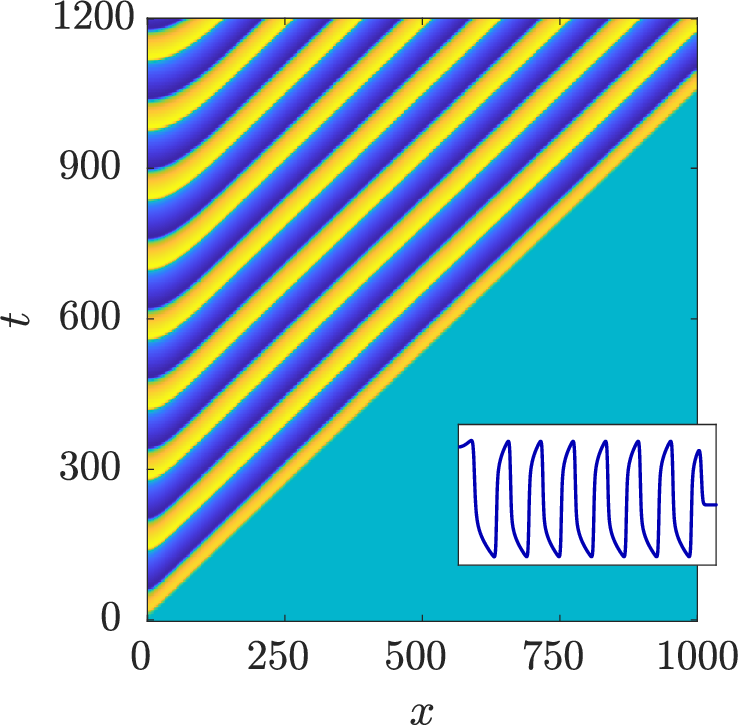}
	\end{subfigure}
	\hfill
	\begin{subfigure}{0.47\textwidth}
	\includegraphics[width=1\textwidth]{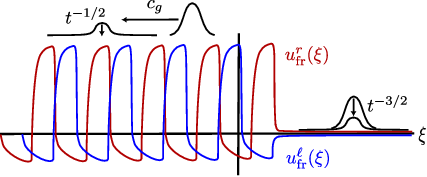}
	\end{subfigure}
	\caption{Left and center: Space-time plot of $u$ from direct simulations of~\eqref{e: fhn} with initial condition $u=w=0$ perturbed by small white noise (left) and a small perturbation on the left end (center); $a=0.4,\gamma=0.1,\eps=0.005$. Insets show the $u$-component at time $t=1000$.  Right: schematic of first component $u_\mathrm{fr}^r$ (red) and $u_\mathrm{fr}^\ell$ (blue) of invasion fronts $\ufr^{r/\ell} (\xi)$ in the co-moving frame, together with the dynamics of perturbations (black) in the leading edge and in the wake. Perturbations in the leading edge decay with improved rate $t^{-3/2}$, while perturbations in the wake are transported outward by the group velocity and decay with diffusive rate $t^{-1/2}$. Note also that the patterns in the wake are large amplitude, highly nonlinear structures which do not resemble a pure cosine wave.
 }
	\label{fig: fronts}
\end{figure}

\begin{thmlocal}[Existence of pulled pattern-forming fronts {\cite[Theorem 1.2]{CarterScheel}}]\label{t: existence}
Fix $(3-\sqrt{6})/6 < a < \frac{1}{2}$ and $0 < \gamma < 4$. There exists $\eps_0 > 0$ such for all $0<\eps<\eps_0$~\eqref{e: fhn}  admits two pulled pattern-forming front solutions $\u(x,t) = \ufr^\ell (x-\clin t)$ and $\u(x,t) = \ufr^r(x-\clin t)$, which satisfy
\begin{align}\label{eq:asymp_est_front}
\ufr^{\ell/r} (\xi) = \left[ \u_0^{\ell/r} \xi + \u^{\ell/r}_1 + b^{\ell/r} \u^{\ell/r}_0 \right] \re^{-\etalin \xi} + \mathrm{O}(\re^{-(\etalin + \eta) \xi}), \quad \xi \to \infty ,
\end{align}
and
\begin{align}\label{eq:asymp_est_wt}
    \ufr^{\ell/r}(\xi) = \uwt^{\ell/r} (\xi) + \mathrm{O}(\re^{\eta \xi}), \quad \xi \to -\infty, 
\end{align}
for some $b^{\ell/r} \in \R$, $\eta > 0$, $\u_0^{\ell/r} \in \R^2 \setminus \{0\}$, and $\u_1^{\ell/r} \in \R^2$, where $\uwt^{\ell/r}$ is a wave train solution with selected period $L > 0$ and wavenumber $\frac{2 \pi}{L}$, $c_\mathrm{lin}$ is the linear spreading speed, and $\eta_\mathrm{lin}$ is the characteristic exponential decay rate.  Both fronts select the same wave trains, $\uwt^{\ell}(\xi) = \uwt^{r}(\xi + \xi_0)$ for some $\xi_0\in\R$ but have opposite monotonicity in the leading edge, $\u_j^{\ell}=-\u_j^{r}$, $j=0,1$.
\end{thmlocal}
See Figure~\ref{fig: fronts} for a schematic of the fronts. We shall discuss below results that establish that both fronts $\ufr^{r}$ and $\ufr^{\ell}$ are marginally stable due to marginal stability in the leading edge, that is, they are pulled fronts.  
We emphasize that the wave trains generated in the wake are ``far-from-equilibrium patterns'', in the sense that they possess large amplitude and are highly nonlinear, in contrast to the low-amplitude, weakly nonlinear patterns selected by fronts near a Turing instability~\cite{ColletEckmann, EckmannWayne1, HaragusSchneider}.

\paragraph{Nonlinear stability of pulled fronts --- main result.} To study stability of the front solutions established in Theorem~\ref{t: existence}, we consider the frame $\xi = x - \clin t$ moving with the speed $\clin$, in which~\eqref{e: fhn} reads
\begin{align}
    \u_t = D \u_{\xi \xi} + \clin \partial_\xi \u + F(\u; a, \gamma, \eps), \label{e: fhn comoving}
\end{align}
so that the front is a stationary solution to~\eqref{e: fhn comoving}. Our main result may be informally stated as follows.

\begin{thmlocal}\label{t: main}
Let $\ufr$ be a uniformly translating pattern-forming front solution to~\eqref{e: fhn}, as established in Theorem~\ref{t: existence}. Assume that $\ufr$ satisfies the following marginal spectral stability assumptions:
\begin{itemize}
    \item The wave train generated by the front is diffusively spectrally stable, that is, its spectrum touches the origin in a single parabolic tangency and is otherwise stable;
    \item In the frame moving with the front speed $\clin$, the \emph{group velocity} of the wave train points to the left, away from the front interface;
    \item The linearization of~\eqref{e: fhn comoving} about $\ufr$ possesses no unstable point spectrum, and no embedded eigenvalue at $\lambda = 0$. 
\end{itemize}
Then, $\ufr$ is nonlinearly stable as a stationary solution to~\eqref{e: fhn comoving} against sufficiently localized perturbations, and perturbations decay pointwise in time with sharp diffusive rate $t^{-1/2}$, and with enhanced rate $t^{-3/2}$ in the leading edge of the front.
\end{thmlocal}

Before we prepare a more precise statement of this result in the next section (see Theorem~\ref{t: main detailed}), where we will also point out the conceptual nature of our methods with potential broader applicability, we mention some subtleties. 

The assumptions in Theorem~\ref{t: main} are made precise  in Hypotheses~\ref{hyp: wt spectral} through~\ref{hyp: point spectrum} below  and are established for~\eqref{e: fhn} in natural parameter ranges in the companion papers~\cite{SpectralFronts, spectral}; {\color{blue} see Remark \ref{rmk: spectral assumptions}}. The assumptions guarantee that the fronts studied here are pulled fronts, since they propagate at the linear spreading speed and do not possess marginally stable or unstable point spectrum. The diffusive spectral stability that we require for the wave trains in the wake is generic for wave trains in reaction-diffusion systems and a standard assumption for their nonlinear stability analysis~\cite{SchneiderAbstract, JONZ, SSSU}. The group velocity of the wave train describes the direction of transport of small perturbations; see~\cite{DSSS} for further background. We assume that the group velocity points away from the front interface so that the front naturally acts as a source of patterns in the sense of~\cite{SandstedeScheelDefects}. Spreading at the linear spreading speed guarantees diffusive decay in exponentially weighted spaces in the leading edge; see Figure~\ref{fig: fronts}.

We reiterate that the fronts studied here are uniformly translating, i.e.~equilibria in a comoving frame. In particular, they are not time-periodic in a comoving frame (or \emph{modulated}) as is the case for many other pattern-forming fronts.  Such modulated fronts have been constructed in the wake of a Turing instability~\cite{ColletEckmann} but also for large amplitudes in spinodal decomposition~\cite{ScheelCoarsening2}. We do not treat such modulated fronts in this paper, but we expect that our methods and proofs can be adapted to this case; see also~\cite{SandstedeScheelModulatedTW, BeckSandstedeZumbrun} for stability results of modulated waves in the presence of essential spectrum. In this context, it is worth noting the stability result in the complex Ginzburg-Landau equation in~\cite{AveryScheelGL}, which does give sharp stability estimates towards a front that describes pattern formation near a Turing instability, albeit in an amplitude equation setting that averages oscillations so that the front is in fact uniformly translating, and which allows for coordinate choices that are not available in more general situations, in particular the present one. 

The $u$-component of both fronts in Theorem~\ref{t: existence} are monotone in the leading edge: for left fronts, $u_\mathrm{fr}^\ell$ is increasing and for right fronts $u_\mathrm{fr}^r$ is  decreasing. Both fronts select the same state in the wake (up to a phase shift of the wave train), a curious phenomenon not observed in order-preserving systems. Initial conditions will converge to left or to right fronts depending on their behavior in the leading edge. One indeed observes this selected wave number in simulations when the unstable state is perturbed by ``shot'' noise, locally in space. In contrast, spatially distributed white noise perturbations lead to long-wavelength modulations of spatially homogeneous oscillations with a slow coarsening process toward synchrony. In this sense, the fronts studied here mediate a rapid frequency synchronization in an otherwise disorganized system; see Figure~\ref{fig: fronts}. This work can be seen as a first instance where this rapid frequency synchronization through growth is mathematically corroborated. 

Finally, we point out that the 
asymptotic estimate~\eqref{eq:asymp_est_front}, while not explicitly stated in~\cite{CarterScheel}, follows from an extension of the analysis of~\cite[Section~4]{AronsonWeinberger}; for a detailed proof, see for instance~\cite[Appendix A]{KellerSegelAHS} in the context of the Fisher-KPP equation. The  estimate~\eqref{eq:asymp_est_wt} follows directly from the existence construction in~\cite{CarterScheel} which finds the pattern-forming front  in the unstable manifold of a hyperbolic periodic orbit representing the wave train in the corresponding traveling wave formulation.

\section{Overview, challenges, setup, and main result}

A standard approach to studying nonlinear stability of a given coherent structure $\u_*$, say in a reaction-diffusion system $\u_t = D \u_{\xi \xi} + c \u_\xi + f(\u)$, is to derive an equation for the perturbation $\w(\xi, t) = \u(\xi, t) - \u_*(\xi)$, of the form
\begin{align*}
    \w_t = \mathcal{A} \w + \mathcal{N} (\w), 
\end{align*}
where $\mathcal{A}$ is the linearization about $\u_*$ and $\mathcal{N}$ is the resulting quadratic nonlinear remainder. One then studies the behavior of $\w$ by analyzing the associated \emph{variation of constants formula} 
\begin{align*}
    \w(t) = \re^{\mathcal{A} t} \w_0 + \int_0^t \re^{\mathcal{A} (t-s)} \mathcal{N} (\w(s)) \de s, 
\end{align*}
where $\w_0(\xi) = \u(\xi, 0) - \u_*(\xi)$ is the initial perturbation, and $\re^{\mathcal{A} t}$ is a strongly continuous semigroup generated by $\mathcal{A}$. In general, one needs to prove that $\mathcal{A}$ indeed generates a suitable semigroup, although this is automatic by standard results if, for instance, $\mathcal{A}$ is elliptic. One then hopes to establish decay of $\w(t)$ by combining decay estimates on $\re^{\mathcal{A} t}$ with a contraction mapping or iterative argument on the variation of constants formula. Hence, the key first step is to obtain sharp linear decay estimates on the associated semigroup $\re^{\mathcal{A} t}$. In Sections~\ref{s: intro linear} and~\ref{s: nonlinear diffusive stability}, we explain challenges to obtaining suitable linear and nonlinear estimates to close a nonlinear iteration argument in our present context, before setting up and precisely formulating our main result in Sections~\ref{s: function spaces} through~\ref{s: main result}. 

\subsection{Linear diffusive stability}\label{s: intro linear}

\paragraph{Classical semigroup estimates.} If the spectrum of the linearization $\mathcal{A}$ is strictly contained in the left half-plane and a spectral mapping theorem  holds, solutions to the linearized equation decay exponentially in time yielding nonlinear stability by classical arguments. If the spectrum is stable except for some discrete eigenvalues at the origin, typically related to translation invariance or other symmetries of the original equation, then one can define spectral projections which separate these neutral modes and recover nonlinear stability of the family of underlying traveling waves; see~\cite{KapitulaPromislow, Henry} and references therein. In some cases when the essential spectrum is unstable or marginally stable, exponential weights can be used to push the essential spectrum into the left half-plane and return to a setting where classical arguments give stability~\cite{Sattinger}. 

\paragraph{Exponential weights.} In our case, the essential spectrum of the linearization about $\ufr$ is unstable in $L^2(\R, \C^2)$ due to the instability of the background state $\u \equiv 0$ to which $\ufr$ converges at $\xi = +\infty$. Since the front is \emph{critical}, traveling with the linear spreading speed, the essential spectrum cannot be fully stabilized with an exponential weight. The optimal choice of weight renders the spectrum marginally stable, touching the imaginary axis at the origin but otherwise contained in the left half-plane. To make this precise, we let
\begin{align*}
    \afr = D \partial_\xi^2 + \clin \partial_\xi + F'(\ufr) 
\end{align*}
denote the linearization of~\eqref{e: fhn comoving} about the invasion front $\ufr$, and define a smooth, positive, monotonically increasing weight function $\omega$ satisfying
\begin{align*}
    \omega(\xi; a, \gamma, \eps) = \begin{cases}
    1, &\xi \leq 0, \\
    \re^{\etalin(a, \gamma, \eps) \xi}, &\xi \geq 1.  
    \end{cases}
\end{align*}
We will from now on suppress the dependence of $\omega$ and $\etalin$ on $a$, $\gamma$ and $\eps$. We will restrict to perturbations $\w$ of $\ufr$ for which $\omega \w \in L^2(\R)$. The spectral problem for perturbations of this type is equivalent to the spectral problem for the conjugate operator 
\begin{align} \label{e: lfr def} 
\lfr = \omega \afr \omega^{-1},
\end{align} 
acting on $L^2(\R)$. Here and everywhere after, $\omega^{-1}$ denotes the function $\xi \mapsto \frac{1}{\omega(\xi)}$, not the inverse of the function $\xi \mapsto \omega(\xi)$. 

\begin{figure}
	\centering
	\includegraphics[width = 0.8\textwidth]{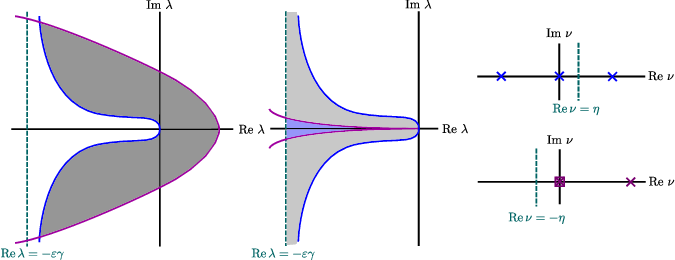}
	\caption{Spectrum of the unweighted linearization $\afr$ (left) and the weighted linearization $\lfr$ (middle). The purple and blue curves denote the essential spectrum of the limiting operators at $\pm\infty$, respectively. These curves are the Fredholm borders of $\afr/\lfr$: in the shaded regions, the operator is Fredholm but with nonzero index, with the index changing across the Fredholm borders. The Fredholm index is 1 in the dark grey region, -1 in the light grey region, and would be determined by the cubic term in the expansion of $\lambda_+(\nu)$ in the light blue region. Top right: spatial Floquet exponents of $\lwt$. Bottom right: spatial eigenvalues of $\mcl_+$; the box indicates a double spatial eigenvalue at the origin. See Section~\ref{s: spectral stability} for details about spatial Floquet exponents and eigenvalues.} 
	\label{fig: spectrum}
\end{figure}

\paragraph{Critical diffusive modes.} The essential spectrum of $\lfr$ touches the imaginary axis at the origin in two different ways: one curve of essential spectrum is associated to the dynamics at $+\infty$, and has the expansion
\begin{align}
    \lambda_+(\ri k) = -\Deff^+ k^2 + \mathrm{O}(k^3), \label{e: right curve}
\end{align}
while the other is associated to dynamics near the wave train at $-\infty$, and has the expansion
\begin{align}
    \lambdawt(\ri k) = -\ri c_g k - \Deff^\mathrm{wt} k^2 + \mathrm{O}(k^3), \label{e: left curve}
\end{align}
where $\Deff^+, \Deff^\mathrm{wt} > 0$, and $c_g < 0$. The expansion of the first curve naturally arises from marginal pointwise stability in the leading edge, which characterizes the linear spreading speed. The second expansion follows from standard diffusive spectral stability of the wave train in the wake with negative group velocity $c_g$ in a frame moving with the linear spreading speed; see Figure~\ref{fig: spectrum}, middle panel, for a depiction of these critical curves. These two curves, though both marginally stable, are quite different in that the first has a branch point at the origin while the second does not. We refer to the dynamics associated to the first curve as a \emph{branched diffusive mode}, and those associated to the second as an \emph{outgoing diffusive mode}. The fact that $c_g < 0$ means that perturbations to the wave train in the wake are, at least on the linear level, transported to the left, away from the front interface~\cite{DSSS}, which is why we refer to these modes as \emph{outgoing}. 

Branched diffusive modes occur naturally in the stability of pulled invasion fronts~\cite{AveryScheelSelection, AveryScheel, AverySelectionRD, FayeHolzer, FayeHolzerLV}, and also of stationary periodic patterns~\cite{SchneiderSH, SchneiderAbstract}, degenerate viscous shock waves~\cite{Howard1}, layer solutions in phase separation models~\cite{Howard2}, and contact defects~\cite{SandstedeScheelDefects, SandstedeScheelBlowup}. Outgoing diffusive modes arise naturally in the study of the stability of undercompressive shock waves~\cite{ZumbrunHoward}, traveling periodic waves~\cite{DSSS, JONZ, JNRZ_13_1}, and source defects~\cite{SandstedeScheelDefects, BeckNguyenSandstedeZumbrun}. 

The curve~\eqref{e: right curve} of essential spectrum associated to the branched diffusive mode cannot be stabilized with an exponential weight, since a stable and an unstable spatial eigenvalue must collide  at the branch point; in the language of~\cite{AbsoluteSpectrum}, there is \emph{absolute spectrum} at the origin. See Section~\ref{s: spectral stability} for further details on spatial eigenvalues and their relation to essential spectrum.  At the linear level, the curve~\eqref{e: left curve} of essential spectrum associated to the outgoing diffusive mode can be moved into the left half-plane with exponential weights, but these weights turn out to be incompatible with a nonlinear argument due to the outgoing group velocity. Indeed, conjugating the equation with an exponential weight which stabilizes this spectral curve introduces coefficients in the nonlinearity which grow exponentially as $\xi \to -\infty$, and so one is not able to control the nonlinearity with this approach. One therefore has no choice but to address a stability problem where continuous curves of essential spectrum touch the imaginary axis, and, as a result, work with algebraic decay of perturbations in the chosen function space.

When the underlying coherent structure is a purely periodic wave train, one can represent the corresponding semigroup via the Bloch wave transform, analogous to the Fourier representation of the heat semigroup, and use this representation to obtain algebraic time decay estimates for perturbations in the linearized equation~\cite{SchneiderAbstract, SchneiderSH, JONZ, JNRZ_13_1, JNRZInventiones}. This approach does not work for front solutions connecting two different end states, since no analogue of the Bloch or Fourier transform is available. An exception is in the study of supercritical pattern-forming fronts, that is,~fronts traveling faster than the selected speed predicted by the marginal stability conjecture. In this case, the dynamics associated to the front tail and interface are exponentially stable, and the stability problem can be reduced to the stability of the periodic pattern in the wake~\cite{Schneider1, Schneider2}. We emphasize that these supercritical fronts are not observed when invasion originates from highly localized initial conditions, and the methods of~\cite{Schneider1, Schneider2} do not appear to generalize to critical invasion fronts. 

\paragraph{Pointwise semigroup estimates.} Originally developed to analyze stability of undercompressive viscous shock waves~\cite{ZumbrunHoward}, pointwise semigroup methods are a powerful tool for establishing linear and nonlinear stability of front-like solutions in the presence of critical diffusive modes. Since the obstruction to obtaining linear stability estimates via a classical approach is the essential spectrum near the origin, the idea is to closely analyze the behavior of the resolvent operator near the critical essential spectrum. The hope is that with detailed enough information on this resolvent, one can deform the inverse Laplace integration contours arbitrarily close to or even {into} the essential spectrum, and recover algebraic decay estimates. In the now-seminal pointwise semigroup approach, the resolvent is analyzed by constructing its integral kernel (or spatial Green's function) $G_\lambda(x,y)$, which solves
\begin{align}
    (\mcl - \lambda) G_{\lambda} (x,y) = \delta(x-y) I, \label{e: resolvent kernel equation}
\end{align}
where $\mcl$ is the linearization about the coherent structure of interest, $\delta$ is the Dirac delta function, and $I$ is the identity. The solution to the linearized equation $u_t = \mcl u$, with initial data $u_0$, is then represented by
\begin{align*}
    u(x,t) = \int_\R G(x,y,t) u_0(y) \de y, 
\end{align*}
where $G(x,y,t)$ is the temporal Green's function, expressed through the inverse Laplace transform
\begin{align}
    G(x,y,t) = - \frac{1}{2 \pi \ri} \int_{\Gamma} \re^{\lambda t} G_\lambda (x,y) \de  \lambda, \label{e: pointwise greens function}
\end{align}
and $\Gamma$ is some contour initially lying to the right of the spectrum of $\mcl$. Of course, the resolvent operator $(\mcl - \lambda)^{-1}$ is singular near the spectrum of $\mcl$. However, if all critical modes are strictly outgoing or ingoing diffusive modes, then the resolvent kernel $G_\lambda(x,y)$ remains analytic or at least meromorphic in $\lambda$ for fixed $x$ and $y$ {even as $\lambda$ passes into the essential spectrum}. The loss of analyticity of the resolvent is due to a loss of spatial localization of $G_\lambda(x,y)$, so that the resolvent is no longer a bounded operator between fixed function spaces. Such a loss of spatial localization can be repaired by conjugating with an exponential weight and, thus, $G_\lambda(x,y)$ remains {pointwise}, that is, for fixed $x,y$,  analytic. Consequently, the integration contour $\Gamma$ in~\eqref{e: pointwise greens function} may be deformed into the essential spectrum of $\mcl$. Sharp linear decay estimates can then be obtained by deforming to pointwise contours $\Gamma_{x,y,t}$, conveniently chosen for each combination of $x, y,$ and $t$, to extract spatio-temporal behavior of the inverse Laplace integral~\eqref{e: pointwise greens function}, using sharp estimates on the resolvent kernel $G_\lambda (x,y)$. Estimates on the resolvent kernel are typically established by solving~\eqref{e: resolvent kernel equation} via a construction of the spatial Green's function through \emph{exponential dichotomies}~\cite{ZumbrunHoward}. 

In the presence of branched diffusive modes, the resolvent kernel $G_\lambda(x,y)$ itself has a branch point at $\lambda = 0$ for each fixed $x,y \in \R$, and so integration contours can no longer be deformed past this singularity. Nonetheless, pointwise semigroup methods have been successfully adapted to problems with branched diffusive modes, using pointwise defined contours which pass near the essential spectrum but remain to the right of the branch point~\cite{Howard1, Howard2, howard3, FayeHolzer, FayeHolzerLV}. 

From this perspective, our results rely on a linear problem with outgoing and branched modes as in~\cite{howard3}, albeit in a context where both the linear and the nonlinear argument are considerably more delicate as we shall explain in the remainder of this introduction. Closer to the problems that we encounter here are results on the stability of invasion fronts in the real Ginzburg-Landau equation~\cite{BricmontKupiainenGLfronts, EckmannWayneGL, AveryScheelGL}.  The pulled front there is special in the sense that, although it selects a spatially constant state rather than a periodic pattern, the selected constant state is only diffusively stable instead of exponentially stable. This relates to the fact that the real Ginzburg-Landau equation is a universal modulation equation for pattern-forming systems near a Turing instability, and the pulled front inherits spectral properties of pattern-forming invasion fronts. In~\cite{AveryScheelGL}, the stability of fronts and the interaction of branched and outgoing modes on the linear level are handled by a far-field/core decomposition of the resolvent, a technique that we rely on here as well, and that we describe next. 

\paragraph{Linear estimates via resolvent decompositions.}
As an alternative to pointwise estimates, we rely on a somewhat more direct and efficient  functional analytic approach to stability problems developed in~\cite{AveryScheel, AveryScheelGL}. Following the strategy outlined above based on an inverse Laplace transform, we carefully analyze the behavior of the resolvent near the origin, and then deform  contours close to the essential spectrum to extract temporal decay. We do not however construct the very detailed and at times cumbersome pointwise integral kernel $G_\lambda(x,y)$ and associated temporal Green's function, but rather solve the resolvent equation
\begin{align}
    (\lfr - \lambda) \u = \g \label{e: intro resolvent eqn}
\end{align}
by decomposing both the data $\g$ and the solution $\u$ using a partition of unity, into two far-field parts which are supported on the wake and on the leading-edge of the front, and a core part which is strongly localized. We then rely on somewhat straightforward Fredholm properties in exponentially localized spaces rather than the more subtle constructions based on for instance the Gap Lemma~\cite{GardnerZumbrun,KapitulaSandstede}; see for instance~\cite{ssrelativemorse} for comparisons. 

We naturally arrive at an explicit spatial decomposition of the solution to the resolvent equation~\eqref{e: intro resolvent eqn}, in which some terms only ``see'' the outgoing diffusive mode, some terms ``see'' the branched diffusive mode, and another, crucial term encodes their interaction. We can then adapt our integration contours specifically to each term in this decomposition, using classical pointwise semigroup contours for the terms which only see the outgoing mode, contours adapted to branched modes as in~\cite{AveryScheel, AveryScheelSelection} for the branched mode, and contours adapted to the interaction terms. 
Somewhat more explicitly, the  decomposition of the linear semigroup $\re^{\lfr t}$ that we obtain is of the form
\begin{align}
    [\re^{\lfr t} \g] (\xi) = \ufr'(\xi) [S_p(t) \g] (\xi) + [S_c(t) \g] (\xi) + [S_e(t) \g] (\xi), \label{e: intro semigroup decomp}
\end{align}
where, roughly speaking, we have
\begin{align}
    \| S_p(t) \g\|_{L^\infty} \sim (1+t)^{-\frac12}, \quad \| S_c (t) \g \|_{L^\infty} \sim (1+ t)^{-1}, \quad \| S_e(t) \g \|_{L^\infty} \sim \re^{-\mu t}
     \label{e: intro linear decay rates}
\end{align}
for some $\mu > 0$. We refer to Theorem~\ref{t: linear estimates} for precise linear estimates. For now, we point out that the first term in~\eqref{e: intro semigroup decomp} is the slowest decaying, with diffusive rate $(1+t)^{-1/2}$, and is in fact supported on the wake of the front, $\xi < 0$, only.

\subsection{Nonlinear diffusive stability} \label{s: nonlinear diffusive stability}
To study the nonlinear stability of the pulled pattern-forming front $\ufr$, we analyze the dynamics of a perturbed solution to~\eqref{e: fhn comoving} of the form $\u(\xi, t) = \ufr(\xi) + \w(\xi, t)$, with sufficiently localized initial perturbation $\w(\xi, 0) = \w_0 (\xi)$. To enforce the exponential localization needed to stabilize the unstable state in the leading edge, we define the weighted perturbation $\vt(\xi, t) = \omega(\xi) \w(\xi, t)$, with induced equation
\begin{align}
\vt_t = \El_{\mathrm{fr}} \vt + \NT(\vt), \label{e:umodpert intro}
\end{align}
where the nonlinearity $\NT$ is given through
\begin{align*}
\NT(\vt) = \omega \widetilde{N}\left(\omega^{-1} \vt\right), \qquad \widetilde{N}(\w) = F(\u_{\mathrm{fr}}+\w) - F(\u_{\mathrm{fr}}) - F'(\u_{\mathrm{fr}}) \w.
\end{align*}
A standard approach to nonlinear stability is to attempt to close an iteration argument based on a variation-of-constants formula for the perturbation equation~\eqref{e:umodpert intro}.

From~\eqref{e: intro linear decay rates}, we see that the decay rates exhibited by the full semigroup $\re^{\El_{\mathrm{fr}} t}$ are diffusive, that is, they coincide with the decay rates of the heat semigroup $\smash{\re^{\partial_\xi^2 t}}$. On the other hand, one finds that the nonlinearity $\NT(\vt)$ in~\eqref{e:umodpert intro} contains quadratic terms in $\vt$ which in general does not allow one to close a nonlinear argument, as is well known for instance in the case of the nonlinear heat equation $u_t = u_{xx} + u^2$ where all nonnegative nontrivial initial data blow up in finite time~\cite{FUJI}. This difficulty does not arise in the stability of viscous shocks~\cite{howard3} since quadratic terms involve derivatives that induce stronger decay. It also does not arise when establishing nonlinear stability of pulled fronts with an exponentially stable state in the wake since the linear diffusive decay in the leading edge is stronger, with rate $t^{-\frac32}$, sufficient to close a nonlinear stability argument; see~\cite{AveryScheel}. The improved decay results from the fact that the front interface provides an absorption mechanism similar to an absorbing boundary condition for the heat equation on the half line.

The weak stability in the wake at rate $t^{-\frac12}$ does arise in the analysis of pulled fronts in the Ginzburg-Landau equation~\cite{BricmontKupiainenGLfronts, EckmannWayneGL, AveryScheelGL}. In this situation, there is however an astute change of coordinates that exploits the symmetry from the gauge invariance and exhibits gradients in the relevant parts of the nonlinearity. Explicitly, after switching to polar coordinates, the linearization diagonalizes with the outgoing diffusive mode only manifesting itself in the phase variable, and the associated nonlinearities at the relevant quadratic and cubic order carry derivatives with stronger associated decay. Such an explicit transformation does not seem to be available or convenient in our setting; see~\cite{scheelwu} for an analysis from this perspective of normal forms.

\paragraph{Forward and backward modulation of perturbations.} Motivated by the fact that the terms in~\eqref{e: intro semigroup decomp} with the slowest decay stem from the diffusively stable pattern in the wake of the front, we take inspiration from the nonlinear stability analysis of wave trains, see e.g.~\cite{JNRZ_13_1,JONZ,SSSU, JNRZInventiones} and references therein. It was observed in~\cite{DSSS} that the most critical diffusive dynamics of perturbed wave trains can be captured by a spatio-temporal phase modulation. Intuitively speaking, such a phase modulation accounts for the translational invariance of the pattern, which causes spectrum to touch the imaginary axis at the origin. Such a phase modulation has been incorporated in the nonlinear stability analysis of wave trains, see~\cite{JNRZ_13_1,JONZ}, by considering the \emph{inverse-modulated perturbation}
\begin{align}
\v(\xi,t) = \omega(\xi)\left(\u(\xi - \psi(\xi,t),t) - \u_{\mathrm{fr}}(\xi)\right). \label{e:modpert intro}
\end{align}
Here, the phase modulation function $\psi(t)$ is chosen a posteriori such that it accounts for the most critical terms in the Duhamel formulation for $\v(t)$. Our analysis shows that the ansatz~\eqref{e:modpert intro} is also successful in the case of pulled pattern-forming fronts to control the diffusive dynamics of the periodic pattern in the wake. Key to this argument is the spatial decomposition~\eqref{e: intro semigroup decomp} of the linearized dynamics, which shows that the slowest decaying terms are supported only in the wake of the invasion front. The resulting spatio-temporal modulation argument is then quite similar to that used for the nonlinear stability of pure wave trains~\cite{JNRZ_13_1, JONZ, SSSU, JNRZInventiones}. 

A natural alternative to modulating the perturbed solution $\u(t)$ is to modulate the pattern-forming front $\u_{\mathrm{fr}}$ itself which leads us to the \emph{forward-modulated perturbation}
\begin{align*}
\vf(\xi,t) = \omega(\xi)\left(\u(\xi,t) - \u_{\mathrm{fr}}(\xi + \psi(\xi,t))\right).
\end{align*}
It has recently been shown in~\cite{ZUM22} that the $W^{k,p}$-norms of the forward- and inverse-modulated perturbations are equivalent up to  controllable norms of $\psi_\xi$. Thus, one finds that the forward- and inverse-modulated perturbations exhibit the same decay rates and it suffices to close a nonlinear argument for one of them.
Nevertheless, it turns out to be advantageous to use {both} the forward- and inverse-modulated perturbation variables in the nonlinear argument. The reason is twofold. On the one hand, the equation for the inverse-modulated perturbation is quasilinear, introducing an apparent loss of derivatives in the nonlinear iteration scheme. On the other hand, the equation for the forward-modulated perturbation is semilinear, but contains terms which are not sufficiently rapidly decaying to close a nonlinear argument through iterative estimates on the associated Duhamel formulation, see~\cite[Section~5.2]{ZUM22} for more details. We emphasize that these observations do not rely on the fact that the underlying solution is a periodic wave train and hold for any traveling wave, thus, in particular, for the pattern-forming fronts under consideration. Hence, we follow the approach, as proposed in~\cite{ZUM22}, of establishing sharp bounds through iterative estimates on the Duhamel formula for the inverse-modulated perturbation $\v(t)$, while controlling regularity through nonlinear damping estimates on $\vf(t)$, thereby using the equivalence of $W^{k,p}$-norms of $\v(t)$ and $\vf(t)$ modulo controllable errors.

\paragraph{Nonlinear damping estimates.} We use energy estimates to effectively control $H^k$-norms of the perturbation $\vf(t)$ in terms of its $L^2$-norm and the $H^k$-norm of its initial condition for $k \geq 0$. Such ``nonlinear damping estimates'' generally rely on damped high-frequency spectrum of the linearization. In fact, the same linear terms in the FitzHugh-Nagumo system~\eqref{e: fhn comoving} that yield high-frequency resolvent estimates in Appendix~\ref{s: contour shifting} are crucial for obtaining a nonlinear damping estimate. More concisely, high-frequency resolvent bounds are equivalent to \emph{linear} damping estimates, which readily yield associated nonlinear damping estimates as long as solutions stay small; see for instance~\cite{RZ16} for further discussion of this principle in the setting of the Saint-Venant equations. Here, the nonlinear damping estimate for $\vf(t)$ is induced by the second derivative $\partial_{\xi\xi} u_1$ in the first component and the term $-\epsilon \gamma u_2$ in the second component of~\eqref{e: fhn comoving}. 
We remark that in reaction-diffusion systems of the form $\u_t = D\u_{xx} + F(\u)$, with positive definite, nondegenerate diffusion matrix $D$, damping comes from the highest-order derivative $D\u_{xx}$, allowing control even of quasilinear terms so that, in that setting, one can directly obtain a nonlinear damping estimate for the inverse-modulated perturbation and it is not necessary to consider the forward-modulated perturbation, see~\cite{JONZ}.

\subsection{Function spaces and notation}\label{s: function spaces}
We introduce function spaces used in the analysis and in the precise statement of our main result. 
\paragraph{Exponentially weighted spaces.}
Given $\eta_\pm \in \R$, we define a smooth positive weight function $\omega_{\eta_-, \eta_+}$ satisfying
\begin{align*}
    \omega_{\eta_-, \eta_+} (\xi) = \begin{cases}
    \re^{\eta_- \xi}, &\xi \leq -1, \\
    \re^{\eta_+ \xi}, &\xi \geq 1. 
    \end{cases}
\end{align*}
Given additionally non-negative integers $k$ and $m$, a field $\mathbb{F} \in \{\R, \C\}$, and a real number $1 \leq p \leq \infty$, we define a corresponding weighted Sobolev space $W^{k,p}_{\mathrm{exp}, \eta_-, \eta_+}(\R, \mathbb{F}^m)$ through the norm
\begin{align*}
    \| f \|_{W^{k,p}_{\mathrm{exp}, \eta_-, \eta_+} } = \| \omega_{\eta_-, \eta_+} f \|_{W^{k,p}},
\end{align*}
where $W^{k,p}(\R, \mathbb{F}^m)$ is the standard Sobolev space with differentiability index $k$ and integrability index $p$. When $p = 2$, we write $W^{k,2}_{\mathrm{exp}, \eta_-, \eta_+} (\R, \F^m) = H^k_{\mathrm{exp},\eta_-,\eta_+} (\R, \F^m)$. When domain and codomain can be readily inferred, we write $W^{k,p}_{\mathrm{exp},\eta_-, \eta_+} (\R, \F^m) = W^{k,p}_{\mathrm{exp},\eta_-, \eta_+}$. Finally, we write $W^{0, p}_{\mathrm{exp}, \eta_-, \eta_+} = L^p_{\mathrm{exp}, \eta_-, \eta_+}$.

Given $\eta \in \R$ and $1 \leq p \leq \infty$, we define $X^p_{\eta} = L^p_{\mathrm{exp}, -\eta, \eta} (\R, \C^2)$ and let 
\begin{align*}
    Y^p_\eta = W^{2,p}_{\mathrm{exp}, -\eta, \eta} (\R, \C) \times W^{1,p}_{\mathrm{exp}, -\eta, \eta} (\R, \C) 
\end{align*}
denote the domain of the linearization $\lfr$, given by~\eqref{e: lfr def}, on $X^p_\eta$, so that $\lfr : Y^p_\eta \subset X^p_\eta \to X^p_\eta$ is a closed operator. 

\paragraph{Algebraically weighted spaces.} 
To exert finer control over the spatial localization of perturbations, we also introduce algebraically weighted spaces. Given $r_\pm \in \R$, we define a smooth positive weight function
\begin{align*}
    \rho_{r_-, r_+}(\xi) = \begin{cases}
    |\xi|^{r_-}, & \xi \leq -1, \\
    |\xi|^{r_+}, &\xi \geq 1. 
    \end{cases}
\end{align*}
As in the definition of exponentially weighted spaces, given additionally non-negative integers $k$ and $m$, a field $\F \in \{ \R, \C \}$, and a real number $1 \leq p \leq \infty$, we define a corresponding algebraically weighted Sobolev space $W^{k,p}_{r_-, r_+}(\R, \F^m)$ through the norm
\begin{align*}
    \| f \|_{W^{k,p}_{r_-, r_+}} = \| \rho_{r_-, r_+} f \|_{W^{k,p}}.
\end{align*}
We again suppress the notation for the domain and codomain when clear in context, and  we write $W^{0,p}_{r_-, r_+} = L^p_{r_-, r_+}$. 

\paragraph{Additional function spaces.}
An additional technical challenge we must address is that the FitzHugh-Nagumo system is not fully parabolic, so the linearization $\lfr$ is not a sectorial operator. Spectral mapping in such contexts is not automatic. We will first prove spectral mapping properties for initial data $\u_0 \in C^\infty_c (\R)$, and then extend the semigroup to larger function spaces. To carry out this extension, we consider spaces in which such test functions are dense. Therefore, given non-negative integers $k$ and $m$ and a field $\F \in \{ \R, \C\}$, we let $C^k_0(\R, \F^m)$ denote the space of bounded continuous functions $\u: \R \to \F^m$ satisfying
\begin{align*}
    \lim_{\xi \to \pm \infty} |\partial_\xi^\ell \u(\xi)| = 0
\end{align*}
for $\ell = 0, 1, ..., k$. Note that such functions (and their $\ell$-th derivatives) are automatically uniformly continuous. We equip this space with the $W^{k, \infty}$ norm and recall the well-known fact that test functions are dense in this space. 

We will use a mix of $C_0$ and $L^2$-based spaces in our nonlinear argument, and therefore define
the spaces
\begin{align*}
Z_k(\R, \F^m) := H^k(\R, \F^m) \cap C^k_0 (\R, \F^m),
\end{align*}
endowed with norm $\| \cdot \|_{H^k} + \| \cdot\|_{W^{k,\infty}}$, for any non-negative integers $k$ and $m$ and  $\F \in \{\R, \C\}$. 

\paragraph{Spaces involving time.} Given an interval $I \subseteq [0, \infty)$ and a Banach space $X$, we let $C^k(I, X)$ denote the space of $k$-continuously differentiable functions on $I$ taking values in $X$.

\paragraph{Additional notation.} We will sometimes abuse notation by writing a function $\u(\xi, t)$ of space and time as $\u(t)$, viewing it as a function of time taking values in a particular Banach space. Similarly, we may write a function $\u(\xi; \lambda)$ of a spectral parameter $\lambda$ as $\u(\lambda)$. 
We also let $(\chi_-, \chi_c, \chi_+)$ be a partition of unity on $\R$ such that 
\begin{align*}
	\chi_+(\xi) = \begin{cases}
		1, & \xi \geq 1 \\
		0, & \xi \leq 0, 
	\end{cases}
\end{align*}
and $\chi_-(\xi) = \chi_+(-\xi)$. 

\paragraph{Suppression of constants.} Let $S$ be a set, and let $A, B \colon S \to \R$. The expression ``$A(x) \lesssim B(x)$ for $x \in S$'', means that there exists a constant $C>0$, independent of $x$, such that $A(x) \leq CB(x)$ holds for all $x \in S$.

\subsection{Spectral assumptions}

We formulate spectral stability hypotheses that precisely describe pulled uniformly translating pattern-forming fronts. We expect these assumptions to hold in open classes of systems, generically for pulled uniformly translating pattern-forming fronts, and we summarize results on their validity for~\eqref{e: fhn} in Remark~\ref{rmk: spectral assumptions}. 

We decompose the spectrum $\Sigma(\mcl)$ of an operator $\mcl$ as follows. We say that $\lambda \in \C$ is in the essential spectrum $\Sigma_\mathrm{ess}(\mcl)$ if either $\mcl - \lambda$ is not Fredholm, or it is Fredholm with nonzero index. We say $\lambda \in \C$ is in the point spectrum $\Sigma_\mathrm{pt}(\mcl)$ if $\mcl - \lambda$ is Fredholm with index zero, but not invertible. The essential spectrum is invariant under compact perturbations. Thus, for the operator $\El_{\rm fr}$, the essential spectrum may be determined from the linearizations about the asymptotic end states at $\pm \infty$, which are
\begin{align}
    \mcl_+ = D(\partial_\xi - \etalin)^2 + \clin (\partial_\xi - \etalin) + F'(0), \label{e: L plus def}
\end{align}
at $\xi = +\infty$, and
\begin{align}  \label{e: L wt def}
    \lwt = D\partial_\xi^2 + \clin \partial_\xi + F'(\uwt(\xi)),
\end{align}
at $\xi = -\infty$. The spectrum of $\mcl_+$ is marginally stable, a fact that effectively characterizes the linear spreading speed; see Lemma~\ref{l: linear spreading speed}. We now state assumptions on the essential spectrum of $\lwt$ and on the point spectrum of $\lfr$. 

\paragraph{Essential spectrum of the wave train.} The limiting operator $\lwt$ in the wake has coefficients which are periodic in $\xi$ with period $L$. By Floquet-Bloch theory (see for instance~\cite{ReedSimon}), $\lwt$ is conjugate to a family of Bloch operators 
\begin{align*}
    \hat{\mcl}_\mathrm{wt} (\nu) : H^2 (\R / L \Z, \C) \times H^1 (\R / L \Z, \C) \subset L^2(\R/L\Z, \C^2) \to L^2 (\R / L \Z, \C^2)
\end{align*}
given by
\begin{align*}
    \hat{\mcl}_\mathrm{wt} (\nu) = D (\partial_\xi + \nu)^2 + \clin (\partial_\xi + \nu) + F'(\uwt(\xi))
\end{align*}
for $\nu$ purely imaginary, $\nu \in \left[-\ri\kwt, \ri \kwt \right)$, with $\kwt = \frac{\pi}{L}$. In particular, we have
\begin{align*}
    \Sigma \left( \lwt \right) = \bigcup_{\nu \in [-\ri\kwt, \ri \kwt)} \Sigma(\hat{\mcl}_\mathrm{wt} (\nu)). 
\end{align*}
We assume that $\lwt$ satisfies the following standard diffusive spectral stability assumptions for periodic wave trains~\cite{SchneiderAbstract, JONZ, SSSU}. 
\begin{hyp}[Diffusive spectral stability of the wave train]\label{hyp: wt spectral}
We assume that the following spectral stability conditions hold for the operator $\lwt : H^2(\R, \C) \times H^1(\R, \C) \subset L^2(\R,\C^2) \to L^2 (\R, \C^2)$ given by~\eqref{e: L wt def}:
\begin{enumerate}
    \item We have $\Sigma(\lwt) \subset \{ \lambda \in \C : \Re \lambda < 0 \} \cup \{ 0 \}$;
    \item There exists $\theta > 0$ such that for any $\kappa \in [-\kwt, \kwt)$, we have $\Re \Sigma(\hat{\mcl}_\mathrm{wt} (\ri \kappa)) \leq - \theta \kappa^2$;
    \item $\lambda = 0$ is an algebraically simple eigenvalue of $\hat{\mcl}_\mathrm{wt} (0)$. 
\end{enumerate}
\end{hyp}
We note that $0$ is always in the spectrum of $\hat{\El}_{\rm wt}(0)$ with associated eigenfunction $\partial_\xi \uwt$ by translation invariance of the wave train. The spectral information encoded in Hypothesis~\ref{hyp: wt spectral} is sufficient to establish nonlinear stability of pure wave trains~\cite{JONZ, SSSU}, for strictly parabolic systems. For degenerate systems such as~\eqref{e: fhn}, one has to additionally control high frequency modes to obtain a spectral mapping estimate; see Theorem~\ref{t: wave train stability}. Since the Bloch operators $\hat{\mcl}_{\mathrm{wt}}(\nu)$ depend analytically on the Floquet exponent $\nu$, Hypothesis~\ref{hyp: wt spectral}  yields open neighborhoods $U, V \subset \C$ of the origin and an analytic function $\lambdawt \colon U \to \C$ with the expansion
\begin{align} \label{e: wt spectral curve}
    \lambdawt(\nu) = - c_g \nu + \Deff^\mathrm{wt} \nu^2 + \mathrm{O}(\nu^3),
\end{align}
with $c_g \in \R$ and $\Deff^\mathrm{wt} > 0$ such that the spectrum of $\hat{\mcl}_{\rm wt}(\nu)$ in $V$ is given by the simple eigenvalue $\lambdawt(\nu)$ for $\nu \in U$, see~\cite{DSSS}. Consequently, the spectrum of $\lwt$ in $V$ is given by the curve $\{\lambdawt(\ri\kappa) : \ri\kappa \in U\} \cap V$ touching the origin in a quadratic tangency. The equation $\lambdawt(\nu) = 0$ is called the \emph{linear dispersion relation} and the coefficient $c_g$ is the group velocity of the wave train (in the frame moving with the speed $\clin$). With a standard Lyapunov-Schmidt reduction procedure~\cite{DSSS} one obtains
\begin{align} \label{e: group velocity}
c_g &= 2\left\langle \u_{\mathrm{ad}}, D \partial_{\xi\xi} \uwt\right\rangle_{L^2(\R/L\Z,\C^2)} + \clin = 2\left\langle u_{\mathrm{ad},1}, \partial_{\xi\xi} u_{\mathrm{wt},1}\right\rangle_{L^2(\R/L\Z,\C)} + \clin, \end{align}
where $\u_{\mathrm{ad}}$ spans the kernel of the adjoint operator $\hat{\mcl}_{\rm wt}(0)^*$ normalized such that $\langle\u_\mathrm{ad},\partial_\xi \uwt\rangle_{L^2(\R/L\Z,\C^2)} = 1$. 

For pattern-forming fronts, it is important to characterize the group velocity of the wave train in the wake, relative to the front speed, so that the front acts as a \emph{source} of patterns~\cite{SandstedeScheelDefects}. Thinking of the invasion process as a source of patterns, the group velocity of the wave train should be negative in the frame moving with the spreading speed $\clin$. We make this precise as follows.

\begin{hyp}[Outgoing group velocity for the wave train]\label{hyp: wake spectral curve expansion}
Assuming, in accordance with Hypothesis~\ref{hyp: wt spectral}, that $0$ is an algebraically simple eigenvalue of $\hat{\mcl}_{\rm wt}(0)$, we require  that the group velocity $c_g$, given by~\eqref{e: group velocity}, is negative. 
\end{hyp}

Hypothesis~\ref{hyp: wake spectral curve expansion} implies that the critical diffusive mode associated to the dynamics near the wave train is an outgoing diffusive mode, in the terminology of Section~\ref{s: intro linear}. 

\paragraph{Point spectrum.} Hypothesis~\ref{hyp: wt spectral}, together with the calculation of the linear spreading speed in Lemma~\ref{l: linear spreading speed} below, imply that the essential spectrum of $\lfr$ is marginally stable; see Proposition~\ref{p: fr essential spectrum}. We exclude unstable point spectrum as follows. 
\begin{hyp}[No unstable point spectrum]\label{hyp: point spectrum}
    We assume that the operator $\lfr \colon H^2(\R, \C) \times H^1(\R, \C) \subset L^2(\R, \C^2) \to L^2(\R, \C^2)$, given by~\eqref{e: lfr def} has no eigenvalues $\lambda$ with $\Re \lambda \geq 0$. Moreover, we assume that there is no nontrivial bounded solution $\u$ to the equation $\lfr \u = 0$. 
\end{hyp}
A nontrivial bounded solution to $\lfr \u = 0$ would indicate that the fronts are not strictly pulled, but rather at the boundary between pushed and pulled front propagation~\cite{pp}, and hence we exclude this possibility here. Such a bounded solution is sometimes referred to as a resonance and would correspond to a  zero of the associated Evans function, appropriately extended into the essential spectrum via the gap lemma~\cite{GardnerZumbrun, KapitulaSandstede}. The improved $t^{-3/2}$ decay rate in the leading edge, when compared with diffusive decay in one dimension,  is associated to this lack of a resonance at $\lambda = 0$~\cite{AveryScheel}. 

\begin{remark}\label{rmk: spectral assumptions}
    Hypothesis~\ref{hyp: wake spectral curve expansion}, that is, the fact that  wave trains considered in Theorem~\ref{t: existence} possess outgoing group velocity in the co-moving frame,  was established in~\cite{CarterScheel}. We show in~\cite{spectral} that Hypothesis~\ref{hyp: wt spectral} holds for the wave trains in Theorem~\ref{t: existence} within the parameter range $(3-\sqrt{5})/6 < a < \frac12$. In this regime, the wave trains are of `phase wave' type~\cite{CarterScheel}, i.e.,~they possess long wavelength while their amplitude is roughly independent of $a$, as opposed to the somewhat more narrowly spaced ‘trigger’ wave trains which are selected for $(3-\sqrt{6})/6 < a < (3-\sqrt{5})/6$. A spectral stability analysis of trigger wave trains in the FitzHugh-Nagumo system with $\gamma = 0$ can be found in~\cite{eszter1999evans}. The vanishing of linear damping in the $v$-equation, $\gamma = 0$, leads to marginally stable high-frequency spectrum, that is, a spectral branch converging to $\pm\ri\infty$, which we cannot allow in our diffusive stability analysis. We do expect however that the analysis in~\cite{eszter1999evans} can be extended to positive $\gamma$, resulting in a proof that the wave trains in Theorem~\ref{t: existence} also obey Hypothesis~\ref{hyp: wt spectral} for $(3-\sqrt{6})/6 < a < (3-\sqrt{5})/6$. Furthermore, we show in~\cite{SpectralFronts} that Hypothesis~\ref{hyp: point spectrum} is satisfied for all of the fronts described in Theorem~\ref{t: existence} in the phase wave regime $(3-\sqrt{5})/6 < a < \frac12$. Finally, we mention that in the regime of larger $\gamma$, dynamics are bistable, with wave trains emerging from heteroclinic loops.  Recent work \cite{li2025nonlinearstabilitylargeperiodtraveling} establishes stability of wave trains also in this regime. We however do not expect these wave trains to be selected by invasion fronts.
\end{remark}

\subsection{Main result --- precise statement}\label{s: main result}

Having introduced the necessary notation, we are ready to formulate the precise statement of our main result.

\begin{thmlocal}[Nonlinear stability --- precise formulation]\label{t: main detailed}
Assume $\ufr$ is a pulled front solution from Theorem~\ref{t: existence} which obeys Hypotheses~\ref{hyp: wt spectral} through~\ref{hyp: point spectrum}. Then, there exist constants $M, \delta > 0$ such that if we consider~\eqref{e: fhn comoving} with initial data $\u_0 = \ufr + \w_0$, where $\w_0 \in L^1_{0,1} (\R, \R^2) \cap (H^3(\R, \R) \times H^2 (\R, \R))$ satisfies
\begin{align}
    E_0 := \| \omega \w_0\|_{L^1_{0,1}} + \|\omega \w_0\|_{H^3 \times H^2} < \delta, \label{e: theorem 1 E0}
\end{align}
then there exist functions $\w: [0, \infty) \times \R \to \R^2$ and $\psi : [0, \infty) \times \R \to \R$ satisfying
\begin{align}
    \omega \w &\in C\left( [0, \infty), H^3 (\R, \R) \times H^2 (\R, \R) \right) \cap C^1 \left([0, \infty), H^1(\R, \R^2) \right), \qquad \psi \in C^\infty \left( [0, \infty) \times \R, \R \right)
\label{e:regularity}
\end{align}
and 
\begin{align*}
\partial_t^k \psi (\cdot, t) \in H^{\ell} (\R, \R), \qquad t > 0, \, k,\ell \in \N_0,
\end{align*}
with initial conditions $\w(0) = \w_0$ and $\psi(0) = 0$, such that $\u(t) = \ufr + \w(t)$ is the unique global classical solution of~\eqref{e: fhn comoving} with initial data $\u_0$, enjoying the following estimates
\begin{align}
\begin{split}
\left\|\omega\left(\u(t) - \ufr\right)\right\|_{H^3 \times H^2}+ \left\|\psi(t)\right\|_{L^2} &\leq \frac{ME_0}{(1+t)^{\frac14}}, \qquad \left\|\omega\left(\u(t) - \ufr\right)\right\|_{L^\infty}+ \left\|\psi(t)\right\|_{L^\infty} \leq \frac{ME_0}{(1+t)^{\frac12}},
\end{split} \label{e:MTmodder}
\end{align}
and
\begin{align}
\begin{split}
\left\|\omega\big(\u(t) - \ufr\left(\cdot + \psi(\cdot,t)\right)\!\big)\right\|_{H^3 \times H^2}+ \left\|\left(\psi_\xi(t),\partial_t \psi(t)\right)\right\|_{H^3 \times H^2} &\leq \frac{ME_0}{(1+t)^{\frac34}}, \\
\left\|\omega\big(\u(t) - \ufr\left(\cdot + \psi(\cdot,t)\right)\!\big)\right\|_{L^\infty}+ \left\|\left(\psi_\xi(t),\partial_t \psi(t)\right)\right\|_{L^\infty} &\leq \frac{ME_0}{1+t},
\end{split} \label{e:MTmodder2}
\end{align}
for all $t\geq 0$. Furthermore, the following refined estimate holds on the leading edge:
\begin{align}
    \sup_{\xi \geq 1} \left| \frac{\omega(\xi)}{1 + \xi} \left(\u(\xi, t) - \ufr(\xi) \right) \right| \leq \frac{M E_0}{(1+t)^{\frac32}} \label{e: improved leading edge decay}
\end{align}
for all $t \geq 0$. Finally, the phase modulation function is supported on the wake: we have $\mathrm{supp}(\psi(\cdot,t)) \subset (-\infty, 0]$ for all $t \geq 0$. 
\end{thmlocal}

Rephrasing, Theorem~\ref{t: main detailed} establishes the following:
\begin{enumerate}
    \item The weighted, unmodulated perturbation to the front, $\omega (\u(t) - \ufr)$, decays in $L^\infty(\R)$ with diffusive rate $(1+t)^{-1/2}$, provided the initial perturbation $\w_0$ is sufficiently small and localized.
    \item The leading-order behavior of the perturbation is described by a spatio-temporal phase modulation $\psi(\xi,t)$ of the pattern-forming front $\ufr$, where $\psi(\cdot,t)$ is supported on $(-\infty,0]$. That is, we only need to modulate the phase in the wake. 
    \item In the leading edge, $\xi \geq 1$, the perturbation decays with an improved rate $(1+t)^{-3/2}$ characteristic of pulled invasion fronts~\cite{AveryScheel, AveryScheelSelection, AverySelectionRD}. 
\end{enumerate}

\paragraph{Choice of function spaces.} We briefly explain the choices of function spaces we make in Theorem~\ref{t: main detailed}. The additional algebraic localization of the initial data enforced by requiring $\omega \w_0 \in L^1_{0,1}$ is needed to obtain optimal linear decay rates for the pulled front dynamics in the leading edge~\cite{AveryScheel, AveryScheelGL, AverySelectionRD}. A nonlinear stability argument could then be closed measuring the solution only in $L^2$ based spaces, but we choose to additionally measure in $L^\infty$ in order to capture the sharp pointwise decay rate of the perturbation --- notice that the $L^\infty$ decay rates in Theorem~\ref{t: main detailed} are faster than the $L^2$ decay rates. On the other hand, even if we are primarily interested in $L^\infty$ decay rates, we control regularity in the resulting quasilinear iteration scheme through energy estimates, for which we must additionally measure the solution in $L^2$-based spaces. From the structure of the equation, it is natural to expect to require $H^2 \times H^1$ regularity of the initial data, and indeed if we were only interested in $L^2$ decay rates we could relax the regularity requirement in~\eqref{e: theorem 1 E0} to $\omega \w_0 \in H^2 \times H^1$. However, in also controlling $L^\infty$ decay rates, we lose a degree of regularity through the embedding $H^1 \hookrightarrow L^\infty$, explaining the $H^3 \times H^2$ requirement in~\eqref{e: theorem 1 E0}. 

\subsection{Nonlinear stability of pure wave trains}
The periodic wave trains constructed in~\cite{CarterScheel} and references therein are interesting objects of their own right. We establish their spectral stability in the companion paper~\cite{spectral}, but their nonlinear stability is not quite included in general results on stability of wave trains in reaction-diffusion systems under spectral assumptions~\cite{JONZ, JNRZ_13_1, SSSU}, since the degenerate diffusion matrix in the FitzHugh-Nagumo system~\eqref{e: fhn} introduces additional technical difficulties in establishing spectral mapping properties of the linear semigroup and controlling regularity in the nonlinear argument. One could attack these issues by proceeding as in~\cite{HJPR21, RZ16}, that is, using a Bloch wave decomposition and the Gearhart-Pr{\"u}ss theorem to obtain high frequency resolvent bounds and associated linear decay estimates on the semigroup. These techniques are not available for the pattern-forming fronts, which are not purely periodic in space, preventing the use of the Bloch wave decomposition. A simplified version of our proof of Theorem~\ref{t: main detailed}, simply ignoring parts which have to do with dynamics in the leading edge, then gives an alternative proof of nonlinear stability of the wave trains, formulated as follows.

\begin{thmlocal}[Nonlinear stability of wave trains]\label{t: wave train stability}
Let $\uwt$ be a stationary periodic solution to~\eqref{e: fhn comoving} which is spectrally stable in the sense of Hypothesis~\ref{hyp: wt spectral}. Then, there exist constants $M, \delta > 0$ such that if we consider~\eqref{e: fhn comoving} with initial data $\u_0 = \uwt + \w_0$, where $\w_0 \in L^1(\R, \R^2) \cap (H^3 (\R, \R) \times H^2(\R, \R))$ satisfies
\begin{align*}
    E_0 := \| \w_0 \|_{L^1} + \| \w_0 \|_{H^3 \times H^2} < \delta, 
\end{align*}
then there exist functions $\w : [0, \infty) \times \R \to \R^2$ and $\psi: [0, \infty) \times \R \to \R$ satisfying
\begin{align*}
    \w &\in C \left( [0, \infty), H^3 (\R, \R) \times H^2 (\R, \R) \right) \cap C^1 \left( [0, \infty), H^1 (\R, \R^2) \right), \qquad \psi \in C^k \left([0, \infty), H^\ell(\R, \R) \right)
\end{align*}
for any $k, \ell \in \N$, with $\w(0) = \w_0$ and $\psi(0) = 0$ such that $\u(t) = \uwt + \w(t)$ is the unique global solution of~\eqref{e: fhn comoving} with initial data $\u_0$, and we have the following estimates
\begin{align*}
\begin{split}
\left\|\u(t) - \uwt \right\|_{H^3 \times H^2}+ \left\|\psi(t)\right\|_{L^2} &\leq \frac{ME_0}{(1+t)^{\frac14}}, \qquad \left\|\u(t) - \uwt \right\|_{L^\infty}+ \left\|\psi(t)\right\|_{L^\infty} \leq \frac{ME_0}{(1+t)^{\frac12}},
\end{split} 
\end{align*}
and
\begin{align*}
\begin{split}
\left\|\u(t) - \uwt\left(\cdot + \psi(\cdot,t)\right)\right\|_{H^3 \times H^2}+ \left\|\left(\psi_\xi(t),\partial_t \psi(t)\right)\right\|_{H^3 \times H^2} &\leq \frac{ME_0}{(1+t)^{\frac34}}, \\
\left\|\u(t) - \uwt\left(\cdot + \psi(\cdot,t)\right)\right\|_{L^\infty}+ \left\|\left(\psi_\xi(t),\partial_t \psi(t)\right)\right\|_{L^\infty} &\leq \frac{ME_0}{1+t},
\end{split} 
\end{align*}
for all $t\geq 0$.
\end{thmlocal}
Since the proof of Theorem~\ref{t: wave train stability} is a simplified version of the proof of Theorem~\ref{t: main detailed}, we do not present it in detail. 

We note that it is possible to relax the regularity requirement on the initial data in Theorem~\ref{t: wave train stability}, to requiring smallness in $L^1(\R, \R^2) \cap \big(W^{2, \infty}(\R, \R) \times W^{1, \infty}(\R, \R)\big)$, only, by using an $L^1$-$L^\infty$ iteration scheme and replacing the $L^2$ based nonlinear damping estimates with alternative arguments developed in~\cite{deRijkSandstede, HJPR21} to control regularity of the modulated perturbation. It is more difficult to relax the regularity requirement on the initial data for the nonlinear stability of the pattern-forming fronts in Theorem~\ref{t: main detailed}, since the alternative argument for controlling regularity in the iteration scheme relies on obtaining estimates on the operator $\partial_x^\ell S_p(t) \partial_x^m$, where $S_p(t)$ arises in the decomposition~\eqref{e: intro semigroup decomp}. In the pure wave train case, the leading-order terms in the resolvent decomposition roughly have the structure of a convolution, so it is easy to integrate by parts to pass derivatives from one side of $S_p(t)$ to another, while it is not trivial to obtain estimates on $S_p(t)$ composed with derivatives on the right through the far-field/core decomposition used here in the linear analysis of the pattern-forming fronts. We expect that this technical issue could be overcome with more effort, but do not pursue it further, instead focusing on the most interesting aspects of the dynamics. 

\subsection{Outline of the paper}
The remainder of the paper is organized as follows. In Section~\ref{s: spectral stability}, we determine the essential spectrum of $\lfr$ under Hypotheses~\ref{hyp: wt spectral} and~\ref{hyp: wake spectral curve expansion} and explore related Fredholm properties. In Section~\ref{s: contour rep}, we use standard results from semigroup theory to represent the semigroup $\re^{\lfr t}$ via the inverse Laplace transform of the resolvent.  In Section~\ref{s: abstract linear}, we state and prove some general linear estimates based on different possible types of behavior of the resolvent near the essential spectrum at the origin. With these abstract estimates informing what kind of behavior we want to extract from the resolvent, we perform a detailed analysis of the resolvent in Section~\ref{s: resolvent}. In Section~\ref{s: linear estimates}, we combine the resolvent analysis of Section~\ref{s: resolvent} with the abstract linear estimates of Section~\ref{s: abstract linear} to establish a detailed decomposition of the semigroup $\re^{\lfr t}$. In Section~\ref{sec:itscheme}, we set up our nonlinear iteration scheme and establish equivalence of the forward-modulated and inverse-modulated perturbations. Finally, we close our nonlinear stability argument in Section~\ref{s: nonlinear argument}, proving Theorem~\ref{t: main detailed}. We conclude in Section~\ref{s: discussion} by discussing applications of our methods to related problems. Throughout, we relegate some technical computations to the appendices. 

\section{Spectral and Fredholm properties}\label{s: spectral stability}
We compute the linear spreading speed and characterize the resulting essential spectrum and Fredholm properties of $\lfr$, under Hypotheses~\ref{hyp: wt spectral} and~\ref{hyp: wake spectral curve expansion}. 

\subsection{Essential spectrum in the wake} \label{sec: essential spectrum wake}

Recall that the limiting operator in the wake is the linearization $\lwt$ of~\eqref{e: fhn comoving} about the wave train $\uwt$, see~\eqref{e: L wt def}. Hypotheses~\ref{hyp: wt spectral} and~\ref{hyp: wake spectral curve expansion} yield that the spectrum of the operator $\lwt$ (on, for instance, $L^2(\R)$) is bounded away from the imaginary axis, except for a quadratic tangency at the origin. In a neighborhood of the origin the spectrum of $\lwt$ is given by the curve $\kappa \mapsto \lambdawt(\ri \kappa)$, where $\lambdawt$ is given by~\eqref{e: wt spectral curve} with $c_g < 0$ and $D_{\rm eff}^{\rm wt} > 0$. We consider the eigenvalue problem $(\lwt - \lambda) \u = 0$ as a first-order system of linear differential equations with periodic coefficients. We call the Floquet exponents $\nu$ associated to this system the \emph{spatial Floquet exponents} of $\lwt - \lambda$. They arise by solving the linear dispersion relation $\lambdawt(\nu) = 0$. An application of the implicit function theorem then readily yields the following result.

\begin{prop}[Configuration of spatial Floquet exponents] \label{prop: wake spectral curve expansion}
Assume Hypotheses~\ref{hyp: wt spectral} and~\ref{hyp: wake spectral curve expansion}. For $\lambda$ to the right of $\Sigma(\lwt)$, the operator $\lwt-\lambda$ has two unstable spatial Floquet exponents and one stable spatial Floquet exponent, which are analytic in $\lambda$ in a neighborhood of the origin. As $\lambda$ approaches the origin from the right, one spatial Floquet exponent $\nuwt(\lambda)$ crosses the imaginary axis from right to left, with expansion
\begin{align*}
    \nuwt(\lambda) = \nuwt^1 \lambda - \nuwt^2 \lambda^2 + \mathrm{O}(\lambda^3), \quad \nuwt^1 = - \frac{1}{c_g} > 0, \quad \nuwt^2 = - \frac{\Deff^\mathrm{wt}}{c_g^3} > 0. 
\end{align*}
The other Floquet exponents remain uniformly bounded away from the imaginary axis for $|\lambda|$ small.
\end{prop}

\subsection{Essential spectrum in the leading edge}
Recall that the limiting operator in the leading edge is given by the linearization $\El_+$ of~\eqref{e: fhn comoving} about the rest state $0$, see~\eqref{e: L plus def}. Since $\mcl_+$ has constant coefficients, its spectrum (on, for instance, $L^2(\R)$) can be readily computed via the Fourier transform as
\begin{align*}
    \Sigma(\mcl_+) = \{ \lambda \in \C : d_{\clin}(\lambda, ik-\etalin) = 0 \text{ for some } k \in \R \},
\end{align*}
where, for a given speed $c$, $d_{c} \colon \C \times \C \to \C$ is the \emph{linear dispersion relation} given by
\begin{align*}
    d_{c}(\lambda, \nu) = \det \left(D \nu^2 + c \nu I + F'(0) - \lambda I \right).
\end{align*}
We determine the essential spectrum of $\mcl_+$ in the following lemma, whose proof is a straightforward computation; see also~\cite[\S2]{CarterScheel} for some computational details. 
\begin{lemma}\label{l: linear spreading speed}
Fix $(3 - \sqrt{6})/6 < a < \frac{1}{2}$, $0 < \gamma < 4$, and $\eps > 0$ sufficiently small. Then the spreading speed $\clin(a, \gamma, \eps)$ and spatial decay rate $\etalin(a, \gamma, \eps) < 0$ from Theorem~\ref{t: existence} are smooth in all arguments and satisfy the following. 
\begin{enumerate}
    \item (Simple double root) For $\lambda$ and $\nu$ small, we have the expansion
    \begin{align*}
        d_\clin(\lambda, \nu-\etalin) = -d_{10} \lambda + d_{02} \nu^2 + \mathrm{O}(\nu^3, \lambda^2, \lambda \nu)
    \end{align*}
    for some $d_{10}, d_{02} \in \R$ with $d_{10} d_{02} > 0$. 
    \item (Minimal critical spectrum) If $d_{\clin} (\ri \omega, \ri k -\etalin) = 0$ for some $\omega, k \in \R$, then $\omega = k = 0$. 
    \item (No unstable spectrum) $d_{\clin} (\lambda, \ri k - \etalin) \neq 0$ for any $k \in \R, \lambda \in \C$ with $\Re \lambda > 0$. 
\end{enumerate}
\end{lemma}
\begin{remark}
Lemma~\ref{l: linear spreading speed} together with~\cite{HolzerScheelPG} imply that the dynamics of the linearization in the leading edge at the speed $\clin$, which we call the \emph{linear spreading speed},  given by
\begin{align*}
    \u_t = D \partial_\xi^2 \u + \clin \partial_\xi \u + F'(0) \u,
\end{align*}
are marginally pointwise stable, that is, neither exponentially growing nor decaying in time at any fixed $\xi$. See~\cite[Section 1.2]{AveryScheelSelection} for a concise explanation of how Lemma~\ref{l: linear spreading speed} relates to marginal pointwise stability, and why this leads to a prediction for a selected front speed. 
\end{remark}

Lemma~\ref{l: linear spreading speed} implies that the spectrum of $\mcl_+$ is stable and uniformly bounded away from the imaginary axis except for a neighborhood of the origin, where it is given by the curve
\begin{align*}
    \lambda_+(\ri k) = - \Deff^+ k^2 + \mathrm{O}(k^3), 
\end{align*}
where $\Deff^+ = \frac{d_{02}}{d_{10}} > 0$, as suggested in Section~\ref{s: intro linear}. In particular, the spectrum of $\mcl_+$ is marginally stable, and the critical spectral curve is associated to a branched diffusive mode, in the terminology of Section~\ref{s: intro linear}.

We will also need information on the \emph{spatial eigenvalues} of $\mcl_+ - \lambda$, that is, the eigenvalues of the matrix $\mathcal{M}(\lambda)$ obtained from writing $(\mcl_+ - \lambda) \u = 0$ as a first-order system $\partial_\xi U = \mathcal{M}(\lambda) U$ in $U = (u_1, \partial_\xi u_1, u_2)^\top$. 
\begin{corollary}\label{c: leading edge spatial eval}
Assume $a, \gamma$ and $\eps$ are such that Lemma~\ref{l: linear spreading speed} holds. Let $\sigma = \sqrt{\lambda}$, with branch cut chosen along the negative real axis. If $\sigma^2$ is small, $\mathcal{M}(\sigma^2)$ has precisely two eigenvalues $\nufr^\pm(\sigma)$ in a neighborhood of the origin, which are analytic in $\sigma$ and satisfy the expansions
\begin{align}
    \nufr^\pm(\sigma) = \pm \nufr^1 \sigma + \mathrm{O}(\sigma^2), \label{e: leading edge spatial eval expansion intro}
\end{align}
where $\nufr^1 > 0$. The spatial eigenvalues $\nufr^\pm(\sigma)$ cross the imaginary axis precisely when $\lambda = \sigma^2$ crosses $\Sigma(\mcl_+)$. The remaining third eigenvalue $\nu_3(\sigma^2)$ of $\mathcal{M}(\sigma^2)$ is analytic in $\sigma^2$, satisfies $\Re \nu_3 (\sigma^2) > 0$ for $|\sigma|$ small, and is uniformly bounded away from the imaginary axis for $|\sigma|$ small. 
\end{corollary}
\begin{proof}
The expansion~\eqref{e: leading edge spatial eval expansion intro} follows from Lemma~\ref{l: linear spreading speed} and the Newton polygon. The fact that $\Re \nu_3 (\sigma^2) > 0$ follows from a direct computation. 
\end{proof}

\subsection{Essential spectrum and Fredholm properties of \texorpdfstring{$\lfr$}{lfr}}

Fredholm properties of the linearization about a traveling front may be determined through properties of the asymptotic rest states through Palmer's theorem~\cite{Palmer2}; see for instance any of~\cite{KapitulaPromislow, SandstedeReview, FiedlerScheel} for a review of spectral properties of traveling waves. Specifically, in our setting, the operator $\lfr: W^{2,p} (\R, \C) \times W^{1,p}(\R, \C) \subset L^p(\R, \C^2) \to L^p (\R, \C^2)$ is Fredholm if and only if the limiting rest states $\u = 0$ and $\uwt$ are \emph{hyperbolic}, that is, if and only if $\mcl_+$ has no purely imaginary spatial eigenvalues and $\lwt$ has no purely imaginary spatial Floquet exponents. If this is the case, then the Fredholm index of $\lfr$ is
\begin{align}
    \mathrm{ind}(\lfr) = i_M (\lwt) - i_M(\mcl_+), \label{e: fredholm index}
\end{align}
where the asymptotic Morse indices $i_M(\lwt)$ and $i_M(\mcl_+)$ measure the number of unstable spatial Floquet exponents of $\lwt$ and the number of unstable spatial eigenvalues of $\mcl_+$, respectively. 

\begin{prop}\label{p: fr essential spectrum}
Assume $a, \gamma,$ and $\eps$ are such that Hypotheses~\ref{hyp: wt spectral} and~\ref{hyp: wake spectral curve expansion} and Lemma~\ref{l: linear spreading speed} hold. Then the essential spectrum of $\lfr$ is marginally stable, touching the imaginary axis only at the origin. In particular, in a neighborhood of the origin the rightmost boundary of $\Sigma_\mathrm{ess}(\lfr)$ coincides with $\Sigma(\lwt)$. 
\end{prop}
\begin{proof}
This follows from a short computation using the formula~\eqref{e: fredholm index}, Proposition~\ref{prop: wake spectral curve expansion} and Lemma~\ref{l: linear spreading speed}; see the right panel of Figure~\ref{fig: spectrum} for a depiction of the configuration of spatial Floquet exponents and spatial eigenvalues. 
\end{proof}

We can also use Palmer's theorem to determine Fredholm properties of $\lfr$ acting on exponentially weighted spaces, as follows. 
\begin{prop}\label{p: fredholm properties}
Fix $\eta > 0$ sufficiently small. Recall from Section~\ref{s: function spaces} the notation $X^p_\eta = L^p_{\mathrm{exp}, -\eta, \eta}(\R, \C^2)$ and $Y^p_\eta = W^{2,p}_{\mathrm{exp}, -\eta, \eta} (\R, \C) \times W^{1,p}_{\mathrm{exp}, -\eta, \eta} (\R, \C)$. The operator
\begin{align*}
    \lfr : Y^p_\eta \subset X^p_\eta \to X^p_\eta 
\end{align*}
is Fredholm with index -2. 
\end{prop}
\begin{proof}
Conjugating with exponential weights and using Palmer's theorem, one finds $\lfr$ is Fredholm on these spaces, with index given by
\begin{align*}
    \mathrm{ind}(\lfr: Y^p_\eta \to X^p_\eta) = i_M^{-\eta} (\lwt) - i_M^{\eta} (\mcl_+),
\end{align*}
where $i_M^{-\eta}(\lwt)$ is the number of spatial Floquet exponents $\nu$ of $\lwt$ with $\Re \nu > \eta$, and $i_M^{\eta}(\mcl_+)$ is the number of spatial eigenvalues $\nu$ of $\mcl_+$ with $\Re \nu > -\eta$. Using Corollary~\ref{c: leading edge spatial eval} and Proposition~\ref{prop: wake spectral curve expansion}, we find
\begin{align*}
    \mathrm{ind}(\lfr: Y^p_\eta \to X^p_\eta) = 1 - 3 = -2, 
\end{align*}
provided $\eta > 0$ is sufficiently small, as desired; again, see the right panel of Figure~\ref{fig: spectrum} for a depiction of the configuration of spatial Floquet exponents and spatial eigenvalues. 
\end{proof}

Throughout the remainder of the paper, we assume that the parameters $a, \gamma,$ and $\eps$ are such that Hypotheses~\ref{hyp: wt spectral} through~\ref{hyp: point spectrum} together with Lemma~\ref{l: linear spreading speed} hold. 

\section{Contour integral representation of semigroup}\label{s: contour rep}

Standard results from semigroup theory imply that for any fixed non-negative integer $k$, $\lfr$ generates a strongly continuous semigroup on $H^k(\R,\C^2)$. 

\begin{prop}\label{p: Zk generation}
Fix a non-negative integer $k$. The operator 
\begin{align*}
\lfr \colon H^{k+2} (\R, \C) \times H^{k+1} (\R, \C) \to H^k (\R, \C^2)
\end{align*}
is closed and densely defined, and generates a strongly continuous semigroup $\re^{\lfr t}$ on $H^k(\R, \C^2)$. Furthermore, there exist constants $M, \mu_* > 0$ such that $\| \re^{\lfr t} \|_{H^k \to H^k} \leq M \re^{\mu_* t}$ holds for $t \geq 0$. 
\end{prop}
\begin{proof}
The principal part of $\lfr$ is the diagonal diffusion-advection operator
\begin{align*}
    L_0 = \begin{pmatrix}
   \partial_{\xi \xi} + (c+a_1) \partial_\xi & 0 \\
     0 & c \partial_\xi 
     \end{pmatrix}
\end{align*}
on $H^k(\R,\C^2)$ with dense domain $H^{k+2} (\R, \C) \times H^{k+1} (\R, \C)$, where we denote $a_1 = 2\omega(\omega^{-1})'$. Since $a_1$ is smooth and bounded, standard theory of analytic semigroups, cf.~\cite[Theorem~3.1.3]{Lunardi}, yields that $\partial_{\xi \xi} + (c+ a_1)\partial_\xi$ is sectorial on $H^k(\R)$ with domain $H^{k+2}(\R)$, and hence generates an analytic semigroup, which is in particular strongly continuous~\cite[p.~34]{Lunardi}. On the other hand, since $c\partial_\xi$ is a skew-adjoint operator on the Hilbert space $H^k(\R)$ with dense domain $H^{k+1}(\R)$, Stone's Theorem, cf.~\cite[Theorem~II.3.24]{EngelNagel}, implies that it generates a $C_0$-semigroup. We conclude that $L_0$ is the generator of a strongly continuous semigroup on $H^k(\R,\C^2)$. Since $\lfr$ is a bounded perturbation of $L_0$, the same holds for $\lfr$ by~\cite[Theorem~III.1.3]{EngelNagel}. The existence of the constants $M,\mu_* > 0$ follows from~\cite[Proposition~I.5.5]{EngelNagel}.
\end{proof}

From now on we consider $\El_{\rm fr}$ as an operator on $H^k(\R,\C^2)$ for a fixed nonnegative integer $k$. Using another standard result, we can write the semigroup $\re^{\lfr t}$ as the inverse Laplace transform of the resolvent $(\lfr - \lambda)^{-1}$, at least for regular enough initial data. 
\begin{corollary}[{\cite[Chapter 1, Corollary 7.5]{Pazy}}]\label{c: inverse laplace} There exists $\eta > 0$ such that for all test functions $\u_0 \in C^\infty_c (\R, \C^2)$, we have the representation
\begin{align}
    \re^{\lfr t} \u_0 = -\frac{1}{2 \pi \ri} \lim_{R \to \infty} \int_{\eta - \ri R}^{\eta+\ri R} \re^{\lambda t} (\lfr - \lambda)^{-1} \u_0 \de  \lambda, \qquad t > 0. \label{e: contour integral representation}
\end{align}
\end{corollary}
We will eventually relax the restriction $\u_0 \in C^\infty_c (\R, \C^2)$ via an approximation argument using that test functions are dense in $H^k(\R,\C^2)$. To extract decay from this contour integral representation, we would like to shift the contour as far left as possible. As a first step, we show that we can shift to a contour which is in the left half-plane except in a neighborhood of the origin. This allows us to extract exponential decay from the high frequency parts of the semigroup. 

\begin{prop}\label{p: first shift}
    For $\u_0 \in C^\infty_c (\R, \R^2)$ and $t > 0$, we have the representation
    \begin{align*}
        \re^{\lfr t} \u_0 = - \frac{1}{2 \pi \ri} \lim_{R \to \infty} \int_{\Gamma^1_R} \re^{\lambda t} (\lfr - \lambda)^{-1} \u_0 \de  \lambda,
    \end{align*}
    where the contour $\Gamma^1_R$ is depicted in the middle panel of Figure~\ref{fig:contour shifting}. Moreover, for each $1 \leq p \leq \infty,$ there exists a constant $C > 0$ such that
    \begin{align*}
        \left\| \lim_{R \to \infty} \int_{\Gamma^{\mathrm{1},+}_R \cup \Gamma^{\mathrm{1},-}_R} \re^{\lambda t} (\lfr - \lambda)^{-1} \g \de  \lambda \right\|_{L^p} \leq C \re^{- \frac{\eps \gamma}{4} t} \| \g \|_{L^p}
    \end{align*}
    for all $t > 0$ and $\g \in C^\infty_c (\R, \R^2)$, where the vertical segments $\Gamma_{R}^{1,\pm}$ are depicted in the middle panel of Figure~\ref{fig:contour shifting}. 
\end{prop}

The main difficulty in proving Proposition~\ref{p: first shift} is in uniformly controlling the resolvent $(\lfr - \lambda)^{-1}$ for $| \Im \lambda| \gg 1$. We do this by computing an expansion of the resolvent in this limit, identifying the first few terms explicitly until we can truncate with an error integrable in $\lambda$, an approach similar to~\cite{MasciaZumbrun}. Since the main idea is intuitive but the details are fairly technical, we carry out this procedure in Appendix~\ref{s: contour shifting}. 

\begin{figure}
    \centering
    \includegraphics[width=1\textwidth]{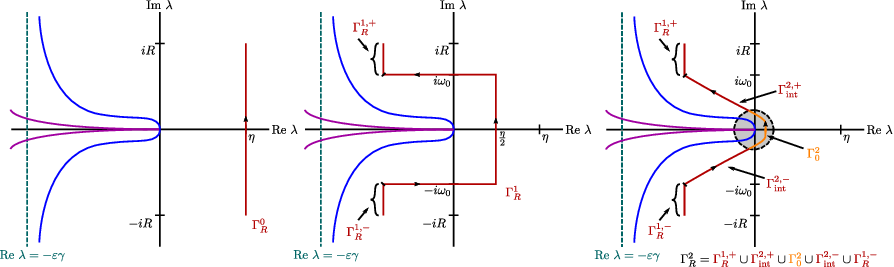}
    \caption{All three panels include the Fredholm borders of $\lfr$, associated to essential spectrum at $-\infty$ (blue) and $+\infty$ (purple). Left: the starting contour $\{ \Re \lambda = \eta, |\Im \lambda| \leq R\}$ from Corollary~\ref{c: inverse laplace}. Middle: the contour $\Gamma^1_R$ (dark red) of Proposition~\ref{p: first shift}, with high frequency parts shifted into the left half-plane. Right: the grey ball shows a neighborhood of the origin in which the resolvent analysis of Section~\ref{s: resolvent} holds; the contour $\Gamma^2_R$ is the disjoint union of the vertical segments $\Gamma_R^{1,\pm}$ (dark red), the diagonal segments $\Gamma_\mathrm{int}^{2,\pm}$ (dark red), and $\Gamma^2_0$ (orange). All parts of $\Gamma^2_R$ are contained in the open left half-plane except for $\Gamma^2_0$. } 
    \label{fig:contour shifting}
\end{figure}

\begin{corollary}\label{c: second shift}
    For $\u_0 \in C^\infty_c (\R, \R^2)$ and $t > 0$, we have the representation
    \begin{align}
        \re^{\lfr t} \u_0 = -\frac{1}{2 \pi \ri} \lim_{R \to \infty} \int_{\Gamma^2_R} \re^{\lambda t} (\lfr - \lambda)^{-1} \u_0 \de  \lambda, \label{e: shifting contour 1}
    \end{align}
    where the contour $\Gamma^2_R = \Gamma^{1,+}_R \cup \Gamma_\mathrm{int}^{2, +} \cup \Gamma^2_0 \cup \Gamma_\mathrm{int}^{2,-} \cup \Gamma_R^{1,-}$ is depicted in the right panel of Figure~\ref{fig:contour shifting}. Note that $\Gamma^{1,\pm}_R$ and $\Gamma_\mathrm{int}^{2, \pm}$ lie in the open left half-plane. Hence, for each $1 \leq p \leq \infty$ there exist constants $C, r > 0$ such that 
  \begin{align}
        \left\| \lim_{R \to \infty} \int_{\Gamma^{1,+}_R \cup \Gamma^{2,+}_\mathrm{int} \cup \Gamma^{2, -}_\mathrm{int} \Gamma^{1,-}_R} \re^{\lambda t} (\lfr - \lambda)^{-1} \g \de  \lambda \right\|_{L^p} \leq C \re^{-  r t} \| \g \|_{L^p} \label{e: shifting contour estimate}
    \end{align}
    for all $t > 0$ and $\g \in C^\infty_c (\R, \R^2)$.
\end{corollary}
\begin{proof}
    By the spectral stability captured in Proposition~\ref{p: fr essential spectrum} and Hypothesis~\ref{hyp: point spectrum}, we can homotope $\Gamma^1_R$ to $\Gamma^2_R$ while remaining in the resolvent set of $\lfr$, so~\eqref{e: shifting contour 1} follows. The estimate~\eqref{e: shifting contour estimate} follows from Proposition~\ref{p: first shift} together with the fact that the compact portion $\Gamma^{2,+}_\mathrm{int} \cup \Gamma^{2,-}_\mathrm{int}$ is contained both in the left half-plane, bounded away from the imaginary axis, and in the resolvent set of $\lfr$.
\end{proof}

\section{Abstract linear estimates and choice of contours near the origin}\label{s: abstract linear}
One observes from Corollary~\ref{c: second shift} that the only part of the semigroup that does not decay exponentially is the contribution from $\Gamma^2_0$. This is due to two curves of essential spectrum of $\lfr$ touching the origin, one associated to the branched diffusive mode in the leading edge, and the other associated to the outgoing diffusive mode in the wake. In Section~\ref{s: resolvent} we decompose the resolvent in a neighborhood of the origin, identifying terms associated to the branched mode, terms associated to the outgoing mode, and terms associated to the interaction of both modes. In this section, we choose integration contours which are adapted to the expected behavior of these three type of terms and use these contours to prove abstract linear decay estimates. We will combine these estimates with the resolvent analysis and decomposition, carried out in Section~\ref{s: resolvent}, to establish sharp bounds on the semigroup $\re^{\lfr t}$ in Section~\ref{s: linear estimates}. 

\subsection{Regions of analyticity near the origin}\label{s: analyticity regions}

We identify regions near the origin in which parts of the resolvent associated to the branched or outgoing diffusive mode, or parts associated to their interaction, can be rendered analytic in $\lambda$. Our choice of integration contours for the different type of terms arising in the resolvent decomposition will thus be determined by these regions. Some critical parts of the resolvent can be represented as a convolution with a suitable integral kernel (or spatial Green's function), allowing one to consider the analyticity properties of the integral kernel after swapping integrals.

We distinguish between regions of pointwise analyticity and regions of analyticity in a fixed function space (typically $L^2$- or $L^\infty$-spaces). In regions where certain parts of the resolvent (or their associated integral kernels) are pointwise analytic, but perhaps not analytic in fixed function spaces due to loss of spatial localization, we can deform our integration contour pointwise and employ pointwise estimates. We explain in Section~\ref{sec: estimates outgoing and interaction} below the need for pointwise estimates, i.e., why normed estimates are insufficient to obtain sharp bounds on certain parts of the resolvent. 

\begin{figure}
    \centering
    \includegraphics[width=1\textwidth]{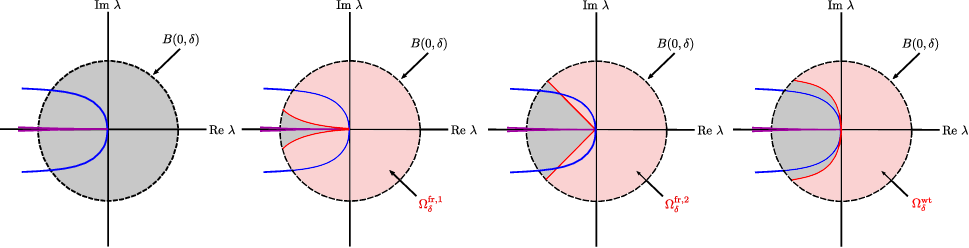}
    \caption{Far left: Fredholm borders of $\lfr$ (blue, purple) in a neighborhood $B(0,\delta)$ (shaded, grey) of the origin. Middle left: the region $\Omega^{\mathrm{fr},1}_\delta$ (shaded, red) defined in~\eqref{e: Omega fr 1 def} . Middle right: the region $\Omega^{\mathrm{fr},2}_\delta$ (shaded, red) defined in~\eqref{e: Omega fr 2 def}. Far right: the region $\Omega^\mathrm{wt}_\delta$ (shaded, red) defined in~\eqref{e: Omega wt def}.}
    \label{fig: resolvent regions}
\end{figure}

\paragraph{Terms associated to branched diffusive modes.} The essential spectrum in the leading edge gives rise to branched diffusive modes, for which the associated parts of the resolvent are roughly of the form
\begin{align*}
    \u_+(\xi; \sigma) \sim \alpha_+(\sigma) \chi_+ (\xi) \re^{\nufr^- (\sigma) \xi},
\end{align*}
where we denote $\sigma = \sqrt{\lambda}$ with branch cut along the negative real axis, and $\alpha_+(\sigma)$ is some analytic function of $\sigma$ defined on a neighborhood of the origin. We recall that $\chi_+(\xi)$ is a smooth positive cutoff function supported on $\{ \xi \geq 0 \}$, and $\nufr^-(\sigma)$ is the spatial eigenvalue from Corollary~\ref{c: leading edge spatial eval}, which is analytic in $\sigma$ on a neighborhood of the origin. The obstructions to analyticity of $\u_+$ are twofold:
\begin{enumerate}
    \item Due to the branch point in $\alpha_+(\sigma)$ and $\nufr^-(\sigma)$ at $\lambda = 0$, $\u_+(\xi; \sigma)$ is only pointwise analytic in $\lambda = \sigma^2$ away from the negative real axis.
    \item Since $\Re \nufr^-(\sigma)$ changes sign as $\lambda = \sigma^2$ crosses $\Sigma(\mcl_+)$, $\u_+(\xi; \sigma)$ loses spatial localization as $\lambda$ passes through $\Sigma(\mcl_+)$, and hence loses analyticity in a fixed function space. 
\end{enumerate}
The first obstruction prevents us, even in a pointwise sense, from using contours which pass through the branch point at the origin. The second obstruction prevents us from passing $\Sigma(\mcl_+)$ with our integration contours when measuring in a fixed function space. 

What contours we may use depends on what space we want to measure in. In our nonlinear argument, we will bound the solution (and some derivatives) in both the $L^2$- and $L^\infty$-norms. Taking the $L^2$-norm, we find
\begin{align}
    \| \u_+(\cdot; \sigma) \|_{L^2} \sim \frac{1}{\sqrt{\Re \nufr^-(\sigma)}} \sim \frac{1}{\sqrt{\Re \sigma}}. \label{e: leading edge blowup}
\end{align}
To extract temporal decay, we want to move the contours as far left as possible. Since we can't pass through the branch point, here this means we want to shrink the contour into as small of a neighborhood of the origin as possible. By~\eqref{e: leading edge blowup}, $\u_+(\cdot; \sigma)$ blows up in $L^2(\R)$ as $\sigma$ approaches zero. Nonetheless, if we restrict $\lambda = \sigma^2$ to a sector 
\begin{align}
    \Omega^{\mathrm{fr},2}_\delta = \left\{r \re^{\ri \phi} : 0 < r < \delta, -\theta_0 < \phi < \theta_0 \right\}, \label{e: Omega fr 2 def}
\end{align}
for some $\delta,\theta_0 > 0$, then there exists a constant $C > 0$ such that $\Re \sigma \geq C |\sigma|$ for all $\sigma \in \Omega^{\mathrm{fr},2}_\delta$; see Figure~\ref{fig: resolvent regions}. Then, for $\sigma \in \Omega^{\mathrm{fr},2}_\delta$ we can quantify the rate of blowup of the resolvent as 
\begin{align*}
    \| \u_+(\cdot; \sigma) \|_{L^2} \sim \frac{1}{ |\sigma|} = \frac{1}{|\lambda|^{1/2}}.
\end{align*}
Hence, when estimating branched diffusive modes in $L^2(\R)$, we will use contours which remain in $\Omega^{\mathrm{fr},2}_\delta$, but can now be shrunk arbitrarily small, since the blowup rate $|\lambda|^{-1/2}$ is integrable in $\lambda$. 

On the other hand, we find that $\| \u_+(\cdot; \sigma) \|_{L^\infty}$ remains bounded all the way up to $\Sigma(\mcl_+)$, at which point $\Re \nufr^-(\sigma)$ becomes zero. If we measure in the even weaker $L^\infty_{0, -1}$-norm, then by Taylor expansion, one sees that $\u_+$ obeys the Lipschitz estimate $\| \u_+ (\cdot; \sigma) - \u_+ (\cdot; 0) \|_{L^\infty_{0, -1}} \lesssim |\sigma|$ for $\sigma$ to the right of $\Sigma(\mcl_+)$. We demonstrate in Proposition~\ref{p: front estimate sharp decay} that such a Lipschitz estimate leads to faster decay compared to decay estimates in norms in which $\u_+(\cdot,\sigma)$ is merely bounded in $\sigma$. 

All in all, we will obtain the fastest decay estimates for branched modes when measuring in $L^\infty_{0,-1}(\R)$, and when measuring in this space, we can use contours which pass into a larger region than $\Omega^{\mathrm{fr},2}_\delta$, as long as we remain to the right of $\Sigma(\mcl_+)$. We therefore define the curve
\begin{align*}
    K^{\mathrm{fr},1} = \left\{ -k^2 + \ri C^{\mathrm{fr},3} k^3 : k \in \R \right\}.
\end{align*}
Choosing $C^{\mathrm{fr},3} > 0$ sufficiently large, we can guarantee that in a small enough neighborhood of the origin, $K^{\mathrm{fr},1}$ lies to the right of the essential spectrum of $\mcl_+$. Thus, given $\delta > 0$, we define
\begin{align}
    \Omega^{\mathrm{fr},1}_\delta = \left\{ \lambda \in B(0, \delta) : \lambda \text{ lies to the right of } K^{\mathrm{fr},1} \right\},  \label{e: Omega fr 1 def}
\end{align} 
see Figure~\ref{fig: resolvent regions}. When measuring in $L^\infty_{0,-1}(\R)$, we will use contours which remain in the larger region $\Omega^{\mathrm{fr},1}_\delta$. 

\paragraph{Terms associated to the outgoing diffusive mode.} The parts of the resolvent that are associated to the outgoing diffusive mode only are either roughly of the form
\begin{align} \label{e: model term outgoing mode 2}
    \u_\mathrm{out} (\xi; \lambda) \sim  \alpha_\mathrm{out} (\lambda) \chi_-(\xi) \re^{\nuwt(\lambda) \xi},
\end{align}
or can be represented as a convolution with a suitable integral kernel and are then roughly of the form
\begin{align} \label{e: model term outgoing mode}
    \u_\mathrm{out} (\xi; \lambda) \sim  \int_\R \alpha_\mathrm{out} (\lambda) \chi_-(\xi) \chi_-(\xi-\zeta) \re^{\nuwt(\lambda) (\xi-\zeta)} Q(\zeta) \de \zeta,
\end{align}
where $\alphaout(\lambda)$ is some function which is analytic in $\lambda$ in a full neighborhood of the origin and $Q(\zeta)$ is a bounded function of $\zeta$. We recall that $\chi_-(z)$ is a smooth cutoff function supported on $\{ z \leq 0 \}$ and $\nuwt(\lambda)$ is the spatial Floquet exponent from Proposition~\ref{prop: wake spectral curve expansion}, which is analytic in $\lambda$ in a neighborhood of the origin. The only obstruction to analyticity is that $\chi_-(z) \re^{\nuwt (\lambda) z}$ loses spatial localization as $\lambda$ crosses $\Sigma(\lwt)$, since $\Re \nuwt(\lambda)$ then changes sign by Proposition~\ref{prop: wake spectral curve expansion}. After swapping the convolution integral with respect to $\zeta$ with the complex line integral with respect to $\lambda$ in case of the representation~\eqref{e: model term outgoing mode}, we may deform integration contours pointwise, i.e., we may let the choice of integration contour depend on the values of $\xi$ and $t$ in case of~\eqref{e: model term outgoing mode 2} and on $\xi,\zeta$, and $t$ in case of~\eqref{e: model term outgoing mode}. Since this only requires pointwise analyticity, we can pass through $\Sigma(\lwt)$. Proceeding as in~\cite{ZumbrunHoward, MasciaZumbrun}, we make judicious choices for our pointwise integration contours to extract sharp spatio-temporal localization of the inverse Laplace transform of~\eqref{e: model term outgoing mode 2} or of the inverse Laplace transform of the integral kernel associated to~\eqref{e: model term outgoing mode}, which eventually yields sharp normed estimates on the corresponding parts of the semigroup. 

\paragraph{Interaction terms.} Terms which capture the interaction between the branched and outgoing modes are roughly of the form
\begin{align} \label{e: model interaction term}
    \uint (\xi; \sigma) \sim \alphaint(\sigma) \chi_-(\xi) \re^{\nuwt (\sigma^2) \xi}, 
\end{align}
where $\alphaint(\sigma)$ is some analytic function of $\sigma = \sqrt{\lambda}$ defined on a full neighborhood of the origin, which satisfies $\alphaint(0) = 0$ without loss of generality.\footnote{If $\alphaint(\sigma)$ does not vanish at $\sigma = 0$, we can always subtract $\alphaint(0)\chi_-(\xi) \re^{\nuwt (\sigma^2) \xi}$ from~\eqref{e: model interaction term}, which is a term of the form~\eqref{e: model term outgoing mode 2}, associated to the outgoing mode only.} We cannot estimate these terms using the pointwise contours of~\cite{ZumbrunHoward, MasciaZumbrun}. Indeed, these contours pass through the origin which is not allowed here due to the branch point of $\alphaint(\sigma)$ at $\lambda = 0$. Pointwise contours designed to estimate terms like~\eqref{e: model interaction term}, which thereby avoid passing through the branch point, were developed in~\cite{HowardDegenerateLinear}. However, instead of using the delicate pointwise contours of~\cite{HowardDegenerateLinear}, we obtain linear estimates sufficient to establish sharp nonlinear decay rates by measuring $\uint$ in norm using simple contours which remain to the right of $\Sigma(\lwt)$. 
We therefore define a curve
\begin{align*}
    K^\mathrm{wt} = \left\{ \ri k - C^\mathrm{wt} k^2 : k \in \R \right\}.
\end{align*}
Provided $C^\mathrm{wt} > 0$ is sufficiently small, this curve lies to the right of $\Sigma(\lwt)$, in a small enough neighborhood of the origin. Given $\delta > 0$ small, we then define the region
\begin{align}
    \Omega^\mathrm{wt}_\delta = \left\{ \lambda \in B(0,\delta) : \lambda \text{ lies to the right of } K^{\mathrm{wt}} \right\}. \label{e: Omega wt def}
\end{align}
We will restrict our contours to the region $\Omega^\mathrm{wt}_\delta$; see Figure~\ref{fig: resolvent regions}. A risk of our approach is that normed estimates do not always afford sharp bounds on terms associated to outgoing diffusive modes, as explained in~\S\ref{sec: estimates outgoing and interaction} below. Nevertheless, the observation that $\alphaint(\sigma)$ vanishes at $
\sigma = 0$ yields an additional factor $\sigma$ in~\eqref{e: model interaction term}, rendering additional temporal decay which we can exploit to compensate for the use of normed estimates.

\subsection{Estimates for branched modes} 
Fix $\delta > 0$ small, and let $\Omega^\mathrm{fr, 1}_\delta$, $\Omega^\mathrm{fr, 2}_\delta$, and $\Omega^\mathrm{wt}_\delta$ be defined by~\eqref{e: Omega fr 1 def},~\eqref{e: Omega fr 2 def}, and~\eqref{e: Omega wt def}, respectively. Moreover, let $\Delta^{\mathrm{fr},j}_\delta$ be the image of $\Omega^{\mathrm{fr},j}_\delta$ under the principal square root $\lambda \mapsto \sqrt{\lambda}$ for $j = 1,2$. We start by establishing normed estimates which will be used for parts of the resolvent associated to branched modes. 

\begin{prop}[Normed estimates with sharp decay rates for branched modes]\label{p: front estimate sharp decay}
Let $X$ be a Banach space, and suppose that for some small $\delta > 0$, we have a function $\sigma \mapsto u(\sigma) : \Delta_\delta^{\mathrm{fr},1} \to X$ which is analytic in $\sigma^2$ on $\Omega_\delta^{\mathrm{fr},1}$, and extends Lipschitz-continuously to $\sigma = 0$ with constant $L > 0$ such that
    \begin{align}
        \|{\color{black} u(\sigma)} - u(0)\|_X \leq L |\sigma|, \label{e: branched lipschitz}
    \end{align}
    for $\sigma \in \Delta_\delta^{\mathrm{fr},1}$.
Then, there exists a constant $C > 0$ such that for all $t > 0$, we have
    \begin{align*}
        \left\| \int_{\Gamma^2_0} \re^{\sigma^2 t} u(\sigma) \de  (\sigma^2) \right\|_X \leq \frac{C}{(1+t)^{\frac32}}. 
    \end{align*}
\end{prop}
\begin{proof}
    The proof is identical to~\cite[Proof of Proposition 4.1]{AveryScheel}, but we sketch it here for completeness. We will integrate in $\sigma$ rather than in $\lambda = \sigma^2$. Using analyticity of $\sigma^2 \mapsto {\color{black} u(\sigma)}$ on $\Omega^{\mathrm{fr},1}_\delta$ and continuity up to $\sigma = 0$, we can shift to an integration contour which runs tangent to the imaginary axis in the $\sigma$-plane, given by
    \begin{align*}
        \tilde{\Gamma}^{\mathrm{fr}, 1} = \tilde{\Gamma}^{\mathrm{fr},1}_- \cup \tilde{\Gamma}^{\mathrm{fr},1}_0 \cup \tilde{\Gamma}^{\mathrm{fr},1}_+,
    \end{align*}
    where
    \begin{align}
        \tilde{\Gamma}^{\mathrm{fr},1}_0 = \{\sigma(a) : a \in [-a_*, a_*] \}, \qquad  \sigma(a) := \ri a + c_2 a^2, \label{e: branched contour parameterization}
    \end{align}
    where $c_2 > 0$ is chosen sufficiently large and $a_* > 0$ is chosen sufficiently small so that the image $\Gamma^{\mathrm{fr,1}}_0$ of $\tilde{\Gamma}^{\mathrm{fr},1}_0$ under $\sigma \mapsto \sigma^2$ lies within $\Omega^{\mathrm{fr},1}_\delta$, touching its boundary only at the origin. Let $\Gamma^{\mathrm{fr,1}}_\pm$ be straight line segments which connect the ends of $\Gamma^{\mathrm{fr},1}_0$ to points on the boundary of $B(0, \delta)$ which are in the left half-plane but to the right of $\Sigma(\lfr)$. Let $\tilde{\Gamma}^{\mathrm{fr,1}}_\pm$ denote the images of $\Gamma^{\mathrm{fr,1}}_\pm$ under the map $\lambda \mapsto \sqrt{\lambda}$. See Figure~\ref{fig: small contours} for a depiction of the integration contour $\Gamma^{\mathrm{fr},1} = \Gamma^{\mathrm{fr},1}_- \cup \Gamma^{\mathrm{fr},1}_0 \cup \Gamma^{\mathrm{fr},1}_+$ in the $\lambda$-plane. 
    
    Integrating in $\sigma$ and deforming to this new contour, we then have
    \begin{align*}
        \int_{\Gamma^2_0} \re^{\sigma^2 t} u(\sigma) \de  (\sigma^2) = \int_{\tilde{\Gamma}^{\mathrm{fr},1}_0} \re^{\sigma^2 t} u(\sigma) 2 \sigma \de  \sigma + \sum_{\iota = \pm} \int_{\tilde{\Gamma}^{\mathrm{fr},1}_\iota} \re^{\sigma^2 t} u(\sigma) 2 \sigma \de  \sigma. 
    \end{align*}
    Since $\Gamma^{\mathrm{fr},1}_\pm$ lie in the left half-plane, and $u(\sigma)$ is uniformly bounded along these segments, the contributions from these portions of the contour decay exponentially in time, so we focus on the first integral. We use~\eqref{e: branched lipschitz} to expand
    \begin{align*}
        \int_{\tilde{\Gamma}^{\mathrm{fr},1}_0} \re^{\sigma^2 t} u(\sigma) 2 \sigma \de  \sigma = 2 \int_{\tilde{\Gamma}^{\mathrm{fr},1}_0} \re^{\sigma^2 t} u(0)  \sigma \de  \sigma + \int_{\tilde{\Gamma}^{\mathrm{fr},1}_0} \re^{\sigma^2 t} (u(\sigma)-u(0))  \sigma \de  \sigma.
    \end{align*}
    Since $u(0)$ does not depend on $\sigma$, we can change variables in the first integral to integrate again over $\sigma^2$, with integrand analytic in $\sigma^2$ in a full neighborhood of the origin, so that we may shift the integration contour fully into the left half-plane and thereby extract exponential decay. For the second integral, we use~\eqref{e: branched contour parameterization} to explicitly parameterize the integral as
    \begin{align*}
        \int_{\tilde{\Gamma}^{\mathrm{fr},1}_0} \re^{\sigma^2 t} (u(\sigma)-u(0))  \sigma \de  \sigma = \int_{-a_*}^{a_*} \re^{(-a^2 + c_2^2 a^4) t} \re^{2 \ri c_2 a^3 t} (u(\ri a + c_2 a^2) - u(0)) (\ri a + c_2 a^2) (\ri + 2 c_2 a) \de a. 
    \end{align*}
    Using the estimate~\eqref{e: branched lipschitz} and choosing $a_*$ sufficiently small, we arrive at
    \begin{align*}
        \left\| \int_{\tilde{\Gamma}^{\mathrm{fr},1}_0} \re^{\sigma^2 t} (u(\sigma)-u(0))  \sigma \de  \sigma \right\| \lesssim \int_{-a_*}^{a_*} \re^{- a^2 t/2} |a|^2 \de a \lesssim \frac{1}{(1+t)^{\frac32}}, 
    \end{align*}
    for $t > 0$, where the last inequality follows from applying the change of variables $z = a \sqrt{t}$ for $t > 1$, whereas for $t \in [0,1]$ we use that $|a|^2 \re^{-a^2 t/2} \leq a_*^2$ for $a \in (-a_*,a_*)$. 
\end{proof}

\begin{figure}
    \centering
    \includegraphics[width=1\textwidth]{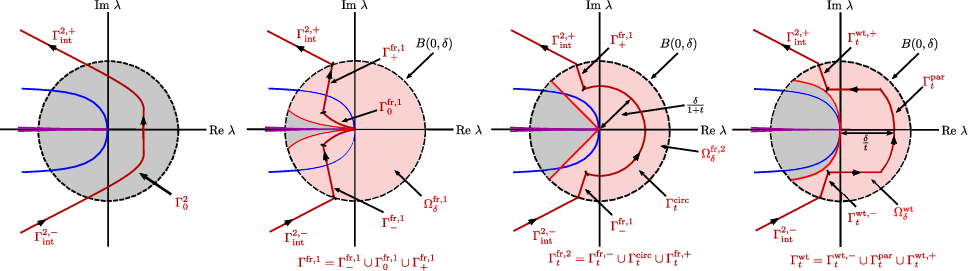}
    \caption{Far left: the original contour $\Gamma^2_0$ (dark red) in a neighborhood $B(0, \delta)$ (grey, shaded) of the origin. Middle left: the contour $\Gamma^{\mathrm{fr},1}$ (dark red) used in the proof of Proposition~\ref{p: front estimate sharp decay}. Middle right: the contour $\Gamma^{\mathrm{fr},2}_t$ used in the proof of Proposition~\ref{p: L2 front estimate}. Far right: the contour $\Gamma^\mathrm{wt}_t$ (dark red) used in the proof of Proposition~\ref{p: wake norm estimates}.}
    \label{fig: small contours}
\end{figure}

We will only be able to obtain the estimate~\eqref{e: branched lipschitz} and the associated $t^{-3/2}$ decay rate when measuring in fairly weak norms. When measuring in $L^2(\R)$, as needed for our nonlinear argument, the resolvent will instead blowup at $\sigma = 0$ as explained in~\S\ref{s: analyticity regions}, but with a quantified rate in the region $\Omega^{\mathrm{fr},2}_\delta$. The next estimate is adapted to precisely this scenario. 

\begin{prop}[Estimates for branched modes in stronger norms]\label{p: L2 front estimate}
	Let $X$ be a Banach space, and suppose that for some $\delta > 0$ small, we have a map $\sigma \mapsto u(\sigma) : \Delta_\delta^{\mathrm{fr},2} \to X$, which is analytic in $\sigma^2$ on $\Omega_\delta^{\mathrm{fr},2}$, and there exist $r \in [0,2)$ and $K > 0$ such that the blowup estimate
	\begin{align}
	\| {\color{black} u(\sigma)} \| \leq \frac{K}{|\sigma|^r} \label{e: blow up estimate} 
	\end{align}
	holds for $\sigma \in \Delta_ \delta^{\mathrm{fr},2}$. Then, there exists a constant $C > 0$ such that for all $t > 0$, we have
	\begin{align*} 
	\left\| \int_{\Gamma^2_0} \re^{\sigma^2 t} u(\sigma) d (\sigma^2) \right\|_X \leq \frac{C}{(1+t)^{1-\frac{r}{2}}}. 
	\end{align*}
\end{prop}
\begin{proof}
This is a standard argument, but we sketch a proof here for completeness. Using analyticity of $\sigma^2 \mapsto {\color{black} u(\sigma)}$ on $\Omega^{\mathrm{fr},2}_\delta$, for any fixed $t > 0$, we can deform the integration contour from $\Gamma^2_0$ to the $t$-dependent integration contour $\Gamma^{\mathrm{fr},2}_t =   \Gamma^{\mathrm{fr},+}_t \cup \Gamma^\mathrm{circ}_t \cup \Gamma^{\mathrm{fr,-}}_t$ lying in $\Omega^{\mathrm{fr},2}_\delta$, where
$\Gamma^{\mathrm{fr},\pm}_t$ are line segments connecting the ends of  $\Gamma^{2,\pm}_\mathrm{int}$ to points $\delta \re^{\pm\ri(\theta_0-\tilde{\delta})}/(1+t)$, where $\tilde{\delta} > 0$ is chosen such that the extension of $\Gamma^{\mathrm{fr},\pm}_t$ passes through the origin. We note here that $\tilde{\delta}$ only depends on $\Gamma^{2,\pm}_{\mathrm{int}}$ and is thus independent of $t$. Finally, $\Gamma^{\mathrm{circ}}_t$ is the circle segment
\begin{align*}
	   \Gamma^\mathrm{circ}_t = \left\{ \frac{\delta}{1+t} \re^{\ri \theta} : \theta \in \left[-\theta_0+\tilde{\delta}, \theta_0 - \tilde{\delta} \right] \right\};
\end{align*}
see Figure~\ref{fig: small contours}. Note that if $\sigma^2 \in \Gamma^\mathrm{circ}_t$, then we have $|\sigma|^2 = \frac{\delta}{1+t}$ and $\Re \sigma^2 \leq \frac{\delta}{1+t}$. Hence, using~\eqref{e: blow up estimate} and the fact that $\Gamma_t^\mathrm{circ}$ has length bounded by $2\pi \delta/(1+t)$, we establish
\begin{align*}
 \left\| \int_{{\Gamma}^{\mathrm{circ}}_t} \re^{\sigma^2 t} u(\sigma) \de( \sigma^2)\right\|_X \lesssim \exp \left(\frac{\delta t}{1+t}\right) \left(\frac{\delta}{1+t}\right)^{1-\frac{r}2} \lesssim \frac{C_2}{(1+t)^{1-\frac{r}{2}}},
\end{align*}
for $t > 0$. On the other hand, explicitly parameterizing the line segment $\Gamma_t^{\mathrm{fr},+}$ which lies in the open left half-plane and using~\eqref{e: blow up estimate} we obtain
\begin{align*}
 \left\| \int_{{\Gamma}^{\mathrm{fr},+}_t} \re^{\sigma^2 t} u(\sigma) \de( \sigma^2)\right\|_X \lesssim \int_{\frac{\delta}{1+t}}^\delta
 \exp\left(\cos(\theta_0 - \tilde{\delta})
  st\right) |s|^{-\frac{r}2} \de s \lesssim \frac{1}{(1+t)^{1-\frac{r}{2}}},
\end{align*}
for $t > 0$, where the last inequality follows for $t > 1$ from the change of variables $z = st$ and the fact that the integral $\int_0^\infty \re^{-az} z^{-r/2} \de z$ is bounded for $a > 0$ and $r \in [0,2)$, whereas for $t \in [0,1]$ we use $\exp(\cos(\theta_0 - \tilde{\delta}) st) \leq 1$. The integral over $\Gamma_t^{\mathrm{fr},-}$ is estimated analogously.
\end{proof}

\subsection{Estimates for the outgoing mode} \label{sec: estimates outgoing and interaction}

For estimating those parts of the resolvent associated to outgoing diffusive modes only, we may use either pointwise contours which pass fully into the essential spectrum, or contours which remain in the region $\Omega^\mathrm{wt}_\delta$. The latter turn out to give sharp estimates only in certain norms; see the discussion below. Nonetheless, for terms in the resolvent decomposition which capture the interaction between outgoing and branched modes, we must use contours which remain in $\Omega^\mathrm{wt}_\delta$ as pointed out in~\S\ref{s: analyticity regions}. We capture decay using contours of this type in the following result. 

\begin{prop}[Normed estimates for outgoing modes]\label{p: wake norm estimates}
	Let $X$ be a Banach space, and suppose that for some $\delta > 0$ small, we have an analytic function $\lambda \mapsto {\color{black}u(\lambda)} : \Omega^\mathrm{wt}_\delta \to X$, and there exists $r > -1$ such that
	\begin{align*}
	\| {\color{black}u(\lambda)} \|_X \leq C |\lambda|^{r}
	\end{align*}
	for all $\lambda \in \Omega^\mathrm{wt}_\delta$. Then, there exists a constant $C > 0$ such that for all $t > 0$,
	\begin{align*}
	\left\| \int_{\Gamma^2_0} \re^{\lambda t} u (\lambda) \de  \lambda \right\|_X \leq \frac{C}{(1+t)^{\frac12 + \frac{r}2}}.  
	\end{align*}
\end{prop}
\begin{proof}
Estimates of this type have been established and used in~\cite{Kapitula, HowardDegenerateLinear, AveryScheelGL}, but we sketch a proof for completeness. Estimates for small times may be obtained using the fixed contour $\Gamma_0^2$ which remains to the right of the spectrum, so we focus on $t > 1$. Using analyticity of $\lambda \mapsto {\color{black} u(\lambda)}$ on $\Omega^{\mathrm{fr},2}_\delta$, for any fixed $t > 1$, we can deform the integration contour from $\Gamma^2_0$ to the $t$-dependent integration contour $\Gamma^{\mathrm{wt}}_t = \Gamma^{\mathrm{wt},-}_t \cup \Gamma^{\mathrm{par}}_t \cup \Gamma^{\mathrm{wt},+}_t$ lying in $\Omega_\delta^{\mathrm{wt}}$ with
\begin{align*}
\Gamma^{\mathrm{par}}_t = \left\{\frac{\delta}{t} + \ri a - d a^2 : a \in [-a_*, a_*] \right\},
\end{align*}
for some some $t$-independent constants $a_*, d > 0$. Here, $a_* > 0$ is chosen sufficiently small and $d>0$ sufficiently large such that the the contours $\Gamma^{\mathrm{par}}_t$ lie in $B(0,\delta)$ and their end-points are in the open left half-plane for any $t \geq 1$. Moreover,  we let $\Gamma^{\mathrm{wt},\pm}_t$ be the line segments connecting the ends of $\Gamma^{\mathrm{par}}_t$ to $\Gamma_{\mathrm{int}}^{2,\pm}$; see Figure~\ref{fig: small contours}. Since the contours $\Gamma^{\mathrm{wt},\pm}_t$ are contained in the left half-plane and are $t$-uniformly bounded away from the imaginary axis, their contributions are exponentially decaying in time. So, we focus on the integral over $\Gamma^{\mathrm{par}}_t$. For a point $\lambda = \frac{\delta}{t} + \ri a - da^2 \in \Gamma^{\mathrm{par}}_t$ with $a \in [-a_*,a_*]$, we have $|\lambda|^r \leq |\Im \lambda|^r = |a|^r$ in case $r \in (-1,0)$. On the other hand, in case $r \geq 0$, there exists a $t$- and $a$-independent constant $C > 0$ such that we have $|\lambda|^r \leq ((\delta/t - da^2)^2 + a^2)^{r/2} \leq C (t^{-r} + |a|^r)$ for $a \in [-a_*,a_*]$ and $t \geq 1$. Thus, we arrive at
\begin{align*}
	    \left\| \int_{\Gamma^\mathrm{par}_t} \re^{\lambda t} u(\lambda) \de  \lambda \right\|_X \lesssim \int_{-a_*}^{a_*} \re^{\delta - d a^2 t } \left(t^{-\max\{r,0\}} + |a|^r\right) \de  a \lesssim \frac{1}{t^{\frac{1}{2}+\frac{r}{2}}},
	\end{align*}
for $t > 1$, where the last inequality follows from applying the change of variables $z = a\sqrt{t}$.
\end{proof}

\paragraph{Need for pointwise estimates.} We claim that Proposition~\ref{p: wake norm estimates} gives estimates with sharp decay rates when we measure the solution in $L^\infty(\R)$, but not when we measure it, for instance, in $L^2(\R)$. To see this heuristically, recall from~\S\ref{s: intro linear} that a good model for the dynamics in the wake is the advection-diffusion equation $u_t = D_\mathrm{eff} u_{\xi \xi} + c_g u_\xi$ with $D_\mathrm{eff} > 0$ and  $c_g < 0$. Since advection does not affect decay rates in translation invariant norms, we have for this equation the well-known diffusive decay estimates, i.e.,~for $k \in \N_0$ and $p \in [1,\infty]$ the solution $u(t)$ with initial condition $u(0) = u_0$ obeys the bound
\begin{align}
    \| \partial_\xi^k u(t) \|_{L^p} \lesssim \frac{1}{t^{\frac{1}{2}-\frac{1}{2p} + \frac{k}{2}}} \| u_0 \|_{L^1}, \label{e: heat eqn estimates}
\end{align}
for $t > 0$. On the other hand, the resolvent kernel (or spatial Green's function) $G_\lambda$ for this equation, which solves $(D_\mathrm{eff} \partial_\xi^2 + c_g \partial_\xi - \lambda) G_\lambda = - \delta_0$, resembles
\begin{align*}
    G_\lambda(\xi) \sim \chi_- (\xi) \re^{\nuwt(\lambda) \xi},
\end{align*}
cf.~\S\ref{sec: essential spectrum wake}. By simple scaling, we obtain, at best,\footnote{The $\xi$-derivatives may fall on $\chi_-(\xi)$ instead of on $\re^{\nuwt(\lambda) \xi}$, but terms where this happens are better behaved since $\chi_-'$ is compactly supported.}
\begin{align*}
    \| \partial_\xi^k G_\lambda \|_{L^p} \leq C\frac{|\nuwt(\lambda)|^k}{|\Re \nuwt(\lambda)|^{1/p}},
\end{align*}
for some $\lambda$-independent constant $C>0$. Along the parabolic contours used in Proposition~\ref{p: wake norm estimates}, we find $|\nuwt(\lambda)| \sim |\lambda|$ and $|\Re \nuwt(\lambda)| \sim |\lambda|^2$; see Proposition~\ref{prop: wake spectral curve expansion}. Using these bounds one arrives at the estimate
\begin{align*}
    \| \partial_\xi^k G_\lambda \|_{L^p} \lesssim |\lambda|^{k - \frac{2}{p}}, 
\end{align*}
for points $\lambda$ on these contours. Hence, if $\tilde{u}(\lambda)$ solves the resolvent equation $(D_\mathrm{eff} \partial_\xi^2 + c_g \partial_\xi - \lambda) \tilde{u} = g$, then using Young's convolution inequality, we find the estimate
\begin{align*}
    \|\partial_\xi^k \tilde{u}(\lambda) \|_{L^p} \lesssim |\lambda|^{k - \frac{2}{p}} \| g \|_{L^1},
\end{align*}
for points $\lambda$ on these contours. Using Proposition~\ref{p: wake norm estimates}, this then yields, for each  $k \in \N_0$ and $p \in [1,\infty]$,
\begin{align}
    \|\partial_\xi^k u(\cdot, t)\|_{L^p} \lesssim \frac{1}{t^{\frac{1}{2} + \frac{k}{2} - \frac{1}{p}}}, \label{e: model norm estimate}
\end{align}
for $t > 0$. When $p = \infty$, these estimates agree with the sharp estimates of~\eqref{e: heat eqn estimates}, but when $p < \infty$ the decay rates are worse by a factor of $t^{\frac{1}{p} - \frac{1}{2p}}$. The reason is that when we measure in $L^p(\R)$ for $p < \infty$, we must control spatial localization of the resolvent kernel $G_\lambda$. This spatial localization is controlled by $\Re \nuwt(\lambda) \sim \Re \lambda$, but along the parabolic contours used in Proposition~\ref{p: wake norm estimates} we have $\Re \lambda \sim (\Im \lambda)^2,$ which ultimately leads to the discrepancy in decay rates. To close a nonlinear argument, we will need sharp bounds on the solution in the wake at least in $L^p(\R)$ for some $p < \infty$, in order to leverage spatial localization of quadratic terms in the nonlinearity (in our nonlinear argument we choose $p = 2$). We will therefore rely on pointwise estimates for this particular aspect to prove sharp bounds on the solution in $L^2(\R)$. 

\paragraph{Pointwise estimates.} 
The estimates we need were essentially proven in~\cite[Section 8]{ZumbrunHoward}, with an updated statement closer to our exact needs given in~\cite[Section 7, example 2]{MasciaZumbrun}. We give an adapted exposition of the proof from~\cite{MasciaZumbrun} in Appendix~\ref{s: pointwise estimates} for completeness.  

\begin{prop}[Pointwise estimates for outgoing diffusive modes]\label{p: pointwise estimate}
	Suppose $\delta > 0$ is small. Let $\lambda_0$ denote the point at which $\Gamma^{2, +}_\mathrm{int}$ intersects $B(0, \delta)$, and let $\lambda_0^*$ denote its complex conjugate; see Figure~\ref{fig: small contours}. Let $g \colon \R^2 \times B(0,\delta) \to \C^2$ be a function such that
	\begin{align*}
	\lambda \mapsto g(\cdot, \cdot, \lambda) : B(0, \delta) \to L^\infty(\R^2, \C^2)
	\end{align*}
	is analytic and there exist a constant $C > 0$ and an integer $m \geq 0$ such that $\|g(\cdot, \cdot, \lambda) \|_{L^\infty} \leq C |\lambda|^m$. Then, for each $j,\ell \in \N_0$, there exist a function $G^{j,\ell,m} \colon [0,\infty)\times \R \to \R$ and a constant $C_{j,\ell,m} > 0$ such that, for each $(\xi, \zeta, t) \in \R^2 \times [0, \infty)$, there exists a contour $\Gamma_{\xi,\zeta,t}$ lying in $B(0,\delta)$, which connects $\lambda_0^*$ to $\lambda_0$, satisfying
	\begin{align}
	\int_{\Gamma_{\xi,\zeta,t}} |\lambda|^j |\nuwt(\lambda)|^\ell \re^{\Re (\lambda t + \nuwt(\lambda) (\xi-\zeta))} |g(\xi,\zeta,\lambda)| \chi_- (\xi-\zeta) \, |\de \lambda| \leq G^{j, \ell, m} (t, \xi-\zeta), \label{e: ptwise estimate} 
	\end{align}
	and
	\begin{align*} 
	\| G^{j, \ell, m}(t, \cdot) \|_{L^1} &\leq \frac{C_{j,\ell,m}}{(1+t)^\frac{\ell + j + m}{2}}, \qquad 
	\| G^{j, \ell, m}(t, \cdot) \|_{L^\infty} \leq \frac{C_{j,\ell,m}}{(1+t)^{\frac{1+\ell+j+m}{2}}},
	\end{align*}
    which immediately yields the interpolated estimate
	\begin{align*}
	    \| G^{j, \ell, m} (t, \cdot) \|_{L^p} \leq \frac{C_{j,\ell,m}}{(1+t)^{\frac{1}{2} - \frac{1}{2p} + \frac{j + \ell + m}{2}}},
	\end{align*}
for $p \in [1,\infty]$. 
\end{prop}

We note that the pointwise estimates in Proposition~\ref{p: pointwise estimate} are tailored to pointwise analytic terms in the resolvent decomposition of the form~\eqref{e: model term outgoing mode} which can be represented as a convolution with a suitable integral kernel. However, fixing $\zeta = 0$ and working with a function $g$ that does not depend on $\zeta$ in Proposition~\ref{p: pointwise estimate}, the result can also be applied to those terms in the resolvent decomposition of the form~\eqref{e: model term outgoing mode 2} which are pointwise analytic in $\lambda$, but are not necessarily of convolution type.

Returning to the simple example $u_t = D_\mathrm{eff} u_{\xi \xi} + c_g u_\xi$, we find for large times that the solution $u(t)$ with initial condition $u(0) = u_0$ behaves as
\begin{align*}
    \partial_\xi^k u(\xi,t) \sim \int_\R \int_{\Gamma^2_0} (\nuwt(\lambda))^k \re^{\lambda t + \nuwt(\lambda) (\xi-\zeta)} u_0 (\zeta) \de  \lambda \de \zeta, 
\end{align*}
for $k \in \N_0$, $\xi \in \R$ and $t > 0$. Using Proposition~\ref{p: pointwise estimate} and Young's convolution inequality,  we find, for each $p \in [1,\infty]$, that the sharp decay estimate
\begin{align*}
    \| \partial_\xi^k u(\cdot, t) \|_{L^p} \lesssim \| G^{0, k, 0} \|_{L^p} \| u_0 \|_{L^1}  \lesssim \frac{1}{(1+t)^{\frac{1}{2} - \frac{1}{2 p} + \frac{k}{2}}} \| u_0 \|_{L^1},
\end{align*}
holds for $t > 0$, thus recovering~\eqref{e: heat eqn estimates}. 

\paragraph{Estimates for interaction terms.} Recall from our discussion in~\S\ref{s: analyticity regions} that we cannot use the contours of Proposition~\ref{p: pointwise estimate} for estimating terms involving the interaction between branched and outgoing modes, since these contours pass through the branch point. To handle these terms, which are roughly of the form 
\begin{align*}
    \uint (\xi; \sigma) \sim \alphaint(\sigma) \chi_-(\xi) \re^{\nuwt (\sigma^2) \xi},
\end{align*}
where $\alphaint(\sigma)$ is analytic in $\sigma = \sqrt{\lambda}$ in a full neighborhood of the origin and satisfies $\alphaint(0) = 0$, we turn to Proposition~\ref{p: wake norm estimates}. The fact that $\alphaint(\sigma) = \mathrm{O}(\sigma)$, yields an extra decaying factor of $t^{-1/4}$. Returning to the model problem $u_t = D_\mathrm{eff} u_{\xi \xi} + c_g u_\xi$, and (heuristically) giving this extra factor of decay to the model estimate~\eqref{e:  model norm estimate}, we find
\begin{align*}
    \| \partial_\xi^k u(\cdot, t) \|_{L^p} \lesssim \frac{1}{t^{\frac{3}{4} + \frac{k}{2} - \frac{1}{p}}},
\end{align*}
for $t > 0$. This agrees with the sharp model estimate~\eqref{e: heat eqn estimates} when $p=2$. These estimates are sufficient to close a nonlinear stability argument. 

\section{Resolvent analysis near the origin}\label{s: resolvent}

We analyze the resolvent near the origin with a far-field/core decomposition, aiming to identify terms associated to the branched diffusive modes, terms associated to the outgoing diffusive mode, and interaction terms. We start by considering the asymptotic contributions from $\xi = \pm \infty$.

\subsection{Resolvent analysis for the wave train in the wake}\label{s: resolvent wavetrain}

In this section, we will use Floquet theory to study the resolvent $(\lwt-\lambda)^{-1}$ of the linearization about the wave train $\uwt$, near $\lambda = 0$. To begin, we transform the resolvent equation 
\begin{align}
(\lwt - \lambda)\begin{pmatrix} u \\ v\end{pmatrix} = \g, \label{e: resolvent equation wave train} 
\end{align}
for $\g = (g_1,g_2)^\top \in C^\infty (\R, \C^2)$ to the first-order system 
\begin{align}
(\partial_\xi - A(\xi, \lambda))U = G, \label{e: first order resolvent problem wave train}
\end{align}
where $U = (u, u_\xi, v)^\top$ and $G = (0, \g)^\top$. For translating between the original formulation and the first-order formulation, we define the linear operators $\Pi_{13} : \C^3 \to \C^2$ and $\Lambda_1: \C^2 \to \C^3$ by 
\begin{align*}
\Pi_{13}\begin{pmatrix} 
u \\ u_\xi \\ v \end{pmatrix} = \begin{pmatrix} u \\ v \end{pmatrix}, \quad \Lambda_1 \begin{pmatrix} g_1 \\ g_2 \end{pmatrix}
= \begin{pmatrix}
0 \\ g_1 \\ g_2
\end{pmatrix}. 
\end{align*} 

The matrix $A(\xi, \lambda)$ is periodic in $\xi$ and analytic in $\lambda$, and so by standard Floquet theory we have the following result.
\begin{lemma}[Floquet theory]\label{l: floquet theory}
Let $1 \leq p \leq \infty$. For $\delta > 0$ sufficiently small, there exists a change of coordinates $Q \colon \R \times B(0,\delta) \to \C^{3 \times 3}$ which is smooth and periodic in its first component and analytic in its second component and an analytic map $B \colon B(0,\delta) \to \C^{3 \times 3}$ such that, if we have $G \in L^p(\R,\C^3)$ and $U(\cdot,\lambda) \in W^{1,p}(\R,\C^3)$ is a solution of~\eqref{e: first order resolvent problem wave train}, then $V(\cdot, \lambda) = Q(\cdot, \lambda) U(\cdot, \lambda) \in W^{1,p}(\R,\C^3)$ solves
\begin{align}
    (\partial_\xi - B(\lambda)) V = QG. \label{e: resolvent floquet equation}
\end{align}
\end{lemma}

For $\lambda \in B(0,\delta) \setminus \Sigma(\lwt)$ we solve~\eqref{e: resolvent floquet equation} by constructing the matrix Green's function $\Twt_\lambda$ which solves
\begin{align*}
    (\partial_\xi - B(\lambda)) \Twt_\lambda = \delta_0 I, 
\end{align*}
and is given by 
\begin{align*}
    \Twt_\lambda (\xi) = \begin{cases}
    \re^{B(\lambda) \xi} \Pwt^\mathrm{s} (\lambda), & \xi \geq 0, \\
    -\re^{B(\lambda) \xi} \Pwt^\mathrm{u} (\lambda), & \xi < 0, 
    \end{cases}
\end{align*}
where $\Pwt^\mathrm{s}(\lambda)$ and $\Pwt^\mathrm{u}(\lambda)$ are the spectral projections onto the stable and unstable subspaces of $B(\lambda)$, respectively. Eigenvalues of $B(\lambda)$ are spatial Floquet exponents of $\lwt$. It follows from Proposition~\ref{prop: wake spectral curve expansion} that for $|\lambda|$ small, the only eigenvalue of $B(\lambda)$ near the imaginary axis is 
\begin{align}
\nuwt(\lambda) = \nuwt^1 \lambda - \nuwt^2 \lambda^{\color{blue}2} + \mathrm{O}(\lambda^3), \label{e: nu expansion} 
\end{align}
where $\nuwt^1, \nuwt^2 > 0$. Thus, for $|\lambda|$ small, $\nuwt(\lambda)$ stays bounded away from the other eigenvalues of $B(\lambda)$ and, if $\Re \lambda > 0$, this eigenvalue has positive real part. We can therefore separate $\Pwt^\mathrm{u}(\lambda)$ as  
\begin{align*}
    \Pwt^\mathrm{u}(\lambda) = \Pwt^\mathrm{cu} (\lambda) + \Pwt^\mathrm{uu}(\lambda),
\end{align*}
for $\lambda \in B(0,\delta)$ with $\Re \lambda > 0$, where $\Pwt^\mathrm{cu}(\lambda)$ is the spectral projection onto the eigenspace associated to $\nuwt(\lambda)$. Note that, since the eigenvalues of $B(\lambda)$ stay separated, the spectral projections $\Pwt^\mathrm{s}(\lambda)$, $\Pwt^\mathrm{uu}(\lambda)$, and $\Pwt^\mathrm{cu}(\lambda)$ are defined for all $\lambda$ in a full neighborhood of the origin. We may write 
\begin{align*}
    \Twt_\lambda (\xi) = -\re^{\nuwt(\lambda) \xi} \Pwt^\mathrm{cu}(\lambda) 1_{\{ \xi < 0\}} - \re^{B(\lambda) \xi} \Pwt^\mathrm{uu} (\lambda) 1_{\{\xi < 0\}} + \re^{B(\lambda) \xi} \Pwt^\mathrm{s} (\lambda) 1_{\{\xi \geq 0\}},
\end{align*}
for $\xi \in \R$ and $\lambda \in B(0,\delta)$ with $\Re \lambda > 0$. We now modify this decomposition slightly, using the cutoff function $\chi_-$, which is supported on $(-\infty,0]$, to write 
\begin{align*}
\begin{split}
    \Twt_\lambda(\xi) &= - \re^{\nuwt(\lambda) \xi} \Pwt^\mathrm{cu}(\lambda) \chi_- (\xi) + \re^{\nuwt(\lambda) \xi} \Pwt^\mathrm{cu} (\lambda) (\chi_- (\xi) - 1_{\{ \xi < 0\}}) - \re^{B(\lambda) \xi} \Pwt^\mathrm{uu} (\lambda) 1_{\{ \xi < 0 \}}\\ 
    &\qquad + \, \re^{B(\lambda) \xi} \Pwt^\mathrm{s} (\lambda) 1_{\{\xi \geq 0\}} \\
    &=: -\re^{\nuwt(\lambda) \xi} \Pwt^\mathrm{cu}(\lambda) \chi_- (\xi) + T^\mathrm{e}_\lambda (\xi).
\end{split}
\end{align*}
Note that $T^\mathrm{e}_\lambda \in L^1(\R,\C^3)$ is in fact analytic in $\lambda$ in a full neighborhood of the origin due to its uniform exponential localization. We then decompose the solution to the resolvent equation~\eqref{e: resolvent equation wave train} as 
\begin{align}
    \begin{pmatrix} u(\xi, \lambda) \\ v (\xi, \lambda) \end{pmatrix}
    = -\Pi_{13} Q(\xi, \lambda)^{-1} \int_\R \re^{\nuwt(\lambda) (\xi-\zeta)} \chi_- (\xi-\zeta) \Pwt^\mathrm{cu} (\lambda) Q(\zeta, \lambda) \Lambda_1 \g (\zeta) \de  \zeta + \bar{s}^\mathrm{wt}_e (\lambda) \g, \label{e: definition projections}
\end{align}
where 
\begin{align}
    \bar{s}^\mathrm{wt}_e (\lambda) \g (\xi) = \Pi_{13} Q(\xi, \lambda)^{-1} \int_\R T^\mathrm{e}_\lambda (\xi-\zeta) Q(\zeta, \lambda) \Lambda_1 \g(\zeta) \de  \zeta. \label{e: resolvent S0 exp def}
\end{align}

We further modify the decomposition by expanding $\Pwt^\mathrm{cu}(\lambda)$, $Q(\xi, \lambda),$ and $Q(\xi, \lambda)^{-1}$ in $\lambda$ in order to separate the most critical modes from those with improved algebraic decay in time, so that we may finally write 
\begin{align*}
     \begin{pmatrix} u(\xi, \lambda) \\ v (\xi, \lambda) \end{pmatrix} = -\Pi_{13} Q(\xi, 0)^{-1} \int_\R \re^{\nuwt(\lambda) (\xi-\zeta)} \chi_- (\xi-\zeta) \Pwt^\mathrm{cu} (0)  Q(\zeta, 0) \Lambda_1 \g(\zeta) \de  \zeta + \bar{s}^\mathrm{wt}_c(\lambda) \g + \bar{s}^\mathrm{wt}_e(\lambda) \g, 
\end{align*}
with 
\begin{align} \label{e: sc wt def}
\begin{split}
    \bar{s}^\mathrm{wt}_c (\lambda) \g(\xi) &= \Pi_{13} \bigg[Q(\xi, 0)^{-1} \int_\R - \re^{\nuwt (\lambda) (\xi-\zeta)} \chi_- (\xi-\zeta) \Pwt^\mathrm{cu}(0) Q(\zeta, 0) \Lambda_1 \g(\zeta) \de  \zeta \\ 
    &\qquad \qquad \qquad - Q(\xi, \lambda)^{-1} \int_\R \re^{\nuwt (\lambda) (\xi-\zeta)} \chi_- (\xi-\zeta) \Pwt^\mathrm{cu}(\lambda) Q(\zeta, \lambda) \Lambda_1  \g(\zeta) \de  \zeta \bigg],
\end{split}
\end{align}
and $\bar{s}^\mathrm{wt}_e (\lambda)$ given by~\eqref{e: resolvent S0 exp def}. 

\begin{lemma}\label{l: resolvent S e estimates}
    Fix $1 \leq p \leq \infty$. There exists $\delta > 0$ so that the map $\bar{s}^\mathrm{wt}_e \colon B(0, \delta) \to B(L^p(\R,\C^2), L^p(\R,\C^2))$ is analytic. Moreover, there exists a constant $C > 0$ so that 
    \begin{align*}
        \|\bar{s}^\mathrm{wt}_e(\lambda) \g \|_{L^p} \leq C \| \g \|_{L^p}
    \end{align*}
    for $\lambda \in B(0, \delta)$ and $\g \in L^p(\R,\C^2)$. 
\end{lemma}
\begin{proof}
Analyticity follows readily from pointwise analyticity of $Q(\xi, \lambda)$ in $\lambda$ and analyticity of $T^\mathrm{e}_\lambda \in L^1(\R,\C^3)$ in $\lambda$ in a full neighborhood of the origin. 
\end{proof}

Undoing the Floquet change of coordinates at $\lambda = 0$, it follows that 
\begin{align*}
    -\Pi_{13} Q(\xi, 0)^{-1} \int_\R \re^{\nuwt(\lambda) (\xi-\zeta)} \chi_- (\xi-\zeta) \Pwt^\mathrm{cu} (0)  Q(\zeta, 0) \Lambda_1 \g(\zeta) \de  \zeta = \uwt'(\xi) \bar{s}_p^\mathrm{wt} (\lambda) \g (\xi),
\end{align*}
where 
\begin{align}
    \bar{s}_p^\mathrm{wt} (\lambda) \g (\xi) = \phi \left( \int_\R \re^{\nuwt(\lambda) (\xi-\zeta)} \chi_- (\xi-\zeta)\Pwt^\mathrm{cu} (0)  Q(\zeta,0) \Lambda_1 \g(\zeta) \de  \zeta \right) \label{e: sp wt def}
\end{align}
and $\phi \colon \C^3 \to \C$ is some linear map. 

We will use the following estimate in analyzing the center piece of the resolvent, which captures dynamics near the front interface. 
\begin{lemma}\label{l: wake resolvent localized estimate}
For any $\eta > 0$ sufficiently small, there exist positive constants $C$ and $\delta$ and a bounded limiting operator $R_0^\mathrm{wt} : L^1 (\R, \C^2) \to W^{1, \infty} (\R, \C) \times L^\infty (\R, \C)$ such that the resolvent $(\lwt-\lambda)^{-1}$ is a well-defined map on the exponentially weighted space $L^1_{\mathrm{exp},\eta,\eta}(\R,\C^2)$ for any $\lambda \in B(0,\delta)$. Moreover, for any $\g \in L^1(\R, \C^2)$ with $\mathrm{supp} (\g) \subseteq (-\infty, 0]$, we have  
	\begin{align} \label{e: wake resolvent localized estimate}
	\| \re^{\eta \cdot} \chi_- \left( (\lwt-\lambda)^{-1} - R_0^\mathrm{wt} \right) \g \|_{W^{1, \infty} \times L^\infty} \leq C |\lambda| \| \g \|_{L^1} 
	\end{align}
for all $\lambda \in B(0, \delta)$. 
\end{lemma}
\begin{proof}
First, we observe that Proposition~\ref{prop: wake spectral curve expansion} yields that, for $\eta > 0$ and $\delta > 0$ sufficient small, the eigenvalues $\nu$ of $B(\lambda)$ stay bounded away from the line $\Re(\nu) = -\eta$ for any $\lambda \in B(0,\delta)$. Hence, the $L^1_{\mathrm{exp},\eta,\eta}$-spectrum of the operator $\lwt$ does not intersect the ball $B(0,\delta)$, implying that the resolvent $(\lwt - \lambda)^{-1}$ is well-defined on the exponentially weighted space $L^1_{\mathrm{exp},\eta,\eta}(\R,\C^2)$ for any $\lambda \in B(0,\delta)$.
    
Upon defining $R_0^\mathrm{wt} \colon L^1(\R,\C^2) \to W^{1,\infty}(\R,\C) \times L^\infty(\R,\C)$ by $R_0^\mathrm{wt}\g(\xi) = \uwt'(\xi)\bar{s}^\mathrm{wt}_p(0) \g(\xi) + \bar{s}^\mathrm{wt}_e(0)\g(\xi)$, the proof of estimate~\eqref{e: wake resolvent localized estimate} is similar to~\cite[Lemma 3.6]{AveryScheelGL}, but we sketch it for completeness. We focus only on the most critical term in the resolvent, i.e.~the one involving $\bar{s}_p^\mathrm{wt}(\lambda) \g$, since obtaining the corresponding estimate for the other terms is strictly easier. Since we are assuming $\g$ is supported on $\{ \zeta \leq 0 \}$, the integrand in~\eqref{e: sp wt def} is nonzero only for $\xi \leq \zeta \leq 0$. 
    
Now, suppose $|\nuwt(\lambda) (\xi - \zeta)| \leq 1$. Then, using the mean value theorem, Proposition~\ref{prop: wake spectral curve expansion} and the fact that $|\zeta| \leq |\xi|$, we establish    
  \begin{align*}
        \re^{\eta \xi} |\re^{\nuwt(\lambda) (\xi - \zeta)} - 1 | \lesssim  \re^{\eta \xi}|\nuwt(\lambda)| |\xi-\zeta| \lesssim \re^{\eta \xi} |\lambda| |\xi| \lesssim |\lambda|. 
    \end{align*}
    for $\xi \leq \zeta \leq 0$ and $\lambda \in B(0,\delta)$.  Now, instead suppose  $|\nuwt(\lambda) (\xi-\zeta)| \geq 1$. Then, using that $|\zeta| \leq |\xi|$, we establish, if $|\lambda|$ is sufficiently small relative to $\eta$, that
    \begin{align*}
        \re^{\eta \xi} |\re^{\nuwt(\lambda) (\xi - \zeta)} - 1 | \lesssim \re^{\frac{\eta}{2} \xi} \lesssim \re^{\frac{\eta}{2} \xi} |\nuwt(\lambda) (\xi - \zeta)| \lesssim |\lambda| \re^{\frac{\eta}{2} \xi} |\xi| \lesssim |\lambda|,
    \end{align*}
    for $\xi \leq \zeta \leq 0$ and $\lambda \in B(0,\delta)$. Hence, in either case we have
    \begin{align*}
        \re^{\eta \xi} |\re^{\nuwt(\lambda) (\xi - \zeta)} - 1 | \lesssim |\lambda|,
    \end{align*}
    for $\xi \leq \zeta \leq 0$ and $|\lambda|$ sufficiently small, from which the result readily follows. 
\end{proof}

\subsection{Resolvent analysis in the leading edge}\label{s: leading edge resolvent}
  We analyze the behavior of the limiting resolvent in the leading edge $(\mcl_+ - \sigma^2)^{-1}$ near $\sigma = 0$. A general analysis in reaction-diffusion systems with non-degenerate diffusion was carried out in~\cite[Section 5]{AverySelectionRD}, based on the analysis for scalar parabolic equations in~\cite{AveryScheel, AveryScheelSelection}. So our main emphasis here is in explaining the necessary modifications to handle the degenerate diffusion in~\eqref{e: fhn}. 

As in the preceding section, we recast the resolvent equation 
\begin{align}
(\mcl_+ - \sigma^2) \begin{pmatrix}
u \\ v 
\end{pmatrix} = 
\begin{pmatrix}
g_1 \\ g_2 
\end{pmatrix} \label{e: leading edge resolvent eqn}
\end{align}
as a first-order system in the variable $U = (u, u_\xi, v)^\top$ with $G = (0,\g)^\top$ and $\g = (g_1,g_2)^\top$. This system takes the form
\begin{align}
(\partial_\xi - \mathcal{M}(\sigma^2)) U = G,  \label{e: first order formulation leading edge}
\end{align} 
where the matrix $\mathcal{M}(\sigma^2)\in \C^{3 \times 3}$ is analytic in $\sigma^2$. From a short calculation, one sees that the eigenvalues of $\mathcal{M}(\sigma^2)$ correspond precisely to roots $\nu(\sigma)$ of the dispersion relation $d_{\clin} (\sigma^2, \nu)$, which we refer to as spatial eigenvalues. Recall from Corollary~\ref{c: leading edge spatial eval} that when $|\sigma|$ is small, there are precisely two spatial eigenvalues $\nufr^\pm(\sigma)$ in a neighborhood of the origin, with expansion
\begin{align}
		\nu_\mathrm{fr}^\pm(\sigma) = \pm \nufr^1 \sigma + \mathrm{O}(\sigma^2), \label{e: leading edge spatial eval expansion}
		\end{align}
where $\nufr^1 > 0$. In particular, $\nu_\mathrm{fr}^+(\sigma)$ has strictly positive real part and $\nu_\mathrm{fr}^-(\sigma)$ has strictly negative real part for $\sigma^2 \in B(0,\delta)$ lying to the right of $\Sigma(\mcl_+)$. We record the following result, adapted from~\cite[Lemma 5.1]{AverySelectionRD}.

\begin{lemma} \label{l: center projections pole}	
    Let $\Pfr^\mathrm{cu}(\sigma)$ and $\Pfr^\mathrm{cs}(\sigma)$ denote the spectral projections onto the eigenspaces of $\mathcal{M}(\sigma^2)$ associated with $\nu^+_\mathrm{fr}(\sigma)$ and $\nu^-_\mathrm{fr}(\sigma)$, respectively, which are one-dimensional provided $\sigma \neq 0$. Then, $\Pfr^\mathrm{cu}(\sigma)$ and $\Pfr^\mathrm{cs}(\sigma)$  are meromorphic in $\sigma$ in a neighborhood of the origin, with expansions
		\begin{align*}
		\Pfr^\mathrm{cu}(\sigma) = -\frac{1}{\sigma} P_\mathrm{pole} + \mathrm{O}(1), \qquad \Pfr^\mathrm{cs}(\sigma) = \frac{1}{\sigma} P_\mathrm{pole} + \mathrm{O}(1) 
		\end{align*}
		for some matrix $P_\mathrm{pole}\in \C^{3 \times 3}$.
\end{lemma}

Similar to Section~\ref{s: resolvent wavetrain}, our strategy is to recover the solution to~\eqref{e: leading edge resolvent eqn} from the first-order formulation, which we solve via a (somewhat) explicit Green's function. For $\sigma^2 \in B(0,\delta) \setminus \Sigma(\mcl_+)$ the matrix Green's function $\Tfr_\sigma$ associated to the first-order formulation~\eqref{e: first order formulation leading edge} solves
\begin{align*}
	(\partial_\xi - \mathcal{M}(\sigma^2)) \Tfr_\sigma =  \delta_0 I,
 \end{align*}
so that the solution to~\eqref{e: leading edge resolvent eqn} is given by
\begin{align}
\begin{pmatrix}
u(\xi;\sigma) \\ v (\xi;\sigma)
\end{pmatrix}
= \int_\R \Pi_{13} \Tfr_\sigma (\xi-\zeta) \Lambda_1 \g(\zeta) \de \zeta, \label{e: leading edge resolvent formula}
\end{align}
where we recall the definition~\eqref{e: definition projections} of the linear operators $\Pi_{13}$ and $\Lambda_1$. Following the reasoning in~\cite[Section~2]{AveryScheel} and using Corollary~\ref{c: leading edge spatial eval}, we may decompose $\Tfr_\sigma$ as 
\begin{align}
\Tfr_\sigma(\zeta) =  \re^{\nufr^-(\sigma) \zeta} \Pfr^\mathrm{cs}(\sigma)1_{\{\zeta \geq 0 \}} - \left( \re^{\nufr^+(\sigma) \zeta} \Pfr^\mathrm{cu}(\sigma) + \re^{\nu_3(\sigma^2) \zeta} \Pfr^\mathrm{uu} (\sigma^2) \right) 1_{\{\zeta < 0\}}, \label{e: matrix greens fcn decomposition}
\end{align}
for $\sigma^2 \in B(0,\delta)$ to the right of $\Sigma(\mcl_+)$, where $\Pfr^\mathrm{uu}(\sigma^2)$ is the spectral projection onto the strong unstable subspace of $\mathcal{M}(\sigma^2)$ corresponding to the remaining third eigenvalue $\nu_3(\sigma^2)$ of $M(\sigma^2)$, which resides in the open right half-plane and remains bounded away from the imaginary axis for $|\sigma|$ small. The leading-order temporal behavior of the solution to the linearization $\u_t = \mcl_+ \u$ will be governed by those terms in~\eqref{e: matrix greens fcn decomposition} which are most singular in $\sigma$. Standard spectral theory implies that $\Pfr^\mathrm{uu}(\sigma^2)$ is analytic in $\sigma^2$ in a full neighborhood of the origin. The only singularities then arise from the terms involving $P_\mathrm{pole}$, and so from the point of view of the resulting temporal dynamics, we have
\begin{align*}
	\Tfr_\sigma(\zeta) \approx \frac{1}{\sigma} P_\mathrm{pole}\left( \re^{\nufr^+ (\sigma) \zeta} 1_{\{\zeta <0\}} + \re^{\nufr^-(\sigma) \zeta} 1_{\{\zeta \geq 0\}}\right). 
\end{align*}
Using the expansion~\eqref{e: leading edge spatial eval expansion}, we see that $-\nufr^+(\sigma) \approx \nufr^-(\sigma) \approx - \nufr^1 \sigma$, and hence
\begin{align*}
\Tfr_\sigma(\zeta) \approx \frac{1}{\sigma} \re^{- \nufr^1 \sigma |\zeta|} P_\mathrm{pole} =: G_\sigma^\mathrm{heat} (\zeta) P_\mathrm{pole}.
\end{align*}
Note that, up to scaling, $G_\sigma^\mathrm{heat}$ is the Laplace transform of the fundamental solution of the heat equation. The solution to the resolvent equation~\eqref{e: leading edge resolvent eqn} is then given to leading order by
\begin{align}
\begin{pmatrix}
u(\xi;\sigma) \\
v(\xi;\sigma)
\end{pmatrix} \approx \Pi_{13} P_\mathrm{pole} \Lambda_1 \int_\R G_\sigma^\mathrm{heat} (\xi-\zeta) \g(\zeta) \de  \zeta. \label{e: right resolvent leading order}
\end{align}
To make this reasoning rigorous, we can follow the analysis of~\cite[Section 2]{AveryScheel}, but for now ignoring estimates on derivatives of the solution, to obtain the following description of the far-field resolvent for $\sigma \in \Delta_\delta^{\mathrm{fr},1}$, where we recall that $\Delta_\delta^{\mathrm{fr},1}$ is the image of the set $\Omega_{\delta}^{\mathrm{fr},1}$ under the principal square root, see Figure~\ref{fig: resolvent regions} and~\eqref{e: Omega fr 1 def}.

\begin{lemma}\label{l: leading edge resolvent no regularity}
	Fix $r > 2$. There exist positive constants $C$ and $\delta$ and a bounded limiting operator $R_0^+ : L^1_{1,1} (\R, \C^2) \to L^\infty_{-1, -1} (\R, \C^2)$ such that for all odd functions $\g \in L^1_{1,1} (\R, \C^2)$, we have
	\begin{align*}
		\| (\mcl_+ - \sigma^2)^{-1} \g - R_0^+ \g\|_{L^\infty_{-1, -1}} &\leq C |\sigma| \| \g \|_{L^1_{1,1}}, \qquad 
				\| (\mcl_+ - \sigma^2)^{-1} \g - R_0^+ \g\|_{L^1_{-r, -r}} \leq C |\sigma| \| \g \|_{L^1_{1,1}}, 
	\end{align*}
	for all $\sigma \in \Delta^{\mathrm{fr},1}_\delta$. 
\end{lemma}

The leading-order description~\eqref{e: right resolvent leading order} corresponds to the resolvent problem for the heat equation. For the heat equation, restricting to odd initial data is equivalent to imposing the Dirichlet boundary condition $u = 0$ at $\xi = 0$. Restricting to odd initial data $\g$ here models an effective absorption mechanism behind the leading edge of the front, due to stability of the state in the wake. We will enforce this oddness for the leading edge resolvent in our far-field/core decomposition in the next section. 

In passing estimates to the full resolvent $(\lfr - \sigma^2)^{-1}$, we will need to control commutators $[\mcl_+, \chi_+]$ applied to the solution $\u^+$ of the leading edge resolvent equation $(\mcl_+ - \sigma^2) \u^+ = \g_+$. These commutator terms will involve derivatives applied to the first component of $\u^+$, but not the second component, and so we need to upgrade Lemma~\ref{l: leading edge resolvent no regularity} to control regularity in the first component. 
\begin{prop}\label{p: leading edge resolvent}
		Fix $r > 2$. There exist positive constants $C$ and $\delta$ and a bounded limiting operator $R_0^+ : L^1_{1,1} (\R, \C^2) \to W^{1, \infty}_{-1, -1} (\R, \C) \times L^\infty_{-1, -1} (\R, \C)$ such that for all odd functions $\g = (g_1, g_2)^\top \in L^1_{1,1} (\R, \C^2)$, we have
	\begin{align} \label{e: leading edge resolvent lipschitz estimate} 
 \begin{split}
	\| (\mcl_+ - \sigma^2)^{-1} \g - R_0^+ \g\|_{W^{1, \infty}_{-1, -1} \times L^\infty_{-1, -1}} &\leq C |\sigma| \| \g \|_{L^1_{1,1}}, \\
	\| (\mcl_+ - \sigma^2)^{-1} \g - R_0^+ \g\|_{W^{1, 1}_{-r, -r} \times L^1_{-r, -r}} &\leq C |\sigma| \| \g \|_{L^1_{1,1}},
 \end{split}
	\end{align}
	for all $\sigma \in \Delta^{\mathrm{fr},1}_\delta$.
\end{prop}
\begin{proof}
	Using the first-order formulation in $U = (u, u_\xi, v)^\top$, we see that, if $(\mcl_+ - \sigma^2) u = \g$, then we have
	\begin{align*}
	u_\xi (\xi; \sigma) = \int_\R \Pi_2 \Tfr_\sigma(\xi-\zeta) \Lambda_1 \begin{pmatrix}
	g_1(\zeta) \\ g_2(\zeta) 
	\end{pmatrix} \de \zeta, \qquad \text{where}\quad 
	\Pi_2 \begin{pmatrix}
	f_1 \\ f_2 \\ f_3 
	\end{pmatrix} = 
	f_2. 
	\end{align*}
	Comparing with~\eqref{e: leading edge resolvent formula}, the only difference between the expressions for $(u,v)^\top$ and $u_\xi$ is that $\Pi_{13}$ is replaced by $\Pi_2$, both of which are independent of $\sigma$. Therefore, following the same reasoning, we readily find the same estimates for $u_\xi(\xi; \sigma)$ as for $u(\xi;\sigma)$, which implies the desired result. 
\end{proof}

We will use Proposition~\ref{p: leading edge resolvent} to obtain the sharp $t^{-3/2}$ decay when measuring the solution in $L^\infty_{0, -1}$. Indeed, we saw in Proposition~\ref{p: front estimate sharp decay} how the Lipschitz estimate~\eqref{e: leading edge resolvent lipschitz estimate} leads to the $t^{-3/2}$ decay rate. In closing a nonlinear argument, however, we will also measure the solution in weaker norms, such as $L^\infty(\R)$ and $L^2(\R)$. In these spaces, we lose the Lipschitz expansion of the resolvent, but retain some boundedness or quantifiable blowup for $\sigma \in \Delta_\delta^{\mathrm{fr},2}$, where we recall that $\Delta_\delta^{\mathrm{fr},2}$ is the image of the set $\Omega_\delta^{\mathrm{fr},2}$ under the principal square root; see Figure~\ref{fig: resolvent regions} and~\eqref{e: Omega fr 2 def}.

\begin{lemma}[$L^\infty$ estimate in the leading edge]\label{l: right resolvent L1 Linf boundedness estimate}
	There exist positive constants $C$ and $\delta$ such that for $\sigma \in \Delta^{\mathrm{fr},2}_\delta$ the following estimates hold.
	\begin{itemize}
		\item For all odd functions $\g \in L^1_{0,1} (\R, \C^2)$, we have
		\begin{align}
			\left\|(\mcl_+ - \sigma^2)^{-1} \g \right\|_{W^{2, \infty} \times W^{1,\infty}} \leq C \| \g\|_{L^1_{0,1}}. \label{e: right resolvent L101 estimate}
		\end{align}
		\item For all odd functions $\g \in L^1 (\R, \C^2)$, we have 
		\begin{align}
			\left\| \chi_+  (\mcl_+ - \sigma^2)^{-1} \g \right\|_{L^\infty_{0,-1}} \leq C \| \g\|_{L^1}. \label{e: right resolvent Linf -1 estimate}
		\end{align}
	\end{itemize}
\end{lemma}

The estimate~\eqref{e: right resolvent L101 estimate} will be used to estimate decay in the leading edge in $L^\infty(\R)$. The estimate~\eqref{e: right resolvent Linf -1 estimate} will be used in controlling contributions to the center dynamics near the front interface. Finally, to estimate decay in the leading edge in $L^2(\R)$ we will rely on the following result.

\begin{lemma}[$L^2$ estimate in the leading edge]\label{l: leading edge resolvent L2 estimate}
	There exist positive constants $C$ and $\delta$ such that for all odd functions $\g \in L^1_{0,1}(\R, \C^2)$, we have
	\begin{align*}
	\| (\mcl_+ - \sigma^2)^{-1} \g \|_{L^2} \leq \frac{C}{|\sigma|^{1/2}} \| \g \|_{L^1_{0,1}}
	\end{align*}
	for all $\sigma \in \Delta^{\mathrm{fr},2}_\delta$. 
\end{lemma}

We relegate the  proofs of Lemmas~\ref{l: right resolvent L1 Linf boundedness estimate} and~\ref{l: leading edge resolvent L2 estimate} to Appendix~\ref{s: appendix leading edge resolvent}. 

\subsection{Far-field/core decomposition}

Our goal is to solve $(\lfr-\sigma^2) \u = \g$ by using a far-field/core decomposition which takes advantage of Fredholm properties of $\lfr$ on spaces of exponentially localized functions, captured in Proposition~\ref{p: fredholm properties}. To do this, we must first reduce to a problem where the data $\g$ is exponentially localized. Thus, we first decompose data $\g \in L^1(\R,\C^2) \cap L^\infty(\R,\C^2)$ as 
\begin{align*}
\g = \chi_- \g + \chi_c \g + \chi_+ \g =: \g_- + \g_c + \g_+,
\end{align*}
where $(\chi_-, \chi_c, \chi_+)$ is the partition of unity defined in Section~\ref{s: function spaces}. We then let $\g_+^\mathrm{odd}(\xi) = \g_+(\xi) - \g_+(-\xi)$ be the odd extension of $\g_+$, in order to take advantage of the improved behavior of $(\mcl_+-\sigma^2)^{-1}$ when acting on odd functions. We let $\u_+$ solve
\begin{align*}
(\mcl_+ - \sigma^2) \u_+ = \g_+^\mathrm{odd}.
\end{align*}
By Proposition~\ref{p: leading edge resolvent} we can solve this equation in $L^\infty_{-1,-1}(\R)$, thus allowing for algebraic growth, by requiring that $\g_+^\mathrm{odd}$ is algebraically localized, i.e.~$\g_+^\mathrm{odd} \in L^1_{1,1}(\R)$, and $\sigma^2$ lies in the set $\Omega^{\mathrm{fr},1}_\delta$, which lies to the right of $\Sigma_\mathrm{ess}(\mcl_+)$. In addition, if $\sigma^2$ lies in the smaller set $\Omega^{\mathrm{fr},2}_\delta$, we obtain by Proposition~\ref{l: right resolvent L1 Linf boundedness estimate} a solution $\u_+$ in $L^\infty_{-1,-1}(\R)$ without requiring extra conditions on $\g_+^\mathrm{odd}$.

We let $\u_-$ solve
\begin{align*}
(\lwt-\sigma^2) \u_- = \g_-.
\end{align*}
By Proposition~\ref{l: wake resolvent localized estimate}, there exists $\delta > 0$ and $\eta > 0$ small, such that we can solve this equation uniquely for any $\sigma^2 \in B(0,\delta)$ by allowing the solution $\u_-$ to grow exponentially on $(-\infty,0]$ with rate $\eta$. 

We then decompose $\u$ as 
\begin{align*}
\u = \chi_- \u_- + \u_c + \chi_+ \u_+,
\end{align*}
and see that in order for $\u$ to solve $(\lfr - \sigma^2) \u = \g$, the center correction $\u_c$ must solve
\begin{align} \label{e: center resolvent equation}
(\lfr-\sigma^2) \u_c = \tilde{\g}(\sigma),
\end{align}
where
\begin{align} \label{e: definition g}
\tilde{\g}(\sigma) = \g - (\lfr-\sigma^2) (\chi_- \u_-) - (\lfr-\sigma^2) (\chi_+ \u_+). 
\end{align}
The next result confirms that the above procedure has led to a reduced problem~\eqref{e: center resolvent equation} in which the data $\tilde{\g}$ is indeed exponentially localized. In addition, it provides control on the data $\tilde{\g}$ in terms of the original data $\g$ and the spectral parameter $\sigma$.

\begin{lemma}[Control on center data]\label{l: tilde f control}
	Fix $1 \leq p \leq \infty$. There exist $\delta>0$ and $\eta > 0$ small such that, for any $\g \in L^1(\R,\C^2)$, the map $\tilde{\g} \colon \Delta^{\mathrm{fr,1}}_\delta \to X^p_\eta$ given by~\eqref{e: definition g} is well-defined and analytic. Moreover, there exists a constant $C > 0$ such that
	\begin{enumerate}
		\item for all any $\g \in L^1_{0,1} (\R, \C^2)$ the map $\tilde{\g}$ extends Lipschitz continuously to $\sigma = 0$, i.e.,~for all $\sigma \in \Delta_\delta^{\mathrm{fr},1}$ we have	
         \begin{align}
		\| \tilde{\g}(\sigma) - \tilde{\g}(0) \|_{X^p_\eta} \leq C |\sigma| \| \g \|_{L^1_{0,1}}, \qquad \|\tilde{\g}(0)\| \leq C \|\g\|_{L^1_{0,1}}; \label{e: f tilde L1 gamma estimate}
		\end{align}
		\item for all $\sigma \in \Delta^{\mathrm{fr},2}_\delta$ and any $\g \in L^1 (\R, \C^2)$ we have
		\begin{align}
		\| \tilde{\g}(\sigma) \|_{X^1_{\eta}} \leq C \| \g \|_{L^1}; \label{e: f tilde L 1 boundedness estimate}
		\end{align}
		\item for all $\sigma \in \Delta^{\mathrm{fr},2}_\delta$ and for any $\g \in L^1 (\R, \C^2) \cap L^\infty(\R, \C^2)$ we have
		\begin{align}
		\| \tilde{\g}(\sigma) \|_{X^p_{\eta}} \leq C \left( \| \g \|_{L^1} + \| \g \|_{L^\infty} \right) \label{e: f tilde L inf boundedness estimate}
		\end{align}
		for all $1 \leq p \leq \infty$. 
	\end{enumerate}
\end{lemma}
\begin{proof}
	We rewrite $\tilde{\g}(\sigma)$ as 
	\begin{align}
 \begin{split}
	\tilde{\g}(\sigma) &= \g - (\lwt - \sigma^2) (\chi_- \u_-) - (\lfr - \lwt) (\chi_- \u_-) - (\mcl_+ - \sigma^2) (\chi_+ \u_+) - (\lfr - \mcl_+) (\chi_+ \u_+) \\
	&= \g - \chi_- \g_- - \chi_+ \g_+ - [\lwt, \chi_-] \u_- - (\lfr - \lwt) (\chi_- \u_-) - [\mcl_+, \chi_+] \u_+\\ 
 &\qquad - \, (\lfr - \mcl_+) (\chi_+ \u_+), \end{split} \label{e: f tilde expanded formula}
	\end{align}
	where we have used the facts that $\chi_+ (\mcl_+ - \sigma^2) \u_+ = \chi_+ \g_+^\mathrm{odd} = \chi_+ \g_+$, and $\chi_- (\lwt - \sigma^2) \u_- = \chi_- \g_-$ by construction. 
	
	We first prove~\eqref{e: f tilde L1 gamma estimate}. The dependence of $\tilde{\g}$ on $\sigma$ enters through the terms involving $\u_-$ and $\u_+$. Using that the derivatives $\chi_\pm'$ are compactly supported and the coefficients of $\lfr$ converge to their limits exponentially quickly as $\xi \to \pm \infty$, we see that all coefficients of $\u_\pm$ in~\eqref{e: f tilde expanded formula} are exponentially localized with rate $\eta_0 > 0$ independent of $\sigma$. 
    Note also that all commutator terms only involve derivatives in the first component. Hence, choosing $\eta \leq \frac{\eta_0}{2}$, we have by Proposition~\ref{p: leading edge resolvent} and H\"older's inequality,
	\begin{align*}
	\| [\mcl_+, \chi_+] (\u_+ (\sigma) - \u_+(0)) \|_{X^p_{\eta}} & \lesssim \left\| \re^{-\frac{\eta_0}{2} |\cdot|}  (\u_+(\sigma) - \u_+(0)) \right\|_{L^p} +  \left\| \re^{- \frac{\eta_0}{2}|\cdot|} \Pi_1 \partial_x (\u_+(\sigma) - \u_+(0)) \right\|_{L^p} \\ &\lesssim |\sigma| \| \g_+ \|_{L^1_{1,1}} \lesssim |\sigma| \| \g \|_{L^1_{0,1}},
	\end{align*}
	for $\sigma \in \Delta^{\mathrm{fr,1}}_\delta$, where $\Pi_1 (u,v)^\top = u$. By the same argument relying on exponential localization, we obtain
	\begin{align*}
	\| (\lfr - \mcl_+) [\chi_+ (\u_+(\sigma) - \u_+(0))] \|_{X^1_{\eta}} \lesssim |\sigma| \| \g \|_{L^1_{0,1}} ,
	\end{align*}
	for $\sigma \in \Delta^{\mathrm{fr,1}}_\delta$. The argument controlling the terms involving $\u_-$ is analogous, with Lemma~\ref{l: wake resolvent localized estimate} replacing Proposition~\ref{p: leading edge resolvent}. Note that Lemma~\ref{l: wake resolvent localized estimate} only controls a localized $L^\infty$ norm, not the $L^1$ norm, but here the integral can be absorbed by the uniform exponential factor so that the $L^\infty$ estimate is sufficient even in the case $p = 1$. 
	
	The proof of the estimate~\eqref{e: f tilde L 1 boundedness estimate} is similar: after exploiting uniform exponential localization and the structure of the commutators, control on the terms involving $\u_+$ is obtained using Lemma~\ref{l: right resolvent L1 Linf boundedness estimate}, while Lemma~\ref{l: wake resolvent localized estimate} still suffices for control on the terms involving $\u_-$. The proof of~\eqref{e: f tilde L inf boundedness estimate} is again similar, the only difference being that the term $\| \g \|_{L^\infty}$ is needed on the right hand side to control $\|\g - \chi_- \g_- - \chi_+ \g_+\|_{L^p}$ from the point of view of spatial regularity. 
\end{proof}

With control on the exponentially localized right-hand side $\tilde{\g}$ in hand, we now aim to solve~\eqref{e: center resolvent equation} for $\u_c$ using a far-field/core decomposition. To understand the solvability properties of~\eqref{e: center resolvent equation} on the space $X_\eta^p$ of exponentially localized functions we can look at the Fredholm index of $\lfr$, which is $-2$ by Proposition~\ref{p: fredholm properties}. Thus, the problem~\eqref{e: center resolvent equation} cannot be solved for exponentially localized $\u_c$ as we need to allow for two additional degrees of freedom. This leads us to consider the following two neutral modes.

\begin{lemma}[Neutral mode on the left]\label{l: left mode}
	For $\delta > 0$ sufficiently small, there exists a solution $\e_- (\xi, \lambda)$ to $(\lwt - \lambda) \u = 0$ given by 
	\begin{align*}
	\e_-(\xi, \lambda) = \q(\xi, \lambda) \re^{\nuwt (\lambda) \xi} 
	\end{align*}
	where $\q \colon \R \times B(0,\delta) \to \C^2$ is smooth and periodic in its first argument and analytic in its second argument. Moreover, we have $\q(\cdot,0) = \uwt$.
\end{lemma}
\begin{proof}
    This follows readily from Lemma~\ref{l: floquet theory} and Proposition~\ref{prop: wake spectral curve expansion}.
\end{proof}

\begin{lemma}[Neutral mode on the right]\label{l: right mode}
For $\delta > 0$ sufficiently small, there exists a solution $\e_+(\xi, \sigma)$ to $(\mcl_+ - \sigma^2) \u = 0$ given by 
	\begin{align*}
	\e_+(\xi, \sigma) = \v_+(\sigma) \re^{\nufr^-(\sigma) \xi}, 
	\end{align*}
	where $\v_+ \colon B(0,\sqrt{\delta}) \to \C^2$ and $\nufr^- \colon B(0,\sqrt{\delta}) \to \C$ are analytic, with $\nufr^-(\sigma)$ given in Corollary~\ref{c: leading edge spatial eval}. 
\end{lemma}
\begin{proof}
	This follows readily from Corollary~\ref{c: leading edge spatial eval}.  
\end{proof}

Motivated by the previous Fredholm analysis, we make the ansatz
\begin{align*}
\u_c (\xi; \sigma) = \alpha_- \chi_- (\xi) \e_- (\xi, \sigma^2) + \w(\xi) + \alpha_+ \chi_+(\xi) \e_+(\xi, \sigma). 
\end{align*}
Inserting this ansatz into~\eqref{e: center resolvent equation} leads to an equation
\begin{align*}
G(\w, \alpha_-, \alpha_+; \sigma) = \tilde{\g}(\sigma),
\end{align*}
where
\begin{align*}
G(\w, \alpha_-, \alpha_+; \sigma) = (\lfr-\sigma^2) \left[ \alpha_- \chi_-  \e_- (\cdot, \sigma^2) + \w + \alpha_+ \chi_+ \e_+(\cdot, \sigma) \right]. 
\end{align*}
\begin{lemma}
	Fix $1 \leq p \leq \infty$. There exist constants $\eta > 0$ and $\delta > 0$ sufficiently small so that the map 
	\begin{align*}
	G : Y^p_\eta \times \C \times \C \times B(0,\sqrt{\delta}) \to X^p_\eta
	\end{align*}
	is well defined and analytic in $\sigma$. 
\end{lemma}
\begin{proof}
	This follows from the fact that $\e_\pm$ solve the equations in the far-field (see Lemmas~\ref{l: left mode} and~\ref{l: right mode}), together with exponential convergence of the coefficients of $\lfr$ and localization of commutator terms. Analyticity in $\sigma$ follows as in~\cite[Proposition 5.11]{PoganScheel} or~\cite[Lemma 3.9]{AveryScheelSelection}.
\end{proof}

Note that $G$ is linear in $w, \alpha_-$, and $\alpha_+$. Although $D_w G = \lfr$ is Fredholm with index $-2$ by Proposition~\ref{p: fredholm properties}, provided $\eta > 0$ is sufficiently small, the Fredholm bordering lemma and continuity of the Fredholm index imply that, for each fixed $|\sigma|$ small, the map $(w, \alpha_-, \alpha_+) \mapsto G(w, \alpha_-, \alpha_+; \sigma)$ is Fredholm with index $0$. Hence, the two additional degrees of freedom in the ansatz for $\u_c$, represented by the coefficients $\alpha_\pm$ of the neutral modes, led to a problem which can potentially be inverted. The next result shows that this is in fact the case. 

\begin{prop}\label{p: ff core invertibility}
	Fix $1 \leq p \leq \infty$. There exist constants $\eta > 0$ and $\delta > 0$ sufficiently small so that for each $\sigma \in B(0, \sqrt{\delta})$, the map 
	\begin{align*}
	(\w, \alpha_-, \alpha_+) \mapsto G(w, \alpha_-, \alpha_+;\sigma) : Y^p_\eta \times \C \times \C \to X^p_\eta 
	\end{align*}
	is invertible. We denote the solution to $G(\w, \alpha_-, \alpha_+;\sigma) = \g$ by 
	\begin{align*}
	(\w, \alpha_-, \alpha_+) = (T(\sigma) \g, A_- (\sigma) \g, A_+(\sigma) \g). 
	\end{align*}
	Moreover, the inverse maps
	\begin{align*}
	T(\sigma) : X^p_\eta \to Y^p_\eta, \quad A_\pm(\sigma) : X^p_\eta \to \C
	\end{align*}
	are analytic in $\sigma$ in $B(0,\sqrt{\delta})$. 
\end{prop}
\begin{proof}
	The result will follow from the implicit function theorem provided we can verify invertibility at $\sigma = 0$. Since $(\w, \alpha_-, \alpha_+) \mapsto G(\w, \alpha_-, \alpha_+; 0)$ is Fredholm with index 0, this map is invertible if and only if it has trivial kernel. Since $\e_-(\xi, 0)$ and $\e_+(\xi, 0)$ are bounded on the real line, a nontrivial kernel of this map would correspond to a nontrivial bounded solution to $\lfr \u = 0$, which is excluded by Hypothesis~\ref{hyp: point spectrum}. 
\end{proof}

We have  constructed a solution to $(\lfr - \sigma^2) \u = \g$ given by
\begin{align*}
\u (\xi; \sigma) = \chi_-(\xi) \u_- (\xi; \sigma^2) + \alpha_- (\sigma) \chi_-(\xi) \e_- (\xi; \sigma) + \w(\xi; \sigma) + \alpha_+(\sigma) \chi_+(\xi) \e_+(\xi; \sigma) + \chi_+(\xi) \u_+ (\xi; \sigma),
\end{align*}
where
\begin{align*}
\alpha_- (\sigma) = A_- (\sigma) \tilde{\g} (\sigma), \quad \w(\cdot; \sigma) = T(\sigma) \tilde{\g}(\sigma), \quad \alpha_+(\sigma) = A_+(\sigma) \tilde{\g}(\sigma). 
\end{align*}
Our goal is now to isolate the terms which correspond to the leading-order temporal dynamics. Thus, we decompose the solution as 
\begin{align} \label{e: solution decomp resolvent}
\u(\xi) = \chi_- (\xi) \u_-(\xi; \sigma^2) + \alpha_- (\sigma) \chi_- (\xi) \e_- (\xi; \sigma^2) + R_\mathrm{fr}(\xi; \sigma),
\end{align}
where 
\begin{align*}
R_\mathrm{fr}(\xi; \sigma) = \chi_+(\xi) \u_+(\xi; \sigma) + T(\sigma) \tilde{\g} (\sigma) + \alpha_+(\sigma) \chi_+(\xi) \e_+(\xi; \sigma)
\end{align*}
captures contributions from the branched mode in the leading edge. To estimate time decay rates of the branched mode in various norms, we use the following estimates on $R_\mathrm{fr}(\xi; \sigma)$. The first corresponds to the sharp $t^{-3/2}$ pointwise decay rate, which is observed when measuring in sufficiently weak norms. The next two correspond to slower decay rates when measuring in stronger norms. 

\begin{lemma}[Expansion for branched mode in $L^\infty_{0,-1}$] \label{l: fr Linf 0 1 resolvent estimate}
	There exist positive constants $C$ and $\delta$ such that for any $\g \in L^1_{0,1}(\R, \C^2)$, we have the estimate
	\begin{align*}
		\| R_\mathrm{fr} (\cdot; \sigma) - R_\mathrm{fr} (\cdot; 0) \|_{L^\infty_{0, -1}} \leq C |\sigma| \| \g \|_{L^1_{0,1}} ,
	\end{align*}
	for all $\sigma \in \Delta^{\mathrm{fr},1}_\delta$. 
\end{lemma}
\begin{proof}
	This follows by combining Corollary~\ref{c: leading edge spatial eval}, Proposition~\ref{p: leading edge resolvent}, Lemma~\ref{l: tilde f control}, and Proposition~\ref{p: ff core invertibility}, and Taylor expanding the exponential in $\e_+(\xi;\sigma)$.
\end{proof}
\begin{lemma}[Blowup of the branched mode in $L^2$]\label{l: fr L2 resolvent estimate}
	There exist positive constants $C$ and $\delta$ such that for any $\g \in L^1_{0,1} (\R, \C^2) \cap L^\infty(\R, \C^2)$, we have the estimate
	\begin{align*}
	\| R_\mathrm{fr}(\cdot; \sigma) \|_{L^2} \leq \frac{C}{|\sigma|^{1/2}} \| \g \|_{L^1_{0,1} \cap L^\infty}
	\end{align*}
	for all $\sigma \in \Delta^{\mathrm{fr},2}_\delta$. 
\end{lemma} 
\begin{lemma}[Boundedness of the branched mode in $L^\infty$]\label{l: fr Linf resolvent estimate}
    There exist positive constants $C$ and $\delta$ such that for any $\g \in L^1_{0,1} (\R, \C^2) \cap L^\infty(\R, \C^2)$, we have the estimate
    \begin{align*}
    \| R_\mathrm{fr}(\cdot; \sigma) \|_{L^\infty} \leq C \| \g \|_{L^1_{0,1}\cap L^\infty}
    \end{align*}
    for all $\sigma \in \Delta^{\mathrm{fr},2}_\delta$.  
\end{lemma}

Lemmas~\ref{l: fr L2 resolvent estimate} and~\ref{l: fr Linf resolvent estimate} readily follow from the control on the leading edge resolvent $(\mcl_+-\sigma^2)^{-1}$ in Lemmas~\ref{l: leading edge resolvent L2 estimate} and~\ref{l: right resolvent L1 Linf boundedness estimate}, control on $\nufr^-(\sigma)$ and $\e_+(\cdot,\sigma)$ from Corollary~\ref{c: leading edge spatial eval} and Lemma~\ref{l: right mode}, and control on $T(\sigma) \tilde{\g}(\sigma)$ and $A_+(\sigma)\tilde{\g}(\sigma)$ from Proposition~\ref{p: ff core invertibility} and Lemma~\ref{l: tilde f control}. 

We now analyze the terms $\chi_-(\xi)\u_-(\xi;\sigma^2)$ and $\alpha_-(\sigma)\chi_-(\xi)\e_-(\xi;\sigma^2)$ in~\eqref{e: solution decomp resolvent}, which correspond to the leading-order temporal dynamics. We begin with the term $\chi_-(\xi)\u_-(\xi;\sigma^2)$. By the resolvent decomposition for the wave train in the wake, performed in Section~\ref{s: resolvent wavetrain}, this contribution can be expressed as
\begin{align*}
	\chi_- (\xi) \u_- (\xi; \sigma^2) = \chi_- (\xi) \uwt'(\xi) [\bar{s}_p^\mathrm{wt}(\sigma^2) \g_-](\xi) + \chi_-(\xi) [\bar{{s}}_c^\mathrm{wt} (\sigma^2) \g_-](\xi)_- + \chi_-(\xi) [\bar{s}_e^\mathrm{wt}(\sigma^2) \g_-](\xi),
\end{align*}
where $\bar{s}_p^\mathrm{wt}$, $\bar{s}^\mathrm{wt}_c$ and $\bar{s}_e^\mathrm{wt}$ are defined in~\eqref{e: sp wt def},~\eqref{e: sc wt def} and~\eqref{e: resolvent S0 exp def}, respectively. Since we are interested in the stability of the pattern-forming front $\ufr$, it is more natural to identify the leading-order dynamics with $\ufr'$ rather than $\uwt'$. We can make this replacement up to a manageable error, writing
\begin{align*}
\chi_-(\xi) \uwt'(\xi) [\bar{s}_p^\mathrm{wt} (\sigma^2) \g_-] (\xi) = \chi_- (\xi) \ufr'(\xi) [\bar{s}_p^\mathrm{wt} (\sigma^2) \g_- ](\xi) + \chi_- (\xi) (\uwt'(\xi) - \ufr'(\xi)) [\bar{s}_p^\mathrm{wt} (\sigma^2) \g_-] (\xi). 
\end{align*}
Since $\chi_- (\uwt' - \ufr')$ is uniformly exponentially localized in space, contributions from this term can be seen to decay much faster in time. 

Similarly, using that $\e_- (\xi, \sigma^2) = \q(\xi, \sigma^2) \re^{\nuwt (\sigma^2) \xi}$, and that $\q(\xi, \sigma^2)$ is analytic in $\sigma^2$ in $L^\infty(\R)$, with $\q(\cdot, 0) = \uwt'$, we have
\begin{align*}
\alpha_-(\sigma) \chi_- (\xi) \e_-(\xi; \sigma^2) &= \alpha_- (\sigma) \chi_- (\xi) \ufr'(\xi) \re^{\nuwt (\sigma^2) \xi} + \alpha_- (\sigma) \chi_-(\xi) \re^{\nuwt (\sigma^2) \xi} [\ufr'(\xi) - \uwt'(\xi)] \\ 
&\qquad + \, \alpha_- (\sigma) \chi_- (\xi) \re^{\nuwt(\sigma^2) \xi} [ \uwt'(\xi)- \q(\xi; \sigma^2)]. 
\end{align*}
The first term contributes to the leading-order time dynamics, while the second term is faster decaying since $\chi_- [\ufr'-\uwt']$ is uniformly exponentially localized. Since $\q(\cdot; \sigma^2) = \uwt' + \mathrm{O}(|\sigma|^2)\in L^\infty(\R)$ the last term is also faster decaying in time. Note also that $\omega$ is identically equal to $1$ on the support of $\chi_-$, so we may replace all factors $\chi_- \ufr'$ in the above with $\chi_- \omega \ufr'$. This notation is more natural for the nonlinear argument.

In summary, we have obtained the resolvent decomposition
\begin{align}
\u (\cdot; \sigma) = (\lfr-\sigma^2)^{-1}\g = \textcolor{blue}{\omega}\ufr' [\bar{s}_p^\mathrm{fr} (\sigma) \g] + \bar{s}_c (\sigma) \g + \bar{s}_e(\sigma) \g \label{e: resolvent decomp},
\end{align}
where
\begin{align}
\bar{s}_p^\mathrm{fr}(\sigma) \g &= \chi_- \bar{s}_p^\mathrm{wt}(\sigma^2) \g_- + \chi_-  \alpha_- (\sigma) \re^{\nuwt(\sigma^2) \cdot}, \label{e: sp gamma} \\
\bar{s}_c(\sigma) \g &= R_\mathrm{fr}(\cdot; \sigma) + \alpha_-(\sigma) \chi_- \re^{\nuwt (\sigma^2) \cdot} \left( [\ufr' - \uwt'] + [\uwt' - \q(\cdot; \sigma^2)] \right) + \chi_- \bar{{s}}_c^\mathrm{wt}(\sigma^2) \g_- \label{e: sc gamma}\\ 
&\qquad + \, \chi_- (\uwt' - \ufr') \bar{s}_p^\mathrm{wt} (\sigma^2) \g_-, \nonumber\\
\bar{s}_e (\sigma) \g &= \chi_- \bar{s}_e^\mathrm{wt}(\sigma^2) \g. \label{e: se gamma}
\end{align}

In the next section we will extract temporal decay by combining the abstract linear estimates in Section~\ref{s: abstract linear} with bounds on $\bar{s}_p^\mathrm{fr}(\sigma), \bar{s}_c(\sigma)$ and $\bar{s}_e(\sigma)$. We recall that control on $\bar{s}_e(\sigma)$ follows from Lemma~\ref{l: resolvent S e estimates}. Moreover, control on $\bar{s}_c(\sigma)$ is provided by the estimates on $R_\mathrm{fr}(\cdot;\sigma)$ in Lemmas~\ref{l: fr Linf 0 1 resolvent estimate},~\ref{l: fr L2 resolvent estimate},  and~\ref{l: fr Linf resolvent estimate} in combination with the control on $\nuwt(\sigma^2)$, $Q(\cdot;\sigma^2)$, $P_\mathrm{wt}^\mathrm{cu}(\sigma^2)$ and  $\alpha_-(\sigma)$, obtained in Propositions~\ref{prop: wake spectral curve expansion} and~\ref{p: ff core invertibility} and Lemmas~\ref{l: floquet theory} and~\ref{l: tilde f control}, using the explicit expressions~\eqref{e: sc wt def} and~\eqref{e: sp wt def} of $\bar{s}_c^\mathrm{wt}(\sigma)$ and $\bar{s}_p^\mathrm{wt}(\sigma)$. Finally, using again the explicit expressions~\eqref{e: sp wt def} of $\bar{s}_p^\mathrm{wt}(\sigma)$, control on $\bar{s}_p^\mathrm{fr}(\sigma)$ follows by control on $\nuwt(\sigma^2)$ in combination with the following estimate on the interaction term $\chi_-\alpha_-(\sigma)\re^{\nuwt(\sigma^2) \cdot}$.

\begin{lemma}[Estimate on interaction term]\label{l: interaction norm resolvent estimate}
   Fix $1 \leq p \leq \infty$. There exist constants $C,\delta > 0$ such that we have the estimates 
    \begin{align}
        |\alpha_-(\sigma) - \alpha_-(0)| \leq C |\sigma| \| \g \|_{L^1_{0,1}}, &\qquad |\alpha_-(0)| \leq C\| \g \|_{L^1_{0,1}} \label{e: interaction term uniform estimate 2},
        \end{align}
   for all $\sigma \in \Delta^{\mathrm{fr},1}_\delta$ and $\g \in L_{0,1}^1(\R,\C^2)$, and 
     \begin{align}   
        \| [\alpha_-(\sigma) - \alpha_-(0)] \chi_- \re^{\nuwt (\sigma^2) \cdot} \|_{L^p} &\leq C |\sigma|^{1-\frac4p} \|\g\|_{L^1_{0,1}} ,\label{e: interaction term uniform estimate}
    \end{align}
    for all $\sigma^2 \in \Omega^{\mathrm{wt}}_\delta$ and $\g \in L_{0,1}^1(\R,\C^2)$.
    \end{lemma}
\begin{proof}
The first estimate~\eqref{e: interaction term uniform estimate 2} follows readily from Lemma~\ref{l: tilde f control}  and Proposition~\ref{p: ff core invertibility}. Noting that $\Omega_\delta^{\mathrm{wt}} \subset \Omega_\delta^{\mathrm{fr},1}$ (see Figure~\ref{fig: resolvent regions}), we observe that estimate~\eqref{e: interaction term uniform estimate 2} holds for all $\sigma^2 \in \Omega_\delta^\mathrm{wt}$ and $\g \in L_{0,1}^1(\R,\C^2)$.

For proving the second estimate, we fix $d > 0$ and define the curve
    \begin{align*}
    \Gamma_\mathrm{par}^{\tilde{\eps}} = B(0, \delta) \cap \{ \lambda_{\tilde{\eps}}(b) : b \in \R\}, \qquad \lambda_{\tilde{\eps}}(b) := \tilde{\eps} + \ri b - db^2, 
    \end{align*}
    for each $\tilde{\eps} > 0$. By estimate~\eqref{e: interaction term uniform estimate 2} and a basic scaling argument, we obtain
    \begin{align}
        \| [\alpha_-(\sigma) - \alpha_-(0)] \chi_- \re^{\nuwt (\sigma^2) \cdot} \|_{L^p} \lesssim \frac{|\sigma|}{|\Re \nuwt(\sigma^2)|^{1/p}} \|\g\|_{L^1_{0,1}}, \label{e: wt norm resolvent estimate 1}
    \end{align}
   for all $g \in L_{0,1}^1(\R,\C^2)$ and $\sigma^2 \in \Omega_\delta^\mathrm{wt}$. Using the expansion~\eqref{e: nu expansion}, we see that along $\Gamma_\mathrm{par}^{\tilde{\eps}}$, we have 
    \begin{align*}
        \Re \nuwt(\lambda_{\tilde{\eps}}(b)) = \nu_1 \tilde{\eps} - \nu_2 \tilde{\eps}^2 + b^2 (-\nu_1 d + \nu_2) + \mathrm{O}((\tilde{\eps}+b)^3). 
    \end{align*}
    From this expansion, we observe that, upon taking $d > 0$ smaller if necessary, there are constants $k_1, k_2 > 0$, uniform in $\tilde{\eps} > 0$ sufficiently small, such that
    \begin{align*}
        \Re \nuwt(\lambda_{\tilde{\eps}}(b)) \geq k_1 \tilde{\eps} + k_2 b^2, 
    \end{align*}
    holds along $\Gamma_\mathrm{par}^{\tilde{\eps}}$. Note that along $\Gamma_\mathrm{par}^{\tilde{\eps}}$, we have $|\sigma|^2 \sim \sqrt{\eps^2 + b^2},$ and so we have
    \begin{align*}
        \Re \nuwt(\sigma^2) \geq k |\sigma|^4,
    \end{align*}
    for $\sigma^2 \in \Gamma^{\tilde{\eps}}_\mathrm{par}$, where $k > 0$ is some constant independent of $\tilde{\eps}$ and $\sigma$. Combining this estimate and~\eqref{e: wt norm resolvent estimate 1}, we obtain~\eqref{e: interaction term uniform estimate} for $\sigma^2 \in \Gamma^{\tilde{\eps}}_\mathrm{par}$ with constant $C > 0$ independent of $\tilde{\eps}$ and $\sigma$. If $\delta > 0$ is small, we can write $\Omega^\mathrm{wt}_\delta$ as the union of all such $\Gamma^{\tilde{\eps}}_\mathrm{par}$, parameterized over $\tilde{\eps} > 0$ sufficiently small, and so the result follows. 
\end{proof}

\section{Linear estimates}\label{s: linear estimates}
In this section, we will apply the general linear estimates of Section~\ref{s: abstract linear} to the resolvent description of Section~\ref{s: resolvent} to establish a decomposition of the semigroup $\re^{\lfr t}$ with corresponding estimates, which are suitable for proving sharp nonlinear stability results. The outcome of our linear analysis may be summarized as follows.

\begin{thmlocal}[Linear estimates]\label{t: linear estimates}
    The semigroup $\re^{\lfr t}$, generated by $\lfr$, may be extended to an operator on $L^p(\R, \R^2)$ for $ 1 \leq p < \infty$ or on $C_0 (\R, \R^2)$. Moreover, $\re^{\lfr t}$ admits a decomposition
	\begin{align}
	[\re^{\lfr t} \g] (\xi) = \omega(\xi) \ufr'(\xi) [s_p(t) \g] (\xi) + [S_c (t) \g] (\xi) + [S_e (t) \g] (\xi), \qquad \xi \in \R, \, t > 0, \label{e: semigroup decomp}
	\end{align}
	with $s_p(0) = 0$, and there exist constants $C,\mu > 0$ such that the following estimates hold. 
	\begin{enumerate}
		\item \textbf{Estimates on exponentially damped part.} For all $t > 0$ we have
		\begin{align*}
		\| S_e (t) \|_{L^2 \to L^2} \leq C \re^{-\mu t}, \quad \| S_e (t) \|_{C_0 \to C_0} \leq C \re^{-\mu t}.
		\end{align*}
		\item \textbf{Estimates on smoothing principal part.} For each pair of non-negative integers $m$ and $\ell$, there exists a constant $C_{m, \ell} > 0$ such that
		\begin{align}
  \begin{split}
		\| \partial_\xi^\ell \partial_t^m s_p(t) \g \|_{L^2} &\leq \frac{C_{m, \ell}}{(1+t)^{\left(\frac{1}{4} + \frac{m+\ell}{2} \right) \wedge \frac{3}{2} }} \| \g \|_{L^1_{0,1}}, 
  \\
		\| \partial_\xi^\ell \partial_t^m s_p(t) \g \|_{L^\infty} &\leq \frac{C_{m, \ell}}{(1+t)^{\left( \frac{1}{2}+\frac{m+\ell}{2} \right) \wedge \frac{3}{2}}} \| \g \|_{L^1_{0,1}}, 
  \end{split}
  \label{e: linear principal 2}
		\end{align}
		for all $\g \in L_{0,1}^1(\R,\R^2)$ and $t > 0$, where $a \wedge b = \min\{a,b\}$.
		\item \textbf{Estimates on residual part.} We have
		\begin{align*}
  \begin{split}
		\| {S}_c (t) \g \|_{L^2} &\leq \frac{C}{(1+t)^{\frac34}} \left( \| \g \|_{L^1_{0,1}} + \| \g \|_{L^\infty} \right), \qquad
		\| {S}_c (t) \g \|_{L^\infty} \leq \frac{C}{1+t} \left( \| \g \|_{L^1_{0,1}} + \| \g\|_{L^\infty} \right),
  \end{split} 
		\end{align*}
  for all $\g \in L^1_{0,1} (\R,\R^2) \cap C_0 (\R,\R^2)$ and $t > 0$. 
	\end{enumerate}
	Moreover, we have the improved decay estimate in the leading edge,
	\begin{align}
	    \| \chi_+ \re^{\lfr t} \g \|_{L^\infty_{0, -1}} \leq \frac{C}{(1+t)^{\frac32}} \left(\| \g \|_{L^1_{0,1}} + \|\g\|_{L^\infty}\right), \label{e: leading edge linear decay}
	\end{align}
	for all $\g \in L^1_{0,1}(\R, \R^2) \cap C_0(\R, \R^2)$ and $t > 0$. 
\end{thmlocal}

Fix $\delta > 0$ small enough so that the decomposition~\eqref{e: resolvent decomp} of the resolvent holds in the ball $B(0, \delta)$, which is depicted in gray in the right panel of Figure~\ref{fig:contour shifting}. By Corollary~\ref{c: second shift}, we may write the action of the semigroup $\re^{\lfr t}$ on a test function $\g \in C^\infty_c (\R, \R^2)$ via the inverse Laplace representation as
\begin{align}
    \re^{\lfr t} \g = - \frac{1}{2 \pi \ri} \lim_{R \to \infty} \int_{\Gamma^2_R} \re^{\lambda t} (\lfr - \lambda)^{-1} \g \de  \lambda, \label{e: linear estimates inverse laplace formula}
\end{align}
where the contour $\Gamma^2_R$ is depicted in the right panel of Figure~\ref{fig:contour shifting}. To prove Theorem~\ref{t: linear estimates}, we will deform the portion $\Gamma^2_0 = \Gamma^2_R \cap B(0,\delta)$ of the contour $\Gamma^2_R$ near the origin in several different ways depending on the behavior of the different terms arising in the decomposition~\eqref{e: resolvent decomp} of the resolvent so that we can subsequently apply the abstract linear estimates obtained in Section~\ref{s: abstract linear}.

First, we identify the various terms in the decomposition~\eqref{e: semigroup decomp}. Recall by Corollary~\ref{c: second shift} that we can write the semigroup $\re^{\lfr t}$ via the inverse Laplace transform as
\begin{align*}
    \re^{\lfr t}\g = -\frac{1}{2 \pi \ri} \lim_{R \to \infty} \int_{\Gamma^{1,-}_R \cup \Gamma^{2, -}_\mathrm{int} \cup \Gamma^2_0 \cup \Gamma^{2,+}_\mathrm{int} \cup \Gamma^{1,+}_R} \re^{\lambda t} (\lfr - \lambda)^{-1}\g \de  \lambda,
\end{align*}
for $\g \in C_c^\infty(\R,\R^2)$. The segments $\Gamma^{1,\pm}_R$ and $\Gamma^{2, \pm}_\mathrm{int}$ of the contour are contained strictly in the left half-plane, and so contribute to the exponentially decaying part of the semigroup, see Corollary~\ref{c: second shift}. By Lemma~\ref{l: resolvent S e estimates} the term $\bar{s}_e^\mathrm{wt}(\lambda)$ in the decomposition~\eqref{e: resolvent decomp} of the resolvent is analytic in $\lambda$ in $L^p(\R)$ on the full ball $B(0,\delta)$. As a consequence, the integral over $\Gamma^2_0$ of this term is also exponentially decaying in time as the contour $\Gamma_0^2$ can be deformed within $B(0,\delta)$ so that it lies in the open left half-plane. Thus, collecting these exponentially decaying terms, we define
\begin{align}
\begin{split}
    S_e (t) \g &= -\frac{\color{black}{\varsigma(t)}}{2 \pi \ri} \lim_{R \to \infty} \int_{\Gamma^{1,-}_R \cup \Gamma^{2, -}_\mathrm{int} \cup \Gamma^{2, +}_\mathrm{int} \cup \Gamma^{1, +}_R} \re^{\lambda t} (\lfr - \lambda)^{-1} \g \de  \lambda - \frac{\color{black}{\varsigma(t)}}{2 \pi \ri} \int_{\Gamma^2_0} \re^{\lambda t} \chi_- \bar{s}_e^\mathrm{wt} (\lambda) \g_- \de  \lambda \\
    & \qquad +\, \color{black}{\left(1-\varsigma(t)\right) \re^{\El_{\mathrm{fr}} t} \g},
    \end{split}
    \label{e: Se t def}
\end{align}
{\color{black} where $\varsigma \colon [0,\infty) \to \R$ satisfies $\varsigma(t) = 0$ for $t \in [0,1]$ and $\varsigma(t) = 1$ for $t \in [2,\infty)$.\footnote{\color{black} The reason for introducing the temporal cut-off function is to assure that the critical diffusively decaying part of the semigroup vanishes at $t = 0$, that is,~we have $s_p(0) = 0$. This has the advantage that temporal derivatives of the phase modulation function simplify in the nonlinear stability argument.}} On the other hand, we collected all terms, contributing to the leading-order temporal dynamics, in the term $\bar{s}_p^\mathrm{fr}(\sigma)$ in the decomposition~\eqref{e: resolvent decomp} of the resolvent. Thus, we define
\begin{align}
    s_p(t) \g = - \frac{\color{black}{\varsigma(t)}}{2 \pi \ri} \int_{\Gamma^2_0} \re^{\sigma^2 t} \bar{s}_p^\mathrm{fr}(\sigma) \g \de  (\sigma^2), \label{e: sp t def}
\end{align}
for the principal part, which leaves only the algebraically decaying residual part
\begin{align}
    S_c(t) \g = - \frac{\color{black}{\varsigma(t)}}{2 \pi \ri} \int_{\Gamma^2_0} \re^{\sigma^2 t} \bar{s}_c(\sigma) \g \de  (\sigma^2). \label{e: sc t def}
\end{align}
Recall that expressions for $\bar{s}_e^\mathrm{wt}(\lambda), \bar{s}_p^\mathrm{fr}(\sigma)$, and $\bar{s}_c(\sigma)$ are given in~\eqref{e: resolvent S0 exp def},~\eqref{e: sp gamma}, and~\eqref{e: sc gamma}, respectively. 

Having identified the terms in the decomposition~\eqref{e: semigroup decomp}, we now estimate them one by one by choosing appropriate integration contours based on the analysis in Section~\ref{s: abstract linear}, thereby proving Theorem~\ref{t: linear estimates}.

\begin{lemma}[Estimates on exponentially damped part]
    The operator $S_e (t)$, defined by~\eqref{e: Se t def}, extends to a bounded linear operator on $L^2(\R,\R^2)$ and on $C_0(\R,\R^2)$ for each fixed $t > 0$. Moreover, there exist constants $C, \mu > 0$ such that for all $t > 0$ we have
    \begin{align*}
        \| S_e(t) \|_{L^2 \to L^2} \leq C \re^{-\mu t}, \quad \| S_e (t) \|_{C_0 \to C_0} \leq C \re^{-\mu t}. 
    \end{align*}
\end{lemma}
\begin{proof}
    The estimates on the first two terms in~\eqref{e: Se t def} follow for $\g \in C^\infty_c (\R, \R^2)$ from Corollary~\ref{c: second shift} together with analyticity in $\lambda$ of $\bar{s}^{\mathrm{wt}}_e(\lambda)$ on the full ball $B(0,\delta)$ obtained in Lemma~\ref{l: resolvent S e estimates}. From the proof of Proposition~\ref{p: Zk generation}, it is clear that $\re^{\lfr t}$ also generates a strongly continuous semigroup on $L^2(\R, \R^2)$ or on $C_0 (\R, \R^2)$. An exponential decay estimate on the last term in~\eqref{e: Se t def} follows from these facts together with compact support in time of $1-\varsigma(t)$. With these estimates in hand, we can then naturally extend $S_e(t)$ to the larger spaces $L^2 (\R, \R^2)$ and $C_0 (\R, \R^2)$ by approximating with test functions. 
\end{proof}

\begin{lemma}[Estimates on smoothing principal part] \label{l: principal temporal estimates}
    Let $s_p(t)$ be defined by~\eqref{e: sp t def}. For each pair of non-negative integers $\ell$ and $j$, the operator $\partial_\xi^\ell \partial_t^j s_p(t)$ extends to a bounded linear operator from $L^1_{0,1} (\R, \R^2)$ into $L^2(\R, \R^2)$ or from $L^1_{0, 1} (\R, \R^2)$ into $C_0(\R, \R^2)$. Moreover, there exists a constant $C_{j, \ell} > 0$ such that 
    \begin{align}
    \| \partial_\xi^\ell \partial_t^j s_p(t) \g \|_{L^2} &\leq \frac{C_{j, \ell}}{(1+t)^{\frac{1}{4} + \frac{j+\ell}{2}\wedge \frac{3}{2}}} \| \g \|_{L^1_{0,1}}, \qquad 
        \| \partial_\xi^\ell \partial_t^j s_p(t) \g \|_{L^\infty} \leq \frac{C_{j, \ell}}{(1+t)^{\frac{1}{2}+\frac{j+\ell}{2} \wedge \frac{3}{2} }} \| \g \|_{L^1_{0,1}}, \label{e: linear principal local 2}
    \end{align}
    for all $t>0$ and $\g \in L^1_{0,1}(\R)$.
\end{lemma}
\begin{proof}
    By~\eqref{e: sp gamma}, the action of $s_p(t)$ on a test function $\g \in C_c^\infty(\R,\R^2)$ decomposes as
    \begin{align}
        s_p(t) \g = - \frac{\color{black}{\varsigma(t)}}{2 \pi \ri} \int_{\Gamma^2_0} \re^{\lambda t} \chi_-  \bar{s}_p^\mathrm{wt} (\lambda) \g_- \de  \lambda - \frac{\color{black}{\varsigma(t)}}{2 \pi \ri} \int_{\Gamma^2_0} \re^{\sigma^2 t} \chi_- \alpha_- (\sigma) \re^{\nuwt (\sigma^2) \cdot} \de  (\sigma^2).\label{e: spt decomp 1}
    \end{align}
    We start estimating the first integral. With the aid of~\eqref{e: sp wt def} we rewrite this integral as
    \begin{align*}
        \int_{\Gamma^2_0} \re^{\lambda t} [\chi_- \bar{s}_p^\mathrm{wt} (\lambda) \g_-] (\xi) \de  \lambda  \,  = \phi \left(\int_{\Gamma^2_0} \re^{\lambda t} \chi_-(\xi) \int_\R \re^{\nuwt (\lambda) (\xi-\zeta)} \chi_-(\xi-\zeta) \Pwt^\mathrm{cu}(0) Q(\zeta, 0) \Lambda_1 \g_- (\zeta) \de \zeta \de \lambda \right). 
    \end{align*}
    Using that $\chi_-(\xi-\zeta) \re^{\nuwt(\lambda)(\xi-\zeta)}$ is uniformly exponentially localized for $\lambda$ in the compact contour $\Gamma^2_0$, we swap the order of integration to obtain
    \begin{align*}
       \int_{\Gamma^2_0} \re^{\lambda t} [\chi_- \bar{s}_p^\mathrm{wt} (\lambda) \g_-] (\xi) \de  \lambda = \phi \left( \chi_-(\xi) \int_\R \Pwt^\mathrm{cu}(0) Q(\zeta,0) \Lambda_1 \g_-(\zeta) \int_{\Gamma^2_0} \re^{\lambda t + \nuwt(\lambda) (\xi-\zeta)} \chi_- (\xi-\zeta) \de  \lambda \de\zeta \right).
    \end{align*}
    Since the map $\lambda \mapsto \re^{\lambda t + \nuwt(\lambda)(\xi-\zeta)}$ is analytic in $\lambda$ on the full ball $B(0,\delta)$ for each fixed $\xi, \zeta \in \R$ and $t \geq 0$ by Proposition~\ref{prop: wake spectral curve expansion}, we can deform the integration contour $\Gamma^2_0$ to the pointwise contour $\Gamma_{\xi,\zeta,t}$ obtained in Proposition~\ref{p: pointwise estimate}. Then, using Proposition~\ref{p: pointwise estimate} to estimate the resulting integral and applying Young's convolution inequality, we obtain
    \begin{align*}
        \left\| \int_{\Gamma^2_0} \re^{\lambda t} \chi_-  \bar{s}_p^\mathrm{wt} (\lambda) \g_- \de  \lambda  \right\|_{L^2} \lesssim  \| G^{0, 0, 0} (t, \cdot) \|_{L^2} \| Q(\cdot, 0) \g_-\|_{L^1} \lesssim \frac{\| \g \|_{L^1} }{(1+t)^{\frac14}} \leq \frac{\| \g \|_{L^1_{0,1}}}{(1+t)^{\frac14}} ,
    \end{align*}
    and
    \begin{align*}
        \left\| \int_{\Gamma^2_0} \re^{\lambda t} \chi_-  \bar{s}_p^\mathrm{wt} (\lambda) \g_- \de  \lambda  \right\|_{L^\infty} \lesssim  \| G^{0, 0, 0} (t, \cdot) \|_{L^\infty} \| Q(\cdot, 0) \g_-\|_{L^1} \lesssim \frac{\| \g \|_{L^1}}{(1+t)^{\frac12}}  \leq \frac{\| \g \|_{L^1_{0,1}}}{(1+t)^{\frac12}},
    \end{align*}
    for $\g \in C_c^\infty(\R,\R^2)$ and $t > 0$, which correspond to the estimates~\eqref{e: linear principal local 2} for $j = \ell = 0$. The estimate for $j > 0$ and $ \ell = 0$ follows in the same way: differentiating $j$ times the factor $\re^{\lambda t}$ produces a factor of $\lambda^j$, and we then can apply Proposition~\ref{p: pointwise estimate} with $j > 0$ to obtain the improved decay. Considering $\ell > 0$, the $\xi$-derivatives may fall on either $\chi_-(\xi), \re^{\nuwt(\lambda) (\xi-\zeta)}$ or $\chi_-(\xi-\zeta)$. If all $\ell$ derivatives fall on $\re^{\nuwt(\lambda) (\xi-\zeta)}$, then the desired estimate follows directly by applying Proposition~\ref{p: pointwise estimate} with $\ell > 0$. If any derivatives fall on $\chi_-(\xi-\zeta)$, then we use Proposition~\ref{prop: wake spectral curve expansion} and the fact that $\chi_-'$ is compactly supported to conclude that $\lambda \mapsto \re^{\nuwt(\lambda) \cdot}\, \chi_-'$ is analytic in $\lambda$ on $B(0,\delta)$ in $L^p(\R)$ for any $1 \leq p \leq \infty$, so that the contour integral can be shifted into the left half-plane, implying exponential decay in time. Similarly, if any derivatives fall on $\chi_-(\xi)$, then note that the map $(\xi,\zeta) \mapsto \chi_-'(\xi) \g_-(\zeta) \chi_-(\xi-\zeta) = \chi_-'(\xi)\chi_-(\zeta)\chi_-(\xi-\zeta)\g(\zeta)$ is supported on the compact set $\{ 1 \leq \xi \leq \zeta \leq 0 \}$, so that we do not have to worry about loss of spatial localization of $\re^{\nuwt(\lambda) (\xi - \zeta)}$. Hence, the resulting integral is again analytic in $\lambda$ on $B(0,\delta)$ in $L^p(\R)$ in this case and we gain exponential decay in time. This completes the estimates for the first term in~\eqref{e: spt decomp 1} for $\g \in C_c^\infty(\R,\R^2)$. 
    
    Next, we analyze the second term in~\eqref{e: spt decomp 1}, which encodes the interaction between branched and outgoing modes. We start by further decomposing this integral as
    \begin{align}
    \begin{split}
        \int_{\Gamma^2_0} \re^{\sigma^2 t} \chi_- \alpha_- (\sigma) \re^{\nuwt(\sigma^2) \cdot}\, \de  (\sigma^2) &= \int_{\Gamma^2_0} \re^{\sigma^2 t} \chi_- \alpha_-(0) \re^{\nuwt(\sigma^2) \cdot}\, \de  (\sigma^2)\\ &\qquad\,+ \int_{\Gamma^2_0} \re^{\sigma^2 t} \chi_- [\alpha_-(\sigma) - \alpha_-(0)] \re^{\nuwt(\sigma^2) \cdot}\, \de  (\sigma^2). 
        \end{split}\label{e: principal second decomp} 
    \end{align}
    The estimates~\eqref{e: linear principal local 2} for the first integral in~\eqref{e: principal second decomp} with $j = \ell = 0$ follow by applying the pointwise estimates in Proposition~\ref{p: pointwise estimate} with $\zeta = 0$ fixed. Using Lemma~\ref{l: interaction norm resolvent estimate} to estimate $|\alpha_-(0)|$, this leads to
    \begin{align*}
    \left\|\int_{\Gamma^2_0} \re^{\lambda t} \chi_- \alpha_-(0) \re^{\nuwt(\lambda) \cdot} \,\de \lambda\right\|_{L^p} \lesssim \|G^{0,0,0}(t,\cdot)\|_{L^p} |\alpha_-(0)| \lesssim  \frac{\| \g \|_{L^1_{0,1}}}{(1+t)^{\frac12 - \frac1{2p}}},
    \end{align*}
    for $p \in\{ 2,\infty\}$, $\g \in C_c^\infty(\R,\R^2)$ and $t > 0$. This can be extended to positive $j$ and $\ell$ by handling $\xi$- and $t$-derivatives exactly as for the first integral in~\eqref{e: spt decomp 1}. 
    
    We cannot use Proposition~\ref{p: pointwise estimate} for the second term in~\eqref{e: principal second decomp}, since $\alpha_-(\sigma) - \alpha_-(0)$ is not analytic in $\sigma^2$ in a full neighborhood of the origin. Instead, by Lemma~\ref{l: interaction norm resolvent estimate}, we have 
    \begin{align} \label{e: alpha Lipschitz estimate}
        \| \chi_- [\alpha_-(\sigma) - \alpha_-(0)] \re^{\nuwt(\sigma^2) \cdot} \|_{L^p} \lesssim |\sigma|^{1-\frac4p} \| \g \|_{L^1_{0,1}} = |\sigma|^{2\left(\frac12-\frac2p\right)} \| \g \|_{L^1_{0,1}},
    \end{align}
    for $p=2,\infty$, $\sigma^2 \in \Omega^\mathrm{wt}_\delta$ and $\g \in C_c^\infty(\R,\R^2)$. Since $\frac12 - \frac{2}{p} > -1$ for $p > \frac43$, we can apply Proposition~\ref{p: wake norm estimates} to obtain
    \begin{align}
        \left\| \int_{\Gamma^2_0} \re^{\sigma^2 t} \chi_- [\alpha_-(\sigma) - \alpha_-(0)] \re^{\nuwt(\sigma^2) \cdot} \de  (\sigma^2) \right\|_{L^p} \leq \frac{\| \g \|_{L^1_{0,1}}}{(1+t)^{\frac{3}{4} - \frac{1}{p}} }, \label{e: interaction estimate}
    \end{align}
    for $\g \in C_c^\infty(\R,\R^2)$, $t > 0$ and $p = 2, \infty$, which leads to the estimates~\eqref{e: linear principal local 2} for $\ell = j = 0$. The same argument applies to obtain faster decay rates for $j > 0$, since differentiating $\re^{\sigma^2 t}$ with respect to $t$ produces additional factors of $|\sigma|^2$, which contribute to faster decay by Proposition~\ref{p: wake norm estimates}. When $\ell > 0$, if all derivatives fall on $\re^{\nuwt(\sigma^2) \cdot}\,$, then the same argument applies by Proposition~\ref{prop: wake spectral curve expansion}. However, if some derivatives fall on $\chi_-$, then, since $\chi_-'$ is compactly supported, the function $\partial_x^k \chi_- \partial_x^{\ell - k} \re^{\nuwt(\sigma^2)\cdot}$ is analytic in $\sigma^2$ on $B(0,\delta)$ in $L^p(\R)$ provided $k \geq 1$. Hence, if we define
    \begin{align*}
        u(\sigma) = [\alpha_-(\sigma) - \alpha_- (0)] \partial_\xi^{k} \chi_- \partial_\xi^{\ell-k} \re^{\nuwt(\sigma^2) \cdot}\,, 
    \end{align*}
    then this analyticity together with the regularity of $\alpha_-(\sigma)$ from Lemma~\ref{l: tilde f control} and Proposition~\ref{p: ff core invertibility} implies that $u(\sigma)$ is analytic in $L^p(\R)$ in $\sigma^2$ on $\Omega^{\mathrm{fr},1}_\delta$. It enjoys the estimate $\| u(\sigma) \|_{L^p} \lesssim |\sigma| \|g\|_{L_{0,1}^1}$
    for $\sigma \in \Delta^{\mathrm{fr},1}_\delta$ and $\g \in \smash{C_c^\infty(\R,\R^2)}$ by Lemma~\ref{l: interaction norm resolvent estimate}. Applying Proposition~\ref{p: front estimate sharp decay}, we see that this term decays with rate $(1+t)^{-3/2}$. 
    
    Altogether, after using approximation with test functions, we have established~\eqref{e: linear principal local 2} for any pair of non-negative integers $(\ell,j)$, as desired. 
\end{proof}

\begin{lemma}[Estimates on smoothing residual part]
    Let $S_c(t)$ be defined by~\eqref{e: sc t def}. The operator $S_c(t)$ extends to a bounded linear operator from $L^1_{0,1}(\R, \R^2) \cap C_0(\R,\R^2)$ into $ L^2(\R,\R^2)$ or from $L^1_{0,1}(\R, \R^2) \cap C_0(\R,\R^2)$ into $ C_0(\R,\R^2)$. Moreover, there exists a constant $C > 0$ such that
    \begin{align*}
        \| S_c(t) \g \|_{L^2} \leq \frac{C}{(1+t)^{\frac34}} \| \g \|_{L^1_{0,1} \cap L^\infty}, \qquad \| S_c(t) \g \|_{L^\infty} \leq \frac{C}{1+t} \| \g \|_{L^1_{0,1} \cap L^\infty},
    \end{align*}
    for all $\g \in L^1_{0,1}(\R, \R^2) \cap C_0(\R,\R^2)$ and $t > 0$.
\end{lemma}
\begin{proof}
    Expanding $\bar{s}_c(\sigma) \g$ for a test function $\g \in C_c^\infty(\R,\R^2)$ via the formula~\eqref{e: sc gamma}, we find
    \begin{align}
    \begin{split}
        S_c(t) \g &= - \frac{\color{black}{\varsigma(t)}}{2 \pi \ri} \bigg( \int_{\Gamma^2_0} \re^{\sigma^2 t} R^\mathrm{fr}(\sigma) \de  (\sigma^2) + \int_{\Gamma^2_0} \re^{\sigma^2 t} \chi_- \bar{{s}}_c^\mathrm{wt}(\sigma^2)  \g_- \de  (\sigma^2) \\ 
        &\qquad\qquad\qquad + \int_{\Gamma^2_0} \re^{\sigma^2 t} \left(\alpha_-(\sigma) - \alpha_-(0)\right) \re^{\nuwt(\sigma^2) \cdot} \chi_- (\ufr'-\uwt') \de  (\sigma^2)\\ 
        &\qquad\qquad\qquad + \int_{\Gamma^2_0} \re^{\sigma^2 t} \chi_- (\ufr'-\uwt')\left(\alpha_-(0) + \bar{s}_p^\mathrm{wt} (\sigma^2)\g_-\right) \de  (\sigma^2) \\
        &\qquad\qquad\qquad + \int_{\Gamma^2_0} \re^{\sigma^2 t} \alpha_-(\sigma) \chi_- \re^{\nuwt (\sigma^2) \cdot} (\uwt' - \q(\cdot; \sigma^2)) \de  (\sigma^2) \bigg).
        \end{split}\label{e: estimating Sc t}
    \end{align}
    
    For the first integral, applying Lemma~\ref{l: fr L2 resolvent estimate} and Proposition~\ref{p: L2 front estimate}, we obtain the estimate
    \begin{align*}
        \left\| \int_{\Gamma^2_0} \re^{\sigma^2 t} R^\mathrm{fr}(\sigma) \de  (\sigma^2) \right\|_{L^2} \lesssim \frac{\|\g\|_{L^1_{0,1} \cap L^\infty} }{(1+t)^{\frac34}} , 
    \end{align*}
    for $\g \in C_c^\infty(\R,\R^2)$ and $t > 0$. Similarly, using Lemma~\ref{l: fr Linf resolvent estimate} instead of~\ref{l: fr L2 resolvent estimate}, we obtain
    \begin{align*}
        \left\| \int_{\Gamma^2_0} \re^{\sigma^2 t} R^\mathrm{fr}(\sigma) \de  (\sigma^2) \right\|_{L^\infty} \leq \frac{\|\g\|_{L^1_{0,1} \cap L^\infty}}{1+t}  ,
    \end{align*}
     for $\g \in C_c^\infty(\R,\R^2)$ and $t > 0$.
     
    The desired estimates on the second integral in~\eqref{e: estimating Sc t} follow exactly as those on the first integral in~\eqref{e: spt decomp 1} in the proof of Lemma~\ref{l: principal temporal estimates}, noting that each term in the formula~\eqref{e: sc wt def} for $\bar{{s}}^\mathrm{wt}_c(\sigma)$ carries by Lemma~\ref{l: floquet theory} an extra factor of $\lambda$ compared to $\bar{s}^\mathrm{wt}_p(\lambda)$, which improves the temporal decay rate by a factor $(1+t)^{-1/2}$. 
    
    For the third integral in~\eqref{e: estimating Sc t}, note that $\chi_- \re^{\nuwt(\sigma^2) \cdot} (\ufr' - \uwt')$ is analytic in $\sigma^2$ in a full neighborhood of the origin in $L^p(\R)$ for any $1 \leq p \leq \infty$ using Proposition~\ref{prop: wake spectral curve expansion} and the fact that $\chi_- (\ufr'-\uwt')$ is exponentially localized. Hence, the worst behavior in this term is in $\alpha_-(\sigma) - \alpha_-(0)$, which can be bounded with the aid of Lemma~\ref{l: interaction norm resolvent estimate}. Altogether, we can use Proposition~\ref{p: front estimate sharp decay} to obtain 
    \begin{align*}
        \left\| \int_{\Gamma^2_0} \re^{\sigma^2 t} \left(\alpha_-(\sigma) - \alpha_-(0)\right) \chi_- \re^{\nuwt(\sigma^2) \cdot} (\ufr'-\uwt') \de  (\sigma^2) \right\|_X \lesssim \frac{\|\g\|_{L^1_{0,1}}}{(1+t)^{\frac32}},
    \end{align*}
    for $g \in C_c^\infty(\R,\R^2)$ and $t > 0$, where $X = L^2(\R)$ or $X = L^\infty(\R)$. 

    For the fourth integral in~\eqref{e: estimating Sc t}, we observe that the integrand is analytic in $\sigma^2$ in a full neighborhood of the origin in $L^p(\R)$ for any $1 \leq p \leq \infty$, where we use the expression~\eqref{e: sp wt def} and the facts that $\nuwt(\lambda)$ is analytic in $\lambda$ by Proposition~\ref{prop: wake spectral curve expansion} and $\chi_- (\ufr'-\uwt')$ is exponentially localized. Thus, we gain exponential decay in time in $L^2(\R)$ and $L^\infty(\R)$ of the fourth integral.
    
    The last integral in~\eqref{e: estimating Sc t} can be estimated exactly as the second term in~\eqref{e: spt decomp 1} in the proof of Lemma~\ref{l: principal temporal estimates}. That is, we further decompose
    \begin{multline*}
        \int_{\Gamma^2_0} \re^{\sigma^2 t} \alpha_-(\sigma) \chi_- \re^{\nuwt (\sigma^2) \cdot} (\uwt' - \q(\cdot; \sigma^2)) \de  (\sigma^2) = \int_{\Gamma^2_0} \re^{\sigma^2 t} \alpha_-(0) \chi_- \re^{\nuwt (\sigma^2) \cdot} (\uwt' - \q(\cdot; \sigma^2)) \de  (\sigma^2) \\ + \int_{\Gamma^2_0} \re^{\sigma^2 t} [\alpha_-(\sigma)-\alpha_-(0)] \chi_- \re^{\nuwt (\sigma^2) \cdot} (\uwt' - \q(\cdot; \sigma^2)) \de  (\sigma^2).
    \end{multline*}
    Lemma~\ref{l: left mode} then  yields 
    \begin{align} \label{e: left neutral mode bound}
    \|\uwt' - \q(\cdot; \lambda)\|_{L^\infty} \lesssim |\lambda|,
    \end{align}
    for $|\lambda|$ sufficiently small. The desired estimates on the first term then follow by applying the pointwise estimates of Proposition~\ref{p: pointwise estimate}, using Lemma~\ref{l: interaction norm resolvent estimate} to bound $\alpha_-(0)$, and noting that we carry an extra factor of $|\lambda|$ by estimate~\eqref{e: left neutral mode bound}. Similarly, the second term carries an extra factor of $|\sigma|^2$ compared to~\eqref{e: interaction estimate}. Hence, we obtain the desired estimates by combining Proposition~\ref{p: wake norm estimates} with the estimates~\eqref{e: alpha Lipschitz estimate} and~\eqref{e: left neutral mode bound}. 
    
    Altogether, the result now follows by approximation with test functions.
\end{proof}

To complete the proof of Theorem~\ref{t: linear estimates}, it only remains to establish the improved decay~\eqref{e: leading edge linear decay} in the leading edge. 
\begin{lemma}[Improved decay in the leading edge] There exists a constant $C > 0$ such that for all $\g \in L^1_{0,1}(\R, \R^2) \cap C_0(\R,\R^2)$ and all $t > 0$, we have the estimate
\begin{align*}
    \| \chi_+ \re^{\lfr t} \g \|_{L^\infty_{0, -1}} \leq \frac{C}{(1+t)^{\frac32}} \| \g \|_{L^1_{0,1}\cap L^\infty}.  
\end{align*}
\end{lemma}
\begin{proof}
Using the inverse Laplace representation~\eqref{e: linear estimates inverse laplace formula}, we write
\begin{align*}
    \chi_+ \re^{\lfr t} \g = - \frac{1}{2 \pi \ri} \lim_{R \to \infty} \int_{\Gamma^{1,-}_R \cup \Gamma^{2, -}_\mathrm{int} \cup \Gamma^{2,+}_\mathrm{int} \cup \Gamma^{1, +}_R} \re^{\lambda t} \chi_+ (\lfr - \lambda)^{-1} \g \de  \lambda - \frac{1}{2 \pi \ri} \int_{\Gamma^2_0} \re^{\lambda t} \chi_+ (\lfr - \lambda)^{-1} \g \de  \lambda. 
\end{align*}
For the first term, we apply Corollary~\ref{c: second shift} to find a constant $r > 0$ such that
\begin{align}
\begin{split}
    \left\| \lim_{R \to \infty} \int_{\Gamma^{1,-}_R \cup \Gamma^{2, -}_\mathrm{int} \cup \Gamma^{2,+}_\mathrm{int} \cup \Gamma^{1, +}_R} \re^{\lambda t} \chi_+ (\lfr - \lambda)^{-1} \g \de  \lambda \right\|_{L^\infty_{0, -1}} 
    &\lesssim \re^{-rt} \| \g \|_{L^\infty} ,
    \end{split}
    \label{e: leading edge estimate exponential part}
\end{align}
for all $\g \in C_c^\infty(\R,\R^2)$ and $t > 0$, where we use the continuous embedding $L_{0,-1}^\infty(\R) \hookrightarrow L^\infty(\R)$.

For the integral over $\Gamma^2_0$, we find from the resolvent decomposition~\eqref{e: resolvent decomp} that, for $\lambda = \sigma^2$ in a neighborhood of the origin, we have
\begin{align*}
    \chi_+ (\lfr - \sigma^2)^{-1} \g = \chi_+ R_\mathrm{fr} (\cdot; \sigma), 
\end{align*}
which satisfies the estimate
\begin{align*}
    \| \chi_+ [R_\mathrm{fr} (\cdot; \sigma) - R_\mathrm{fr}(\cdot; 0)]\|_{L^\infty_{0, -1}} \lesssim |\sigma| \| \g \|_{L^1_{0,1}},
\end{align*}
for all $\sigma \in \Delta^{\mathrm{fr},1}_\delta$ and $\g \in C_c^\infty(\R,\R^2)$ by Lemma~\ref{l: fr Linf 0 1 resolvent estimate}. Using Proposition~\ref{p: front estimate sharp decay}, we therefore obtain 
\begin{align*}
    \left\| \int_{\Gamma^2_0} \re^{\lambda t} \chi_+ (\lfr - \lambda)^{-1} \g \de  \lambda \right\|_{L^\infty_{0, -1}} \lesssim \frac{ \| g \|_{L^1_{0,1}}}{(1+t)^{\frac32}},
\end{align*}
for $t > 0$ and $\g \in C_c^\infty(\R,\R^2)$. Combining this with~\eqref{e: leading edge estimate exponential part} completes the proof of the lemma by approximation with test functions. 
\end{proof}

\section{Nonlinear iteration scheme and nonlinear estimates} \label{sec:itscheme}

In this section we formulate a nonlinear iteration scheme and state associated nonlinear estimates. In the next section we then prove our nonlinear stability result by closing an iterative argument based on this scheme with the nonlinear estimates.

We are interested in the long-time dynamics of the perturbed solution $\u(t)$ of~\eqref{e: fhn comoving} with initial condition $\u(0) = \u_{\mathrm{fr}} + \w_0$, where $\v_0 := \omega \w_0 \in L^1_{0,1}(\R,\R^2) \cap \left(H^3(\R,\R) \times H^2(\R,\R)\right)$ is sufficiently small. The initial perturbation $\w_0$ is $L^1$-localized on $(-\infty,0]$ and exponentially localized on $[0,\infty)$. Recall that the latter is necessary to marginally stabilize the essential spectrum associated with the leading edge of the front. 

The perturbed solution can be written as $\u(t) = \u_{\mathrm{fr}} + \w(t)$, where the perturbation $\w(t)$ has initial condition $\w(0) = \w_0$. Consequently, the weighted perturbation $\vt(t) = \omega \w(t)$ with $\vt(0) = \v_0$ satisfies
\begin{align}
\vt_t = \El_{\mathrm{fr}} \vt + \NT(\vt), \label{e:umodpert}
\end{align}
with locally Lipschitz continuous nonlinearity $\NT \colon H^1(\R,\R^2) \to H^1(\R,\R^2)$ given by
\begin{align*}
\NT(\vt) = \omega \widetilde{N}\left(\frac{\vt}{\omega}\right), \qquad \widetilde{N}(\w) = F(\u_{\mathrm{fr}}+\w) - F(\u_{\mathrm{fr}}) - F'(\u_{\mathrm{fr}}) \w.
\end{align*}
As stated in Proposition~\ref{p: Zk generation}, the linearization $\El_{\mathrm{fr}}$ generates a $C^0$-semigroup on the Hilbert space $H^1(\R)$ with domain $H^3(\R) \times H^2(\R)$. Hence, the following local well-posedness result follows from classical semigroup theory, see e.g.~\cite[Theorems~6.1.4 and~6.1.6]{Pazy}. 

\begin{prop}[Local well-posedness of the unmodulated perturbation] \label{p:local_unmod}
There exists a maximal time $T_{\max} \in (0,\infty]$ such that~\eqref{e:umodpert} possesses a unique classical solution
\begin{align*} \vt \in C\big([0,T_{\max}),H^3(\R,\R) \times H^2(\R,\R)\big) \cap C^1\big([0,T_{\max}),H^1(\R,\R^2)\big), \end{align*}
with initial condition $\vt(0) = \v_0$. Moreover, if $T_{\max} < \infty$, then we have
\begin{align*} \limsup_{t \nearrow T_{\max}} \left\|\vt(t)\right\|_{H^1} = \infty.\end{align*}
\end{prop}

First, we recall the discussion from Section~\ref{s: nonlinear diffusive stability} about the difficulties and overall strategy of the nonlinear iteration argument. The principal difficulty is that the overall linear decay rate $(1+t)^{-1/2}$  of Theorem~\ref{t: linear estimates} is too slow to directly close a nonlinear iteration scheme. To overcome this, we draw on ideas from the nonlinear stability of wave trains~\cite{DSSS, SSSU, JONZ, JNRZ_13_1, JNRZInventiones} and introduce the \emph{inverse-modulated perturbation}
\begin{align}
\v(\xi,t) = \omega(\xi)\left(\u(\xi - \psi(\xi,t),t) - \u_{\mathrm{fr}}(\xi)\right). \label{e:modpert}
\end{align}
We will choose the spatio-temporal phase modulation $\psi$ to capture the leading-order behavior of the solution, crucially relying on the detailed semigroup decomposition of Theorem~\ref{t: linear estimates}. Using this phase modulation, we will obtain improved decay for $\v$ which will ultimately allow us to close a nonlinear iteration.

We then, however, encounter the additional difficulty that the coupled system for $\v$ and $\psi$ is quasilinear, and we need to overcome an apparent loss of regularity in our nonlinear iteration scheme. To control regularity, we will introduce the forward-modulated perturbation,
\begin{align}
\vf(\xi,t) = \omega(\xi)\left(\u(\xi,t) - \u_{\mathrm{fr}}(\xi + \psi(\xi,t))\right). \label{e:fmodpert}
\end{align}
By putting the phase modulation in $\ufr$ rather than the solution $\u$ itself, we ensure that the equation for $\vf$ remains semilinear, so that it is easier to control regularity for $\vf$. We rely on ideas recently developed in~\cite{ZUM22} to show that decay rates for $\v$ and $\vf$ are equivalent, so that we may enjoy improved temporal decay of $\v$ while using nonlinear damping estimates on $\vf$ to control regularity, ultimately closing a nonlinear iteration argument to prove Theorem~\ref{t: main detailed}. 

The rest of this section is structured as follows. First, we derive an equation for the inverse-modulated perturbation $\v(t)$ and establish estimates on the corresponding nonlinearity. We then choose the phase modulation function $\psi(t)$ in such a way that it accounts for supercritical terms in the Duhamel formula for $\v(t)$. Subsequently, we consider the equation for the forward-modulated perturbation and derive an associated nonlinear damping estimate, which yields regularity control for the inverse-modulated perturbation by bounding the relevant norms of $\v(t)$ in terms of those of $\vf(t)$ plus controllable error terms in $\psi_\xi(t)$.

\subsection{The inverse-modulated perturbation}

Using that both the perturbed solution $\u(t)$ and the pattern-forming front $\u_{\mathrm{fr}}$ solve the FitzHugh-Nagumo system~\eqref{e: fhn comoving}, we derive an equation for the inverse-modulated perturbation~\eqref{e:modpert} in Appendix~\ref{app: modulated sys derivation}, which reads
\begin{align}
\left(\partial_t - \El_{\mathrm{fr}}\right)\left[\v + \omega \u_{\mathrm{fr}}' \psi\right] = \mathcal{N}(\v,\psi,\partial_t \psi) + \left(\partial_t - \El_{\mathrm{fr}}\right)\left[\psi_\xi \v\right], \label{e:modpertbeq}
\end{align}
with nonlinearity $\mathcal{N}$ given by
\begin{align*}
\mathcal{N}(\v,\psi,\psi_t) &= \omega\left(\mathcal Q\left(\frac{\v}{\omega},\psi\right) + \partial_\xi \mathcal R\left(\frac{\v}{\omega},\psi,\psi_t\right)\right),
\end{align*}
where
\begin{align} \label{e:defNLQ}
\mathcal Q(\z,\psi) &= \left(F(\u_{\mathrm{fr}}+\z) - F(\u_{\mathrm{fr}}) - F'(\u_{\mathrm{fr}}) \z\right)\left(1-\psi_\xi\right),
\end{align}
is quadratic in $\z$ and 
\begin{align} \label{e:defNLR}
\mathcal R(\z,\psi,{\color{blue} \psi_t}) = \z\left(c\psi_\xi - \psi_t\right) + D\left(\frac{\left(\z_\xi + \ufr'\psi_\xi\right)\psi_\xi}{1-\psi_\xi} + \left(\z\psi_\xi\right)_\xi\right), 
\end{align}
contains all terms which are linear in $\z$. We observe that equation~\eqref{e:modpertbeq} is quasilinear in $\v(t)$. Using the continuous embedding $H^1(\R) \hookrightarrow L^\infty(\R)$ and the fact that $\omega(\xi)^{-1}$ is bounded and decays exponentially as $\xi \to \infty$, we establish the relevant nonlinear estimate, where we recall the function spaces $Z_k(\R) = H^k(\R) \cap C^k_0 (\R)$ for $k \in \mathbb N_0$, which were defined in~\S\ref{s: function spaces}.

\begin{lemma}[Nonlinear estimates] \label{l: nonlinear estimates}
There exists a constant $C > 0$ such that the inequality
\begin{align*}
\begin{split}
\left\|\mathcal{N}\left(\v,\psi,\psi_t\right)\right\|_{L^1_{0,1} \cap Z_0} &\leq C\left(\|v_1\|_{Z_0}^2 + \left\|\omega \left(\psi_\xi,\psi_t\right)\right\|_{Z_2 \times Z_1}\left(\left\|\v\right\|_{Z_2 \times Z_1} + \left\|\omega \psi_\xi\right\|_{Z_0}\right)\right),
\end{split}
\end{align*}
holds for each $\v = (v_1,v_2) \in Z_2(\R) \times Z_1(\R)$ and $(\psi,\psi_t) \in Z_3(\R) \times Z_1(\R)$ satisfying $\|v_1\|_{L^\infty}, \|\psi_\xi\|_{W^{1,\infty}} \leq \frac{1}{2}$ and $\|\omega \left(\psi_\xi,\psi_t\right)\|_{Z_2 \times Z_1} < \infty$.
\end{lemma}

For the moment we assume that the phase modulation function $\psi(t)$ vanishes identically at $t = 0$. Then, the Duhamel formulation of~\eqref{e:modpertbeq} reads
\begin{align}
\v(t) + \omega\u_{\mathrm{fr}}'\psi(t) = \re^{\El_{\mathrm{fr}} t} \v_0 + \int_0^t \re^{\El_{\mathrm{fr}}(t-s)}\mathcal{N}(\v(s),\psi(s),\partial_t \psi(s))\de s + \psi_\xi(t)\v(t). \label{e:intv}
\end{align}
We make a choice for $\psi(t)$ such that the linear term $\u_{\mathrm{fr}}'\psi(t)$ compensates for the  critical nonlinear contributions on the right-hand side of equation~\eqref{e:intv}. More precisely, motivated by the semigroup decomposition~\eqref{e: semigroup decomp}, we make the implicit choice
\begin{align}
\psi(t) = s_p(t)\v_0 + \int_0^t s_p(t-s) \mathcal{N}(\v(s),\psi(s),\partial_t \psi(s))\de s, \label{e:intpsi}
\end{align}
which, upon substitution in~\eqref{e:intv}, leads to the integral equation
\begin{align}
\begin{split}
\v(t) &= \left(S_c(t) + S_e(t)\right)\v_0+\int_0^t\left(S_c(t-s)+S_e(t-s)\right)\mathcal{N}(\v(s),\psi(s),\partial_t \psi(s)) \de s + \psi_\xi(t)\v(t),
\end{split}\label{e:intv2}
\end{align}
for the inverse-modulated perturbation, where we use that the identities~\eqref{e: sp gamma} and~\eqref{e: sp t def} imply that $\psi(t)$ vanishes on $[-1,\infty)$. We note that, due to the introduction of the temporal cut-off function $\varsigma(t)$, we have $s_p(0) = 0$ and thus we have $\psi(0) = 0$. Moreover, we emphasize that, with this choice of $\psi(t)$, the inverse-modulated perturbation exhibits, at least on the linear level, higher-order algebraic pointwise decay at rate $t^{-1}$, see~Theorem~\ref{t: linear estimates}, whereas the phase-modulation $\psi(t)$ decays diffusively at rate $\smash{t^{-\frac12}}$. Nevertheless, the nonlinearity $\mathcal{N}$ in~\eqref{e:intpsi} and~\eqref{e:intv2} only contains spatial and temporal derivatives of $\psi$, which satisfy
\begin{align}
\partial_\xi^\ell \partial_t^j \psi(t) = \partial_\xi^\ell \partial_t^j s_p(t)\v_0 + \int_0^t \partial_\xi^\ell \partial_t^j s_p(t-s) \mathcal{N}(\v(s),\psi(s),\partial_t \psi(s))\de s, \label{e:intpsi2}
\end{align}
for $\ell,j \in \N_0$, where we use that $s_p(0)=0$. The estimates in Theorem~\ref{t: linear estimates} show that the linear terms in~\eqref{e:intpsi2} also decay pointwise at rate $t^{-1}$ for $\ell + j \geq 1$. These higher-order decay rates suggest that, as in the stability analyses~\cite{JNRZ_13_1,JONZ} of periodic wave trains, an iteration scheme consisting of $\v(t)$ and \emph{derivatives} of $\psi(t)$ might close, which we indeed confirm in the next section. Note that we also need to address the loss of derivatives in the quasilinear integral scheme~\eqref{e:intv2}-\eqref{e:intpsi2}, an issue we tackle in the next section through nonlinear damping estimates on the forward-modulated perturbation~\eqref{e:fmodpert}. 

We end this section by establishing local well-posedness of the phase modulation and the inverse-modulated perturbation. To that end, we note that we can express the inverse-modulated perturbation as
\begin{align}
\v(\xi,t) = \vt(\xi-\psi(\xi,t),t) + \omega(\xi)\left(\u_{\mathrm{fr}}(\xi - \psi(\xi,t)) - \u_{\mathrm{fr}}(\xi)\right), \label{e:definvmod2}  
\end{align}
and recall that local well-posedness of the unmodulated perturbation $\vt(t)$ has been established in Proposition~\ref{p:local_unmod}. Thus, the integral equation~\eqref{e:intpsi} forms a closed system in $\psi$ for which a standard contraction mapping argument yields local well-posedness.

\begin{prop}[Local well-posedness of the phase modulation] \label{p:psi}
Let $T_{\max}$ and $\vt(t)$ be as in Proposition~\ref{p:local_unmod}. Then, there exist a constant $r_0 > 0$ and a maximal time $\tau_{\max} \in (0,T_{\max}]$ such that~\eqref{e:intpsi}, with $\v$ given by~\eqref{e:definvmod2}, has a unique solution
\begin{align*} \psi \in C\big([0,\tau_{\max}),H^4(\R)\big) \cap C^{1}\big([0,\tau_{\max}), H^{2}(\R)\big), \end{align*}
satisfying $\psi(t)=0$ for all $t \in [0,\tau_{\max})$ with $t \leq 1$. In addition, it holds $\|(\psi(t),\partial_t \psi(t))\|_{H^4 \times H^2} < r_0$ and $\|\psi(t)\|_{Z_2} \leq \frac12$ for all $t \in [0,\tau_{\max})$. Furthermore, if $\tau_{\max} < T_{\max}$, then we have
\begin{align*} \limsup_{t \nearrow \tau_{\max}} \|(\psi(t),\partial_t \psi(t))\|_{H^4 \times H^2} = r_0.\end{align*}
Finally, it holds $\psi \in C\big([0,\tau_{\max}),H^{4+m}(\R)\big) \cap C^{1+j}\big([0,\tau_{\max}), H^{2+l}(\R)\big)$ for any $j, l, m \in \mathbb{N}_0$.
\end{prop}

For the sake of completeness we provide a proof of Proposition~\ref{p:psi} in Appendix~\ref{app:local}.

Local well-posedness of the inverse-modulated perturbation~\eqref{e:definvmod2} now readily follows.

\begin{corollary}[Local well-posedness of the inverse-modulated perturbation] \label{C:local_v}
Let $\vt(t)$ be the unmodulated perturbation from Proposition~\ref{p:local_unmod} and let $\psi(t)$ and $\tau_{\max}$ be as in Proposition~\ref{p:psi}. Then, the inverse-modulated perturbation, defined by~\eqref{e:definvmod2}, satisfies 
\begin{align*}
\v = (v_1,v_2) \in C\left([0,\tau_{\max}),Z_2(\R) \times Z_1(\R)\right).
\end{align*} 
Moreover, the Duhamel formulation~\eqref{e:intv} holds for all $t \in [0,\tau_{\max})$.
\end{corollary}
\begin{proof}
The function $\v$ maps continuously into $C_0^2(\R) \times C_0^1(\R)$ by Propositions~\ref{p:local_unmod} and~\ref{p:psi}, the mean value theorem, the continuous embedding $H^1(\R) \hookrightarrow C_0(\R)$ and the fact that functions in $C_0(\R)$ are uniformly continuous. Moreover, $\v$ maps continuously into $H^2(\R) \times H^1(\R)$ by combining Proposition~\ref{p:psi} with Lemma~\ref{L:loc_gamma}.
\end{proof}

\subsection{The forward-modulated perturbation}

In this section we derive a nonlinear damping estimate for the forward-modulated perturbation $\vf(t)$ in order to control regularity in the quasilinear iteration scheme~\eqref{e:intv2}-\eqref{e:intpsi2}. That is, we derive an energy estimate controlling the $(H^3 \times H^2)$-norm of $\vf(t)$ in terms of its $L^2$-norm and the $(H^3 \times H^2)$-norm of the initial condition $\v_0$. 

We start by establishing local well-posedness of the forward-modulated perturbation, which readily follows by combining Propositions~\ref{p:local_unmod} and~\ref{p:psi} and applying the mean value theorem. 

\begin{corollary}[Local well-posedness of the forward-modulated perturbation] \label{c:local_vf}
Let $\vt(t)$ be the unmodulated perturbation from Proposition~\ref{p:local_unmod} and let $\psi(t)$ and $\tau_{\max}$ be as in Proposition~\ref{p:psi}. Then, the forward-modulated perturbation, defined by~\eqref{e:fmodpert}, satisfies $\vf \in C\big([0,\tau_{\max}),H^3(\R) \times H^2(\R)\big) \cap C^1\big([0,\tau_{\max}),H^1(\R)\big)$.
\end{corollary}

Using that both the perturbed solution $\u(t)$ and the pattern-forming front $\u_{\mathrm{fr}}$ solve the FitzHugh-Nagumo system~\eqref{e: fhn comoving}, we find that the forward-modulated perturbation $\vf(t)$ satisfies the equation
\begin{align}
\begin{split}
\vf_t &= D\left(\vf_{\xi \xi} + a_1 \vf_\xi + a_0 \vf\right) + \clin \vf_\xi + b_0\vf + \omega \left(F\left(\frac{\vf}{\omega} + \uf_{\mathrm{fr},0}\right) - F(\uf_{\mathrm{fr},0})\right)\\ &\qquad +\, \omega \left(\clin \psi_\xi - \partial_t \psi\right)\uf_{\mathrm{fr},1} + \omega D\left(\psi_{\xi\xi} \uf_{\mathrm{fr},1} + \psi_\xi\left(\psi_\xi + 2\right)\uf_{\mathrm{fr},2}\right),
\end{split}
\label{e:fmodpertbeq}
\end{align}
where we denote $\uf_{\mathrm{fr},j}(\xi,t) = \big(\partial_\xi^j \ufr\big)(\xi + \psi(\xi,t))$ and $a_j, b_j$ are generated by conjugation with the exponential weight $\omega$, with the expressions
\begin{align}
    a_1 = 2 \omega (\omega^{-1})', \qquad a_0 = \omega (\omega^{-1})'', \qquad b_0  = \clin \omega (\omega^{-1})'. \label{e: coefficients from weight}
\end{align}

We emphasize that, in contrast to the equation~\eqref{e:modpertbeq} for the inverse-modulated perturbation $\v(t)$, the equation~\eqref{e:fmodpertbeq} is semilinear in $\vf$, which implies that all nonlinear terms in~\eqref{e:fmodpertbeq} can be controlled by linear damping terms. As mentioned before, the linear damping terms in~\eqref{e:fmodpertbeq} are $\partial_{\xi\xi}\mathring{v}_1$ in the first component and $-\epsilon \gamma \mathring{v}_2$ in the second component. Using these observations we establish the following nonlinear damping estimate for the forward-modulated perturbation.

\begin{prop}[Nonlinear damping estimate for the forward-modulated perturbation] \label{p: nonlinear damping}
Let $\vf(t)$ be as in Corollary~\ref{c:local_vf} and $\psi(t)$ and $\tau_{\max}$ as in Proposition~\ref{p:psi}. Fix $R > 0$. There exist constants $C, \mu > 0$ such that the forward-modulated perturbation $\vf(t)$ satisfies the nonlinear damping estimate
\begin{align} \label{e:dampingineq}
\begin{split}
\|\vf(t)\|_{H^3 \times H^2}^2 &\leq C\left(\re^{-\mu t} \|\v_0\|_{H^3 \times H^2}^2 + \left\|\vf(t)\right\|_{L^2}^2\phantom{\int_0^t}\right. \\
&\qquad \qquad \qquad \left. + \, \int_0^t \re^{-\mu(t-s)} \left(\|\vf(s)\|_{L^2}^2 + \|\psi_\xi(s)\|_{H^3}^2 + \|\partial_s \psi(s)\|_{H^2}^2 \right) \de s\right),
\end{split}
\end{align}
for each $t \in [0,\tau_{\max})$ with
\begin{align}\label{e:upbound}
\sup_{0 \leq s \leq t} \left(\|\mathring{v}_1(s)\|_{W^{1,\infty}} + \|\psi_\xi(s)\|_{W^{1,\infty}}\right) \leq R.
\end{align}  
\end{prop}
\begin{proof}
Our aim is to infer an estimate for the energy 
\begin{align*}E(t) = \frac12 \left\|\partial_\xi^3 \mathring{v}_1(t)\right\|_{L^2}^2 + \frac1{2\epsilon \gamma} \left\|\partial_\xi^2 \mathring{v}_2(t)\right\|_{L^2}^2.\end{align*} 
In order to be able to differentiate $E(t)$ with respect to $t$, we restrict ourselves for the moment to initial conditions $\v_0 \in H^5(\R) \times H^4(\R)$. One then obtains, using the same reasoning as in the proofs of Propositions~\ref{p:local_unmod} and~\ref{p:psi}, analogous local well-posedness statements with two additional degrees of regularity. That is, we have $\vt \in C\big([0,T_{\max}),H^5(\R) \times H^4(\R)\big) \cap C^1\big([0,T_{\max}),H^3(\R)\big)$, $\psi \in C\big([0,\tau_{\max}),H^5(\R)\big) \cap C^1\big([0,\tau_{\max}), H^3(\R)\big)$ and, thus, $\vf \in  C\big([0,\tau_{\max}),H^5(\R) \times H^4(\R)\big) \cap C^1\big([0,\tau_{\max}),H^3(\R)\big)$.

Fix a constant $R> 0$. Let $t \in [0,\tau_{\max})$ be such that~\eqref{e:upbound} holds. Differentiating $E(s)$ and integrating by parts we establish a $t$-independent constant $C > 0$ such that
\begin{align*}
\begin{split}
E'(s) &\leq -\left\|\partial_\xi^4 \mathring{v}_1(s)\right\|_{L^2}^2 - \left\|\partial_\xi^2 \mathring{v}_2(s)\right\|_{L^2}^2 + \left\|\partial_\xi^4 \mathring{v}_1(s)\right\|_{L^2}\left\|\partial_\xi^2 \mathring{v}_2(s)\right\|_{L^2} + C\left(\left\|\partial_\xi^2\mathring{v}_1(s)\right\|_{L^2}\left\|\partial_\xi^2 \mathring{v}_2(s)\right\|_{L^2}\right. \\
&\qquad \left. + \, \left\|\mathring{v}_1(s)\right\|_{H^3}^2 + \left(\left\|\vf(s)\right\|_{H^3 \times H^1} + \left\|\partial_x^4 \mathring{v}_1(s)\right\|_{L^2} + \left\|\partial_\xi^2 \mathring{v}_2(s)\right\|_{L^2}\right)\left(\left\|\psi_\xi(s)\right\|_{H^3} + \left\|\partial_s \psi(s)\right\|_{H^2}\right)\right),
\end{split}
\end{align*}
for $s \in [0,t]$. Here we used the embedding $H^1(\R) \hookrightarrow L^\infty(\R)$ and the facts that $\psi(s)$ vanishes on $[-1,\infty)$ by Proposition~\ref{p:psi}, $a_1,a_0$ and $b_0$ are bounded, and $b_0$ is non-positive since $\omega$ is non-decreasing and $\clin > 0$. Next, we apply Young's inequality to the above estimate to yield a $t$-independent constant $C > 0$ such that
\begin{align*}
\begin{split}
E'(s) &\leq -\frac14\left\|\partial_\xi^4 \mathring{v}_1(s)\right\|_{L^2}^2 - \frac14 \left\|\partial_\xi^2 \mathring{v}_2(s)\right\|_{L^2}^2 + C\left(\left\|\vf(s)\right\|_{H^3 \times H^1}^2 + \left\|\psi_\xi(s)\right\|_{H^3}^2 + \left\|\partial_s \psi(s)\right\|_{H^2}^2\right).
\end{split}
\end{align*}
Subsequently, we use Sobolev interpolation to obtain $t$-independent constants $C,\mu > 0$ such that
\begin{align*}
\begin{split}
E'(s) &\leq -\mu E(s) + C\left(\left\|\vf(s)\right\|_{L^2}^2 + \left\|\psi_\xi(s)\right\|_{H^3}^2 + \left\|\partial_s \psi(s)\right\|_{H^2}^2\right),
\end{split}
\end{align*}
for $s \in [0,t]$. Integrating the above inequality we arrive at
\begin{align*}
E(t) \leq \re^{-\mu t} E(0) + C\int_0^t \re^{-\mu(t-s)}\left(\left\|\vf(s)\right\|_{L^2}^2 + \left\|\psi_\xi(s)\right\|_{H^3}^2 + \left\|\partial_s \psi(s)\right\|_{H^2}^2\right) \de s,
\end{align*}
Finally, using Sobolev interpolation again, we establish a $t$-independent constant $C > 0$ such that~\eqref{e:dampingineq} holds.

For the case $\v_0 \in H^3(\R) \times H^2(\R)$, we approximate $\v_0$ in $(H^3 \times H^2)$-norm by a sequence $\left(\v_{0,n}\right)_{n \in \N}$ in $H^5(\R) \times H^4(\R)$. By continuity with respect to initial data, see~\cite[Proposition~4.3.7]{CA98}, we obtain sequences of solutions $\vt_n(t)$ of~\eqref{e:umodpert} with $\vt_n(0) = \v_{0,n}$ and of solutions $\psi_n(t)$ of~\eqref{e:intpsi} (with $\v_0$ replaced by $\v_{0,n}$) such that $\vt_n(t)$ converges to $\vt(t)$ in $H^3(\R) \times H^2(\R)$, $\psi_n(t)$ converges to $\psi(t)$ in $H^4(\R)$ and $\partial_t \psi_n(t)$ converges to $\partial_t \psi(t)$ in $H^2(\R)$. Since~\eqref{e:dampingineq} only depends on the $(H^3\times H^2)$-norm of $\vf(t) = \vt(t) + {\color{red}\omega}(\uf_{\mathrm{fr}} - \ufr)$, on the $H^4$-norm of $\psi(t)$ and on the $H^2$-norm of $\partial_t \psi(t)$, the desired result follows by approximation.
\end{proof}

As long as the phase modulation $\psi(t)$ and its spatial derivative stay sufficiently small, one can express the forward- and inverse-modulated perturbation in terms of each other by inverting the function $\xi \mapsto \xi - \psi(\xi,t)$. With the aid of the mean value theorem, one then establishes that the $W^{k,p}$-norms of the forward- and inverse-modulated perturbations are equivalent modulo controllable norms of $\psi_\xi$. This has been established in~\cite[Corollary 5.3]{ZUM22} for the case without exponential weight. For the sake of completeness, we obtain this equivalence for the relevant norms in the current setting in the following lemma.

\begin{lemma}[Equivalence of the forward- and inverse-modulated perturbations] \label{l: equivalence}
Let $\v(t)$ be as in Corollary~\ref{C:local_v}, $\vf(t)$ as in Corollary~\ref{c:local_vf} and $\psi(t)$ and $\tau_{\max}$ as in Proposition~\ref{p:psi}. Then, there exists a constant $C > 0$ such that
\begin{align}
\begin{split}
\|\v(t)\|_{Z_2 \times Z_1} &\leq C\left(\|\vf(t)\|_{Z_2 \times Z_1} + \|\psi_\xi(t)\|_{Z_1}\right),
\end{split}
\label{e:bdibf}
\end{align}
and
\begin{align}
\begin{split}\|\vf(t)\|_{L^\infty} \leq C\left(\|\v(t)\|_{L^\infty} + \|\psi_\xi(t)\|_{L^\infty}\right), &\qquad \|\vf(t)\|_{W^{1,\infty}} \leq C\left(\|\v(t)\|_{W^{1,\infty}} + \|\psi_\xi(t)\|_{L^\infty}\right),\\
\|\vf(t)\|_{L^2} \leq C\left(\|\v(t)\|_{L^2} + \|\psi_\xi(t)\|_{L^\infty}\right),\ &\qquad \quad \|\vf(t)\|_{H^1} \leq C\left(\|\v(t)\|_{H^1} + \|\psi_\xi(t)\|_{Z_0}\right),
\end{split}
\label{e:bdibf2}
\end{align}
for any $t \in [0,\tau_{\max})$.
\end{lemma}
\begin{proof}
Let $t \in [0,\tau_{\max})$. We then have $\|\psi(t)\|_{Z_2} \leq \frac12$ by Proposition~\ref{p:psi}. Therefore, the function $h_t \colon \R \to \R$ given by $h_t(\xi) = \xi - \psi(\xi,t)$ is strictly increasing and invertible. Moreover, using $\xi = h_t(h_t^{-1}(\xi)) = h_t^{-1}(\xi) - \psi(h_t^{-1}(\xi),t)$ and $\|\psi(t)\|_{L^\infty} \leq \frac12$, we find $|\xi - h_t^{-1}(\xi)| \leq \frac12$ for all $\xi \in \R$. 

Substituting~\eqref{e:fmodpert} into~\eqref{e:modpert}, we obtain
\begin{align*}
\begin{split}
\v(\xi,t) &= \omega(\xi)\left(\frac{\vf(\xi-\psi(\xi,t),t)}{\omega(\xi - \psi(\xi,t))} + \ufr\left(\xi-\psi(\xi,t) + \psi(\xi-\psi(\xi,t),t)\right) - \ufr(\xi)\right),\\
&= \vf(\xi-\psi(\xi,t),t) + \ufr\left(\xi-\psi(\xi,t) + \psi(\xi-\psi(\xi,t),t)\right) - \ufr(\xi),
\end{split}
\end{align*}
where the second equality follows from the facts that, for $\xi \leq -1$ we have $\xi - \psi(\xi,t) \leq -1 + \|\psi(t)\|_{L^\infty} \leq -\frac12$, for $\xi \leq 0$ we have $\omega(\xi) = 1$, and for $\xi \geq -1$ we have $\psi(\xi,t) = 0$ by Proposition~\ref{p:psi}. Taking spatial derivatives yields
\begin{align*}
\v_\xi(\xi,t) &= \vf_\xi(\xi-\psi(\xi,t),t)\left(1-\psi_\xi(\xi,t)\right) - \ufr'(\xi)\\
&\qquad + \, \ufr'\left(\xi-\psi(\xi,t) + \psi(\xi-\psi(\xi,t),t)\right)\left(1-\psi_\xi(\xi,t) + \psi_\xi(\xi-\psi(\xi,t),t)\left(1-\psi_\xi(\xi,t) \right) \right) ,\\
\v_{\xi\xi}(\xi,t) &= \vf_{\xi\xi}(\xi-\psi(\xi,t),t)\left(1-\psi_\xi(\xi,t)\right)^2 - \ufr''(\xi) - \vf_\xi(\xi-\psi(\xi,t),t)\psi_{\xi\xi}(\xi,t)\\ 
&\qquad + \, \ufr''\left(\xi-\psi(\xi,t) + \psi(\xi-\psi(\xi,t),t)\right)\left(1-\psi_\xi(\xi,t) + \psi_\xi(\xi-\psi(\xi,t),t)\left(1-\psi_\xi(\xi,t) \right) \right)^2\\ 
&\qquad + \, \ufr'\left(\xi - \psi(\xi,t) + \psi(\xi-\psi(\xi,t),t)\right)\left( \psi_\xi(\xi-\psi(\xi,t),t)\left(1-\psi_\xi(\xi,t) \right)^2\right.\\
&\qquad \qquad \qquad \qquad \qquad \qquad \qquad \qquad  \qquad\left. \phantom{\left(1-\psi_\xi(\xi,t) \right)^2} - \psi_{\xi\xi}(\xi,t)\left(1 + \psi_\xi(\xi-\psi(\xi,t),t)\right) \right).
\end{align*}
Using the mean value theorem twice, we estimate
\begin{align*}
\left|\left(\partial_\xi^j\ufr\right)\left(\xi - \psi(\xi,t) + \psi(\xi-\psi(\xi,t),t)\right) - \partial_\xi^j \ufr(\xi)\right| \leq \left\|\partial_\xi^{j+1}\ufr\right\|_{L^\infty}\left\|\psi_\xi(t)\right\|_{L^\infty} \left|\psi(\xi,t)\right|,
\end{align*}
for $j = 0,1,2$ and $\xi \in \R$. In addition, using $\|\psi_\xi(t)\|_{L^\infty} \leq \frac{1}{2}$ and applying the substitution $y = h_t(\xi)$, we bound
\begin{align*}
\left\|f(\cdot - \psi(\cdot,t))\right\|_{L^2}^2 = \int_{\R} \left|f(\xi - \psi(\xi,t))\right|^2 \de \xi = \int_{\R} \frac{\left|f(y)\right|^2}{1-\psi_\xi\left(h_t^{-1}(y),t\right)} \de y \leq {\color{red} 2}\|f\|_{L^2}^2,
\end{align*} 
for $f \in L^2(\R)$. With the aid of the latter two estimates and using the continuous embedding $H^1(\R) \hookrightarrow L^\infty(\R)$ and the fact that $\|\psi(t)\|_{Z_2} \leq \frac12$, we thus establish~\eqref{e:bdibf}.

Conversely, substituting~\eqref{e:modpert} into~\eqref{e:fmodpert}, we obtain
\begin{align*}
\vf(\xi,t) &= \omega(\xi)\left(\frac{\v(h_t^{-1}(\xi),t)}{\omega(h_t^{-1}(\xi))} + \ufr\left(h_t^{-1}(\xi)\right) - \ufr(\xi + \psi(\xi,t))\right)\\ 
&= \v\left(h_t^{-1}(\xi),t\right) + \ufr\left(h_t^{-1}(\xi)\right) - \ufr(\xi + \psi(\xi,t)),
\end{align*}
where we used that for $\xi \leq -\frac12$ we have $h_t^{-1}(\xi) = \xi + \psi(h_t^{-1}(\xi),t) \leq -\frac12 + \|\psi(t)\|_{L^\infty} \leq 0$, for $\xi \leq 0$ we have $\omega(\xi) = 1$, for $\xi \geq -\frac12$ we have $h_t^{-1}(\xi) = \xi + \psi(h_t^{-1}(\xi),t) \geq -\frac12 - \|\psi(t)\|_{L^\infty} \geq -1$, and for $\xi \geq -1$ we have $\psi(\xi,t) = 0$ by Proposition~\ref{p:psi}. Taking spatial derivatives yields
\begin{align*}
\vf_\xi(\xi,t) &= \left(\v_\xi\left(h_t^{-1}(\xi),t\right) + \ufr'\left(h_t^{-1}(\xi)\right)\right)\partial_\xi \left(h_t^{-1}(\xi)\right) - \ufr'(\xi + \psi(\xi,t))(1+\psi_\xi(\xi,t)).
\end{align*}
First, the inverse function theorem implies
\begin{align*} 
\partial_\xi \left(h_t^{-1}(\xi)\right) = 1 + \frac{\psi_\xi\left(h_t^{-1}(\xi),t\right)}{1-\psi_\xi\left(h_t^{-1}(\xi),t\right)}, 
\end{align*}
for $\xi \in \R$. Next, recalling that $h_t^{-1}(\xi) = \xi + \psi(h_t^{-1}(\xi),t)$ and using the mean value theorem twice, we arrive at
\begin{align*}
\left|\left(\partial_\xi^j\ufr\right)\left(h_t^{-1}(\xi)\right) - \left(\partial_\xi^j\ufr\right)(\xi + \psi(\xi,t))\right| \leq \left\|\partial_\xi^{j+1}\ufr\right\|_{L^\infty}\left\|\psi_\xi(t)\right\|_{L^\infty} \left|\psi\left(h_t^{-1}(\xi),t\right)\right|.
\end{align*}
for $j = 0,1$ and $\xi \in \R$. Lastly, using $\|\psi_\xi(t)\|_{L^\infty} \leq \frac{1}{2}$ and applying the substitution $\xi = h_t(y)$, we estimate
\begin{align*}
\left\|f \circ h_t^{-1}\right\|_{L^2}^2 = \int_{\R} \left|f\left(h_t^{-1}(\xi)\right)\right|^2 \de \xi = \int_{\R} \left|f(y)\right|^2 (1+\psi_y(y,t)) \de y \leq {\color{red} 2}\|f\|_{L^2}^2,
\end{align*}
for $f \in L^2(\R)$. With these these three observations, recalling the continuous embedding $H^1(\R) \hookrightarrow L^\infty(\R)$ and the fact that $\|\psi(t)\|_{Z_1} \leq \frac12$, we have established~\eqref{e:bdibf2}.
\end{proof}

\section{Nonlinear stability argument}\label{s: nonlinear argument}

In this section we prove our nonlinear stability result, Theorem~\ref{t: main detailed}, by applying the linear estimates, obtained in Theorem~\ref{t: linear estimates}, and the nonlinear estimates, obtained in Lemma~\ref{l: nonlinear estimates} and Proposition~\ref{p: nonlinear damping}, to the nonlinear interaction scheme consisting of the integral equations~\eqref{e:intv2} and~\eqref{e:intpsi2} for the inverse-modulated perturbation $\v(t)$ and the phase modulation $\psi(t)$, respectively. 

\begin{proof}[Proof of Theorem~\ref{t: main detailed}]
Set $\v_0 = \omega \w_0$ and let $\vt(t), \v(t)$ and $\vf(t)$ be the associated unmodulated, inverse- and forward-modulated perturbations, established in Proposition~\ref{p:local_unmod} and Corollaries~\ref{C:local_v} and~\ref{c:local_vf}. In addition, let $\psi(t)$ be the corresponding phase modulation, established in Proposition~\ref{p:psi}. It follows by Propositions~\ref{p:local_unmod} and~\ref{p:psi} and Corollary~\ref{C:local_v} that the functions $\eta_1, \eta_2 \colon [0,\tau_{\max}) \to \R$ given by
\begin{align*}
\eta_1(t) &= \sup_{0\leq s\leq t} \left[(1+s)\|\v(s)\|_{L^\infty} + (1+s)^{\frac{3}{4}}\left(\|\v(s)\|_{L^2} + \left\|\psi_\xi(s)\right\|_{H^3} + \left\|\partial_s \psi(s)\right\|_{H^2}\right) + (1+s)^{\frac{1}{4}}\left\|\psi(s)\right\|_{L^2} \right],\\ 
\eta_2(t) &= \sup_{0\leq s\leq t} \left[(1+s)^{\frac{3}{4}}\left\|\partial_\xi \v(s)\right\|_{L^\infty} + (1+s)^{\frac14} \left\|\vt(s)\right\|_{H^1}\right],
\end{align*}
are well-defined, continuous, positive and monotonically increasing. We will control terms that appear in $\eta_1(t)$ by iterative estimates on their Duhamel formulas, whereas terms in $\eta_2(t)$ are controlled by the nonlinear damping estimate stated in Proposition~\ref{p: nonlinear damping}. This leads to different type of inequalities for $\eta_1(t)$ and $\eta_2(t)$, which we eventually combine to an inequality for $\eta(t) = \eta_1(t) + \eta_2(t)$. Next, recall that if $\tau_{\max} < T_{\max}$ or $T_{\max} < \infty$, then we have
\begin{align} \label{e:blowup}
 \lim_{t \nearrow \tau_{\max}} \eta(t) \geq r_0,
\end{align}
where $r_0 > 0$ is the constant from Proposition~\ref{p:psi}. As common in nonlinear iteration arguments, our aim is to establish a nonlinear inequality for the template function $\eta(t)$, which, by continuity, implies that $\eta(t)$ must stay small and precludes~\eqref{e:blowup}, yielding global existence. More specifically, our goal is to obtain a $t$-independent constant $C \geq 1$ such that for any $t \in [0,\tau_{\max})$ with $\eta(t) \leq \frac12$ the inequality 
\begin{align}\label{e:key}
\eta(t) \leq C\left(E_0 + \eta(t)^2\right),
\end{align}
holds. Then, upon taking $\delta = \frac{1}{4C}\min\{\frac{1}{C},r_0\}$, it follows from continuity, monotonicity and non-negativity of $\eta$ that, provided $E_0 \in (0,\delta)$, we have $\eta(t) \leq 2CE_0 \leq \frac12$ for all $t \in [0,\tau_{\max})$. Indeed, for any given $t \in [0,\tau_{\max})$ with $\eta(s) \leq 2CE_0$ for each $s \in [0,t]$, we conclude
$$\eta(t) \leq C\left(E_0 + 4C^2E_0^2\right) < 2CE_0,$$
using~\eqref{e:key}. Therefore, once we have established~\eqref{e:key} and, thus, $\eta(t) \leq 2CE_0$ for all $t \in [0,\tau_{\max})$, it follows by~\eqref{e:blowup} that necessarily $\tau_{\max} = T_{\max} = \infty$. Consequently, one obtains that $\eta(t) \leq 2CE_0$ for all $t \geq 0$, which then yields the desired estimates.

Thus, we start by showing the nonlinear inequality~\eqref{e:key}. In fact, we first obtain such an inequality for $\eta_1(t)$, while using that both $\eta_1(t)$ and $\eta_2(t)$ are bounded. We stress that boundedness of $\eta_2(t)$ is required here to apply the nonlinear damping estimate stated in Proposition~\ref{p: nonlinear damping} in order to control derivatives in the nonlinearity $\mathcal{N}$ in~\eqref{e:intv2} and~\eqref{e:intpsi2}. Subsequently, we bound $\eta_2(t)$ in terms of $\eta_1(t)$ using the same nonlinear damping estimate. Finally, adding the inequalities for $\eta_1(t)$ and $\eta_2(t)$ yields an estimate of the form~\eqref{e:key}. 

Thus, let $t \in [0,\tau_{\max})$ be such that $\eta(t) \leq \frac12$. By Proposition~\ref{p:psi} we have 
\begin{align} \label{e:psibound}
\|\psi(s)\|_{Z_2} \leq \frac12
\end{align}
for all $s \in [0,t]$. Lemma~\ref{l: equivalence}, estimate~\eqref{e:psibound} and the continuous embedding 
\begin{align} \label{e:embedding}
H^{k+1}(\R) \hookrightarrow Z_k(\R), \qquad k \in \N_0,
\end{align}
yield the following bounds on the forward-modulated perturbation
\begin{align} \label{e:estf}
\|\vf(\tau)\|_{L^2} &\lesssim \frac{\eta_1(t)}{(1+\tau)^{\frac34}}, \qquad \|\mathring{v}_1(\tau)\|_{W^{1,\infty}} \lesssim \frac{\eta(t)}{(1+\tau)^{\frac34}},
\end{align}
for all $\tau \in [0,t]$. 
Thus, applying the nonlinear damping estimate in Proposition~\ref{p: nonlinear damping} and using~\eqref{e:embedding} and~\eqref{e:estf}, while noting that $\eta(t) \leq \frac12$, we establish
\begin{align} \label{e:damping}
\|\vf(s)\|_{H^3 \times H^2}^2 \lesssim \re^{-\mu s} E_0^2 + \frac{\eta_1(t)^2}{(1+s)^{\frac32}} + \int_0^s \re^{-\mu(s-r)} \frac{\eta_1(t)^2}{(1+r)^{\frac32}} \de r \lesssim \frac{E_0^2 + \eta_1(t)^2}{(1+s)^{\frac32}},
\end{align}
for all $s \in [0,t]$. Hence, combining Lemma~\ref{l: equivalence} with~\eqref{e:damping}, recalling~\eqref{e:psibound} and~\eqref{e:embedding}, and using Young's inequality, we arrive at
\begin{align} \label{e:damping2}
\|\v(s)\|_{Z_2 \times Z_1} &\lesssim \frac{E_0 + \eta_1(t)}{(1+s)^{\frac34}},
\end{align}
for each $s \in [0,t]$. Finally, combining the latter estimate with Lemma~\ref{l: nonlinear estimates}, using $\eta(t) \leq \frac12$ and~\eqref{e:psibound}, recalling~\eqref{e:embedding} and noting  that $\psi(s)$ vanishes on $[-1,\infty)$ by Proposition~\ref{p:psi}, we obtain the nonlinear bound
\begin{align} \label{e:nlest}
\begin{split} 
\left\|\mathcal{N}\left(\v(s),\psi(s),\partial_s \psi(s)\right)\right\|_{L^1_{0,1} \cap Z_0} &\lesssim \frac{\eta_1(t)\left(E_0 + \eta_1(t)\right)}{(1+s)^{\frac32}}, 
\end{split}
\end{align}
for each $s \in [0,t]$. 

We are now in the position to establish an inequality for $\eta_1(t)$ by iterative estimates on the Duhamel formulas~\eqref{e:intv2} and~\eqref{e:intpsi2} for the inverse-modulated perturbation $\v(t)$ and the phase modulation $\psi(t)$, respectively. Thus, we apply the linear estimates in Theorem~\ref{t: linear estimates} and the nonlinear bound in~\eqref{e:nlest} to~\eqref{e:intv2} and establish
\begin{align}
\label{e:duh1}
\|\v(t)\|_{L^2} &\lesssim \frac{E_0 + \eta_1(t)^2}{(1+t)^{\frac34}} + \int_0^t \frac{\eta_1(t)(E_0 + \eta_1(t))}{\re^{\mu(t-s)}(1+s)^{\frac32}} \de s + \int_0^t \frac{\eta_1(t)(E_0 + \eta_1(t))}{(1+t-s)^{\frac34}(1+s)^{\frac32}} \de s \lesssim \frac{E_0 + \eta_1(t)^2}{(1+t)^{\frac34}},
\end{align}
 and
\begin{align}
\label{e:duh11}
\|\v(t)\|_{L^\infty} &\lesssim \frac{E_0 + \eta_1(t)^2}{1+t} + \int_0^t \frac{\eta_1(t)(E_0 + \eta_1(t))}{\re^{\mu(t-s)}(1+s)^{\frac32}} \de s + \int_0^t \frac{\eta_1(t)(E_0 + \eta_1(t))}{(1+t-s)(1+s)^{\frac32}} \de s \lesssim \frac{E_0 + \eta_1(t)^2}{1+t},
\end{align}
where we used $\eta_1(t) \leq \frac12$. Similarly, we apply Theorem~\ref{t: linear estimates} and estimate~\eqref{e:nlest} to~\eqref{e:intpsi2} and obtain
 \begin{align}
 \label{e:duh22}
\|\psi(t)\|_{L^p} &\lesssim \frac{E_0}{(1+t)^{\frac12 - \frac1{2p}}} + \int_0^t \frac{\eta_1(t)(E_0 + \eta_1(t))}{(1+t-s)^{\frac12 - \frac1{2p}}(1+s)^{\frac32}} \de s \lesssim \frac{E_0 + \eta_1(t)^2}{(1+t)^{\frac12 - \frac1{2p}}},
 \end{align}
and 
\begin{align}
\label{e:duh2}
\left\|\partial_\xi^\ell \partial_t^j \psi(t)\right\|_{L^p} \lesssim \frac{E_0}{(1+t)^{1 - \frac1{2p}}} + \int_0^t \frac{\eta_1(t)(E_0 + \eta_1(t))}{(1+t-s)^{1 - \frac1{2p}}(1+s)^{\frac32}} \de s \lesssim \frac{E_0 + \eta_1(t)^2}{(1+t)^{1 - \frac1{2p}}},
\end{align}
for $p = 2,\infty$ and $\ell, j \in \N_0$ with $1 \leq \ell + 2j \leq 4$, where we used $\eta(t) \leq \frac12$. Combining the estimates~\eqref{e:duh1},~\eqref{e:duh11},~\eqref{e:duh22}, and~\eqref{e:duh2}, we establish a $t$-independent constant $C_1 \geq 1$ such that
\begin{align} \label{e:keyin1}
\eta_1(t) \leq C_1\left(E_0 + \eta_1(t)^2\right).
\end{align}

Next, we obtain an inequality for $\eta_2(t)$. First, we express the unmodulated perturbation as
\begin{align} \label{e:forwunm}
\vt(\xi,t) = \vf(\xi,t) + \omega(\xi)\left(\ufr(\xi + \psi(\xi,t)) - \ufr(\xi)\right) = \vf(\xi,t) + \ufr(\xi + \psi(\xi,t)) - \ufr(\xi),
\end{align}
where we use that $\omega(\xi) = 1$ for $\xi \leq 0$ and $\psi(t)$ vanishes on $[-1,\infty)$ by Proposition~\ref{p:psi}. Thus, Lemma~\ref{l: equivalence} and the mean value theorem yield
\begin{align*}
\|\vt(t)\|_{H^1} \lesssim \|\vf(t)\|_{H^1} + \|\psi(t)\|_{H^1} \lesssim \|\v(t)\|_{H^1} + \|\psi(t)\|_{H^1},
\end{align*}
where we use $\eta(t) \leq \frac12$ and~\eqref{e:psibound}. Combining the latter with~\eqref{e:damping2}, while recalling~\eqref{e:embedding}, we establish the bounds
\begin{align*}
\|\vt(t)\|_{H^1} \leq \frac{E_0 + \eta_1(t)}{(1+t)^{\frac14}}, \qquad \|\v_\xi(s)\|_{L^\infty} \leq  \frac{E_0 + \eta_1(t)}{(1+t)^{\frac34}}.
\end{align*}
Thus, we establish a $t$-independent constant $C_2 \geq 1$ such that
\begin{align} \label{e:key2}
\eta_2(t) \leq C_2\left(E_0 + \eta_1(t)\right).
\end{align}
Finally,~\eqref{e:keyin1} and~\eqref{e:key2} admit the estimate 
\begin{align*}
 \eta(t) \leq C_2E_0 + (1+C_2)\eta_1(t) \leq \left(C_2 + \left(1+C_2\right)C_1\right)E_0 + (1+C_2)C_1\eta_1(t)^2,
\end{align*}
which implies~\eqref{e:key} for some $t$-independent constant $C \geq 1$. As mentioned before, this yields $\tau_{\max} = T_{\max} = \infty$ and $\eta(t) \leq 2CE_0$ for all $t \geq 0$. Hence, recalling Propositions~\ref{p:local_unmod} and~\ref{p:psi}, we readily establish~\eqref{e:regularity}. Moreover, the mean value theorem, identity~~\eqref{e:forwunm} and Lemma~\ref{l: equivalence} yield the estimates
\begin{align} \label{e:equibounds}
 \begin{split}
\|\vt(t)\|_{H^3 \times H^2} \lesssim \|\vf(t)\|_{H^3 \times H^2} + &\|\psi(t)\|_{H^3}, \qquad \|\vt(t)\|_{L^\infty} \lesssim \|\vf(t)\|_{L^\infty} + \|\psi(t)\|_{L^\infty},\\ 
\|\vf(t)\|_{L^\infty} &\lesssim \|\v(t)\|_{L^\infty} + \|\psi_\xi(t)\|_{L^\infty},
\end{split}
\end{align}
for $t \geq 0$, where we use $\eta(t) \leq 2CE_0$ and~\eqref{e:psibound}. All in all, combining the estimates~\eqref{e:damping},~\eqref{e:duh22},~\eqref{e:duh2} and~\eqref{e:equibounds} with the fact that $\eta(t) \leq 2CE_0$ holds for all $t \geq 0$, yields a $t$-independent constant $M > 0$ such that~\eqref{e:MTmodder} and~\eqref{e:MTmodder2} are satisfied.

It only remains to prove~\eqref{e: improved leading edge decay}, that is, to transfer the improved linear decay estimate~\eqref{e: leading edge linear decay} in the leading edge to the nonlinear level. Thus, we multiply the Duhamel formula~\eqref{e:intv} with $\chi_+$ and recall that $\psi(t)$ vanishes on $[-1,\infty)$ to arrive at
\begin{align}
\chi_+ \vt(t) = \chi_+ \v(t) = \chi_+\re^{\El_{\mathrm{fr}} t} \v_0 + \int_0^t \chi_+\re^{\El_{\mathrm{fr}}(t-s)}\mathcal{N}(\v(s),\psi(s),\partial_t \psi(s))\de s. \label{e:intv3}
\end{align}
Applying the linear bound~\eqref{e: leading edge linear decay} in Theorem~\ref{t: linear estimates} and the nonlinear bound~\eqref{e:nlest} to~\eqref{e:intv3}, while using $\eta(t) \leq 2CE_0$, we obtain
\begin{align*}
\left\|\rho_{0,-1} \chi_+ \vt(t)\right\|_{L^\infty} \lesssim \frac{E_0}{(1+t)^{\frac32}} + \int_0^t \frac{\eta_1(t)\left(E_0 + \eta_1(t)\right)}{(1+t-s)^{\frac32}(1+s)^{\frac32}} \de s \lesssim \frac{E_0}{(1+t)^{\frac32}},
\end{align*}
for $t \geq 0$. We conclude that there exists a $t$-independent constant $M > 0$ such that estimate~\eqref{e: improved leading edge decay} holds for all $t \geq 0$, which completes the proof.
\end{proof}

\section{Discussion}\label{s: discussion}
We believe that the methods developed here can be useful for analyzing diffusive stability problems in many contexts, and we conclude by discussing potential applications of our methods to several related problems.

\paragraph{Pushed pattern-forming fronts.}
The analysis in~\cite{CarterScheel} also gives a rigorous construction of pushed pattern-forming invasion fronts in the FitzHugh-Nagumo system~\eqref{e: FHN not shifted} for $0 < a < \frac{1}{3}$. For pushed fronts, the propagation is no longer driven by the dynamics in the leading edge, but instead by a localized mode near the front interface. Hence, the essential spectrum in the leading edge may be fully stabilized by introducing an exponential weight, but there is a resonance pole of the Evans function at $\lambda = 0$. As a result, the linearized dynamics exhibit a non-decaying mode associated to spatial translation. Key to the dynamics then is the interaction of this non-decaying mode with the outgoing diffusive mode. We expect that our approach to resolvent decompositions can still be used in this setting to give a precise description of the linearized dynamics which could then be used to close a nonlinear stability argument. The stability problem here is conceptually similar to the stability of source defects, which also focuses on the interaction of non-decaying, translational modes with outgoing diffusive modes, and has been studied in the complex Ginzburg-Landau equation in~\cite{BeckNguyenSandstedeZumbrun}. We emphasize that, for pushed pattern-forming fronts, proving nonlinear stability against the natural class of perturbations readily yields selection from steep initial data, since the exponential weights involved automatically allow one to consider perturbations which cut off the front tail. This strongly contrasts to the case of pulled pattern-forming fronts, where substantial additional steps are necessary to extend sharp nonlinear stability results, as presented in this work, to selection results (see further discussion below).

\paragraph{Pulled pattern-forming fronts beyond the FitzHugh-Nagumo system.} 
The essential feature of the stability of pulled pattern-forming fronts is the interaction between the branched diffusive mode associated with the linear spreading speed and the outgoing diffusive mode associated with the pattern in the wake. The methods developed here should generally be successful in proving nonlinear stability of pulled pattern-forming fronts also in other systems. A particular case of interest is front invasion in the wake of a Turing instability, for instance in the Swift-Hohenberg equation or other reaction-diffusion systems. The new challenge, compared to the FitzHugh-Nagumo fronts studied here, is that the invasion dynamics in the leading edge are oscillatory in time, so that the pattern-forming fronts are \emph{modulated} traveling waves, which are time-periodic in the comoving frame, rather than stationary. We expect that this can be overcome by combining the analysis here with ideas from~\cite{BeckSandstedeZumbrun}, which develops a framework for studying time-periodic diffusive stability problems via an inverse Laplace transform. We expect that this approach can be used to establish a general result on nonlinear stability of pulled pattern-forming invasion fronts in reaction-diffusion systems near a supercritical Turing instability. 

\paragraph{Selection of pulled pattern-forming fronts from steep initial data.} 
The challenge in establishing selection of pulled pattern-forming fronts from steep initial data is that perturbations which cut off the front tail induce a logarithmic delay $-\frac{3}{2 \etalin} \log t$ in the position of the front. In this frame, perturbations in the leading edge no longer decay at all, creating additional difficulties when controlling the long-time behavior. These difficulties have been overcome for fronts selecting constant states in~\cite{AveryScheelSelection, AverySelectionRD}, but they present substantial additional difficulties for pattern-forming fronts, since these non-decaying modes interact with the outgoing diffusive dynamics in the wake. We are hopeful that combining the present analysis with recent work on selection of pulled fronts~\cite{AveryScheelSelection, AverySelectionRD} and on stability of periodic wave trains against nonlocalized perturbations~\cite{deRijknonlocalized}, which also do not exhibit temporal decay, may make progress towards establishing front selection in this context. In particular, the $(1+t)^{-3/2}$ decay for localized perturbations is the key ingredient in closing the front selection argument in~\cite{AveryScheelSelection, AverySelectionRD}.  The fact that we are able to recover this decay rate in the leading edge here then is a promising step towards establishing selection of pulled pattern-forming fronts.

\appendix

\section{Shifting the integration contour}\label{s: contour shifting}

We prove Proposition~\ref{p: first shift} in this section by carefully analyzing the high frequency behavior of the resolvent $(\lfr - \lambda)^{-1}$. The analysis of this section adapts some ideas from~\cite{MasciaZumbrun}, although we use a different method to obtain a description of the high frequency behavior of the resolvent. We again let $u_*$ denote the first component of $\ufr$, and abbreviate $c = \clin > 0$. We have
\begin{align*}
    \lfr = \omega \afr \omega^{-1} = \begin{pmatrix}
    \partial_{\xi \xi} + (c + a_1) \partial_\xi + F_1'(u_*) + a_0 + b_0 & -1 \\
    \eps & c \partial_\xi - \eps \gamma + b_0
    \end{pmatrix},
\end{align*}
where $a_j, b_j$ are given by~\eqref{e: coefficients from weight}. The key observations to retaining spectral mapping properties for the operator $\lfr$ is damping induced by the second derivative in the first component and by the term $-\eps \gamma$ in the second component. In addition, since $\omega$ is non-decreasing, $b_0$ is non-positive and, thus, only contributes to additional damping.

\subsection{Resolvent decomposition for \texorpdfstring{$\Im \lambda \gg 1$}{Im(lambda)>>1}}

Let $\eta > 0$ be as in Corollary~\ref{c: inverse laplace} and fix $\Omega_0 > 0$. We consider the resolvent equation 
\begin{align}
	(\lfr - \lambda) \begin{pmatrix} u \\ v \end{pmatrix} = \begin{pmatrix} g_1 \\ g_2 \end{pmatrix} \label{e: resolvent equation}
\end{align}
for $\g = (g_1,g_2) \in C^\infty(\R,\R^2)$ and $\lambda \in \rho(\lfr)$ of the form $\lambda = b + \ri \Omega$, with $b \in [- \frac{3}{4} \eps \gamma,\eta]$ and $\Omega$ real with $|\Omega| \geq \Omega_0$. We set $\kappa = \sign(\Omega)$, so that $\Omega = \kappa |\Omega|$. We rescale equation~\eqref{e: resolvent equation}
by defining $X = \sqrt{|\Omega|} \xi$ together with the rescaled functions
\begin{align*}
    \U (X) = \begin{pmatrix}
    U(X) \\ V(X) 
    \end{pmatrix} = \begin{pmatrix}
    u\left(\frac{X}{\sqrt{|\Omega|}}\right) \\
    v\left(\frac{X}{\sqrt{|\Omega|}}\right)
    \end{pmatrix},
    \quad
    \G (X) = \begin{pmatrix}
    G_1(X) \\ G_2(X) 
    \end{pmatrix} = \begin{pmatrix}
    g_1\left(\frac{X}{\sqrt{|\Omega|}}\right) \\
    g_2\left(\frac{X}{\sqrt{|\Omega|}}\right)
    \end{pmatrix},
\end{align*}
and
\begin{align*}
    U_* (X) = u_*\left( \frac{X}{\sqrt{|\Omega|}}\right), \quad A_1 (X) = a_1 \left( \frac{X}{\sqrt{|\Omega|}}\right), \quad A_0 (X) = a_0 \left( \frac{X}{\sqrt{|\Omega|}}\right), \quad B_0 (X) = b_0 \left( \frac{X}{\sqrt{|\Omega|}}\right),
\end{align*}
so that~\eqref{e: resolvent equation} becomes 
\begin{align*}
	&\begin{pmatrix}
		|\Omega| (\partial_{XX} - i\kappa) + \sqrt{|\Omega|} (c + A_1) \partial_X + F_1'(U_*) - b + A_0 + B_0 & -1 \\
		\eps & c \sqrt{|\Omega|} \partial_X - \eps \gamma - b - \ri \kappa |\Omega| + B_0 
	\end{pmatrix}
	\begin{pmatrix}
		U \\
		V
	\end{pmatrix}
	\\
 &\qquad =
	\begin{pmatrix}
		G_1 \\
		G_2
	\end{pmatrix}. 
\end{align*}
Setting $\U = (U,V)^\top$, multiplying both sides by $\diag (|\Omega|^{-1}, |\Omega|^{-1/2})$ and setting $\mu = |\Omega|^{-1/2}$, we obtain the system
\begin{align}
    [L_0(\mu) + L_1(\mu)] \U = \begin{pmatrix}
    \mu^2 G_1 \\
    \mu G_2 
    \end{pmatrix}, \label{e: spectral mapping resolvent rescaled}
\end{align}
where
\begin{align*}
    L_0(\mu) = \begin{pmatrix}
    \partial_{XX} - \ri \kappa & 0 \\
    0 & c \partial_X - \tilde{b} \mu - \frac{\ri \kappa}{\mu} + \mu B_0 
    \end{pmatrix}
    =: \begin{pmatrix}
    L_0^{11} & 0 \\
    0 & L_0^{22} (\mu) 
    \end{pmatrix},
\end{align*}
$\tilde{b} = b + \eps \gamma$, $\kappa \in \{ -1, 1 \}$,
and
\begin{align*}
    L_1(\mu) = \begin{pmatrix}
    \mu (c + A_1)\partial_x + \mu^2 (F_1'(U_*) - b + A_0 + B_0) & - \mu^2 \\
    \eps \mu & 0 
    \end{pmatrix} =: \begin{pmatrix}
    L_1^{11}(\mu) & L_1^{12}(\mu) \\
    L_1^{21}(\mu) &0 
    \end{pmatrix}.
\end{align*}

The operator $L_0(\mu)$ represents a principal part which is responsible for the damping, whereas $L_1(\mu)$ contains higher-order corrections which do not interfere with the damping from $L_0(\mu)$. In particular, we will see that the precise form of $U_*$ does not matter, and the coefficients from the exponential weight are irrelevant as long as $B_0$ is non-positive. 

Our strategy is to invert $L(\mu) := L_0(\mu) + L_1(\mu)$ by using the factorization
\begin{align*}
    L(\mu) = [1 + L_1(\mu) L_0(\mu)^{-1}] L_0(\mu), 
\end{align*}
inverting $L_0(\mu)$ explicitly, and showing that $[1 + L_1(\mu) L_0(\mu)^{-1}]$ may be inverted via a Neumann series. Since $L_0(\mu)$ is diagonal, we can invert it by inverting each block separately. The first block may be inverted by standard spectral theory, since $\pm i$ is not in the spectrum of $\partial_{XX}$ on $L^p(\R)$. For the second block, we can use an integrating factor to write the inverse explicitly as 
\begin{align*}
    \left(c \partial_X - \tilde{b} \mu - \frac{i\kappa}{\mu} + \mu B_0\right)^{-1} G_2 (X) = \int_\R \re^{\frac{i\kappa}{\mu c} (X-Y)} G^\mathrm{tr}_\mu (X,Y) G_2(Y) \de  Y,
\end{align*}
where
\begin{align}
    G_\mu^\mathrm{tr}(X,Y) = \begin{cases}
    \frac{1}{c} \re^{\frac{\tilde{b}\mu (X-Y)}{c}} \re^{\frac{\mu}{c}\int_X^Y B_0 (Z) \de  Z}, & X \leq Y, \\
    0, &X > Y. 
    \end{cases} \label{e: spectral mapping G tr def}
\end{align}
Exploiting the fact that $B_0(Z)$ is non-positive, we readily obtain the following estimates.
\begin{lemma}\label{l: spectral mapping L0 inverse bounds}
    Let $1 \leq p \leq \infty$. There exists a constant $C > 0$ such that, for $\mu > 0$ sufficiently small, we have the estimates
    \begin{align}
        \| (L_0^{11})^{-1} \|_{L^p \to W^{2,p}} &\leq C, \label{e: L0 parabolic bound} \qquad
        \| L_0^{22}(\mu)^{-1} \|_{L^p \to L^p} \leq \frac{C}{|\mu|}. 
    \end{align}
\end{lemma}

We also readily obtain the following basic estimates on $L_1(\mu)$. 
\begin{lemma}\label{l: spectral mapping L1 bounds}
    Let $1 \leq p \leq \infty$. There exists a constant $C > 0$ such that, for $\mu > 0$ sufficiently small, we have the estimates
    \begin{align*}
        \| L_1^{11}(\mu) \|_{W^{1,p} \to L^p} &\leq C |\mu|, 
        \qquad
        \| L_1^{12}(\mu) \|_{L^p \to L^p} \leq C|\mu|^2, \qquad
        \| L_1^{21}(\mu) \|_{L^p \to L^p} \leq C|\mu|.
    \end{align*}
\end{lemma}
We may then invert $[1 + L_1(\mu) L_0(\mu)^{-1}]$ via the Neumann series
\begin{align*}
    [1+L_1(\mu) L_0(\mu)^{-1}]^{-1} = \sum_{n = 0}^\infty \left(-L_1(\mu) L_0(\mu)^{-1}\right)^n,
\end{align*}
provided $\mu > 0$ is sufficiently small.

\begin{lemma}\label{l: neumann series}
	Let $1 \leq p \leq \infty$. The series 
	\begin{align*}
		\sum_{n = 0}^\infty (- L_1 (\mu) L_0 (\mu)^{-1})^n
	\end{align*}
	converges in the space of bounded operators on $L^p(\R)$, provided $\mu > 0$ is sufficiently small. Hence, for $\mu > 0$ sufficiently small, the operator $[1 + L_1(\mu) L_0(\mu)^{-1}]$ is invertible on $L^p(\R)$, with its inverse given by this Neumann series. Moreover, there exists a constant $C > 0$ such that
	\begin{align}
	    \left\| \sum_{n = 3}^\infty (-L_1(\mu) L_0(\mu)^{-1})^n \right\|_{L^p \to L^p} \leq C |\mu|^3. \label{e: neumann series remainder estimate}
	\end{align}
\end{lemma}
\begin{proof}
    Combining Lemmas~\ref{l: spectral mapping L0 inverse bounds} and~\ref{l: spectral mapping L1 bounds}, we find that $L_1(\mu) L_0(\mu)^{-1}$ may be written in the form
    \begin{align*}
        L_1(\mu) L_0(\mu)^{-1} = \begin{pmatrix}
        \mu B_{11}(\mu) & \mu B_{12}(\mu) \\
        \mu B_{21}(\mu) & 0 
        \end{pmatrix},
    \end{align*}
    where $B_{ij}(\mu)$ are operators on $L^p$, which are $\mu$-uniformly bounded. Hence, there exists a $\mu$-independent constant $C > 0$ such that
    \begin{align*}
        \| [L_1(\mu) L_0(\mu)^{-1}]^n \| _{L^p \to L^p} \leq (C |\mu|)^n,
    \end{align*}
    which implies that the series converges for $|\mu| < \frac{1}{C}$, and that the estimate~\eqref{e: neumann series remainder estimate} holds. 
\end{proof}

\textbf{Collecting relevant terms.} In order to shift the integration contour in~\eqref{e: contour integral representation}, we want to express the solution to the resolvent equation~\eqref{e: resolvent equation} as an explicit leading-order part, plus a remainder which is integrable in $\Omega$. From Lemma~\ref{l: neumann series}, we can get an expansion of the solution to~\eqref{e: resolvent equation} in powers of $\mu = |\Omega|^{-1/2}$. To get an integrable remainder, we therefore need to explicitly capture all terms up to order $\mu^2$, since $\mu^3 = |\Omega|^{-3/2}$ is integrable near $|\Omega| = \infty$. Thus, we employ Lemma~\ref{l: neumann series} to express the solution to the rescaled resolvent equation~\eqref{e: spectral mapping resolvent rescaled} as
\begin{align}
    L(\mu)^{-1} \begin{pmatrix} \mu^2 G_1 \\ \mu G_2 \end{pmatrix} = L_0(\mu)^{-1} \sum_{n = 0}^\infty (-L_1 (\mu) L_0(\mu)^{-1} )^n \begin{pmatrix} \mu^2 G_1 \\ \mu  G_2 \end{pmatrix}. \label{e: Neumann series inversion}
\end{align}
By Lemmas~\ref{l: spectral mapping L0 inverse bounds} and~\ref{l: spectral mapping L1 bounds}, we see that
\begin{align*}
    L_0 (\mu)^{-1} \sim \begin{pmatrix} 1 & 0 \\ 0 & \mu^{-1} \end{pmatrix}, \quad L_1(\mu) \sim \begin{pmatrix} \mu & \mu^2 \\ \mu & 0 \end{pmatrix},
\end{align*}
which we can use to check the order in $\mu$ of the terms in~\eqref{e: Neumann series inversion}. The first term, corresponding to $n = 0$, in~\eqref{e: Neumann series inversion}, is
\begin{align*}
    \begin{pmatrix} U_0 (\mu) \\ V_0(\mu) \end{pmatrix} = L_0 (\mu)^{-1} \begin{pmatrix} \mu^2 G_1 \\ \mu  G_2 \end{pmatrix}
	\sim \begin{pmatrix}
	1 & 0 \\
	0 & \mu^{-1} 
	\end{pmatrix}
	\begin{pmatrix} \mu^2 \\ \mu \end{pmatrix}
	\sim
	\begin{pmatrix} 
	\mu^2 \\
	1
	\end{pmatrix}.
\end{align*}
The next term, corresponding to $n = 1$ in~\eqref{e: Neumann series inversion}, is 
\begin{align*}
    \begin{pmatrix} \tilde{U}_1 (\mu) \\ \tilde{V}_1 (\mu) \end{pmatrix} = -L_0(\mu)^{-1} L_1(\mu) \begin{pmatrix} U_0 (\mu) \\ V_0(\mu) \end{pmatrix} \sim \begin{pmatrix} 1 & 0 \\ 0 & \mu^{-1} \end{pmatrix} \begin{pmatrix} \mu & \mu^2 \\ \mu & 0 \end{pmatrix} \begin{pmatrix} \mu^2 \\ 1 \end{pmatrix} \sim \begin{pmatrix} \mu^2 \\ \mu^2 \end{pmatrix}.
\end{align*}

We denote this term by $(\tilde{U}_1, \tilde{V}_1)^\top$ since it contains some higher-order terms which we will not compute explicitly, and so we will later extract the principal contributions $U_1, V_1$. The term corresponding to $n = 2$ in~\eqref{e: Neumann series inversion} is 
\begin{align}
    \begin{pmatrix} U_2 (\mu) \\ V_2 (\mu) \end{pmatrix} = -L_0(\mu) L_1(\mu) \begin{pmatrix} \tilde{U}_1 (\mu) \\ \tilde{V}_1 (\mu) \end{pmatrix} \sim \begin{pmatrix} 1 & 0 \\ 0 & \mu^{-1} \end{pmatrix} \begin{pmatrix} \mu & \mu^2 \\ \mu & 0 \end{pmatrix} \begin{pmatrix} \mu^2 \\ \mu^2 \end{pmatrix} \sim \begin{pmatrix} \mu^3 \\ \mu^2 \end{pmatrix}. \label{e: spectral mapping n = 2 formal}
\end{align}
The residual error from truncating the series at at $n = 2$ will then be of order $\mu^3 = \Omega^{-3/2}$. 

Explicitly computing $(U_0, V_0)$, we find
\begin{align}
	 \begin{pmatrix}
	    U^0(X, \mu) \\
	    V^0 (X, \mu)
	\end{pmatrix}
	= L_0 (\mu)^{-1} \begin{pmatrix} \mu^2 G_1 \\ \mu  G_2 \end{pmatrix}
	= \begin{pmatrix}
		\int_\R G^\mathrm{p} (X - Y) \mu^2 G_1(Y) \de  Y \\
		\mu \int_\R \re^{ \frac{\ri \kappa}{\mu c} (X-Y)} G^\mathrm{tr}_\mu (X,Y)  G_2(Y) \de  Y
	\end{pmatrix}, \label{e: spectral mapping tilde U0 V0 def}
\end{align}
where $G^\mathrm{tr}_\mu$ is given by~\eqref{e: spectral mapping G tr def}, and 
\begin{align}
    G^p (X) = \frac{1}{2 \sqrt{\ri \kappa}} \re^{-\sqrt{\ri \kappa} |X|}. \label{e: spectral mapping G parabolic}
\end{align}
is the spatial Green's function corresponding to the resolvent operator $(\partial_{XX} - \ri \kappa)^{-1}$.

Translating back to the original coordinates $u(\xi)$ and $v(\xi)$, these leading-order terms are then given by 
\begin{align}
	u_0 (\xi) &= \frac{1}{\sqrt{|\Omega|}} \int_\R G^\mathrm{p} \left(\sqrt{|\Omega|} (\xi-\zeta)\right) g_1(\zeta) \de \zeta, \label{e: spectral mapping u0} \\
	v_0 (\xi) &= \int_\R \tilde{G}^\mathrm{tr} (\xi, \zeta) \re^{\ri \frac{\Omega}{c} (\xi-\zeta)} g_2(\zeta) \de \zeta \label{e: spectral mapping v0},
\end{align}
where
\begin{align}
    \tilde{G}^\mathrm{tr}(\xi,\zeta) = \begin{cases}
    \frac{1}{c} \re^{\frac{\tilde{b} (\xi-\zeta)}{c}} \re^{\frac{1}{c} \int_\xi^\zeta b_0(z) \de z}, & \xi \leq \zeta, \\
    0, & \xi > \zeta. 
    \end{cases} \label{e: tilde G def}
\end{align}
The next term is
\begin{align*}
    \begin{pmatrix}
    \tilde{U}_1 \\ \tilde{V}_1
    \end{pmatrix}
    = - L_0(\mu)^{-1} L_1(\mu) \begin{pmatrix}
    U_0 \\ V_0
    \end{pmatrix}
    = -L_0(\mu)^{-1} \begin{pmatrix}
    L_1^{11}(\mu) U_0  - \mu^2  V_0 \\
    \eps \mu U_0 
    \end{pmatrix}. 
\end{align*}
Notice from~\eqref{e: spectral mapping tilde U0 V0 def} and~\eqref{e: L0 parabolic bound} that $\| U_0 \|_{W^{2,p}} \lesssim |\mu|^2 \| G_1 \|_{L^p}$ for $G_1 \in C^\infty(\R)$ and $\mu > 0$ sufficiently small. Then, also applying Lemma~\ref{l: spectral mapping L1 bounds}, we find
\begin{align}
\| L_1^{11}(\mu) U_0\|_{L^p}  \lesssim |\mu|^3 \| \G \|_{L^p}, \label{e: spectral mapping U1 tilde residual estimate}
\end{align}
for $\G = (G_1, G_2)^\top \in C^\infty(\R,\R^2)$ and $\mu > 0$ sufficiently small. Hence we identify the leading-order contributions as 
\begin{align*}
    \begin{pmatrix}
    U_1 \\
    V_1 
    \end{pmatrix}
    = \begin{pmatrix}
    \mu^2 (\partial_{XX} - \ri \kappa)^{-1} V_0 \\ 
    - \eps \mu L_0^{22}(\mu)^{-1} U_0 
    \end{pmatrix},
\end{align*}
while the term $-(L_0^{11})^{-1} L_1^{11} (\mu) U_0$ is higher order. 
Returning to original coordinates $u(\xi)$ and $v(\xi)$, we obtain the corresponding expressions
\begin{align}
    u_1 (\xi) &= \frac{1}{\sqrt{|\Omega|}} \int_\R G^\mathrm{p} \left(\sqrt{|\Omega|} (\xi-\zeta)\right) \int_\R \re^{\ri \frac{\Omega}{c} (\zeta-z)} \tilde{G}^\mathrm{tr}(\zeta,z) g_2(z) \de z \de  \zeta ,  \label{e: spectral mapping u1}\\
    v_1 (\xi) &= - \frac{\eps}{\sqrt{|\Omega|}} \int_\R \tilde{G}^\mathrm{tr}(\xi,\zeta) \re^{\frac{i}{c} \Omega (\xi-\zeta)} \int_\R G^\mathrm{p}\left(\sqrt{|\Omega|} (\zeta-z)\right) g_1(z) \de z \de  \zeta. \label{e: spectral mapping v1}
\end{align}

For the final explicit term, corresponding to $n=2$ in~\eqref{e: Neumann series inversion}, we only need to compute $V_2$ explicitly, since $U_2 = \mathrm{O}(\mu^3)$ as argued before. In $X$ coordinates, following the formal calculation~\eqref{e: spectral mapping n = 2 formal} while keeping track of which entries are relevant, we have 
\begin{align*}
    V_2 = -\eps \mu^4 L_0^{22}(\mu)^{-1} \left[ (L_0^{11})^{-1} \left( L_0^{22}(\mu)^{-1} G_2) \right) \right].   
\end{align*}
Note from Lemma~\ref{l: spectral mapping L0 inverse bounds} that each factor of $L_0^{22}(\mu)^{-1}$ contributes a factor of $\mu^{-1}$ in norm, so $V_2$ is in fact of order $\mu^2$ rather than $\mu^4$. Reverting to $\xi$ coordinates, we obtain 
\begin{align}
    v_2 (\xi) = - \eps \int_{\R^3} \tilde{G}^\mathrm{tr} (\xi,\zeta) \re^{\ri \frac{\Omega}{c} (\xi-\zeta)} \frac{G^\mathrm{p}(\sqrt{|\Omega|} (\zeta-z))}{\sqrt{|\Omega|}} \tilde{G}^\mathrm{tr} (z,w) \re^{\ri \frac{\Omega}{c} (z-w)} g_2(w) \de w \de z \de\zeta. \label{e: spectral mapping v2}
\end{align}

It follows from the proof of Lemma~\ref{l: neumann series} that all remaining terms in the Neumann series expansion are $\mathrm{O}(\mu^3)$, and so do not need to be computed explicitly. We now translate this remainder estimate into the original coordinates. 

\begin{lemma}\label{l: spectral mapping large omega expansions}
    There exist constants $C,\Omega_0 > 0$ such that for each $\lambda = b + \ri \Omega \in \C$ with $-\frac{3}{4} \eps \gamma \leq b \leq \eta$ and $|\Omega| \geq \Omega_0$, the solution $(u(\xi; \lambda), v(\xi; \lambda))$ to the resolvent equation~\eqref{e: resolvent equation} with $\g = (g_1,g_2)^\top \in C^\infty(\R,\R^2)$ obeys the estimate
    \begin{align*}
        \left\| \begin{pmatrix}
    u(\cdot; \lambda) \\ v(\cdot; \lambda)
    \end{pmatrix} -  \begin{pmatrix}
    u_0 (\cdot; \lambda) \\ v_0 (\cdot; \lambda) 
    \end{pmatrix} - \begin{pmatrix} u_1(\cdot; \lambda) \\ v_1 (\cdot; \lambda) \end{pmatrix} - \begin{pmatrix} 0 \\ v_2 (x; \lambda) \end{pmatrix} \right\|_{L^p} \leq C \Omega^{-\frac32} \| \g \|_{L^p},
    \end{align*}
where $(u_0, v_0)$ are given by~\eqref{e: spectral mapping u0}-\eqref{e: spectral mapping v0}, $(u_1, v_1)$ are given by~\eqref{e: spectral mapping u1}-\eqref{e: spectral mapping v1}, and $v_2$ is given by~\eqref{e: spectral mapping v2}. 
\end{lemma}
\begin{proof}
    The remainder terms, in $X$ coordinates introduced above, are given by 
    \begin{multline*}
        \U_\mathrm{res}(\mu) = \begin{pmatrix}
        U_\mathrm{res}(\mu) \\
        V_\mathrm{res}(\mu) 
        \end{pmatrix} = -L_0 (\mu)^{-1} \begin{pmatrix} L_1^{11}(\mu) U_0 \\ 0 \end{pmatrix} 
        \\ + P_1 L_0(\mu)^{-1} (-L_1(\mu) L_0(\mu)^{-1})^2 \begin{pmatrix} \mu^2 G_1 \\ \mu G_2 \end{pmatrix}  +
        L_0 (\mu)^{-1} \sum_{n = 3}^\infty (-L_1 (\mu) L_0 (\mu)^{-1})^n \begin{pmatrix} \mu^2 G_1 \\ \mu  G_2 \end{pmatrix},
    \end{multline*}
    where $P_1 (U,V)^\top = (U, 0)^\top$ is the projection onto the first coordinate, with this term accounting for the fact that we computed $V_2$ but not $U_2$ explicitly. 
    
    Using the estimates of Lemmas~\ref{l: spectral mapping L0 inverse bounds} and~\ref{l: spectral mapping L1 bounds}, we have
    \begin{align*}
        \left\| P_1 L_0(\mu)^{-1} (-L_1(\mu) L_0(\mu)^{-1})^2 \begin{pmatrix} \mu^2 G_1 \\ \mu G_2 \end{pmatrix} \right\|_{L^p} \lesssim |\mu|^3 \| \G\|_{L^p}.
    \end{align*}
    for $\mu > 0$ sufficiently small and $\G = (G_1,G_2)^\top \in C^\infty(\R,\R^2)$. Together with~\eqref{e: spectral mapping U1 tilde residual estimate} and~\eqref{e: neumann series remainder estimate}, we obtain
    \begin{align*}
        \| \U_\mathrm{res}(\cdot; \mu) \|_{L^p} \lesssim |\mu|^3 \| \G \|_{L^p},
    \end{align*}
    for $\mu > 0$ sufficiently small and $\G = (G_1,G_2)^\top \in C^\infty(\R,\R^2)$. Hence, reverting to $\xi$ coordinates with $\u_\mathrm{res}(\xi; \lambda) = \U_\mathrm{res}(X; \mu)$, we arrive at
    \begin{align*}
        \| \u_\mathrm{res}(\cdot; \lambda) \|_{L^p} \lesssim  \Omega^{-\frac32} \| \g \|_{L^p},
    \end{align*}
    for $\g \in C^\infty(\R,\R^2)$ and $\lambda = b + \ri \Omega \in \C$ with $-\frac{3}{4} \eps \gamma \leq b \leq \eta$ and $|\Omega|$ sufficiently large. Note that transforming back to $\xi$ coordinates yields a factor of $|\mu|^{-1/p}$ in the $L^p$ norm, but this factor is present on both sides of the above estimate, and hence can be canceled. 
\end{proof}

\subsection{Shifting the integration contour}

Let $\eta > 0$ be as in Corollary~\ref{c: inverse laplace}. Letting $g_1,g_2 \in C^\infty_c(\R)$ and 
\begin{align*}
    \begin{pmatrix} u \\ v \end{pmatrix} = (\lfr - \lambda) ^{-1} \begin{pmatrix} g_1 \\ g_2 \end{pmatrix},
\end{align*}
for $\lambda \in \rho(\lfr)$, we may write 
\begin{align*}
    \re^{\lfr t} \begin{pmatrix} g_1 \\ g_2 \end{pmatrix} = - \frac{1}{2 \pi \ri} \lim_{R \to \infty} \int_{\eta - iR}^{\eta + iR} \re^{\lambda t} \begin{pmatrix} u(\lambda) \\ v(\lambda) \end{pmatrix} \de  \lambda,
\end{align*}
for $t > 0$. Our goal in this section is to use Lemma~\ref{l: spectral mapping large omega expansions} to shift this contour into the left half-plane, except for a region bounded near the origin. To do this, we first fix $\Omega_0$ sufficiently large so that Lemma~\ref{l: spectral mapping large omega expansions} applies for $|\Omega| \geq \Omega_0$. We then let $\Gamma^0_R$, $\Gamma^1_R$, and $\tilde{\Gamma}_R$ be the contours depicted in Figure~\ref{fig:contour shifting spectral mapping}.

\begin{figure}
    \centering
    \includegraphics[width=1\textwidth]{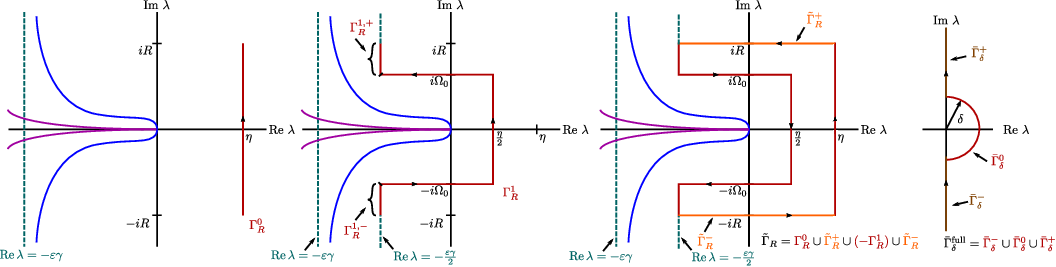}
    \caption{Far left: Fredholm borders of $\mcl$ (blue, purple) and starting contour $\Gamma^0_R$ (dark red) from Corollary~\ref{c: inverse laplace}. Left center: the new contour $\Gamma^1_R$ (dark red). Right center: the closed contour $\tilde{\Gamma}_R$. Far right: the contour $\bar{\Gamma}^\mathrm{full}_\delta$. }
    \label{fig:contour shifting spectral mapping}
\end{figure}

Lemmas~\ref{l: spectral mapping L0 inverse bounds} and~\ref{l: spectral mapping large omega expansions} show that the only term in the expansion of $(u(\cdot;\lambda),v(\cdot;\lambda))^\top$ that does not carry a factor $|\Omega|^{-1/2}$ is the term $(0,v_0(\cdot;\lambda))^\top$. However, this term also vanishes as $|\Omega| \to \infty$ by the Riemann-Lebesgue lemma. We conclude that the contributions from the horizontal segments on the top and bottom of $\tilde{\Gamma}_R$ vanish in the limit $R \to \infty$, so that we have
\begin{align*}
    \lim_{R \to \infty} \int_{\tilde{\Gamma}_R} \re^{\lambda t} \begin{pmatrix} u(\lambda) \\ v(\lambda) \end{pmatrix} \de \lambda = \lim_{R \to \infty} \left( \int_{\Gamma^0_R} - \int_{\Gamma^1_R} \right) \re^{\lambda t} \begin{pmatrix} u(\lambda) \\ v(\lambda) \end{pmatrix} \de \lambda.
\end{align*}
Since 
\begin{align*}
    \lim_{R \to \infty} \int_{\Gamma^0_R}  \re^{\lambda t} \begin{pmatrix} u(\lambda) \\ v(\lambda) \end{pmatrix} \de \lambda 
\end{align*}
exists for $g_1, g_2 \in C^\infty_c(\R)$ by Corollary~\ref{c: inverse laplace}, we will have 
\begin{align*}
    \lim_{R \to \infty} \int_{\Gamma^1_R} \re^{\lambda t} \begin{pmatrix} u(\lambda) \\ v(\lambda) \end{pmatrix} \de \lambda = \lim_{R \to \infty} \int_{\Gamma^0_R} \re^{\lambda t} \begin{pmatrix} u(\lambda) \\ v(\lambda) \end{pmatrix} \de \lambda,
\end{align*}
with the limit on the left-hand side existing, provided 
\begin{align*}
    \lim_{R \to \infty} \int_{\tilde{\Gamma}_R} \re^{\lambda t} \begin{pmatrix} u(\lambda) \\ v(\lambda) \end{pmatrix} \de \lambda  = 0. 
\end{align*}
But this integral is zero for each $R > 0$ by analyticity of $(u(\lambda), v(\lambda))$, where we use that the contour $\tilde\Gamma_R$ lies fully to the right of the spectrum of $\lfr$; see Figure~\ref{fig:contour shifting spectral mapping}. Hence we may write 
\begin{align*}
    \re^{\lfr t} \begin{pmatrix} g_1 \\ g_2 \end{pmatrix}
    = 
    - \frac{1}{2 \pi \ri} \lim_{R \to \infty} \int_{\Gamma^1_R} \re^{\lambda t} (\lfr-\lambda)^{-1} \begin{pmatrix} g_1 \\ g_2 \end{pmatrix} \de  \lambda,
\end{align*}
for $g_1,g_2 \in C^\infty_c (\R)$ and $t > 0$. The goal of the remainder of this section is to obtain the following estimate on the integral over $\Gamma_R^{1,+} \cup \Gamma_R^{1,-}$, which will (together with the subsequent analysis of the resolvent near the origin) allow us to extend the contour definition of the semigroup to $\g \in L^p (\R, \R^2), 1 \leq p < \infty$ or $\g \in C_0(\R, \R^2)$ and prove exponential decay estimates for the high frequency contributions. 

\begin{prop}\label{p: spectral mapping estimate}
    Fix $1 \leq p \leq \infty$.
    Provided $\Omega_0$ is sufficiently large, there exists a constant $C > 0$ such that 
    \begin{align*}
        \left\| \lim_{R \to \infty} \int_{\Gamma^{1,+}_R \cup \Gamma^{1,-}_R} \re^{\lambda t} (\lfr - \lambda)^{-1} \g \de  \lambda \right\|_{L^p} \leq C \re^{- \frac{\eps \gamma}{4} t} \| \g \|_{L^p}
    \end{align*}
    for all $t > 0$ and $\g \in C^\infty_c (\R, \R^2)$.
\end{prop}

To prove this proposition, we split up $(\lfr - \lambda)^{-1} \g$ into several terms according to the decomposition in Lemma~\ref{l: spectral mapping large omega expansions} and estimate the integral for each term separately. 

\begin{lemma}[$u_0$ estimates]\label{l: spectral mapping u0 estimates}
    Fix $1 \leq p \leq \infty$. Let $u_0$ be as in Lemma~\ref{l: spectral mapping large omega expansions}. There exists a constant $C > 0$ such that 
    \begin{align*}
        \left\| \lim_{R \to \infty} \int_{\Gamma^{1,+}_R \cup \Gamma^{1,-}_R} \re^{\lambda t} u_0(\lambda) \de  \lambda \right\|_{L^p} \leq C \re^{- \frac{\eps \gamma}{4} t} \| \g \|_{L^p}
    \end{align*}
    for all $t > 0$ and for all $\g = (g_1, g_2)^\top \in C^\infty_c (\R, \R^2).$
\end{lemma}
\begin{proof}
    Recalling that the contour $\Gamma^{1,-}_R$ is the line segment connecting the point $-\eps \gamma/2 - \ri R$ with $-\eps \gamma/2 - \ri \Omega_0$ and $\Gamma^{1,+}_R$ joins $-\eps\gamma/2 + \ri \Omega_0$ with $-\eps\gamma/2 + \ri R$, we have by~\eqref{e: spectral mapping u0}
    \begin{align*}
        \left[ \lim_{R \to \infty} \int_{\Gamma^{1,+}_R \cup \Gamma^{1,-}_R} \re^{\lambda t} u_0 (\lambda) \de  \lambda \right] (\xi) = \ri \re^{- \frac{\eps \gamma}{2} t} \int_{|\Omega| > \Omega_0} \re^{\ri \Omega t} \frac{1}{\sqrt{|\Omega|}} \int_\R G^\mathrm{p}\left(\sqrt{|\Omega|} (\xi-\zeta)\right) g_1(\zeta) \de  \zeta \de   \Omega,
    \end{align*}
    where we use the notation
   \begin{align*}
        \int_{|\Omega|>\Omega_0} f(\Omega) \de  \Omega = \int_{-\infty}^{-\Omega_0} f(\Omega) + \int_{\Omega_0}^{\infty} f(\Omega) \de  \Omega,
    \end{align*}
    that is, the orientation is consistent with the second panel of Figure~\ref{fig:contour shifting spectral mapping}. 
    
    Using the formula~\eqref{e: spectral mapping G parabolic} for $G^\mathrm{p}$, we obtain 
    \begin{align*}
        \left[\lim_{R \to \infty} \int_{\Gamma^{\mathrm{new},+}_R \cup \Gamma^{\mathrm{new},-}_R} \re^{\lambda t} u_0 (\lambda) \de  \lambda \right] (\xi) = \ri \re^{-\frac{\eps \gamma}{2} t} \int_{|\Omega| > \Omega_0} \re^{\ri \Omega t} \int_\R \frac{\re^{- \sqrt{\ri \Omega} |\xi-\zeta|}}{2 \sqrt{\ri \Omega}} g_1(\zeta) \de  \zeta \de \Omega.
    \end{align*} 
    
    We swap the order of integration and rewrite the inner integral as
    \begin{align}
    \begin{split}
         \ri \re^{-\frac{\eps \gamma}{2} t} \int_\R g_1(\zeta) \int_{|\Omega| > \Omega_0} \re^{\ri \Omega t} \frac{\re^{- \sqrt{\ri \Omega} |\xi-\zeta|}}{2 \sqrt{\ri \Omega}} \de  \Omega \de \zeta &=  \re^{-\frac{\eps \gamma}{2} t} \int_\R g_1(\zeta) \int_{\bar{\Gamma}^\mathrm{full}_{\Omega_0}} \re^{\lambda t} \frac{\re^{- \sqrt{\lambda} |\xi-\zeta|}}{2 \sqrt{\lambda}} \de  \lambda \de \zeta\\
         &\qquad - \re^{-\frac{\eps \gamma}{2} t} \int_\R g_1(\zeta) \int_{\bar{\Gamma}_{\Omega_0}^0} \re^{\lambda  t} \frac{\re^{- \sqrt{\lambda} |\xi-\zeta|}}{2 \sqrt{\lambda}} \de  \lambda \de \zeta,
         \end{split} \label{e: split u0 integral}
    \end{align}
    where for any $\delta > 0$ we denote by $\bar{\Gamma}_\delta^0$ the right semi-circle centered at the origin with radius $\delta$, and by $\bar{\Gamma}_\delta^\mathrm{full}$ the contour that joins this semi-circle to the vertical rays extending to $\pm \ri \infty$; see  Figure~\ref{fig:contour shifting spectral mapping}. Since the singularity $\lambda^{-1/2}$ is integrable, we find that 
    \begin{align*} 
    \begin{split}
        \lim_{\delta \to 0^+} \int_{\bar{\Gamma}^\mathrm{full}_\delta} \re^{\lambda t} \frac{\re^{- \sqrt{\lambda} |\xi-\zeta|}}{2 \sqrt{\lambda}} \de  \lambda &= \ri \, \mathrm{p.v.} \int_\R \re^{\ri \Omega t} \frac{\re^{-\sqrt{\ri \Omega} |\xi-\zeta|}}{2 \sqrt{\ri \Omega}} \de  \Omega 
        = \ri \sqrt{\frac{\pi}{t}} \re^{-|\xi-\zeta|^2/4t} 
        =: \ri G^\mathrm{heat}(t, \xi-\zeta), 
        \end{split}
    \end{align*}    
    recognizing the integral as the inverse Fourier transform (in time) of the Fourier representation of the heat kernel. On the other hand, the value of the integral is independent of $\delta > 0$ since the integrand is analytic away from the origin in the closed right half-plane. In particular, we obtain 
    \begin{align*}
        \int_{\bar{\Gamma}^\mathrm{full}_{\Omega_0}} \re^{\lambda t} \frac{\re^{- \sqrt{\lambda} |\xi-\zeta|}}{2 \sqrt{\lambda}} \de  \lambda = \ri G^\mathrm{heat}(t, \xi-\zeta).
    \end{align*}
    Now to bound the integral over $\bar{\Gamma}_{\Omega_0}^0$ in~\eqref{e: split u0 integral} we use that the integrand is analytic away from the negative real axis to deform $\bar{\Gamma}_{\Omega_0}^0$ to a new contour 
    \begin{align*} \bar{\Gamma}^2 = \bar{\Gamma}^1_- \cup \bar{\Gamma}_{\delta_0}^0 \cup \bar{\Gamma}^1_+\end{align*}
    where $\delta_0 = \frac{\eps \gamma}{4}$, $\bar{\Gamma}^1_-$ is the line segment joining $- \ri \Omega_0$ with $- \ri \delta_0$, and $\bar{\Gamma}^1_+$ is the line segment joining $\ri \delta_0$ with $\ri \Omega_0$. There exist ($\eps$-, $\Omega_0$- and $\gamma$-dependent) constants $C, \alpha > 0$ such that for $\lambda \in \bar{\Gamma}^2$ we have 
    \begin{align}|\lambda|^{-1/2} \left|\re^{-\sqrt{\lambda} |\xi-\zeta|}\right| \leq C \re^{- \alpha |\xi-\zeta|}, \label{e: contour Gamma 2 estimate} \end{align} 
    while also $|\re^{\lambda t}| \leq \re^{\frac{\eps \gamma}{4} t}$. Applying Young's convolution inequality we obtain 
    \begin{align*}
        \left\| \xi \mapsto \int_\R g_1(\zeta)  \int_{\bar{\Gamma}^2} \re^{\lambda t} \frac{\re^{-\sqrt{\lambda} |\xi-\zeta|}}{2 \sqrt{\lambda}} \de  \lambda \de \zeta \right\|_{L^p} \lesssim \re^{\frac{\eps \gamma}{4} t}\| g_1 \|_{L^p},
    \end{align*}
    for $t > 0$ and $g_1 \in C_c^\infty(\R)$. Absorbing the slowly exponentially growing factor into the factor of $\re^{-\frac{\eps \gamma}{2} t}$ from~\eqref{e: split u0 integral} gives the desired estimate for the second term 
    \begin{align*}
\ri \re^{-\frac{\eps \gamma}{2} t} \int_\R g_1(\zeta) \int_{\bar{\Gamma}_{\Omega_0}^0} \re^{\lambda  t} \frac{\re^{- \sqrt{\lambda} |\xi-\zeta|}}{2 \sqrt{\lambda}} \de  \lambda \de \zeta = \ri \re^{-\frac{\eps \gamma}{2} t} \int_\R g_1(\zeta) \int_{\bar{\Gamma}^2} \re^{\lambda  t} \frac{\re^{- \sqrt{\lambda} |\xi-\zeta|}}{2 \sqrt{\lambda}} \de  \lambda \de \zeta,
\end{align*}    
    in~\eqref{e: split u0 integral}. Similarly, the estimate
    \begin{align*}
        \left\| \xi \mapsto \int_\R g_1(\zeta) \ri G^\mathrm{heat}(t, \xi-\zeta) \de  \zeta \right\|_{L^p} \lesssim \| G^\mathrm{heat}(t, \cdot) \|_{L^1} \| g_1 \|_{L^p} \lesssim \| g_1 \|_{L^p} \lesssim \| \g \|_{L^p}
    \end{align*}
    for $t > 0$ and $\g = (g_1,g_2)^\top \in \C_c^\infty(\R,\R^2)$, leads to the desired estimate for the remaining term in~\eqref{e: split u0 integral}, completing the proof.
\end{proof}

Before estimating the other contributions identified in Lemma~\ref{l: spectral mapping large omega expansions}, we first record some basic estimates on $\tilde{G}^\mathrm{tr}$ which can be readily inferred from~\eqref{e: tilde G def}. 
\begin{lemma}[$\tilde{G}^\mathrm{tr}$ estimates]\label{l: tilde G tr estimates}
    Let $\tilde{G}^\mathrm{tr}(\xi,\zeta)$ be defined by~\eqref{e: tilde G def}. There exists a constant $C > 0$ such that
    \begin{align*}
        \| \tilde{G}^\mathrm{tr}(\cdot, \cdot) \|_{L^\infty (\R^2, \R)} &\leq C, \qquad
        \sup_{\zeta \in \R} \| \tilde{G}^\mathrm{tr} (\cdot, \zeta) \|_{L^1} \leq C. 
    \end{align*}
    As a consequence, it follows by Minkowski's integral inequality that for any $1 \leq p \leq \infty$ and $g_2 \in C_c^\infty(\R)$, we have
    \begin{align*}
        \left\| \int_\R \tilde{G}^\mathrm{tr}(\cdot, \zeta) g_2 (\zeta) \de  \zeta \right\|_{L^p} \leq \left(\sup_{\zeta \in \R} \| \tilde{G}^\mathrm{tr} (\cdot, \zeta) \|_{L^1}\right) \| g_2 \|_{L^p} \leq C \| g_2 \|_{L^p}. 
    \end{align*}
\end{lemma}

\begin{lemma}[$v_0$ estimates]
    Fix $1 \leq p \leq \infty$. Let $v_0$ be as in Lemma~\ref{l: spectral mapping large omega expansions}. There exists a constant $C > 0$ such that 
    \begin{align*}
        \left\| \lim_{R \to \infty} \int_{\Gamma^{1,+}_R \cup \Gamma^{1,-}_R} \re^{\lambda t} v_0 (\lambda) \de  \lambda \right\|_{L^p} \leq C \re^{-\frac{\eps \gamma}{2} t} \| \g \|_{L^p}
    \end{align*}
    for all $t > 0$ and all $\g = (g_1, g_2)^\top \in C^\infty_c (\R, \R^2)$. 
\end{lemma}
\begin{proof}
    Here we must take advantage of the symmetry in the principal value integral. The indefinite integrals here are all to be understood in the principal value sense. We have by~\eqref{e: spectral mapping v0} 
    \begin{align*}
        \left[\lim_{R \to\infty} \int_{\Gamma^{1,+}_R \cup \Gamma^{1,-}_R} \re^{\lambda t} v_0 (\lambda) \de  \lambda \right] (\xi) = \ri \re^{- \frac{\eps \gamma}{2} t} \int_{|\Omega|> \Omega_0} \re^{\ri \Omega t} \int_\R  \tilde{G}^\mathrm{tr} (\xi,\zeta) \re^{\ri \frac{\Omega}{c} (\xi-\zeta)} g_2(\zeta) \de \zeta \de \Omega
    \end{align*}
    where $\tilde{G}^\mathrm{tr}(\xi,\zeta)$ is given in~\eqref{e: tilde G def}. 
    Swapping the order of integration, we have 
    \begin{align*}
         \left[ \int_{\Gamma^{1,+}_R \cup \Gamma^{1,-}_R} \re^{\lambda t} v_0 (\lambda) \de  \lambda \right] (\xi) = \ri \re^{-\frac{\eps \gamma}{2} t} \int_\R g_2(\zeta) \tilde{G}^\mathrm{tr} (\xi,\zeta)  \int_{|\Omega| > \Omega_0} \re^{\ri \frac{\Omega}{c} (\xi-\zeta+ct)} \de  \Omega \de\zeta. 
    \end{align*}
    Note that 
    \begin{align*}
        \mathrm{p.v. } \int_{|\Omega| > \Omega_0} \re^{\ri \frac{\Omega}{c} (\xi-\zeta + ct)} \de  \Omega &= \mathrm{p.v. } \int_\R \re^{\ri \frac{\Omega}{c}  (\xi-\zeta + ct)} \de  \Omega - \int_{|\Omega| \leq \Omega_0} \re^{\ri \frac{\Omega}{c} (\xi-\zeta+ct)} \de  \Omega \\
        &= c 2 \pi \delta(\xi-\zeta+ct) - \int_{|\Omega| \leq \Omega_0} \re^{\ri \frac{\Omega}{c} (\xi-\zeta+ct)} \de  \Omega.
    \end{align*}
    Hence 
    \begin{multline*}
        \left[ \int_{\Gamma^{1,+}_R \cup \Gamma^{1,-}_R} \re^{\lambda t} v_0 (\lambda) \de  \lambda \right] (\xi) = \ri \re^{-\frac{\eps \gamma}{2} t} \bigg( 2 \pi c \int_\R g_2(\zeta) \tilde{G}^\mathrm{tr} (\xi,\zeta) \delta(\xi-\zeta+ct) \de \zeta \\ - \int_\R g_2(\zeta) \tilde{G}^\mathrm{tr} (\xi,\zeta) \int_{|\Omega| \leq \Omega_0} \re^{\ri \frac{\Omega}{c} (\xi-\zeta+ct)} \de  \Omega\de\zeta \bigg). 
    \end{multline*}
    For the first term, we have 
    \begin{align*}
        2 \pi c\re^{- \frac{\eps \gamma}{2} t} \int_\R g_2(\zeta) \tilde{G}^\mathrm{tr} (\xi,\zeta)  \delta(\xi-\zeta+ct) \de  \zeta  = 2 \pi c \re^{- \frac{\eps \gamma}{2} t} g_2(\xi+ct) \tilde{G}^\mathrm{tr} (\xi,\xi+ct). 
    \end{align*}
    Measuring in $L^p(\R)$, we obtain
    \begin{align*}
         \left\| 2 \pi c \re^{- \frac{\eps \gamma}{2} t} \int_\R g_2(\zeta) \tilde{G}^\mathrm{tr} (\cdot, \zeta)  \delta(\cdot-\zeta+ct)  \de  \zeta \right\|_{L^p} \lesssim \re^{-\frac{\eps \gamma}{2} t} \| g_2 \|_{L^p} \|\tilde{G}^\mathrm{tr}\|_{L^\infty} \lesssim \re^{-\frac{\eps \gamma}{2} t} \| \g \|_{L^p} ,
    \end{align*}
    for $t > 0$ and $\g = (g_1,g_2)^\top \in C_c^\infty(\R,\R^2)$. For the second term, we use Lemma~\ref{l: tilde G tr estimates} to obtain
    \begin{align*}
        \left\| \re^{-\frac{\eps \gamma}{2} t} \int_\R g_2(\zeta) \tilde{G}^\mathrm{tr} (\xi,\zeta) \int_{|\Omega| \leq \Omega_0} \re^{\ri \frac{\Omega}{c} (\xi-\zeta+ct)} \de  \Omega \de\zeta \right\|_{L^p} &\leq 2 \Omega_0 \re^{- \frac{\eps \gamma}{2} t} \| g_2 \|_{L^p} \left( \sup_{\zeta \in \R} \|\tilde{G}^\mathrm{tr}(\cdot, \zeta) \|_{L^1} \right) \\
        &\lesssim \re^{- \frac{\eps \gamma}{2} t} \| \g \|_{L^p}, 
    \end{align*}
    for $t > 0$ and $\g = (g_1,g_2)^\top \in C_c^\infty(\R,\R^2)$, which concludes the proof of the desired estimate.
\end{proof}

\begin{lemma}[Estimates on $u_1$ and $v_1$]\label{l: spectral mapping u1 v1 estimates}
    Fix $1 \leq p \leq \infty$. Let $u_1$ and $v_1$ be as in Lemma~\ref{l: spectral mapping large omega expansions}. There exists a constant $C > 0$ such that 
    \begin{align*}
        \left\| \lim_{R \to \infty} \int_{\Gamma^{1,+}_R \cup \Gamma^{1,-}_R} \re^{\lambda t} u_1 (\lambda) \de  \lambda \right\|_{L^p} \leq C \re^{- \frac{\eps \gamma}{4} t} \| \g \|_{L^p}, 
    \end{align*}
    and 
    \begin{align*}
        \left\| \lim_{R \to \infty} \int_{\Gamma^{1,+}_R \cup \Gamma^{1,-}_R} \re^{\lambda t} v_1 (\lambda) \de  \lambda \right\|_{L^p} \leq C \re^{- \frac{\eps \gamma}{4} t} \| \g \|_{L^p}, 
    \end{align*}
    for all $t > 0$ and all $\g = (g_1, g_2)^\top  \in C^\infty_c (\R, \R^2)$. 
\end{lemma}
\begin{proof}
    We will prove the cases $p = 1$ and $p = \infty$ directly, the estimates for $1 < p < \infty$ then following by interpolation. Using~\eqref{e: spectral mapping u1} and swapping the order of integration, we have 
    \begin{align*}
        \left[ \lim_{R \to \infty} \int_{\Gamma^{1,+}_R \cup \Gamma^{1,-}_R} \re^{\lambda t} u_1(\lambda) \de  \lambda \right] (\xi) = \ri \re^{-\frac{\eps \gamma}{2} t} \int_\R \int_\R\tilde{G}^\mathrm{tr}(\zeta,z) g_2(z) \int_{|\Omega| > \Omega_0} \re^{\ri \Omega \left( t + \frac{\zeta-z}{c} \right)}  \frac{\re^{-\sqrt{\ri \Omega} |\xi-\zeta|}}{2 \sqrt{\ri \Omega}} \de \Omega \de z \de  \zeta. 
    \end{align*}
    As in the proof of Lemma~\ref{l: spectral mapping u0 estimates}, we would like to close the integration contour so that we may recognize the contour integral as an inverse Fourier transform. However, in this case we can only evaluate this inverse Fourier transform provided $t + \frac{\zeta-c}{c} \geq 0$, in which case we will obtain the heat kernel evaluated at time $t + \frac{\zeta-z}{c} \geq 0$. Since $\tilde{G}^\mathrm{tr}(\zeta, z)$ is supported on $\{ \zeta \leq z \}$, the quantity $t + \frac{\zeta-z}{c}$ will be negative for some portions of the overall integral, and so we estimate the integral over these portions separately. Defining the set
    \begin{align*}
        E_t = \left\{ (\zeta,z) \in \R^2 : t + \frac{\zeta-z}{c} \geq 0 \right\},
    \end{align*}
    and denoting by $E_t^c$ its complement in $\R^2$, we write
    \begin{align*}
       \left[ \lim_{R \to \infty} \int_{\Gamma_R^{1,+}\cup \Gamma_R^{1,-}} \re^{\lambda t} u_1(\lambda) \de  \lambda \right] (\xi) = \frac{\ri}{\re^{\frac{\eps \gamma}{2} t}} \left( \int_{E_t} + \int_{E_t^c} \right) \left[ \tilde{G}^\mathrm{tr}(\zeta,z) g_2(z) \int_{|\Omega| > \Omega_0} \re^{\ri \Omega \left( t + \frac{\zeta-z}{c} \right)}  \frac{\re^{-\sqrt{\ri \Omega} |\xi-\zeta|}}{2 \sqrt{\ri \Omega}} \de \Omega \de z \de  \zeta \right].
    \end{align*}
    For $(\zeta, z) \in E_t$, arguing as in the proof of Lemma~\ref{l: spectral mapping u0 estimates}, we find
    \begin{align*}
        \ri \int_{|\Omega| > \Omega_0} \re^{\ri \Omega \left( t + \frac{\zeta - z}{c} \right)} \frac{\re^{-\sqrt{\ri \Omega} |\xi - \zeta|}}{2 \sqrt{\ri \Omega}} \de  \Omega = \ri G^\mathrm{heat} \left( t + \frac{\zeta - z}{c} , \xi - \zeta \right) - \int_{\bar{\Gamma}^2} \re^{\lambda \left( t + \frac{\zeta - z}{c} \right)} \frac{\re^{-\sqrt{\lambda} |\xi - \zeta|}}{2 \sqrt{\lambda}} \de  \lambda . 
    \end{align*}
    Using Minkowski's inequality and Lemma~\ref{l: tilde G tr estimates} together with the estimates~\eqref{e: contour Gamma 2 estimate} and $|\re^{\lambda (t + (\zeta - z)/c)}| \leq \re^{\frac{\eps \gamma}{4}t}$ for $\lambda \in \bar{\Gamma}^2$, $t \geq 0$ and $(\zeta,z) \in E_t$ with $\zeta \leq z$, we obtain 
    \begin{align*}
        \left\| \xi \mapsto \int_{E_t} \tilde{G}^\mathrm{tr} (\zeta, z) g_2 (z) \int_{|\Omega| > \Omega_0} \re^{\ri \Omega \left( t + \frac{\zeta-z}{c} \right)}  \frac{\re^{-\sqrt{\ri \Omega} |\xi-\zeta|}}{2 \sqrt{\ri \Omega}} \de \Omega \de z \de  \zeta \right\|_{L^1} \lesssim \re^{\frac{\eps \gamma}{4} t} \| g_2 \|_{L^1} ,  
    \end{align*}
    for $t > 0$ and $g_2 \in C_c^\infty(\R)$. The exponentially growing factor may again be absorbed by the original factor of $\re^{-\frac{\eps \gamma}{2} t}$. By a similar argument, we obtain the desired control in $L^\infty(\R)$ as well. 

    It remains to control the integral over $E_t^c$. For $(\zeta, z) \in E_t^c$, we again add an additional piece to close the integration contour, writing
    \begin{align*}
        \ri \int_{|\Omega| > \Omega_0} \re^{\ri \Omega \left( t + \frac{\zeta-z}{c} \right)}  \frac{\re^{-\sqrt{\ri \Omega} |\xi-\zeta|}}{2 \sqrt{\ri \Omega}} \de \Omega = \left(\int_{\bar{\Gamma}^\mathrm{full}_{\Omega_0}} - \int_{\bar{\Gamma}^0_{\Omega_0}} \right) \left[ \re^{\lambda \left( t + \frac{\zeta-z}{c} \right)} \frac{\re^{-\sqrt{\lambda} |\xi - \zeta|}}{2 \sqrt{\lambda}} \de  \lambda \right].
    \end{align*}
    Since $\Re \lambda \geq 0$ for $\lambda \in \bar{\Gamma}_{\Omega_0}^0$ and $t + \frac{\zeta-z}{c} < 0$ if $(\zeta, z) \in E_t^c$, there exist $t$-, $\zeta$-, $z$-, and $\xi$-independent constants $C,\alpha > 0$ such that
    \begin{align*}
        \left| \int_{\bar{\Gamma}_{\Omega_0}^0} \re^{\lambda \left( t + \frac{\zeta-z}{c} \right)}  \frac{\re^{-\sqrt{\lambda} |\xi - \zeta|}}{2 \sqrt{\lambda}} \de  \lambda \right| \leq C \re^{- \alpha |\xi - \zeta|},
    \end{align*}
    for $t \geq 0$ and $(\zeta,z) \in E_t^c$, from which the desired estimates follow readily for this term via Minkowski's inequality. For fixed $\eta > 0$, the integral over $\bar{\Gamma}^\mathrm{full}_{\Omega_0}$ can be deformed to the contour $\Re \lambda = \eta$, since the integrand is analytic away from the negative real axis and vanishes as $\Im \lambda \to \infty$, from which we obtain
    \begin{align*}
    \int_{\bar{\Gamma}^\mathrm{full}_{\Omega_0}} \re^{\lambda \left( t + \frac{\zeta-z}{c} \right)} \frac{\re^{-\sqrt{\lambda} |\xi - \zeta|}}{2 \sqrt{\lambda}} \de  \lambda = \re^{\eta \left(t + \frac{\zeta-z}{c} \right)} \mathrm{p.v.} \int_{-\infty}^\infty \re^{\ri \Omega \left( t + \frac{\zeta - z}{c} \right)} \frac{\re^{-\sqrt{\eta + i\Omega} |\zeta - \xi|}}{2 \sqrt{\eta + \ri \Omega}} \de  \Omega\,. 
    \end{align*}
    The remaining integral is uniformly bounded in $\eta \geq 1, t > 0, \xi \in \R$ and $(\zeta, z) \in E_t^c$ as a principal value integral. Hence, sending $\eta \to \infty$, we conclude that the contribution from $\bar{\Gamma}^\mathrm{full}_{\Omega_0}$ vanishes. The proof of the estimate involving $v_1$ proceeds analogously.     
\end{proof}

\begin{lemma}[Estimate on $v_2$]\label{l: spectral mapping v2 estimate}
    Fix $1 \leq p \leq \infty$. Let $v_2$ be as in Lemma~\ref{l: spectral mapping large omega expansions}. There exists a constant $C > 0$ such that
    \begin{align*}
        \left\| \lim_{R \to \infty} \int_{\Gamma^{\mathrm{new},+}_R \cup \Gamma^{\mathrm{new},-}_R} \re^{\lambda t} v_2 (\lambda) \de \lambda \right\|_{L^p} \leq C \re^{- \frac{\eps \gamma}{4} t} \| \g \|_{L^p} 
    \end{align*}
    for all $t > 0$ and $\g \in C_c^\infty(\R,\R^2)$.
\end{lemma}
\begin{proof}
    The proof is analogous to that of Lemma~\ref{l: spectral mapping u1 v1 estimates}, except that it involves an additional convolution with kernel $\re^{\ri \frac{\Omega}{c} \cdot}\, G^\mathrm{tr}$. As a result we recognize the principal contribution, in the inverse Fourier transform, as the heat kernel convolved twice with two separate Dirac delta distributions, leading to an integral of the form 
    \begin{align*}
        \re^{- \frac{\eps \gamma}{2} t} \int_{\R^3} \tilde{G}^\mathrm{tr}(\xi,\zeta) G^\mathrm{heat} \left( t - \frac{\xi-\zeta + z -w}{c}, \zeta-z \right) \tilde{G}^\mathrm{tr} (z,w) g_2(w) \de w \de \zeta \de z. 
    \end{align*}
    Applying Lemma~\ref{l: tilde G tr estimates} and estimating the innermost integral with the heat kernel scaling in a similar manner to the proof of Lemma~\ref{l: spectral mapping u1 v1 estimates} then gives the desired estimate. 
\end{proof}

Having estimated the principal terms arising in the expansion of the resolvent we are now able to complete the proof of Proposition~\ref{p: spectral mapping estimate}.

\begin{proof}[Proof of Proposition~\ref{p: spectral mapping estimate}]
    We decompose $(\lfr - \lambda)^{-1} \g$ as in Lemma~\ref{l: spectral mapping large omega expansions}. The estimates on terms involving $u_0, v_0, u_1, v_1,$ and $v_2$ have already been established in Lemmas~\ref{l: spectral mapping u0 estimates}-\ref{l: spectral mapping v2 estimate}. The remainder terms are controlled in $L^p (\R)$  by $|\Omega|^{-3/2}$ by Lemma~\ref{l: spectral mapping large omega expansions}, which is integrable on $|\Omega| > \omega_0$, completing the proof. 
\end{proof}

\section{Estimates on the leading edge resolvent}\label{s: appendix leading edge resolvent}
From the analysis in Section~\ref{s: leading edge resolvent}, we know that for $\sigma^2$ to the right of $\Sigma_\mathrm{ess}(\mcl_+)$, we can write the leading edge resolvent $(\mcl_+-\sigma^2)^{-1}$ as
\begin{align*}
    (\mcl_+ - \sigma^2)^{-1} \g (\xi) = 
    \int_\R \Pi_{13} T_\sigma^\mathrm{fr}(\xi-\zeta) \Lambda_1 \g(\zeta) \de  \zeta, 
\end{align*}
with $\g = (g_1, g_2)^\top$ and
\begin{align*}
    \Tfr_\sigma (\zeta) = \Gheat_\sigma (\zeta) P_\mathrm{pole} + \tilde{G}^c_\sigma(\zeta) + [G^c_\sigma(\zeta) - \Gheat_\sigma (\zeta) P_\mathrm{pole}] + G^h_\sigma(\zeta), 
\end{align*}
where
\begin{align*}
    \Gheat_\sigma(\zeta) &= \frac{1}{\sigma} \re^{-\nufr^1 \sigma |\zeta|}, \\
    G^c_\sigma(\zeta) &= \frac{1}{\sigma} P_\mathrm{pole} \left( \re^{\nufr^+(\sigma) \zeta} 1_{\{\zeta < 0 \}} + \re^{\nufr^-(\sigma) \zeta} 1_{\{\zeta \geq 0\}} \right), 
    \\
    \tilde{G}^c_\sigma(\zeta)  &= \left[ \re^{\nufr^-(\sigma) \zeta} \left(\Pfr^\mathrm{cs}(\sigma) - \frac{P_\mathrm{pole}}{\sigma} \right) \right] 1_{\{\zeta \geq 0\}} - \left[ \re^{\nufr^+(\sigma) \zeta} \left(\Pfr^\mathrm{cu}(\sigma) + \frac{P_\mathrm{pole}}{\sigma} \right)\right] 1_{\{ \zeta < 0 \}}, \\
    G^h_\sigma(\zeta) &= - \re^{\nu_3(\sigma^2) \zeta} \Pfr^\mathrm{uu}(\sigma^2) 1_{\{\zeta \leq 0\}}. 
\end{align*}

If $\g$ is odd, then we can rewrite 
\begin{align} \label{e: odd heat kernel}
    \Gheat_\sigma \ast \g (\xi) = \int_0^\infty \Godd_\sigma (\xi,\zeta) \g(\zeta) \de  \zeta,
\end{align}
for $\xi \in \R$, where
\begin{align*}
    \Godd_\sigma(\xi,\zeta) = \frac{1}{\sigma} \left[ \re^{-\nufr^1 \sigma |\xi-\zeta|} - \re^{-\nufr^1 \sigma |\xi+\zeta|}\right]. 
\end{align*}
Note that $\Godd_\sigma(-\xi,\zeta) = - \Godd_\sigma(\xi,\zeta)$, so that $\Gheat_\sigma \ast \g$ is odd provided $\g$ is. 

In the following four lemmas we obtain bounds on $G_\sigma^\mathrm{odd}$, $G_\sigma^c$, $\tilde{G}_\sigma^c$ and $G_\sigma^h$, which eventually lead to the proofs of Lemmas~\ref{l: right resolvent L1 Linf boundedness estimate} and~\ref{l: leading edge resolvent L2 estimate}.

\begin{lemma}\label{l: appendix Godd gamma estimate}
    There exist positive constants $C$ and $\delta$ such that
    \begin{align*}
        |\Godd_\sigma(\xi,\zeta)| + |\partial_\xi \Godd_\sigma(\xi,\zeta)| \leq C \langle \zeta \rangle \re^{-\nufr^1 \Re \sigma |\xi-\zeta|}, 
    \end{align*}
    for all $\xi, \zeta \geq 0$ and $\sigma \in B(0,\sqrt{\delta})$ with $\Re \sigma \geq 0$.
\end{lemma}
\begin{proof}
    We start with the estimate without derivative. First, assume $\xi \geq \zeta \geq 0$, so that $|\xi-\zeta| = \xi-\zeta$. Then, we have
    \begin{align*}
        \Godd_\sigma(\xi,\zeta) = \frac{1}{\sigma} \re^{-\nufr^1 \sigma |\xi-\zeta|} \left[1 - \re^{-2 \nufr^1 \sigma \zeta}\right]. 
    \end{align*}
    Using that there exists a constant $C > 0$ such that $|1-\re^z| \leq C |z|$ for all $z \in \C$ with $\Re z \leq 0$, we have $|1 - \re^{-2 \nufr^2 \sigma \zeta}| \lesssim  |\sigma| |\zeta|$, and so
    \begin{align*}
        |\Godd_\sigma (\xi,\zeta)| \lesssim |\zeta| \re^{-\nufr^1 \Re \sigma |\xi-\zeta|} \leq \langle \zeta \rangle \re^{-\nufr^1 \Re \sigma |\xi-\zeta|}, 
    \end{align*}
    for $\xi \geq \zeta \geq 0$ and $\sigma \in B(0,\sqrt{\delta})$ with $\Re \sigma \geq 0$. 
    
    Now assume $\zeta \geq \xi \geq 0$, so that $|\xi-\zeta| = \zeta-\xi$. Then we have
    \begin{align*}
        \Godd_\sigma(\xi,\zeta) = \frac{1}{\sigma} \re^{-\nufr^1 \sigma |\xi-\zeta|} \left[1 - \re^{- 2 \nufr^1 \sigma \xi}\right].
    \end{align*}
    Since $\Re (- 2\nufr^1 \sigma \xi) \leq 0$, we have $|1- \re^{-2 \nufr^1 \sigma \xi}| \lesssim |\sigma| |\xi|$ for $\xi \geq 0$ and $\sigma \in B(0,\sqrt{\delta})$ with $\Re \sigma \geq 0$. Since in the present regime, $|\xi| \leq |\zeta| \leq \langle \zeta \rangle$, we conclude
    \begin{align*}
        |\Godd_\sigma(\xi,\zeta)| \lesssim |\xi| \re^{-\nufr^1 \Re \sigma |\xi-\zeta|} \lesssim \langle \zeta \rangle \re^{-\nufr^1 \Re \sigma |\xi-\zeta|} ,
    \end{align*}
    for $\zeta \geq \xi \geq 0$ and $\sigma \in B(0,\sqrt{\delta})$ with $\Re \sigma \geq 0$, which completes the proof of the desired estimate for $|\Godd_\sigma(\xi,\zeta)|$.
    
    For the derivative estimates, note that 
    \begin{align*}
        \partial_\xi \Godd_\sigma(\xi,\zeta) &= - \nufr^1 \re^{-\nufr^1 \sigma| |\xi-\zeta|} \sign (\xi-\zeta) + \re^{-\nufr^1 \sigma |\xi+\zeta|} \sign (\xi+\zeta), 
    \end{align*}
    Since we are assuming $\xi,\zeta \geq 0$, we have $|\xi+\zeta| \geq |\xi - \zeta|$, so we readily obtain the stronger estimate $|\partial_\xi \Godd_\sigma(\xi,\zeta)|  \lesssim \re^{-\nu_{\mathrm{fr}}^1 \Re \sigma |\xi - \zeta|}$ for $\xi,\zeta \geq 0$ and $\sigma \in B(0,\sqrt{\delta})$ with $\Re \sigma \geq 0$.
\end{proof}

\begin{lemma}\label{l: appendix G c minus G heat estimate}
    There exist positive constants $C$ and $\delta$ such that
    \begin{align*}
        |G^c_\sigma (\zeta) - \Gheat_\sigma(\zeta) P_\mathrm{pole}| + |\partial_\zeta (G^c_\sigma (\zeta) - \Gheat_\sigma(\zeta) P_\mathrm{pole})| \leq C |\sigma||\zeta|\re^{-\frac{1}{2} \nufr^1 \Re \sigma |\zeta|}
    \end{align*}
    for all $\zeta \in \R$ provided $\sigma \in B(0,\sqrt{\delta})$ with $\Re \sigma \geq 0$. 
\end{lemma}
\begin{proof}
    We assume $\zeta \geq 0$; the case $\zeta < 0$ is similar. In this case, we have
    \begin{align*}
        G^c_\sigma (\zeta) - \Gheat_\sigma(\zeta) P_\mathrm{pole} = \frac{P_\mathrm{pole}}{\sigma} \re^{-\nufr^1 \sigma \zeta} \left[ \re^{\tilde{\nu}_\mathrm{fr}^- (\sigma) \zeta} -1 \right],
    \end{align*}
    where $\tilde{\nu}_\mathrm{fr}^- (\sigma) = \mathrm{O}(\sigma^2)$. First we assume $|\sigma^2 \zeta| \leq 1$, in which case, there exists by Taylor's theorem a $\sigma$- and $\zeta$-independent constant $C > 0$ such that $|\re^{\tilde{\nu}_{\mathrm{fr}}^-(\sigma) \zeta} - 1| \leq C |\sigma|^2 |\zeta|$. So, we have
    \begin{align*}
         |G^c_\sigma (\zeta) - \Gheat_\sigma(\zeta) P_\mathrm{pole}| \lesssim |\sigma||\zeta|\re^{-\nufr^1 \Re \sigma |\zeta|}
    \end{align*}
    for $\zeta \geq 0$ and $\sigma \in B(0,\sqrt{\delta})$ with $\Re \sigma \geq 0$ and  $|\sigma^2 \zeta| \leq 1$.    
    
    On the other hand, the expansion~\eqref{e: leading edge spatial eval expansion} yields
    \begin{align*}
        |G^c_\sigma(\zeta) - \Gheat_\sigma(\zeta) P_\mathrm{pole}| \lesssim \frac{|\sigma|^2 |\zeta|}{|\sigma|} (|\re^{\nufr^-(\sigma) \zeta}| + |\re^{-\nufr^1 \sigma \zeta}|) \lesssim |\sigma||\zeta|\re^{-\frac{1}{2} \nufr^1 \Re \sigma |\zeta|},
    \end{align*}
    for $\zeta \geq 0$ and $\sigma \in B(0,\sqrt{\delta})$ with $\Re \sigma \geq 0$ and $|\sigma^2 \zeta| \geq 1$. So we obtain the desired estimate in either case. The estimate on the derivative is similar, in fact easier since we gain a factor of $\sigma$ after differentiating. 
\end{proof}

\begin{lemma}
    There exist positive constants $C$ and $\delta$ such that 
    \begin{align*}
        |\tilde{G}^c_\sigma (\zeta)| + |\partial_\zeta \tilde{G}^c_\sigma (\zeta)| \leq C \re^{-\frac{1}{2} \nufr^1 \Re \sigma |\zeta|},
    \end{align*}
    for all $\zeta \in \R$ provided $\sigma \in B(0,\sqrt{\delta})$ with $\Re \sigma \geq 0$.
\end{lemma}
\begin{proof}
    This follows readily from the following facts: 1. $\Pfr^\mathrm{cs}(\sigma) - \frac{P_\mathrm{pole}}{\sigma}$ and $\Pfr^\mathrm{cu}(\sigma) + \frac{P_\mathrm{pole}}{\sigma}$ are analytic in $\sigma$ in a full neighborhood of the origin by Lemma~\ref{l: center projections pole}; 2. we have $\Re \nufr^-(\sigma) \leq -\frac{1}{2} \nufr^1 \Re \sigma$ and $\Re \nufr^+(\sigma) \geq \frac{1}{2}\nufr^1 \Re \sigma$ for $\sigma \in B(0,\sqrt{\delta})$ with $\Re \sigma \geq 0$. 
\end{proof}

\begin{lemma}\label{l: appendix G h gamma estimate}
    There exist positive constants $C, \delta$, and $\mu$ such that 
    \begin{align*}
        |G^h_\sigma (\zeta)| + |\partial_\zeta G^h_\sigma (\zeta)| \leq C \re^{-\mu |\zeta|},
    \end{align*}
    for all $\zeta \in \R$ provided $\sigma \in B(0,\sqrt{\delta})$.
\end{lemma}
\begin{proof}
    This follows from the fact that the eigenvalue $\nu_3(\sigma^2)$ stays in the open right half-plane, uniformly bounded away from the imaginary axis, for $|\sigma|$ small, which in turn follows from Corollary~\ref{c: leading edge spatial eval}. 
\end{proof}

We are now able to prove Lemmas~\ref{l: right resolvent L1 Linf boundedness estimate} and~\ref{l: leading edge resolvent L2 estimate}, which provide control over the leading edge resolvent.

\begin{proof}[Proof of Lemma~\ref{l: right resolvent L1 Linf boundedness estimate}]
Let $\u$ satisfy $(\mcl_+-\sigma^2) \u = \g$. 

First we prove~\eqref{e: right resolvent L101 estimate}. Observe that there exists a constant $\tilde{\mu} > 0$ such that 
\begin{align} \label{e: sigma bound Omega fr 2}
- \Re \sigma \leq -\tilde{\mu}|\sigma|,
\end{align}
for all $\sigma \in \Delta^{\mathrm{fr},2}_\delta$. Hence, combining Lemmas~\ref{l: appendix Godd gamma estimate} through~\ref{l: appendix G h gamma estimate} and recalling~\eqref{e: odd heat kernel} and the fact that we have $G^\mathrm{odd}_\sigma(-\xi,\zeta) = -G^\mathrm{odd}_\sigma(\xi,\zeta)$ for $\xi \in \R$ and $\zeta \geq 0$, we obtain a constant $\mu > 0$ such that the estimate
\begin{align}
        |\u(\xi;\sigma)| + |\partial_\xi \u(\xi;\sigma)| \lesssim \int_{\R} \re^{- \mu |\sigma| |\xi-\zeta|}  \langle \zeta \rangle |\g(\zeta)| \de  \zeta \label{e: leading edge resolvent pointwise estimate}
\end{align}
holds for $\sigma \in \Delta^{\mathrm{fr},2}_\delta$ and $\xi \in \R$. Recalling that $\g$ is odd, we obtain
\begin{align*}
  \|\u(\cdot;\sigma)\|_{W^{1,\infty}} \lesssim \|\g\|_{L_{0,1}^1}
\end{align*}
for $\sigma \in \Delta^{\mathrm{fr},2}_\delta$. To upgrade to an $W^{2,\infty}$ estimate in the first component, note that if $\u = (u, v)^\top$, then we also have
    \begin{align}
        u_\xi(\xi; \sigma) = \int_\R \Pi_2 \Tfr_\sigma (\xi-\zeta) \Lambda_1 \begin{pmatrix} g_1 (\zeta) \\ g_2(\zeta) \end{pmatrix} \de  \zeta. \label{e: leading edge resolvent derivative}
    \end{align}
Hence we can apply the same argument to estimate $u_\xi$ in $W^{1,\infty}$, and thereby get the desired $W^{2,\infty}$ control on $u$, which completes the proof of estimate~\eqref{e: right resolvent L101 estimate}.

To prove~\eqref{e: right resolvent Linf -1 estimate}, note that $\Godd(\xi,\zeta) = \Godd(\zeta,\xi)$ and, thus, Lemma~\ref{l: appendix Godd gamma estimate} yields
   \begin{align*}
        |\Godd_\sigma(\xi,\zeta)| \leq C \langle \xi \rangle \re^{-\nufr^1 \Re \sigma |\xi-\zeta|} ,
    \end{align*}
    for $\xi,\zeta \geq 0$ and $\sigma \in \Delta_\delta^{\mathrm{fr},2}$. Then, combining the latter with Lemmas~\ref{l: appendix G c minus G heat estimate} through~\ref{l: appendix G h gamma estimate}, while recalling~\eqref{e: odd heat kernel} and~\eqref{e: sigma bound Omega fr 2} and the fact that $G^\mathrm{odd}_\sigma(-\xi,\zeta) = -G^\mathrm{odd}_\sigma(\xi,\zeta)$ holds for $\xi \in \R$ and $\zeta \geq 0$, we establish
\begin{align*}
        |\u(\xi;\sigma)| + |\partial_\xi \u(\xi;\sigma)| \lesssim \int_{\R}  \langle \xi \rangle |\g(\zeta)| \de  \zeta ,
\end{align*}
for $\sigma \in \Delta_\delta^{\mathrm{fr},2}$ and $\xi \in \R$. We can control $\chi_+(\xi) \langle \xi \rangle$ with the weight $\rho_{0, -1}(\xi)$ in the $L^\infty_{0,-1}$ norm, leading readily to~\eqref{e: right resolvent Linf -1 estimate}, again extending to higher regularity by examining $u_\xi$ through~\eqref{e: leading edge resolvent derivative}. 
\end{proof}

\begin{proof}[Proof of Lemma~\ref{l: leading edge resolvent L2 estimate}]
As above, combining Lemmas~\ref{l: appendix Godd gamma estimate} through~\ref{l: appendix G h gamma estimate} gives the pointwise estimate~\eqref{e: leading edge resolvent pointwise estimate}. Using Young's convolution inequality we readily obtain
\begin{align*}
        \| \u(\cdot,\sigma) \|_{L^2} \lesssim \| \g \|_{L^1_{1,1}} \| \re^{- \mu |\sigma| \cdot} \|_{L^2} \lesssim \frac{1}{|\sigma|^{\frac12}} \| \g \|_{L^1_{0,1}},
\end{align*}
for $\sigma \in \Delta_\delta^{\mathrm{fr},2}$. 
\end{proof}

\section{Proof of pointwise estimates}\label{s: pointwise estimates}
In this appendix, we give a detailed exposition of a proof of Proposition~\ref{p: pointwise estimate}. We emphasize that the proof we give to a large extent rephrases arguments from~\cite[Section 8]{ZumbrunHoward}. Note that we only have to consider $\xi - \zeta \leq 0$ due to the presence of the term $\chi_-(\xi-\zeta)$ on the left-hand side of~\eqref{e: ptwise estimate}. Recall from Proposition~\ref{prop: wake spectral curve expansion} that the critical Floquet exponent $\nuwt(\lambda) = \nu(\lambda)$ occurring in~\eqref{e: ptwise estimate} is analytic in $\lambda$ on $B(0,\delta)$, has positive real part for $\lambda \in B(0,\delta)$ lying to the right of $\Sigma(\lwt)$, and admits the expansion
\begin{align} \nu(\lambda) = \nu_1 \lambda - \nu_2 \lambda^2 + \mathrm{O}(\lambda^3), \label{e: critical Floquet expansion}
\end{align}
with coefficients $\nu_{1,2} > 0$, where we have abbreviated $\nu(\lambda) = \nuwt(\lambda)$ and $\nu_{1,2} = \nu^{1,2}_\mathrm{wt}$ for simplicity of notation. We further assume that $t > 1$, since estimates for small times may be readily obtained using any fixed contour in $B(0,\delta)$, which lies to the right of $\Sigma(\lwt)$ and connects $\lambda_0^*$ to $\lambda_0$. Indeed, for $t \in [0,1]$, $\lambda$ in such a contour, and $\xi \leq \zeta$ we have that $\Re(\lambda t + \nu(\lambda)(\xi-\zeta))$ is uniformly bounded, which readily leads to the estimate~\eqref{e: ptwise estimate} for $t \in [0,1]$.
    
\noindent \textbf{Choice of pointwise contour.} The basic idea is to choose a contour which approximately minimizes the exponent $\Re (\lambda t + \nu(\lambda) (\xi-\zeta))$ in~\eqref{e: ptwise estimate}. We start by computing the minimum of the quadratic approximation $\Re [ \lambda t  + (\nu_1 \lambda-\nu_2 \lambda^2) (\xi-\zeta)]$, restricting to $\lambda \in \R$. From a simple calculation, we find that the minimum is attained at $\lambda = \lambda_{\mathrm{min}}(\xi,\zeta,t)$ with
    \begin{align*}
        \lambda_\mathrm{min}(\xi,\zeta,t) = \frac{t + \nu_1(\xi-\zeta)}{2\nu_2(\xi-\zeta)}.
    \end{align*}
The objective is to integrate along a short segment of the line $\Re \lambda = \lambda_\mathrm{min}$, which then approximately minimizes the exponent $\Re(\lambda t + \nu(\lambda) (\xi - \zeta))$.
    
    \begin{figure}
        \centering
        \includegraphics[width=0.75\textwidth]{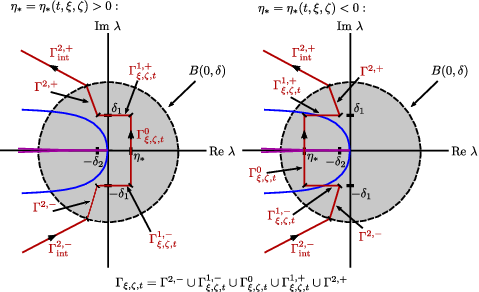}
        \caption{Fredholm borders of $\lfr$ (blue, purple), together with pointwise contours $\Gamma^0_{\xi,\zeta,t}$ (dark red), in the cases where $\eta_*(t,\xi,\zeta) > 0$ (left) and where $\eta_*(t,\xi,\zeta) < 0$ (right). Note that the figure is not to scale, since in particular we require $|\eta_*| + \delta_2 \ll \delta_1^2$.}
        \label{fig: pointwise contours}
    \end{figure}
    
    We first fix $\delta_1 > 0$ small, and then fix $\delta_0, \delta_2 > 0$ such that $\delta_0 + \delta_2 \ll \delta_1^2$, and such that the points $-\delta_2 + \ri \delta_1$ lie in $B(0,\delta)$ to the right of $\Sigma_\mathrm{ess}(\lwt)$. We define 
    \begin{align*}
        \eta_*(t,\xi,\zeta) = \begin{cases}
        \lambda_{\mathrm{min}}, & |\lambda_{\mathrm{min}}| \leq \delta_0, \\
        \delta_0, &\lambda_{\mathrm{min}} > \delta_0, \\
        -\delta_0, &\lambda_{\mathrm{min}} < - \delta_0. 
        \end{cases} 
    \end{align*}
    Let $\Gamma^0_{\xi,\zeta,t}$ be the straight line segment from $\eta_* - \ri \delta_1$ to $\eta_*+\ri\delta_1$. Let $\Gamma^{1,+}_{\xi,\zeta,t}$ be the line segment connecting $\eta_* + \ri \delta_1$ to $-\delta_2 + \ri \delta_1$, and let $\Gamma^{2,+}$ be a fixed curve to the right of $\Sigma_\mathrm{ess}(\lwt)$ which connects $-\delta_2 + \ri \delta_1$ to $\lambda_0$. Similarly, let $\Gamma^{1,-}_{\xi,\zeta,t}$ be the line segment connecting $\eta_* - \ri \delta_1$ to $-\delta_2 - \ri \delta_1$, and let $\Gamma^{2,-}$ be a fixed curve to the right of $\Sigma_\mathrm{ess}(\lwt)$ which connects $-\delta_2 - \ri \delta_1$ to $\lambda_0^*$. Let 
    \begin{align*}
        \Gamma_{\xi,\zeta,t} = \Gamma^{2,-} \cup \Gamma^{1,-}_{\xi,\zeta,t} \cup \Gamma^0_{\xi,\zeta,t} \cup \Gamma^{1,+}_{\xi,\zeta,t} \cup \Gamma^{2, +}_{\xi,\zeta,t}. 
    \end{align*}
    See Figure~\ref{fig: pointwise contours} for schematic of the contours. Note in particular that when $\eta_* < 0$, the contour $\Gamma_{\xi,\zeta,t}$ passes into the essential spectrum of $\lfr$. 
    
    To prove Proposition~\ref{p: pointwise estimate}, we estimate the contribution from each portion of the contour, breaking into the cases $|\lambda_\mathrm{\min}| \leq \delta_0$, and $|\lambda_\mathrm{min}| > \delta_0$. Throughout, we will use the expansion
    \begin{align}
        \Re \nu(\lambda) = \nu_1 \Re \lambda - \nu_2 \left[ (\Re \lambda)^2 - (\Im \lambda)^2 \right] + \mathrm{O}(|\lambda|^3), \label{e: critical Floquet expansion real part} 
    \end{align}
    obtained by taking the real part of~\eqref{e: critical Floquet expansion}.

    \noindent \textbf{Case I: $|\lambda_{\mathrm{min}}| \leq \delta_0$.}
    
    \begin{lemma}\label{l: appendix small eta critical}
     Fix non-negative integers $\ell$ and $j$, and let $g$ satisfy the assumptions of Proposition~\ref{p: pointwise estimate} for some non-negative integer $m$. Suppose $t > 1$ and $\xi, \zeta \in \R$ are such that $\xi \leq \zeta$ and $|\lambda_{\mathrm{min}}(\xi,\zeta,t) | \leq \delta_0$. Then, there exist $\xi$-, $\zeta$- and $t$-independent constants $C, M > 0$ such that
        \begin{align*}
            \int_{\Gamma^0_{\xi,\zeta,t}} |\lambda|^j |\nuwt(\lambda)|^\ell \re^{\Re (\lambda t + \nuwt(\lambda) (\xi-\zeta))} |g(\xi,\zeta,\lambda)| \chi_-(\xi-\zeta) \, |\de \lambda| \leq C\frac{\chi_-(\xi-\zeta)}{(1+t)^{\frac{1}{2} + \frac{\ell + j + m}{2}}} \re^{- \frac{ (\xi-\zeta+\nu_1^{-1} t)^2}{Mt}}.
        \end{align*}
    \end{lemma}
    \begin{proof}
        For $\lambda \in \Gamma^0_{\xi,\zeta,t}$, we have $\Re \lambda = \eta_*$, so by~\eqref{e: critical Floquet expansion real part} we have
        \begin{align*}
            \Re \nu(\lambda) = \nu_1 \eta_* - \nu_2 \eta_*^2 + \nu_2 (\Im \lambda)^2 + \mathrm{O}(|\lambda|^3). 
        \end{align*}
        Noting that $\eta_* = \lambda_{\mathrm{min}}$ in the regime, we find after some algebra
        \begin{align*}
            \eta_* t + \nu_1 \eta_* (\xi-\zeta) - \nu_2 \eta_*^2 (\xi-\zeta) = - \frac{\nu_1^2}{\nu_2} \frac{(\xi-\zeta+\nu_1^{-1} t)^2}{4 |\xi - \zeta|},
        \end{align*}
        and so we have
        \begin{align*}
            \Re (\lambda t + \nuwt (\lambda) (\xi-\zeta)) = - \frac{\nu_1^2}{\nu_2} \frac{(\xi-\zeta+\nu_1^{-1} t)^2}{4 |\xi-\zeta|} + \nu_2 (\Im \lambda)^2 (\xi-\zeta) + \mathrm{O}(|\lambda|^3) (\xi - \zeta). 
        \end{align*}
        The assumption $| \lambda_{\mathrm{min}}| \leq \delta_0$ implies that there exist $t$-, $\xi$- and $\zeta$-independent constants $C_1, C_2 > 0$ such that 
        \begin{align} C_1 (\zeta-\xi) \leq t \leq C_2 (\zeta-\xi),\label{e: regime lmin less than delta0}
        \end{align}
        We obtain
        \begin{align*}
            \Re (\lambda t + \nuwt (\lambda) (\xi-\zeta)) \leq - \frac{(\xi-\zeta+\nu_1^{-1} t)^2}{Mt} - \mu_1 (\Im \lambda)^2 t + \mathrm{O}(|\lambda|^3) (\xi - \zeta)
        \end{align*}
        for some $\lambda$-, $t$-, $\xi$- and $\zeta$-independent constants $M,\mu_1 > 0$. For $\lambda \in \Gamma^0_{\xi,\zeta,t}$ we have $|\lambda|^2 = \eta_*^2 + (\Im \lambda)^2$. There then  exists a $\lambda$-, $t$-, $\xi$- and $\zeta$-independent constant $C_3 > 0$ such that $|\lambda|^3 \leq C_3 [|\eta_*|^3 + |\Im \lambda|^3]$ for $\lambda \in \Gamma^0_{\xi,\zeta,t}$. Hence, we arrive at
        \begin{align*}
            \Re (\lambda t + \nuwt (\lambda) (\xi-\zeta)) \leq - \frac{(\xi-\zeta+\nu_1^{-1} t)^2}{Mt} - \mu_1 (\Im \lambda)^2 t + \mu_2 [|\eta_*|^3 + |\Im \lambda|^3] t
        \end{align*}
        for some $\lambda$-, $t$-, $\xi$- and $\zeta$-independent constant $\mu_2 > 0$. Since $\Im \lambda$ is small, we can absorb $\mu_2 |\Im \lambda|^3 t$ into $- \mu_1 (\Im \lambda)^2 t$ at the cost of changing $\mu_1$. Similarly, note that~\eqref{e: regime lmin less than delta0} yields
        \begin{align}
            |\eta_*|^2 t \leq \frac{C_2^2\nu_1^2}{4\nu_2^2} \frac{(\xi-\zeta + \nu_1^{-1} t)^2}{t}. \label{e: eta squared t estimate}
        \end{align}
        Then, since $|\eta_*| < \delta_0$, we can absorb the factor of $\mu_2 |\eta_*|^3 t$ into the first term by adjusting $M$, provided $\delta_0 > 0$ is sufficiently small. Altogether, we obtain
        \begin{align*}
            \re^{\Re (\lambda t + \nu(\lambda) (\xi-\zeta))} \leq \re^{-\frac{(\xi-\zeta+\nu_1^{-1} t)^2}{Mt}} \re^{-\mu_1 (\Im \lambda)^2 t}
        \end{align*}
        for some  $\lambda$-, $t$-, $\xi$- and $\zeta$-independent constants $M, \mu_1 > 0$. 
        
        For the contour integral over $\Gamma^0_{\xi,\zeta,t}$, we use this together with~\eqref{e: critical Floquet expansion} and the fact that, for any non-negative integer $k$, there exists a $\lambda$-, $t$-, $\xi$- and $\zeta$-independent constant  $C_k > 0$ such that $|\lambda|^k \leq C_k (|\eta_*|^k + |\Im \lambda|^k)$ for $\lambda \in \Gamma^0_{\xi,\zeta,t}$, to find a $t$-, $\xi$- and $\zeta$-independent constant $\tilde{C} > 0$ such that
        \begin{multline}
            \int_{\Gamma^0_{\xi,\zeta,t}} |\lambda|^j |\nuwt(\lambda)|^\ell \re^{\Re (\lambda t + \nuwt(\lambda) (\xi-\zeta))} |g(\xi,\zeta,\lambda)| \chi_-(\xi-\zeta) \, |\de \lambda| \\ \leq \tilde{C} \re^{-\frac{(\xi-\zeta+\nu_1^{-1} t)^2}{Mt}} \int_{\Gamma^0_{\xi,\zeta,t}} (|\eta_*|^{\ell + j + m} + |\Im \lambda|^{\ell + j + m} ) \re^{-\mu_1 (\Im \lambda)^2 t} \, |\de \lambda|. \label{e: ptwise appendix case i gamma 0 estimate}
        \end{multline} 
        Using~\eqref{e: eta squared t estimate} and writing $z = (\xi-\zeta+\nu_1^{-1}t)/\sqrt{t}$, we establish a $t$-, $\xi$- and $\zeta$-independent constant $C > 0$ such that
        \begin{align*}
            \re^{-\frac{(\xi-\zeta+\nu_1^{-1} t)^2}{Mt}} \int_{\Gamma^0_{\xi,\zeta,t}} |\eta_*|^{\ell + j + m}  \re^{-\mu_1 (\Im \lambda)^2 t} \, |\de \lambda| &\leq C \frac{|z|^{\frac{\ell + j + m}{2}}}{t^{\frac{\ell + j + m}{2}}} \re^{-z^2/M} \int_{\Gamma^0_{\xi,\zeta,t}}  \re^{-\mu_1 (\Im \lambda)^2 t} \, |\de \lambda| \\
            &\leq \frac{C}{t^{\frac{\ell + j + m}{2}}} \re^{-z^2/M} \int_{\Gamma^0_{\xi,\zeta,t}}  \re^{-\mu_1 (\Im \lambda)^2 t} |\de \lambda| ,
        \end{align*}
        upon adjusting the values of the constants $C$ and $M$ to absorb the polynomially growing factor in the last inequality. Now, for any non-negative integer $k$, we have the scaling estimate
        \begin{align*}
            \int_{\Gamma^0_{\xi,\zeta,t}} |\Im \lambda|^k \re^{-\mu_1 (\Im \lambda)^2 t} \, | \de \lambda| \leq K \int_{-\infty}^\infty |\xi|^k \re^{-\mu_1 \xi^2 t} \de  \xi \leq \frac{\tilde{C}_k}{t^{\frac{1}{2}+\frac{k}{2}}}. 
        \end{align*}
        for some constants $K, \tilde{C}_k > 0$, which are independent of $\xi$, $\zeta$ and $t$. Using these two estimates in~\eqref{e: ptwise appendix case i gamma 0 estimate}, we obtain the desired result upon noting that $t > 1$. 
    \end{proof}
    
\begin{lemma} 
    Fix non-negative integers $\ell$ and $j$, and let $g$ satisfy the assumptions of Proposition~\ref{p: pointwise estimate} for some non-negative integer $m$. Suppose $t > 1$ and $\xi, \zeta \in \R$ are such that $\xi \leq \zeta$ and $|\lambda_{\mathrm{min}}(\xi,\zeta,t) | \leq \delta_0$. Then, there exist $\xi$-, $\zeta$- and $t$-independent constants $C, \mu_2 > 0$ such that
    \begin{align*}
        \int_{\Gamma^{1,\pm}_{\xi,\zeta,t}} |\lambda|^j |\nuwt(\lambda)|^\ell \re^{\Re (\lambda t + \nuwt(\lambda) (\xi-\zeta))} |g(\xi,\zeta,\lambda)| \chi_-(\xi-\zeta) \, |\de \lambda| \leq C \re^{-\mu_2 t} \re^{-\mu_2 |\xi-\zeta|} \chi_- (\xi - \zeta).
    \end{align*} 
\end{lemma}
\begin{proof}
    For $\lambda \in \Gamma^{1,\pm}_{\xi,\zeta,t}$, we have $|\Im \lambda| = \delta_1$, and $|\Re \lambda| \leq \max (\delta_2, \delta_0)$. Also, the assumption $|\lambda_{\mathrm{min}}| < \delta_0$ implies~\eqref{e: regime lmin less than delta0}. If we assume that $\delta_2$ and $\delta_0$ are sufficiently small relative to $\delta_1$, we can, upon recalling~\eqref{e: critical Floquet expansion real part}, control all terms in
    \begin{align*}
        \Re (\lambda t + \nu(\lambda)(\xi-\zeta)) = \Re \lambda t + \nu_1 \Re \lambda (\xi-\zeta) - \nu_2 (\Re \lambda)^2 (\xi-\zeta) + \nu_2 (\Im \lambda)^2 (\xi-\zeta) + \mathrm{O}(|\lambda|^3) (\xi - \zeta)
    \end{align*}
    by $-\mu_2 (\Im \lambda)^2 |\xi-\zeta| - \mu_2 (\Im \lambda)^2 t$, from which the desired result readily follows. 
\end{proof}

\noindent \textbf{Case II:} $|\lambda_{\mathrm{min}}| > \delta_0$. 

\begin{lemma}
 Fix non-negative integers $\ell$ and $j$, and let $g$ satisfy the assumptions of Proposition~\ref{p: pointwise estimate} for some non-negative integer $m$. Suppose $t > 1$ and $\xi, \zeta \in \R$ are such that $\xi \leq \zeta$ and $|\lambda_{\mathrm{min}}(\xi,\zeta,t)| > \delta_0$. Then, there exist $\xi$-, $\zeta$- and $t$-independent constants $C, \mu > 0$ such that
    \begin{align}
        \int_{\Gamma^0_{\xi,\zeta,t}} |\lambda|^j |\nuwt(\lambda)|^\ell \re^{\Re (\lambda t + \nuwt(\lambda) (\xi-\zeta))} |g(\xi,\zeta,\lambda)| \chi_-(\xi-\zeta) \, |\de \lambda| \leq C \re^{-\mu t} \re^{-\mu |\xi-\zeta|} \chi_-(\xi-\zeta). \label{e: appendix ptwise large eta estimate}
    \end{align}
\end{lemma}
\begin{proof}
    First, we assume $\lambda_\mathrm{min} > \delta_0$, which implies
    \begin{align}
        t < (\zeta-\xi) (\nu_1 - 2 \delta_0 \nu_2). \label{e: appendix ptwise large eta 1}
    \end{align}
    For $\lambda \in \Gamma^0_{\xi,\zeta,t}$ in this regime, we have $\Re \lambda = \delta_0$. Combining these two facts with~\eqref{e: critical Floquet expansion real part}, we find
    \begin{align*}
        \Re (\lambda t + \nu(\lambda)(\xi-\zeta)) &= \delta_0 t + \nu_1 \delta_0 (\xi-\zeta) - \nu_2 \delta_0^2 (\xi-\zeta) + \nu_2 (\Im \lambda)^2 (\xi-\zeta) + \mathrm{O}(|\lambda|^3) (\xi - \zeta) \\
        &\leq \delta_0 (\zeta-\xi) [\nu_1 - 2 \delta_0 \nu_2 - \nu_1 + \nu_2 \delta_0] - \nu_2 (\Im \lambda)^2 (\zeta-\xi) + \mathrm{O}(|\lambda|^3 |\xi-\zeta|) \\
        &= - \nu_2 \delta_0^2 |\zeta-\xi| - \nu_2 (\Im \lambda)^2 |\zeta-\xi| + \mathrm{O}(|\lambda|^3 |\xi-\zeta|). 
    \end{align*}
    Since there exists a $\lambda$-, $\xi$-, $\zeta$- and $t$-independent constant $C_3 >0$ such that $|\lambda|^3 \leq C_3 (|\delta_0|^3 + |\Im \lambda|^3)$ for $\lambda \in \Gamma^0_{\xi,\zeta,t}$, we can absorb the $\mathrm{O}(|\lambda|^3)$ factor into the other two, obtaining
    \begin{align*}
        \Re (\lambda t + \nu(\lambda) (\xi-\zeta)) \leq - \mu_1 |\zeta-\xi| - \mu_2 (\Im \lambda)^2 |\zeta-\xi| \leq - \mu_1 |\zeta-\xi|,
    \end{align*}
    for some $\lambda$-, $\xi$-, $\zeta$- and $t$-independent constants $\mu_1, \mu_2 > 0$. Note that~\eqref{e: appendix ptwise large eta 1} implies that there exists a $\xi$-, $\zeta$- and $t$-independent constant $\tilde{C} > 0$ such that $-|\zeta-\xi| \leq -\tilde{C}t$, so that we may estimate
    \begin{align*}
        \Re (\lambda t + \nu(\lambda) (\xi-\zeta)) \leq -\frac{\mu_1}{2} |\zeta-\xi| - \frac{\mu_1}{2} |\zeta-\xi| \leq -\frac{\mu_1}{2} |\zeta-\xi| - \tilde{C} \frac{\mu_1}{2} t,
    \end{align*}
    from which we conclude the desired estimate~\eqref{e: appendix ptwise large eta estimate} in the case $\lambda_{\mathrm{min}} > \delta_0$. 
    
    The argument for the case $\lambda_{\mathrm{min}} < -\delta_0$ is similar. Specifically, the restriction~\eqref{e: appendix ptwise large eta 1} is replaced by 
    \begin{align}
        (\zeta-\xi) < \frac{t}{\nu_1 + 2 \delta_0 \nu_2}, \label{e: appendix ptwise large eta 2}
    \end{align}
    which leads for $\lambda \in \Gamma_{\xi,\zeta,t}^0$ to an estimate
    \begin{align*}
        \Re (\lambda t + \nu(\lambda) (\xi-\zeta)) &= -\delta_0 t - \nu_1 \delta_0 (\xi-\zeta) - \nu_2 \delta_0^2 (\xi-\zeta) + \nu_2 (\Im \lambda)^2 (\xi-\zeta) + \mathrm{O}(|\lambda|^3) (\xi - \zeta) \\
        &\leq \delta_0 t \left( -1 + \frac{\nu_1 + \delta_0 \nu_2}{\nu_1 + 2 \delta_0 \nu_2} \right) - \nu_2 (\Im \lambda)^2 |\zeta-\xi| + \mathrm{O}(|\lambda|^3 |\xi-\zeta|) \\
        &\leq - M \delta_0^2 t - \nu_2 (\Im \lambda)^2 |\zeta-\xi| + \mathrm{O}(|\lambda|^3) |\xi-\zeta|,
    \end{align*}
    for some $\lambda$-, $\xi$-, $\zeta$-, and $t$-independent constant $M > 0$. Since now in this regime there exists a $\xi$-, $\zeta$- and $t$-independent constant $\tilde{C} > 0$ such that we have $-t \leq -\tilde{C} |\zeta-\xi|$ and we have $|\lambda|^3 \leq C_3 (|\delta_0|^3 + |\Im \lambda|^3)$ for $\lambda \in \Gamma^0_{\xi,\zeta,t}$, we can extract a factor of $-\mu_3 |\xi-\zeta|$ from the first term, and use this term together with the $-\nu_2 (\Im \lambda)^2 |\zeta-\xi|$ term to control the $\mathrm{O}(|\lambda|^3)$ term, and thereby obtain
    \begin{align*}
        \Re (\lambda t + \nu(\lambda) (\xi-\zeta)) \leq -\mu_2 t - \mu_3 |\xi-\zeta|,
    \end{align*}
    for some $\lambda$-, $\xi$-, $\zeta$- and $t$-independent constants $\mu_2, \mu_3 > 0$, from which the desired estimate~\eqref{e: appendix ptwise large eta estimate} follows. 
\end{proof}

\begin{lemma} 
 Fix non-negative integers $\ell$ and $j$, and let $g$ satisfy the assumptions of Proposition~\ref{p: pointwise estimate} for some non-negative integer $m$. Suppose $t > 1$ and $\xi, \zeta \in \R$ are such that $\xi \leq \zeta$ and $|\lambda_{\mathrm{min}}(\xi,\zeta,t)| > \delta_0$. Then, there exist $\xi$-, $\zeta$- and $t$-independent constants $C, \mu > 0$ such that
    \begin{align*}
         \int_{\Gamma^{1, \pm}_{\xi,\zeta,t}} |\lambda|^j |\nuwt(\lambda)|^\ell \re^{\Re (\lambda t + \nuwt(\lambda) (\xi-\zeta))} |g(\xi,\zeta,\lambda)| \chi_-(\xi-\zeta) \, |\de \lambda| \leq C \re^{-\mu t} \re^{-\mu |\xi-\zeta|} \chi_-(\xi-\zeta) .
    \end{align*}
\end{lemma}
\begin{proof}
    First, assume $\lambda_{\mathrm{min}} > \delta_0$, so that~\eqref{e: appendix ptwise large eta 1} holds. For $\lambda \in \Gamma^{1,\pm}_{\xi,\zeta,t}$, we have $|\Im \lambda|^2 = \delta_1^2$. Using~\eqref{e: critical Floquet expansion real part} we arrive at
    \begin{align*}
        \Re (\lambda t + \nu(\lambda) (\xi-\zeta)) = \Re \lambda t + \nu_1 \Re \lambda(\xi-\zeta) - \nu_2 (\Re \lambda)^2 (\xi-\zeta) - \nu_2 \delta_1^2 |\xi-\zeta| + \mathrm{O}(|\lambda|^3) |\xi-\zeta|
    \end{align*}
    for $\lambda \in \Gamma^{1,\pm}_{\xi,\zeta,t}$. If $\Re \lambda > 0$, then the function $t \mapsto \Re \lambda t$ is increasing, while if $\Re \lambda \leq 0$, then $\Re \lambda t \leq 0$. Combining this with~\eqref{e: appendix ptwise large eta 1}, we obtain
    \begin{multline*}
        \Re (\lambda t + \nu(\lambda) (\xi-\zeta)) \leq \max\{0, \Re \lambda\} |\xi-\zeta| (\nu_1 - 2 \delta_0 \nu_2) - \nu_1 \Re \lambda |\xi-\zeta| + \nu_2 (\Re \lambda)^2 |\xi-\zeta| - \nu_2 \delta_1^2 |\xi-\zeta| \\ + \mathrm{O}(|\lambda|^3) |\xi-\zeta|
    \end{multline*}
    Since $|\Re \lambda| \leq \max\{\delta_0, \delta_2\}$, and we are assuming $\delta_0 + \delta_2 \ll \delta_1^2$, we conclude
    \begin{align*}
        \Re  (\lambda t + \nu(\lambda) (\xi-\zeta)) \leq - C \delta_1^2 |\xi-\zeta| \leq  -\mu |\xi-\zeta| - \mu t,
    \end{align*}
    for some $\lambda$-, $\xi$-, $\zeta$- and $t$-independent constants $C, \mu > 0$, using~\eqref{e: appendix ptwise large eta 1} to convert some of the $-|\xi-\zeta|$ localization into $-t$ decay. 
    
    Now assume $\lambda_{\mathrm{min}} < - \delta_0$, so that instead~\eqref{e: appendix ptwise large eta 2} holds. Note that in this case, the contours $\Gamma^{1,\pm}_{\xi,\zeta,t}$ are contained strictly in the left half-plane. Since also $|\Re \lambda| \ll |\Im \lambda|^2 = \delta_1^2$ for $\lambda \in \Gamma^{1,\pm}_{\xi,\zeta,t}$, we readily obtain
    \begin{align*}
        \Re (\lambda t + \nu(\lambda) (\xi-\zeta)) \leq \Re \lambda t - C \delta_1^2 |\xi-\zeta| \leq - \mu t - \mu |\xi-\zeta|,
    \end{align*}
     for some $\lambda$-, $\xi$-, $\zeta$- and $t$-independent constants $C, \mu > 0$, which leads to the desired estimate. 
\end{proof}

Having dealt with the contours $\Gamma_{\xi,\zeta,t}^0$ and $\Gamma_{\xi,\zeta,t}^{1,\pm}$ in both the case $|\lambda_\mathrm{min}| \leq \delta_0$ and $|\lambda_0| > \delta_0$, all that remains is to estimate the integral along the contours $\Gamma^{2,\pm}$, where it is not necessary to distinguish between these cases.

\begin{lemma} \label{l: appendix gamma 2}
    Fix non-negative integers $\ell$ and $j$, and let $g$ satisfy the assumptions of Proposition~\ref{p: pointwise estimate} for some non-negative integer $m$. There exist constants $C, \mu_3 > 0$ such that
    \begin{align*}
        \int_{\Gamma^{2,\pm}} |\lambda|^j |\nuwt(\lambda)|^\ell \re^{\Re (\lambda t + \nuwt(\lambda) (\xi-\zeta))} |g(\xi,\zeta,\lambda)| \chi_-(\xi-\zeta) \, |\de \lambda| \leq C \re^{-\mu_3 t} \re^{-\mu_3 |\xi-\zeta|} \chi_- (\xi - \zeta),
    \end{align*}
    for all $t > 1$ and $\xi,\zeta \in \R$.
\end{lemma}
\begin{proof}
    The compact contours $\Gamma^{2, \pm}$ lie in the left half-plane but strictly to the right of $\Sigma_\mathrm{ess}(\lwt)$, so $\re^{\Re \lambda t}$ contributes uniform exponential decay in time, while $\re^{\Re \nu(\lambda) (\xi-\zeta)}\chi_-(\xi-\zeta)$ contributes uniform spatial localization by Proposition~\ref{prop: wake spectral curve expansion}, leading to the desired estimate. 
\end{proof}

With the obtained control from Lemmas~\ref{l: appendix small eta critical} through~\ref{l: appendix gamma 2} on the integral on the left-hand side of~\eqref{e: ptwise estimate} along all the different portions of the contour $\Gamma_{\xi,\zeta,t}$ we are finally able to prove Proposition~\ref{p: pointwise estimate}.

\begin{proof}[Proof of Proposition~\ref{p: pointwise estimate}]
    Combining Lemmas~\ref{l: appendix small eta critical} through~\ref{l: appendix gamma 2}, we obtain constants $C, \mu, M > 0$ such that the estimate
    \begin{multline*}
        \int_{\Gamma_{\xi,\zeta,t}} |\lambda|^j |\nuwt(\lambda)|^\ell \re^{\Re (\lambda t + \nuwt(\lambda) (\xi-\zeta))} |g(\xi,\zeta,\lambda)| \chi_- (\xi-\zeta) \, |\de \lambda| \\ \leq C \chi_-(\xi-\zeta) \left(\frac{1}{(1+t)^{\frac{1}{2} + \frac{\ell + j + m}{2}}} \re^{-\frac{(\xi-\zeta+\nu_1^{-1} t)^2}{Mt}} + \re^{-\mu t} \re^{-\mu |\xi-\zeta|} \right)
    \end{multline*}
    holds for all $t > 1$ and $\xi,\zeta \in \R$. By a basic scaling argument, one readily finds that the right hand side satisfies the desired $L^p$ estimates of Proposition~\ref{p: pointwise estimate}. 
\end{proof}
 
\section{Derivation of the equation for the inverse-modulated perturbation}\label{app: modulated sys derivation}

Following the strategy in~\cite[Section~4.1]{JONZ}, we derive an equation for the inverse-modulated perturbation $\v(t)$ given by~\eqref{e:modpert}. First, inserting the solution $\ufr$ in~\eqref{e: fhn comoving} and taking spatial derivatives yields
\begin{align}
0 = D \ufr''' + c \ufr'' + F'(\ufr) \ufr' = \mathcal{A}_{\rm fr} \left(\ufr'\right). \label{e:pe0} 
\end{align}
Next, we set $\ut(\xi,t) = \u(\xi-\psi(\xi,t),t)$, $\ut_1(\xi,t) = \u_\xi(\xi-\psi(\xi,t),t), \ut_2(\xi,t) = \u_t(\xi-\psi(\xi,t),t)$ and $\ut_{11}(\xi,t) = \u_{\xi\xi}(\xi-\psi(\xi,t),t)$. Using that the perturbed solution $\u(t)$ solves~\eqref{e: fhn comoving}, we obtain
\begin{align}
\begin{split}
\ut_t - D\ut_{\xi\xi} - c \ut_\xi - F(\ut) &= - \ut_1\psi_t + D\left(\ut_1 \psi_{\xi\xi} + \ut_{11} \psi_\xi (2-\psi_\xi)\right) + c \ut_1 \psi_\xi\\
&= -\ut_1\psi_t  + \ut_2\psi_\xi + D\left(\ut_1 \psi_{\xi\xi} + \ut_{{\color{red}1}1} \psi_\xi (1-\psi_\xi)\right) -  F(\ut)\psi_\xi\\ 
&\qquad + \, (D \ut_{11} + c \ut_1 + F(\ut) - \ut_2)\psi_\xi\\
&= - \ut_1\psi_t +  \ut_2\psi_\xi + D(\ut_1 \psi_\xi)_\xi -  F(\ut)\psi_\xi.
\end{split} \label{e:pe1}
\end{align}
On the other hand, recalling that $\ufr$ solves~\eqref{e: fhn comoving} and applying~\eqref{e:pe0}, we compute
\begin{align}
 \begin{split}
(\partial_t - \mathcal{A}_{\rm fr})\left[\ufr' \psi\right] 
&= \ufr'  \psi_t - \psi \mathcal{A}\left[\ufr'\right] - D\left(2 \ufr''  \psi_\xi + \ufr'  \psi_{\xi\xi}\right) - c \ufr'  \psi_\xi\\
&= \ufr'  \psi_t - \left(D \ufr'' + c\ufr' + F(\ufr)\right)\psi_\xi + F(\ufr)\psi_\xi - D \left(\ufr''  \psi_\xi + \ufr'  \psi_{\xi\xi}\right) \\
&= \ufr'  \psi_t - D\left(\ufr'  \psi_\xi\right)_\xi + F(\ufr)\psi_\xi.
\end{split} \label{e:pe2}
\end{align}
Set $\z(t) := \omega^{-1} \v(t) = \ut(t) - \ufr$. We thus establish
\begin{align}
 \begin{split} \label{e:pe3}
 (\partial_t - \mathcal{A}_{\rm fr})\left[\z \psi_\xi\right] = \z_t \psi_\xi + \z \psi_{\xi t} - D\left(\z\psi_\xi\right)_{\xi\xi} - c\left(\z \psi_\xi\right)_\xi - F'(\ufr) \z {\color{red} \psi_\xi}.
 \end{split}
\end{align}
Next, we express the temporal and spatial derivatives of $\z(t)$ as
\begin{align*}
\z_t = \ut_2 - \ut_1 \psi_t, \qquad
\z_\xi = \ut_1(1-\psi_\xi) - \ufr',
\end{align*}
implying
\begin{align*}
\ut_2 - \z_t &= \ut_1 \psi_t = \frac{\left(\z_\xi + \ufr'\right)\psi_t}{1-\psi_\xi}, \qquad
\ut_1 - \ufr' - \z_\xi = \ut_1 \psi_\xi = \frac{\left(\z_\xi + \ufr'\right)\psi_\xi}{1-\psi_\xi}.
\end{align*}

Hence, with the aid of~\eqref{e:pe1},~\eqref{e:pe2} and~\eqref{e:pe3} we obtain
\begin{align*}
\ut_t - D\ut_{\xi\xi} - c \ut_\xi - F(\ut) &= -(\partial_t - \mathcal{A}_{\rm fr})\left[\ufr'  \psi\right] - \left(\ut_1 - \ufr'\right)\psi_t +  \ut_2\psi_\xi + D\left(\left(\ut_1 - \ufr'\right) \psi_\xi\right)_\xi\\ 
&\qquad - \, \left(F(\ut) - F(\ufr)\right)\psi_\xi\\
&= -(\partial_t - \mathcal{A}_{\rm fr})\left[\ufr'  \psi\right] - \z_\xi \psi_t - \ut_1 \psi_\xi \psi_t + \z_t\psi_\xi + \ut_1\psi_\xi\psi_t\\ 
&\qquad + D \left(\frac{\left(\z_\xi + \ufr'  \psi_\xi\right)\psi_\xi}{1-\psi_\xi}\right)_\xi - \left(F(\ufr + \z) - F(\ufr)\right)\psi_\xi\\
&= (\partial_t - \mathcal{A}_{\rm fr})\left[\z\psi_\xi-\ufr'  \psi \right] + \left(\z\left(c\psi_\xi - \psi_t\right) +  D\left(\frac{\left(\z_\xi + \ufr'  \psi_\xi\right)\psi_\xi}{1-\psi_\xi} + \left(\z \psi_\xi\right)_\xi\right)\right)_\xi\\ &\qquad - \, \left(F(\ufr + \z) - F(\ufr) - F'(\ufr) \z\right)\psi_\xi.
    \end{align*}
Therefore, using that $\ufr$ solves~\eqref{e: fhn comoving}, we establish
\begin{align} \label{e:pe4}
\begin{split}
 \left(\partial_t - \mathcal{A}_{\rm fr}\right) \z &= \left(\partial_t - \mathcal{A}_{\rm fr}\right) \left(\ut - \ufr\right)\\ 
 &= \left(\ut_t - D \ut_{\xi\xi} - c \ut_\xi - F(\ut)\right) +  \left(D \ufr'' + c \ufr' + F(\ufr)\right) +  F(\ut) - F(\ufr) - F'(\ufr) \z\\
 &= (\partial_t - \mathcal{A}_{\rm fr})\left[\z\psi_\xi-\ufr'  \psi\right] + \mathcal{Q}(\z,\psi) + \partial_\xi \mathcal{R}(\z,\psi,{\color{blue}\psi_t}),
\end{split}
\end{align}
where the nonlinearities $\mathcal{Q}(\z,\psi)$ and $\mathcal{R}(\z,\psi,{\color{blue}\psi_t})$ are defined in~\eqref{e:defNLQ} and~\eqref{e:defNLR}, respectively. Finally, to establish the equation~\eqref{e:modpertbeq} for the inverse-modulated perturbation, we multiply~\eqref{e:pe4} with $\omega^{-1}$ and recall $\z(t) = \omega \v(t)$. 

\section{Local Well-Posedness of the Phase Modulation} \label{app:local}

We provide a proof for Proposition~\ref{p:psi} yielding local well-posedness of the phase modulation. To this end, we denote by $Y_k(\R)$ the closed subspace
\begin{align*}
Y_k(\R) = \big\{f \in H^k(\R) : f(\xi) = 0 \text{ for } \xi \in [-1,\infty)\big\}
\end{align*}
of $H^k(\R)$. We start by proving the following auxiliary result.

\begin{lemma} \label{L:loc_gamma}
Let $\vt(t)$ and $T_{\max}$ be as in Proposition~\ref{p:local_unmod}. Take a ball $B$ in $Y_3(\R)$ centered at the origin such that for all $\psi \in B$ it holds $\|\psi\|_{Z_2} \leq \frac12$. Then, the map $V \colon B \times [0,T_{\max}) \to H^2(\R) \times H^1(\R)$ given by
\begin{align*}V(\psi,t)[\xi] &= \vt\left(\xi-\psi(\xi),t\right) + \omega(\xi)\left(\u_{\mathrm{fr}}(\xi-\psi(\xi)) - \u_{\mathrm{fr}}(\xi)\right),\end{align*}
is well-defined, continuous, and Lipschitz continuous in $\psi$ (uniformly in $t$ on compact subintervals of $[0,T_{\max})$).
\end{lemma}
\begin{proof}
Let $\psi_{1,2} \in B$. Then, we have $\|\psi_1'\|_{L^\infty} \leq \frac12$. Therefore, the function $h_1 \colon \R \to \R$ given by $h_1(\xi) = \xi - \psi_1(\xi)$ is strictly increasing and invertible. Hence, using the fundamental theorem of calculus and H\"older's inequality, we obtain the estimate
\begin{align} \label{e:Vineq}
\begin{split}
&\int_\R \left|v(\xi - \psi_1(\xi)) - v(\xi - \psi_2(\xi))\right|^2 \de \xi
= \int_\R \left|\int_0^{\psi_2(\xi)-\psi_1(\xi)} v'(\xi - \psi_1(\xi) - y) \de y\right|^2 \de \xi\\
&\qquad\leq \int_{-\|\psi_2-\psi_1\|_{L^\infty}}^{\|\psi_2-\psi_1\|_{L^\infty}}\int_{-\|\psi_2-\psi_1\|_{L^\infty}}^{\|\psi_2-\psi_1\|_{L^\infty}} \int_\R |v'(\xi - \psi_1(\xi) - y) | |v'(\xi - \psi_1(\xi) - z)| \de \xi \de y \de z\\
&\qquad= \int_{-\|\psi_2-\psi_1\|_{L^\infty}}^{\|\psi_2-\psi_1\|_{L^\infty}}\int_{-\|\psi_2-\psi_1\|_{L^\infty}}^{\|\psi_2-\psi_1\|_{L^\infty}} \int_\R \frac{|v'(w - y) | |v'(w - z)|}{1-\psi_1'\left(h_1^{-1}(w)\right)} \de w \de y \de z\\
&\qquad \leq \frac{4\|v'\|_{L^2}^2}{1-\|\psi_1'\|_{L^\infty}} \|\psi_2-\psi_1\|_{L^\infty}^2
\end{split}
\end{align}
for $v \in C_c^\infty(\R)$. By density of test functions, the latter inequality holds for all $v \in H^1(\R)$. Using the mean value theorem, the continuous embedding $H^1(\R) \hookrightarrow L^\infty(\R)$, and the estimate~\eqref{e:Vineq}, we arrive at
\begin{align}\label{e:estV1}
\begin{split}
\left\|V(\psi_1,t) - V(\psi_2,t)\right\|_{H^2 \times H^1} \lesssim \left(\|\vt(t)\|_{H^3 \times H^2} + \left\|\ufr'\right\|_{W^{2,\infty}}\right) \|\psi_1 - \psi_2\|_{H^2}
\end{split}
\end{align}
for $t \in [0,T_{\max})$ and $\psi_{1,2} \in B$. Hence, using Proposition~\ref{p:local_unmod}, we conclude that $V$ is Lipschitz continuous in $\psi$ (uniformly in $t$ on compact subintervals of $[0,T_{\max})$). Moreover, setting $\psi_2 = 0$ in~\eqref{e:estV1} and noting $V(0,t) = \vt(t) \in H^2(\R) \times H^1(\R)$ by Proposition~\ref{p:local_unmod}, shows that $V$ is well-defined. Finally, using~\eqref{e:Vineq},~\eqref{e:estV1} and the continuous embedding $H^1(\R) \hookrightarrow L^\infty(\R)$, we obtain
\begin{align*}
\|V(\psi_1,t) - V(\psi_2,s)\|_{H^2 \times H^1} &\leq \|V(\psi_2,t) - V(\psi_2,s) - \left(V(0,t) - V(0,s)\right)\|_{H^2 \times H^1}\\ 
&\qquad +  \, \left\|V(0,t) - V(0,s)\right\|_{H^2 \times H^1} + \left\|V(\psi_1,t) - V(\psi_2,t)\right)\|_{H^2 \times H^1}\\
&\lesssim \|\vt(t) - \vt(s)\|_{H^3 \times H^2} + \left(\|\vt(t)\|_{H^3 \times H^2} + \left\|\ufr'\right\|_{W^{2,\infty}}\right) \|\psi_1 - \psi_2\|_{H^2}
\end{align*}
for $\psi_{1,2} \in B$ and $s,t \in [0,T_{\max})$. Continuity of $V$ now follows by combining the latter inequality with Proposition~\ref{p:local_unmod}.
\end{proof}

We are now in the position to establish local well-posedness of the phase modulation.

\begin{prop} 
Let $\vt(t)$ and $T_{\max}$ be as in Proposition~\ref{p:local_unmod}. Then, there exist a constant $r_0 > 0$ and a maximal time $\tau_{\max} \in (0,T_{\max}]$ such that the integral system
\begin{align} \label{e:intsyspsi}
\begin{split}
\psi(t) &= s_p(t)\v_0 + \int_0^t s_p(t-s) \mathcal{N}(V(\psi(s),s),\psi(s),\psi_t(s))\de s, \\
\psi_t(t) &= \partial_t s_p(t)\v_0 + \int_0^t \partial_t s_p(t-s) \mathcal{N}(V(\psi(s),s),\psi(s),\psi_t(s))\de s, 
\end{split}
\end{align}
admits a unique maximally defined solution
\begin{align} (\psi,\psi_t) \in C\big([0,\tau_{\max}),Y_4(\R) \times Y_2(\R)\big) \label{e:maxpsi}\end{align}
satisfying $\|(\psi(t),\psi_t(t))\|_{H^4 \times H^2} < r_0$ and $\|\psi(t)\|_{Z_2} \leq \frac12$ for all $t \in [0,\tau_{\max})$, where $V$ is as in Lemma~\ref{L:loc_gamma}. Moreover, $\tau_{\max} < T_{\max}$ yields
\begin{align} \label{e:psiblowup2}
\limsup_{t \nearrow \tau_{\max}} \left\|\left(\psi(t),\psi_t(t)\right)\right\|_{H^4 \times H^2} = r_0.
\end{align}
In addition, we have $\psi(t) = 0$ for all $t \in [0,\tau_{\max})$ with $t \leq 1$. Finally, it holds $\psi \in C\big([0,\tau_{\max}),Y_{4+m}(\R)\big) \cap C^{1+j}\big([0,\tau_{\max}),Y_{2+l}(\R)\big)$ for any $j,l,m \in \mathbb N_0$ with $\partial_t \psi(t) = \psi_t(t)$ for $t \in [0,\tau_{\max})$.
\end{prop}
\begin{proof} The proof follows from a standard contraction mapping argument applied to the integral system~\eqref{e:intsyspsi}. We collect the relevant details. First, we observe that the identities~\eqref{e: sp gamma} and~\eqref{e: sp t def} and the estimates in Theorem~\ref{t: linear estimates} imply that the propagators $s_p(t) \colon L_{0,1}^1(\R) \to Y_4(\R)$ and $\partial_t s_p(t) \colon L_{0,1}^1(\R) \to Y_2(\R)$ are well-defined and $t$-uniformly bounded. Second, we note that all terms in the nonlinearity $\mathcal{N}(\v,\psi,\psi_t)$, which are nonlinear in $\v$ have the prefactor $\frac1\omega$, which is exponentially localized on $[0,\infty)$. Third, the continuous embedding $H^1(\R) \hookrightarrow L^\infty(\R)$ yields a constant $r_0 > 0$ such that if $\psi \in H^3(\R)$ satisfies $\|\psi\|_{H^3} \leq r_0$, then we have $\|\psi\|_{Z_2} \leq \frac12$. Fourth, all terms in $\mathcal{N}(\v,\psi,\psi_t)$ which contain a factor $\psi$ or $\psi_t$ vanish on $[-1,\infty)$ for $(\psi,\psi_t) \in Y_4(\R) \times Y_2(\R)$. Hence, letting $B$ be the ball in $Y_4(\R) \times Y_2(\R)$ of radius $r_0$ centered at the origin and using Lemma~\ref{L:loc_gamma}, the continuous embedding $H^1(\R) \hookrightarrow L^\infty(\R)$ and the inequality $\|wv\|_{L^1} \leq \|w\|_{L^2}\|v\|_{L^2}$ for $w, v \in L^2(\R)$, we conclude that the nonlinear map $\check{\mathcal{N}} \colon B \times [0,T_{\max}) \to L_{0,1}^1(\R)$ given by
\begin{align*}
\check{\mathcal{N}}(\psi,\psi_t,t) = \mathcal{N}\left(V(\psi,t),\psi,\psi_t\right)
\end{align*}
is continuous in $t$ and Lipschitz continuous in $(\psi,\psi_t)$, uniformly in $t$ on compact subintervals of $[0,T_{\max})$. 

Using these observations, one readily obtains with standard arguments, see e.g.~\cite[Proposition 4.3.3]{CA98} or~\cite[Theorem 6.1.4]{Pazy}, that the right-hand side of the integral system~\eqref{e:intsyspsi} induces for any $\theta \in (0,r_0)$ a contraction mapping on a closed ball $\mathcal{B}_\theta$ of radius $r_0 - \theta$ centered at the origin in $C\big([0,\tau],Y_4(\R) \times Y_2(\R)\big)$ for some $\tau > 0$. By the Banach fixed point theorem this yields a unique solution $(\psi,\psi_t) \in \mathcal{B}_\theta$ of~\eqref{e:intsyspsi}. Letting $\tau_{\max} \in (0,T_{\max}]$ be the supremum of all such $\tau$, we obtain the maximally defined solution~\eqref{e:maxpsi}. 

The condition~\eqref{e:psiblowup2} can be derived arguing by contradiction, that is,~one assumes that $\tau_{\max} < T_{\max}$ and~\eqref{e:psiblowup2} does not hold. Using a standard procedure, see~\cite[Proposition~B.2]{HJPR21}, one then finds a constant $\delta > 0$ such that the maximally defined solution can be extended to a larger interval $[0,\tau_{\max} + \delta]$, contradicting the maximality of $\tau_{\max}$. 

By Theorem~\ref{t: linear estimates} we can differentiate the integral equation for $\psi(t)$ with respect to $t$. Using $s_p(0) = 0$, we obtain $\partial_t \psi(t) = \psi_t(t)$ for $t \in [0,\tau_{\max})$. Finally, the stated regularity properties of $\psi$, as well as the fact that $\psi(t) = 0$ for all $t \in [0,\tau_{\max})$ with $t \leq 1$, follow as the propagators $\partial_t^\ell s_p(t) \colon L_{0,1}^1(\R) \to Y_i(\R)$ are $t$-uniformly bounded for any $i,\ell \in \mathbb N_0$ and vanish for $t \in [0,1]$ by Theorem~\ref{t: linear estimates} and identities~\eqref{e: sp gamma} and~\eqref{e: sp t def}. This completes the proof.
\end{proof}

\bibliographystyle{abbrv}
\bibliography{references}
\end{document}